\documentclass[twoside,11pt]{article}

\usepackage{blindtext}

\usepackage[]{jmlr2e}
\makeatletter
\@ifundefined{assumption}{
	\newtheorem{assumption}[theorem]{Assumption}
}{}
\makeatother

\usepackage[english]{babel}
\usepackage{mathtools}
\usepackage{stmaryrd}
\usepackage{fancyhdr}
\usepackage{tikz}
\usetikzlibrary{er,positioning,bayesnet}
\usepackage[ruled,vlined]{algorithm2e}
\usepackage{nicefrac}

\usepackage{cleveref}
\usepackage{epstopdf}
\usepackage{mathrsfs}  
\usepackage{subcaption}
\usepackage{layouts}
\allowdisplaybreaks

\usepackage{mathtools}

\newcommand{\A}{\mathcal{A}}
\newcommand{\C}{\mathcal{C}}
\newcommand{\N}{\mathbb{N}}
\newcommand{\E}{\mathbb{E}}
\newcommand{\cL}{\mathcal{L}}

\newcommand{\cP}{\mathcal{P}}
\newcommand{\Xbar}{\overline{X}}

\newcommand{\Vbar}{\overline{V}}

\newcommand{\Bbar}{\overline{B}}

\newcommand{\Zbar}{\overline{Z}}

\newcommand{\R}{\mathbb{R}}
\newcommand{\vf}{\Phi}

\newcommand{\lf}{\overline{G}}
\newcommand{\pot}{\psi}

\newcommand{\Exp}{\mathrm{Exp}}
\newcommand{\trace}{\textnormal{tr}}
\newcommand{\Unifsphere}{\textnormal{Unif}(\mathbb{S}^{d-1})}
\newcommand{\sign}{\textnormal{sign}}

\newcommand{\Unif}{\mathrm{Unif}}
\newcommand{\RDBDR}{\textbf{RDBDR} } 
\newcommand{\BDRDB}{\textbf{BDRDB} }
\newcommand{\DRBRD}{\textbf{DRBRD} }
\newcommand{\DBRBD}{\textbf{DBRBD} }
\newcommand{\prox}{\mathrm{prox}}

\usepackage{booktabs} 
\usepackage{array} 
\usepackage{paralist} 
\usepackage{verbatim} 
\usepackage{fancyvrb}
\usepackage{float}
\usepackage{bbm} 

\usepackage{url}
\usepackage[toc,page]{appendix}
\usepackage{color}

\usepackage[shortlabels]{enumitem} 

\newcommand{\po}{\left(}
\newcommand{\pf}{\right)}

\newcommand{\cco}{\llbracket}
\newcommand{\ccf}{\rrbracket}
\newcommand{\dd}{\text{d}}
\newcommand{\na}{\nabla}
\newcommand{\1}{\mathbbm{1}}

\usepackage{lastpage}
\jmlrheading{26}{2025}{1-\pageref{LastPage}}{1/23; Revised
9/25}{10/25}{23-0036}{Andrea Bertazzi, Paul Dobson, and Pierre Monmarch\'e}

\ShortHeadings{Piecewise deterministic sampling with splitting schemes}{Bertazzi, Dobson, and Monmarch\'e}
\firstpageno{1}

\begin{document}

\title{Piecewise deterministic sampling with splitting schemes}

\author{\name Andrea Bertazzi \email andreabertazzi@duck.com \\
       \addr Centre de mathématiques appliquées \\
       École Polytechnique 
       \AND
       \name Paul Dobson \email p.dobson\_1@hw.ac.uk \\
       \addr Department of Mathematics and Computer Science\\
       Heriot-Watt University and Maxwell Institute for Mathematical Sciences
       \AND
       \name Pierre Monmarch\'e \email 	pierre.monmarche@sorbonne-universite.fr \\
       \addr Laboratoire Jacques-Louis Lions and  Laboratoire de Chimie Théorique\\
       Sorbonne Universit{\'e}
       }
\editor{Anthony Lee}

\maketitle

\begin{abstract}
We introduce Markov chain Monte Carlo (MCMC) algorithms based on numerical approximations of piecewise-deterministic Markov processes obtained with the framework of splitting schemes. We present unadjusted as well as adjusted algorithms, for which the asymptotic bias due to the discretisation error is removed applying a non-reversible Metropolis-Hastings filter. In a general framework we demonstrate that the unadjusted schemes have weak error of second order in the step size, while typically maintaining a computational cost of only one gradient evaluation of the negative log-target function per iteration.
Focusing then on unadjusted schemes based on the Bouncy Particle and Zig-Zag samplers, we provide conditions ensuring geometric ergodicity and consider the expansion of the invariant measure in terms of the step size. We analyse the dependence of the leading term in this expansion on the refreshment rate and on the structure of the splitting scheme, giving a guideline on which structure is best. Finally, we illustrate promising results for our samplers with numerical experiments on a Bayesian imaging inverse problem and a system of interacting particles.
\end{abstract}





\section{Introduction}\label{sec:intro}

Piecewise deterministic Markov processes (PDMPs) are non-diffusive Markov processes combining a deterministic motion and random jumps. They appear in a wide range of modelling problems \citep{Cloez_etal,lemaire_thieullen_thomas_2020,Locherbach} and, over the last decade, have gained considerable interest as Markov Chain Monte Carlo (MCMC) methods \citep{PetersDeWith2012,Monmarche2016,BierkensRoberts,BPS,BPS_Durmus,vanetti2017piecewisedeterministic}.  Their dynamics can be described by their infinitesimal generator, which is of the form
\begin{equation}\label{eq:genPDMP}
\cL f(z) = \langle \vf(z), \nabla_zf(z)\rangle + \lambda(z)\int_E (f(y)-f(z)) Q(z,dy)\,,
\end{equation}
where $E$ is the state space and, in this work, $\Phi$ is a smooth and globally Lipschitz vector field, $\lambda:E\to [0,\infty)$ is a continuous function and $Q$ is a probability kernel. The associated process follows the ordinary differential equation (ODE) $\dot z= \Phi(z)$ and, at rate $\lambda(z)$, jumps to a new position distributed according to $Q(z,\cdot)$. We refer to \citet{Davis1984} and \citet{pdmp_inv_meas} for general considerations on PDMP. We denote the deterministic dynamics by $\varphi_t$, the integral curve of $\Phi$, that is the solution to 
\begin{equation*}
\frac{d}{dt}\varphi_t(z)=\vf(\varphi_t(z)), \quad \varphi_0(z)=z, \quad \text{ for all } t\geq 0, z\in E,
\end{equation*}
which exists since $\vf$ is globally Lipschitz. We also assume that $\varphi_t$ leaves $E$ invariant. For $T\sim \Exp(1)$, the random time of the next jump, $\tau$, is given by 
\begin{equation}\label{eq:switching_time_pdmp}
\tau = \inf \left\{t>0: \int_0^t \lambda(\varphi_s(z))\dd s\geq T\right\}.
\end{equation}

This work addresses the question of the simulation of a PDMP with generator \eqref{eq:genPDMP}. The classical method is to use a Poisson thinning procedure \citep{lewis_shedler_thinning,LeThTh17} to sample the jump times, and then to solve the ODE exactly if possible, or otherwise by a standard numerical scheme. Similar to rejection sampling which requires a good reference measure, an efficient Poisson thinning algorithm requires the knowledge of good bounds for the jump rate $\lambda$ along the trajectory of the ODE. In this work, we focus on the case in which such bounds are not available, or are so crude that thinning would not be numerically efficient. In Section \ref{sec:imaging_experiments} we give a concrete example of the latter situation, showing how in a high dimensional setting the Poisson thinning approach makes the exact simulation of a PDMP prohibitive even when the negative log-target distribution is gradient Lipschitz (see Equation \eqref{eq:bounds_zzs_tv} for more details on the bounds, which in the considered case have efficiency that decreases polynomially with the dimension of the process). In this setting, the random event times have to be approximated even if the ODE can be solved exactly. This question has recently been addressed in \citet{bertazzi2021approximations}, \citet{cotter2020nuzz},  \citet{corbella_automatic} with three different schemes. In this paper we define approximations of PDMPs by taking advantage of the core ideas behind splitting schemes, which are widely used and studied for other dynamical systems such as Hamiltonian or underdamped Langevin processes \citep{Leimkuhler_rational,LeimkuhlerMatthewsStoltz,monmarche_splitting}, but that have not been considered in the context of PDMPs before. Following the principle of splitting a PDMP into its elementary components, we obtain novel MCMC algorithms which, as we shall prove, have a numerical error which is of order 2 in the step-size.
Moreover, it is a flexible framework and thus such schemes can be easily combined with multi-time-step or factorization methods \citep{isokin} or integrated in hybrid PDMP/diffusion schemes \citep{weisman,monmarche_kin_walks}. Note that, by using a numerical approximation, we lose one of the interests of PDMP for MCMC purpose, which is the exact simulation by thinning, while in our case the invariant measure of the scheme will have a deterministic bias with respect to the true target measure. However, we still benefit from the good long-time convergence properties of the ballistic non-reversible process and, contrary to Hamiltonian-based dynamics, it is still possible to factorize the target measure and define efficient schemes in terms of number of computations of forces {while using a single step-size} (see \cite{weisman,monmarche_kin_walks} and Section~\ref{subsec:example_particles}). We shall also show how the correct stationary distribution can be recovered by means of a non-reversible Metropolis-Hastings acceptance/rejection step (see Section \ref{sec:metropolis_adjusted_algorithms}). Moreover, for classical velocity jump processes used in MCMC, since the norm of the velocity is constant (between possible refreshments which are independent of the potential), these schemes are numerically  stable (see the numerical experiments in Section~\ref{sec:numerical} where the step-size of PDMP schemes can be taken larger than for the classical ULA), even for non-globally Lipschitz potentials. 

The core idea of splitting schemes is first to split the generator in several parts such that a process associated to each part can be simulated exactly. For instance, when the ODE can be solved exactly, one can write $\cL=\cL_D+\cL_J$ with
\begin{align*}
    & \cL_D f(z) = \langle \vf(z), \nabla_zf(z)\rangle, \\
    & \cL_J f(z) = \lambda(z)\int_E (f(y)-f(z)) Q(z,dy)\,,
\end{align*}
in which case the process associated to $\cL_D$ is simply the solution of the ODE, hence D stands for drift, while the process associated to $\cL_J$ is a continuous-time Markov chain, for which the jump rate is constant between two jumps (so that the jump times are simple exponential random variables), hence J stands for jumps. Then, one approximates the semigroup of the true process by a Strang splitting
\begin{equation}\label{eq:Strang}
    P_\delta = e^{\delta(\cL_D+\cL_J)} \approx e^{\frac{\delta}{2}\cL_D} e^{\delta\cL_J}e^{\frac{\delta}{2}\cL_D}
\end{equation}
for a small step size $\delta>0$. Therefore, over one time step the approximation follows $\cL_D$ for time $\delta/2$, then $\cL_J$ for time $\delta$ and finally  $\cL_D$ again for time $\delta/2$. Given a step size $\delta$, now we illustrate how the $(n+1)$-th iteration works. Starting at time $t_n=n\delta$ at state $\Zbar_{t_n}$ the process first moves deterministically for a half step:
\begin{align*}
    & \Zbar_{t_n+\delta/2} = \varphi_{\delta/2}(\Zbar_{t_n}).
\end{align*}
Then we simulate the pure jump part of the process: we generate an event time $\tau_1\sim \Exp(\lambda(\Zbar_{t_n+\delta/2}))$ and, if $\tau_1<\delta$, we set $\Zbar_{t_n+\delta/2}\sim Q(\Zbar_{t_n+\delta/2},\cdot)$. Then we repeat this step as long as $\sum_i \tau_i < \delta$, though, since we are interested in second order schemes, it is enough to limit ourselves to two jumps per time step. Note that the rate is updated after every jump and is constant between jumps. 
We conclude the iteration by a final half step of deterministic motion:
\begin{equation*}
    \Zbar_{t_{n+1}} = \varphi_{\delta/2}(\Zbar_{t_n+\delta/2}).
\end{equation*}
We refer to this scheme as the splitting scheme \textbf{DJD}, where consistently with above $\textbf{D}$ stands for drift and $\textbf{J}$ for jumps. When the ODE cannot be solved exactly,  any second-order numerical scheme can be used instead of $\varphi_t$. Moreover, in some cases (typically for the Hamiltonian dynamics) the generator $\cL_D$ can be further divided in several ODEs. Similarly, for computational purpose, it can be interesting in some cases to split the jump part $\cL_J$ in several operators. It is also possible to keep in $\cL_D$ a combination of ODE and jump, simulated e.g. by thinning, while some parts of the jump are treated separately in $\cL_J$ (it could make sense for instance in the context of \cite{weisman}). When there are more than two parts in the splitting of $\cL$, a scheme is obtained by starting from \eqref{eq:Strang} and using e.g. $e^{\delta \cL_J} \approx e^{\frac{\delta}{2}\cL_A} e^{\delta\cL_B}e^{\frac{\delta}{2}\cL_A} $ if $\cL_J=\cL_A+\cL_B$, etc.


Such splitting schemes can be used to simulate any PDMP. For some modelling problems, it can be interesting to have estimates on the trajectorial error between the approximated process and the two process, for instance when dynamical properties (like mean squared displacement or transition rates) are of interest. However, in this work, we have mainly in mind the PDMPs which are used for MCMC methods, in particular our recurrent examples will be the Zig-Zag sampler (ZZS) \citep{BierkensRoberts,ZZ} and the Bouncy Particle sampler (BPS) \citep{PetersDeWith2012,Monmarche2016,BPS}. As a consequence, we will not discuss trajectorial errors but rather focus on what is relevant for MCMC purposes, namely the long-time convergence of the Markov chain (which should scale properly as the step size vanishes) and the numerical bias on the invariant measure and on  empirical averages of the chain.

\subsubsection*{Main contributions of the paper}
The main contributions of this paper are the following:
\begin{itemize}
    \item We introduce a novel approach to approximate PDMPs based on splitting schemes, an idea which had not been previously considered for processes of this type and that, as we prove in Theorem \ref{thm:weakerror}, has the key advantage of giving an approximation of second order at the cost of one gradient evaluation per iteration.
    \item We define an unbiased version of our splitting schemes by introducing a non-reversible Metropolis adjustment based on the skew detailed balance condition, thus giving a way to eliminate the discretisation error. For these adjusted algorithms we characterise the average rejection rate.
    \item We prove geometric convergence of the law of the unadjusted splitting schemes to the unique invariant measure and we carefully characterise the dependence of the rate of convergence on the step size (these are Theorems \ref{thm:ergodicBPS} and \ref{thm:ergodicity_zzs}).
    \item We study the asymptotic bias in the invariant measure of the unadjusted schemes and determine what structure of splitting scheme performs best and is most robust to poor choices of the refreshment rate, an important tuning parameter of our algorithms.
    \item We demonstrate the advantages of our algorithm based on ZZS on sampling problems in Bayesian Imaging and Molecular Dynamics. In particular, in the imaging context our algorithm gives faster uncertainty quantification compared the unadjusted Langevin algorithm thanks to its better stability in the step size. In the molecular dynamics setting, we show how to decompose the pairwise interactions between the $N$ particles to reduce the cost of iterations of our algorithm to $\mathcal{O}(N)$, compared to the $\mathcal{O}(N^2)$ of the Hamiltonian Monte Carlo algorithm.
\end{itemize}

\subsubsection*{Organisation of the paper}
The article is organised as follows. We conclude this introduction by presenting the algorithms we focus on in this paper. In Section~\ref{subsec:examples} we discuss our two main examples and their approximation with splitting schemes. In Sections~\ref{sec:metropolis_adjusted_algorithms} and \ref{subsec:subsamp} we discuss respectively how we can Metropolis-adjust our schemes in a non-reversible fashion and how we can modify the algorithms to do subsampling. We conclude our introduction with Section~\ref{sec:algorithms_boundaries}, where we describe how boundaries can be treated with our splitting schemes.
Section~\ref{sec:CVsplitting} is devoted to the analysis of the weak error for the finite-time empirical averages of the scheme \textbf{DJD}. The main result, Theorem~\ref{thm:weakerror}, states that for this scheme the \emph{weak error is of order} $2$ in the step-size. The geometric ergodicity of splitting schemes based on our main examples is established in Section~\ref{sec:ergodicity}, with a consistent dependency of the estimates on the step-size.  In Section~\ref{sec:expansion_mu}, we provide a formal expansion (in terms of the step-size) of the invariant measure of the schemes depending on the choice of the splitting, in the spirit of \cite{Leimkuhler_rational}, with a particular focus in Section~\ref{sec:onedimtargets_expinvmeas} on three one-dimensional examples where everything can be made explicit. In Section \ref{sec:scaling_rej_prob} we study the average rejection rate of our adjusted schemes, then verifying our theoretical results with numerical simulations on two Gaussian distributions. Numerical experiments for applications in Bayesian Imaging and Molecular dynamics are provided in Section~\ref{sec:numerical}. Finally, technical proofs are gathered in an Appendix.

\subsubsection*{Comparison to related works} 
\textit{Comparison to PDMP based approaches.} The work in this paper can be seen as a continuation of the work that two of the authors started with their coauthors in \citet{bertazzi2021approximations}, in which a general framework to approximate PDMPs is introduced and studied. In this previous work, the focus is not a specific scheme and thus the results are mostly general and not tailored for particular processes or schemes, though the ZZS and BPS are considered as recurrent examples. In particular, the schemes introduced in \citet{bertazzi2021approximations} leave considerable freedom to the user in the choice of some crucial components of the algorithm, namely an approximation of the switching rates or a numerical integrator in place of the exact flow map.
On the other hand, in this paper we follow the philosophy of splitting schemes to describe a simple recipe to approximate PDMPs, an approach that was not considered in \citet{bertazzi2021approximations}. The main advantage of splitting schemes is the second order of accuracy with one gradient evaluation per iteration, whereas second order algorithms considered in \citet{bertazzi2021approximations} relied on approximations of second order of the switching rates, which can be usually obtained with the expensive computation of the Hessian of the negative log-target. Moreover, in this work we describe how to remove the bias introduced by our approximation with a non-reversible Metropolis-Hastings step. 
Two other works \citet{cotter2020nuzz,corbella_automatic} focus on approximate simulation of the Zig-Zag sampler, which is one of our two main examples. In \citet{cotter2020nuzz} the authors suggest to approximate event times by using numerical approximations of the integral of the rates along the dynamics \eqref{eq:switching_time_pdmp}, as well as a root finding algorithm. In \citet{corbella_automatic}, the authors suggest using a numerical optimisation algorithm at each iteration to obtain a suitable bound that enables the use of Poisson thinning. The first difference is that we mainly consider our approximations as discrete time Markov chains, whereas the processes of \citet{cotter2020nuzz} and \citet{corbella_automatic} are interpreted in continuous time, although neither resulting process is a Markov process due to the nature of the numerical algorithms that are used. Naturally, one could interpret our algorithms as continuous time processes, which again would not be Markov processes. Secondly, without assuming any properties that we do not verify, we derive theoretical justifications of our proposed algorithms, such as bounds on the weak error and existence, uniqueness, and geometric convergence to a stationary distribution. Moreover, we introduce Metropolis adjusted algorithms to eliminate the error introduced by the numerical approximations, while this aspect is not studied in previous works.

\textit{Comparison to SDE based approaches.} As far as theoretical results are concerned, notice that over the past few years non-asymptotic efficiency bounds for MCMC algorithms like HMC or Langevin-based methods have been obtained, particularly in high-dimensional settings and for specific families of target measures (e.g. Gaussian, log-concave or mean-field models) see for example \citet{Gouraudetal,Whalley,Bou_Rabee_2020,Camrudetal,cheng2018underdamped,durmus2016sampling} and references within. In this paper, we prove a number of theoretical results for our new algorithms. Our theorems are certainly less quantitative and specialised compared to the aforementioned literature, and this is natural for several reasons. First of all, many results known for HMC and Langevin are not yet established even for continuous-time PDMPs.
For instance, direct Wasserstein coupling methods are very efficient for ordinary or stochastic differential equations, but more delicate to implement when the process involves non-homogeneous Poisson jumps (see \citet[Section 4.3]{elementary}  in this direction with a result for a mean-field ZZS). In particular, for Gaussian targets the HMC and unadjusted Langevin algorithms give Markov chains that are Gaussian, therefore the study boils down to linear algebra and sharp non-asymptotic bounds can be obtained, see for example \citet{Gouraudetal,Whalley}. This is not true for PDMPs. In the recent \citet{Kamatani}, some asymptotic study is provided for badly-conditioned Gaussian targets (in a fixed dimension, focusing mainly on the $2$-dimensional case) for BPS and ZZS, and in \citet{Deligiannidis} the marginals of the BPS with separable target are shown to converge to a randomized HMC process in high dimension. These results are only asymptotic, restricted to very specific targets, and require involved technical proofs. It should be possible to adapt them to splitting schemes of PDMP, but this requires a study on its own. On the other hand, our theoretical results provide the necessary qualitative convergence guarantees for the algorithms,  similar to classical results for  Langevin-based algorithm as \citet{Leimkuhler_rational} and \citet{Talay}. 

\subsection{Main examples}\label{subsec:examples}
Let us now introduce two examples  from the computational statistics literature.  In this setting we have a target probability measure with density $\pi(x)\propto \exp(-\pot(x))$ for $x\in \mathbb{R}^d$.

\begin{algorithm}[t]
\SetAlgoLined
\SetKwInOut{Input}{Input}\SetKwInOut{Output}{Output}
\Input{Number of iterations $N$, initial condition $(x,v)$, step size $\delta$.}
\Output{Chain $(\Xbar_{t_{n}},\Vbar_{t_{n}})_{n=0}^N$.}
 Set $n=0$, $(\Xbar_0,\Vbar_0) = (x,v)$\;
 \While{$n < N$}{
  Set $\Xbar_{t_{n+1}} = \Xbar_{t_n} +\frac{\delta}{2} \Vbar_{t_{n}} $\;
  Set $\Vbar_{t_{n+1}}=\Vbar_{t_{n}}$\;
  \For{$i=1\dots,d$}{
  With probability $(1-\exp(-\delta\lambda_i(\Xbar_{t_{n+1}},\Vbar_{t_{n+1}})))$ set $\Vbar_{t_{n+1}} = R_i \Vbar_{t_{n+1}}$\;
  }
  Set $\Xbar_{t_{n+1}} = \Xbar_{t_{n+1}} + \frac{\delta}{2}\Vbar_{t_{n+1}} $\;
  Set $n=n+1$\;
 }
 \caption{Splitting scheme \textbf{DBD} for ZZS}
 \label{alg:splitting_DBD_ZZS}
\end{algorithm}

\begin{example}[Zig-Zag sampler, \citet{ZZ}]\label{ex:ZZ}
	Let $E=\R^d\times\{+1,-1\}^d$. For any $z\in E$, we write $z=(x,v)$ for $x\in \R^d$, $v\in\{+1,-1\}^d$, where $x$ is interpreted as the position of the particle and $v$ denotes the corresponding velocity. The deterministic motion of ZZS is determined by $\Phi(x,v)=(v,0)^T$, i.e. the particle travels with constant velocity $v$.
    For $i=1,\dots,d$ we define the jump rates $\lambda_i(x,v):=(v_i\partial_i\pot(x))_++\gamma_i(x)$, where $\gamma_i(x)$ can be any non-negative function and is often chosen to be zero. The corresponding (deterministic) jump kernels are given by $Q_i((x,v),(dy,dw))=\delta_{(x,R_iv)}(dy,dw)$, where $\delta_z$ denotes the Dirac delta measure and $R_i$ is the operator that flips the sign of the $i$-th component of the vector it is applied to, that is $R_iv=(v_1\ldots,v_{i-1},-v_i,v_{i+1},\ldots,v_d).$
    Hence the $i$-th component of the velocity is flipped with rate $\lambda_i$.
    The ZZS is described by its generator
	\begin{equation}\label{eq:ZZgen}
	\cL f(x,v) = \langle v, \nabla_x f(x,v)\rangle + \sum_{i=1}^d\lambda_i(x,v)[f(x,R_iv)-f(x,v)]. 
	\end{equation}
	Simulating the event times with rates of this form is in general a very challenging problem. 
	
	We can apply the splitting scheme above as follows. For simplicity we consider the process with \emph{canonical rates}, i.e. $\gamma_i=0$ for all $i$. Then we can split the generator as
	\begin{align*}
	    & \cL_D f(x,v) = \langle v, \nabla_x f(x)\rangle, \\
        & \cL_B f(x,v) = \sum_{i=1}^d\lambda_i(x,v)[f(x,R_iv)-f(x,v)]. 
	\end{align*}
    Here we define the scheme \textbf{DBD}, where \textbf{B} stands for \emph{bounces}.
	Given $(\Xbar_{t_n},\Vbar_{t_n})$, we start by a half step of deterministic motion:
	\begin{equation*}
	    \Xbar_{t_n+\frac{\delta}{2}} = \Xbar_{t_n} + \frac{\delta}{2}\Vbar_{t_n}.
	\end{equation*}
	Then for $i=1,\dots,d$ we draw $\tau_i\stackrel{iid}{\sim}\Exp(\lambda_i(\Xbar_{t_n+\delta/2},\Vbar_{t_n}))$, which are homogeneous exponential random variables. Then let $\tau_{(1)}=\min \tau_i$ and set
	\begin{align*}
	    \Vbar_{t_{n+1}} = \begin{cases}
	    \Vbar_{t_n} \quad & \text{ if } \tau_{(1)}>\delta \\
	    R_I\Vbar_{t_n}  \quad & \text{ if } \tau_{(1)}\leq\delta
	    \end{cases}
	\end{align*}
	where $R_I=\prod_{i\in I}R_i$ and $I$ is the set of indices $i$ for which $\tau_i\leq \delta$, and $R_Iv=v$ when $I$ is the empty set. 
    Observe that for canonical rates flipping the sign of a component does not affect the other switching rates, and thus it is not possible to have two flips in the same component when $\gamma_i=0$. Finally, set
	\begin{equation*}
	    \Xbar_{t_{n+1}} = \Xbar_{t_n+\frac{\delta}{2}} + \frac{\delta}{2}\Vbar_{t_{n+1}},
	\end{equation*}
	which concludes the iteration. The procedure is described in pseudo code form in Algorithm~\ref{alg:splitting_DBD_ZZS}. An interesting feature of the algorithm is that the jump part of the chain can be computed \emph{in parallel}, since in that stage a velocity flip in one component does not affect the other components of the process.
\end{example}

\begin{algorithm}[t]
\SetAlgoLined
\SetKwInOut{Input}{Input}\SetKwInOut{Output}{Output}
\Input{Number of iterations $N$, initial condition $(x,v)$, step size $\delta$.}
\Output{Chain $(\Xbar_{t_{n}},\Vbar_{t_{n}})_{n=0}^N$.}
 Set $n=0$, $(\Xbar_0,\Vbar_0) = (x,v)$\;
 \While{$n < N$}{
  Set $\Vbar_{t_{n+1}} = \Vbar_{t_n}$ \;
  With probability $(1-\exp(-\lambda_r \frac{\delta}{2}))$ draw $\Vbar_{t_{n+1}}\sim \Unifsphere$  \;
  Set $\Xbar_{t_{n+1}} = \Xbar_{t_n} + \frac{\delta}{2} \Vbar_{t_{n+1}} $\;
  With probability $(1-\exp(-\delta\lambda_1(\Xbar_{t_{n+1}},\Vbar_{t_{n+1}})))$ set $\Vbar_{t_{n+1}} = R(\Xbar_{t_{n+1}}) \Vbar_{t_{n+1}}$\;
  Set $\Xbar_{t_{n+1}} = \Xbar_{t_{n+1}} + \frac{\delta}{2}\Vbar_{t_{n+1}} $\;
  With probability $(1-\exp(-\lambda_r \frac{\delta}{2}))$ set $\Vbar_{t_{n+1}}\sim \Unifsphere$ \;
  Set $n=n+1$\;
 }
 \caption{Splitting scheme \RDBDR for BPS}
 \label{alg:splitting_RDBDR_BPS}
\end{algorithm}
\begin{example}[Bouncy Particle Sampler, \citet{BPS}]\label{ex:BPSintro}
	\sloppy Let $E=\R^d\times\R^d$, and for any $z\in E$ we write $z=(x,v)$ for $x\in \R^d$, $v\in\R^d$. The deterministic motion is the same as for ZZS: $\Phi(x,v)=(v,0)^T$. The BPS has two types of random events: reflections and refreshments. These respectively have rates $\lambda_1(x,v) =(v^T\nabla_x\pot(x))_+$ and $\lambda_2(x,v)= \lambda_{r}$ for $\lambda_{r}>0$,  and corresponding jump kernels
	$$Q_1((x,v),(dy,dw)) = \delta_{(x,R(x)v)}(y,w), \quad Q_2((x,v),(dy,dw)) =  \delta_x(dy)\nu(dw),$$ 
	where $\nu$ is a rotation-invariant probability measure on $\R^d$ (typically the standard Gaussian measure or the uniform measure on $\mathbb S^{d-1}$), and  
	$$R(x)v=v-2\frac{\langle v,\nabla_x\pot(x)\rangle}{\lvert\nabla_x\pot(x)\rvert^2} \nabla_x\pot(x).$$
    The operator $R$ \emph{reflects} the velocity $v$ off the hyperplane that is tangent to the contour line of $\pot$ passing though point $x$. Importantly, the norm of the velocity is unchanged by the application of $R$, and this corresponds to an elastic collision of the particle on the hyperplane.
	The BPS has generator 
	\begin{equation}\notag
	\cL f(x,v)\! =\! \langle v, \nabla_x f(x)\rangle + \lambda_1(x,v)[f(x,R(x)v)-f(x,v)]+\lambda_2 \!\int\!\! \big(f(x,w) - f(x,v)\big) \nu(dw).
	\end{equation}
	In this case we split the generator in three parts:
	\begin{align*}
        & \cL_D f(x,v) = \langle v, \nabla_x f(x)\rangle, \\
        & \cL_B f(x,v) = \lambda_1(x,v)[f(x,R(x)v)-f(x,v)], \\
        & \cL_R f(x,v) = \lambda_2 \!\int\!\! \big(f(x,w) - f(x,v)\big) \nu(dw),
    \end{align*}
    We then define the scheme \textbf{RDBDR}, where \textbf{R} stands for \emph{refreshments}. Starting at time $t_n=n\delta$ at state $(\Xbar_{t_n},\Vbar_{t_n})$ we begin by drawing $\tau_1\sim \Exp(\lambda_r)$ and setting
    \begin{align*}
	    \tilde{V}_{t_{n}+\frac{\delta}{2}} = \begin{cases}
	    \Vbar_{t_n} \quad & \text{ if } \tau_{1}>\delta/2 \\
	    W_1  \quad & \text{ if } \tau_{1}\leq\delta/2
	    \end{cases}
	\end{align*}
	for $W_1\sim\nu$. Then the process evolves deterministically for time $\delta/2$:
	\begin{equation*}
	    \Xbar_{t_n+\frac{\delta}{2}} = \Xbar_{t_n} + \frac{\delta}{2} \tilde{V}_{t_n+\frac{\delta}{2}}.
	\end{equation*}
	At this point, we check if a reflection takes place by drawing $\tau_2 \sim \Exp(\lambda_1(\Xbar_{t_n+\frac{\delta}{2}},\tilde{V}_{t_{n}+\frac{\delta}{2}}))$ and set
	\begin{align*}
	    \Vbar_{t_{n}+\frac{\delta}{2}} = \begin{cases}
	    \tilde{V}_{t_{n}+\frac{\delta}{2}} \quad & \text{ if } \tau_{2}>\delta \\
	    R(\Xbar_{t_n+\frac{\delta}{2}}) \tilde{V}_{t_{n}+\frac{\delta}{2}}  \quad & \text{ if } \tau_{2}\leq\delta
	    \end{cases}
	\end{align*}
	Importantly, $\lambda_1(\Xbar_{t_{n}+\frac{\delta}{2}}, \Vbar_{t_{n}+\frac{\delta}{2}}) = 0$  if a reflection takes place and thus at most one reflection can happen. This is a consequence of the fact that $\langle R(x) v,\nabla \pot(x)\rangle = -\langle v,\nabla \pot(x)\rangle$ by definition of the reflection operator.
	After this we set
	$$ \Xbar_{t_{n+1}} = \Xbar_{t_n+\frac{\delta}{2}} + \frac{\delta}{2}, $$
	and finally conclude the iteration drawing $\tau_3\sim \Exp(\lambda_r)$ and letting
	\begin{align*}
	    \tilde{V}_{t_{n+1}} = \begin{cases}
	    \Vbar_{t_n+\frac{\delta}{2}} \quad & \text{ if } \tau_{3}>\delta/2 \\
	    W_2  \quad & \text{ if } \tau_{3}\leq\delta/2
	    \end{cases}
	\end{align*}
	where $W_2\sim \nu$. 
    The pseudo code can be found in Algorithm \ref{alg:splitting_RDBDR_BPS}.
\end{example}

\subsection{Metropolis adjusted algorithms}\label{sec:metropolis_adjusted_algorithms}
Naturally, the use of splitting schemes to approximate a PDMP introduces a discretisation error. In this section we discuss how to eliminate this bias with the addition of a  Metropolis-Hastings (MH) acceptance-rejection step. In Section~\ref{sec:nonrev_MH} we describe the general procedure, which is a \emph{non-reversible} MH algorithm, and then apply this to ZZS and BPS.  Similarly this can be applied to other kinetic PDMPs used in MCMC.

\subsubsection{Non-reversible Metropolis-Hastings}\label{sec:nonrev_MH}
The classical MH algorithm builds a $\mu$ invariant Markov chain $P$ by enforcing \emph{detailed balance} (DB): for all $x,y$ it holds that $\mu(d x) P(x,d y) = \mu(d y) P(y,d x).$ The chain is then said \emph{reversible}.
PDMPs such as BPS and ZZS do not satisfy DB and are said to be \emph{non-reversible}. Since this property can lead to a faster converging process (see e.g. \cite{diaconis_nonreversible}), it is reasonable here to Metropolise our splitting schemes in a non-reversible fashion. Moreover, as we shall see below, for our chains based on splitting schemes of PDMPs it is not possible to use the standard MH framework, as in general the chain cannot go back to the previous state.
For this reason, we rely on a different balance equation known as \emph{skew detailed balance}: considering a chain for which the state can be decomposed as $z=(x,v)$, for all $x,v,y,w$
\begin{equation}\label{eq:skewDB}
    \mu(dx,dv) P((x,v),(dy,dw)) = \mu(dy,dw) P((y,-w),(dx,-dv)).
\end{equation}
Integrating both sides with respect to $x$ and $v$ we can see that $\mu$ is a stationary measure for the chain $P$. This condition is at the basis of the classical non-reversible HMC algorithm of \cite{Horowitz_guidedHMC} and was considered in several works on the lifting approach, as for instance \citet{Turitsyn2011,Vucelja_lifting,Hukushima_2013,MichelKapferKrauth2014}.
More generally, skew detailed balance holds when we compose a reversible kernel with a measure preserving involution (see e.g. \citet{involutiveMCMC} or \citet{Thin_nonreversible}), which in our case is the operator that flips the sign of the velocity vector. 
Here we wish to construct skew-reversible Markov chains $P$ by Metropolising kernels $\mathcal{Q}$ which are unadjusted splitting schemes of BPS and ZZS. Because we only need to adjust the \textbf{DBD} part, it is sufficient to consider kernels $\mathcal{Q}$ of the form $$\mathcal{Q}((x,v),(y,w)) = \sum_{j=1}^n p_j(x,v) \1_{(y,w)=F_j(x,v)}\,,$$ where $p_j$ is the probability of applying operator $F_j$, $\sum_{i=1}^n p_i(x,v)=1$ for all $(x,v) \in E$, and finally $F_j:E\to E$ are volume preserving maps. Note that this is a more general setting than that of the HMC algorithm, which corresponds to the case $n=1$ with $F_1$ being a splitting scheme for the Hamiltonian dynamics. For $\mathcal{Q}$ as above the skew-DB holds as long as the move from $(x,v)$ to $(y,w)=F_j(x,v)$ is accepted with probability
\begin{equation}\label{eq:prob_MH_nonrev}
    \alpha((x,v),(y,w)) =  1\wedge \frac{\mu(y,-w)p_j(y,-w)  }{\mu(x,v) p_j(x,v)}.
\end{equation}
If the proposal $(y,w)$ is rejected, the new state of the chain becomes $(x,-v)$, in which case \eqref{eq:skewDB} is trivially satisfied.

\begin{algorithm}[t]
\SetAlgoLined
\SetKwInOut{Input}{Input}\SetKwInOut{Output}{Output}
\Input{Number of iterations $N$, initial condition $(x,v)$, step size $\delta$.}
\Output{Chain $(\Xbar_{t_{n}},\Vbar_{t_{n}})_{n=0}^N$.}
 Set $n=0$, $(\Xbar_0,\Vbar_0) = (x,v)$\;
 \While{$n < N$}{
  Set $\Xbar_{t_n+\delta/2} = \Xbar_{t_n} + \frac{\delta}{2} \Vbar_{t_{n}}$\;
  Set $\tilde{V}=\Vbar_{t_{n}}$\;
  \For{$i=1\dots,d$}{
  With probability $(1-\exp(-\delta\lambda_i(\Xbar_{t_n+\delta/2},\tilde{V})))$ set $\tilde{V} = R_i \tilde{V}$\;
  }
  Set $\tilde{X}= \Xbar_{t_n+\delta/2} + \frac{\delta}{2} \tilde{V} $\;
  Set $(\Xbar_{t_{n+1}},\Vbar_{t_{n+1}}) = (\tilde{X},\tilde{V})$ with probability $$1\wedge \left(\frac{\pi(\tilde{X})  }{\pi(\Xbar_{t_n})} \, \exp \Bigg(\delta \sum_{j=1}^d \left(\lambda_j(\Xbar_{t_n+\delta/2},\Vbar_{t_n})-\lambda_j(\Xbar_{t_n+\delta/2},-\tilde{V})\right)\Bigg) \right)$$  
  \textbf{else} set $(\Xbar_{t_{n+1}},\Vbar_{t_{n+1}}) = (\Xbar_{t_{n}},-\Vbar_{t_{n}})$\;
  Set $n=n+1$\;
 }
 \caption{Non-reversible Metropolis adjusted ZZS}
 \label{alg:Metropolis_DBD_ZZS}
\end{algorithm}

\subsubsection{Non-reversible Metropolis adjusted ZZS}
Taking advantage of the skew-reversible Metropolis-Hastings framework described above we now define an exact version of splitting \textbf{DBD} of ZZS. 
Recall that the splitting \textbf{DBD} of Example \ref{ex:ZZ} proposes moves from $(x,v)$ to states of the form $$ (\tilde{X},\tilde{V}) = (x+ \frac{\delta}{2} (v+R_I v) ,R_I v).$$ As we shall motivate below, in this case the acceptance probability \eqref{eq:prob_MH_nonrev} becomes
\begin{equation}\label{eq:MH_prob_zzs}
    \alpha((x,v),(\tilde X,\tilde V))= 1\wedge \exp\left(\pot(x)-\pot(\tilde{X})  + \delta \sum_{i\notin I}v_i\partial_i \pot(x+v\delta/2) \right).
\end{equation}
In case of rejection the state is set to $(x,-v)$. 
The pseudo-code for the resulting adjusted scheme is shown in Algorithm \ref{alg:Metropolis_DBD_ZZS}, where an equivalent expression of the acceptance probability is used.

\textit{Derivation of \eqref{eq:MH_prob_zzs}.} 
Let $x_{1/2}(x,v) := x + v \delta/2$. After one iteration the algorithm proposes state
$ (\tilde{X},\tilde{V}) = (x_{1/2}+ R_I v\delta/2 ,R_I v)$
with probability
\begin{equation}\label{eq:mh_zzs_kernel}
    \exp\left(-\delta \sum_{i\notin I}\lambda_i(x_{1/2},v) \right) \prod_{i \in I} (1-\exp(-\delta \lambda_i(x_{1/2},v)).
\end{equation}
The classical MH scheme is not directly applicable, as in general the probability that the process goes from $(\tilde{X},\tilde{V})$ to $(x,v)$ is $0$. Hence we enforce skew-DB by first computing the probability that the chain goes from $(\tilde{X},-\tilde{V})$ to $(x,-v)$. This can only be achieved by following the same path of $(x,v)\to(\tilde{X},\tilde{V})$ backwards, hence flipping the sign of the velocity components in $I$. Noticing that $x_{1/2}(x,v)=x_{1/2}(\tilde{X},-\tilde{V}),$ we find that the probability of this path is the same as \eqref{eq:mh_zzs_kernel} but where terms $\lambda_i(x_{1/2},v)$ are substituted by $\lambda_i(x_{1/2},-\tilde V)$.
Observe that for $i\in I$ it holds that $\tilde{V}_i = -v_i$ and thus $\lambda_i(x_{1/2},v) = \lambda_i(x_{1/2},-\tilde{V})$, while for $i\notin I$ we have $\tilde{V}_i = v_i$ and hence $\lambda_i(x_{1/2},v) - \lambda_i(x_{1/2},-\tilde{V}) = v_i\partial \pot(x_{1/2}).$
Therefore applying \eqref{eq:prob_MH_nonrev}, we find that the acceptance probability of state $(\tilde X,\tilde V)$ is \eqref{eq:MH_prob_zzs}.

\begin{algorithm}[t]
\SetAlgoLined
\SetKwInOut{Input}{Input}\SetKwInOut{Output}{Output}
\Input{Number of iterations $N$, initial condition $(x,v)$, step size $\delta$.}
\Output{Chain $(\Xbar_{t_{n}},\Vbar_{t_{n}})_{n=0}^N$.}
 Set $n=0$, $(\Xbar_0,\Vbar_0) = (x,v)$\;
 \While{$n < N$}{
  Set $\Vbar_{t_{n}+\delta/2} = \Vbar_{t_n}$ \;
  With probability $(1-\exp(-\lambda_r \delta/2))$ draw $\Vbar_{t_{n}+\delta/2}\sim \Unifsphere$  \;
  Set $\Xbar_{t_n+\delta/2} = \Xbar_{t_n} +\frac{\delta}{2}\Vbar_{t_{n}+\delta/2} $\;
  Set $\tilde{V}=\Vbar_{t_{n}+\delta/2}$\;
  With probability $(1-\exp(-\delta\lambda_1(\Xbar_{t_n+\delta/2},\tilde{V})))$ set $\tilde{V}= R(\Xbar_{t_n+\delta/2}) \tilde{V}$\;
  Set $\tilde{X} = \tilde{X}+ \frac{\delta}{2}\tilde{V} $\;
  Set $(\Xbar_{t_{n+1}},\Vbar_{t_{n+1}}) = (\tilde{X},\tilde{V})$ with probability  
  $$ 1\wedge \frac{\pi(\tilde{X}) \times \exp(-\delta \lambda(\Xbar_{t_n+\delta/2}, -\tilde{V}))  }{\pi(\Xbar_{t_n})\times \exp(-\delta \lambda(\Xbar_{t_n+\delta/2},\Vbar_{t_n}))}$$
  \textbf{else} set $(\Xbar_{t_{n+1}},\Vbar_{t_{n+1}}) = (\Xbar_{t_{n}},-\Vbar_{t_{n}+\delta/2})$ \;
  With probability $(1-\exp(-\lambda_r \delta/2))$ set $\Vbar_{t_{n+1}}\sim \Unifsphere$ \;
  Set $n=n+1$\;
 }
 \caption{Non-reversible Metropolis adjusted BPS}
 \label{alg:Metropolis_RDBDR_BPS}
\end{algorithm}
\subsubsection{Non-reversible Metropolis adjusted BPS}
Here we consider scheme \RDBDR of BPS and derive the appropriate acceptance probability \eqref{eq:prob_MH_nonrev}. 
The resulting procedure is written in pseudo code form in Algorithm \ref{alg:Metropolis_RDBDR_BPS}. 

\textit{Derivation of Algorithm \ref{alg:Metropolis_RDBDR_BPS}.} First observe that the refreshment steps ensure irreducibility but do not alter the stationary distribution of the process, thus we focus on the \textbf{DBD} part. Recall the notation $x_{1/2}(x,v) = x+\delta v /2$. According to \textbf{DBD}, the process moves from an initial condition $(x,v)$ to 
\begin{equation}\label{eq:tildes_metropolis}
    (\tilde{X},\tilde{V}) = 
    \begin{cases}
	    (x_{1/2}+ \frac{\delta}{2} R(x_{1/2})v ,R(x_{1/2})v) \quad & \text{ with probability } 1-\exp(-\delta \lambda(x_{1/2},v)), \\
	    (x+\delta v,v)  \quad & \text{ with probability } \exp(-\delta \lambda(x_{1/2},v)).
	\end{cases}
\end{equation}
Observe that for both states in \eqref{eq:tildes_metropolis} it holds $x_{1/2}(x,v)=x_{1/2}(\tilde{X},-\tilde{V}) = x_{1/2}$. We now compute the acceptance probability \eqref{eq:prob_MH_nonrev} in either of the two cases. 

Consider first the case in which a reflection took place, which corresponds to the first line of \eqref{eq:tildes_metropolis}. The process goes from $(\tilde{X},-\tilde{V})$ back to $(x,-v)$ with the same probability with which the process has a reflection at $x_{1/2}$.  Recall that by definition of $\lambda$, it holds that $\lambda(x_{1/2},v) = \lambda(x_{1/2},-R(x_{1/2})v)$ and therefore the probability that the process goes from $(\tilde{X},-\tilde{V})$ to $(x,-v)$ is the same of going from $(x,v)$ to $(\tilde{X},\tilde{V})$. This gives that the acceptance probability \eqref{eq:prob_MH_nonrev} is
\begin{equation}\label{eq:MH_bps_1}
    1 \wedge \frac{\pi(x_{1/2}+ \frac{\delta}{2} R(x_{1/2})v )}{\pi(x)} = 1 \wedge \exp\Big(\pot(x)-\pot\left(x_{1/2}+ \delta R(x_{1/2})v/2\right) \Big).
\end{equation}

Consider now the second case in \eqref{eq:tildes_metropolis}. The probability that the process goes from $(x+\delta v,-v)$ to $(x,-v)$ is $\exp(-\delta \lambda(x_{1/2},-v))$, while the probability of going from $(x,v)$ to $(x+\delta v,v)$ is $\exp(-\delta \lambda(x_{1/2},v))$. Observing that $\lambda(x_{1/2},v)- \lambda(x_{1/2},-v) = \langle v, \nabla \pot(x_{1/2}) \rangle$ we find that in this case the MH acceptance probability is 
\begin{equation}\label{eq:MH_bps_2}
    \begin{aligned}
    1\wedge\frac{\pi(x+\delta v) \times \exp(-\delta \lambda(x_{1/2},-v))}{\pi(x)\times \exp(-\delta \lambda(x_{1/2},v))} &
    & = 1\wedge \exp\Big(\pot(x)-\pot(x+v\delta)  + \delta\langle v, \nabla \pot(x_{1/2}) \rangle \Big).
\end{aligned}
\end{equation}

Hence we have shown that the unadjusted proposal $(\tilde X,\tilde V)$ is accepted with probability
\begin{equation}
    \alpha((x,v),(\tilde X,\tilde V)) =1\wedge \exp\Big(\pot(x)-\pot(\tilde{X})  + \delta(\lambda(x_{1/2},v)-\lambda(x_{1/2},-\tilde{V}))\Big).
\end{equation}

\subsection{{Algorithms with subsampling}}\label{subsec:subsamp}
One of the attractive features of ZZS and BPS is \emph{exact subsampling}, i.e. the possibility when the potential is of the form $\pot(x) = \frac{1}{M} \sum_{j=1}^M \pot_j(x)$ of using only a randomly chosen $\pot_j$ to simulate the next event time. 
The typical application of this technique is Bayesian statistics, where $\pot(x)$ is the \emph{posterior distribution}, $x$ is the parameter of the chosen statistical model and, when the data points are independent realisations, $\pot_j$ can be chosen to depend only on the $j$-th batch of data points and not on the rest of the data-set. For large data-sets, this technique can greatly reduce the computational cost per event time. In essence, this property is a consequence of the fact that the ZZS with non-canonical rates
\begin{equation}\label{eq:rates_noncanonical}
    \lambda_i(x,v) =\frac1M \sum_{j=1}^M (v_i \partial_i \pot_j(x))_+ \quad \text{for }i=1,\dots,d
\end{equation}
is invariant with respect to $\pi\propto \exp(-\pot)$ \citep{ZZ}. If the functions $t\mapsto  (v_i \partial_i \pot_j(x+vt))_+$ can be bounded independently of $j$, then one can generate proposals for the next event time using the Poisson thinning technique. A proposal is then accepted evaluating a term $(v_i \partial_i \pot_J(x))_+$ for $J$ that is picked uniformly at random from $\{1,\dots,M\}.$ Overall, this procedure has $O(1)$ cost.
Bayesian statistics is not the only area where this structure of $\pot$ arises, see e.g. the interacting particle system of Section \ref{subsec:example_particles}, where subsampling corresponds to a splitting of the forces.

Algorithm \ref{alg:splitting_DBD_ZZS} can be modified to allow for subsampling by adapting the \textbf{B} part, that is the simulation of the jumps. Here we describe two approaches to achieve this, and we give the pseudo-codes for the corresponding modified jump part in Appendix~\ref{sec:pseudocodes_subsampling}. Whenever $ (v_i \partial_i \pot_j(x))_+$ can be bounded by a constant $\beta$ independently of $j$ (and $v_i$), then we can take advantage of Poisson thinning similarly to the case of the continuous-time ZZS. In this case, we can simulate the jump part exactly for each component $i$ by the following iterative procedure, to be repeated until the end of the time-step: (i) draw a proposed jump time with rate $\beta$, (ii) if the proposed time is smaller than the time left to the current time-step, then we accept it with probability $\beta^{-1} (v_i \partial_i \pot_J(x))_+$ for a random $J\sim\Unif(\{1,\dots,M\})$, and flip the corresponding sign of the velocity, which becomes $-v_i.$ A consequence of the non-canonical rates \eqref{eq:rates_noncanonical} is that for each component multiple jumps can happen at each time step.
We will discuss a concrete example where this approach can be followed in Section~\ref{subsec:example_particles}. 
When this procedure is not applicable, then one can simply simulate for each component $i=1,\dots,d$ a jump process with rate $ (v_i \partial_i \pot_J(x))_+$, where $J$ is refreshed after each jump. This approach does not achieve exact simulation of the jump part of ZZS, hence it introduces further numerical inaccuracies. We remark that this is similar to the algorithm discussed in Example 5.7 of \citet{bertazzi2021approximations}. 
A splitting scheme with subsampling based on BPS can be defined with similar ideas.

\subsection{PDMPs with boundaries}\label{sec:algorithms_boundaries}
Another interesting feature of PDMPs such as BPS and ZZS is that, thanks to the simple deterministic dynamics, boundary conditions can be included and hitting times of the boundary can be easily computed (see \citet{Davis1993} or \citet{pdmp_piecewise_smooth} for a discussion of PDMPs with boundaries). Here we illustrate how to simply adapt splitting schemes to these settings by adding the boundary behaviour to the \textbf{D} part of the scheme.

Boundary terms appear for instance when the target distribution $\pi$ is defined on a restricted domain \citep{pdmp_restricted_domains}. In this case, Algorithms \ref{alg:splitting_DBD_ZZS} and \ref{alg:splitting_RDBDR_BPS} can be modified by incorporating the boundary term in part \textbf{D} of the splitting scheme, as the boundary can be hit only if there is deterministic motion. Hence, the continuous deterministic dynamics are applied as in the exact process, while other jumps are performed in the \textbf{B} steps.

Another example of this setting is when $\pi$ is a mixture of a continuous density and a discrete distribution on finitely many states, as in Bayesian variable selection when a spike and slab prior is chosen. Sticky PDMPs were introduced in \citet{sticky_pdmp} to target a distribution of the form 
$$ \mu(dx) \propto \exp(-\pot(x))\prod_{i=1}^d (dx_i + \frac{1}{c_i}\delta_0(dx_i)),$$
which assigns strictly positive mass to events $\{x_i=0\}$. The sticky ZZS of \citep{sticky_pdmp} is obtained following the usual dynamics of the standard ZZS and in addition freezing the $i$-th component for a time $\tau\sim\Exp(c_i)$ when $x_i$ hits zero. The simulation of this process is challenging for the same reasons of the standard ZZS, since the two processes have the same switching rates $\lambda_i$ for $i=1,\dots,d$. The $i$-th component is either frozen, which is denoted by $(x_i,v_i)\in A_i$, or it evolves as given by the usual dynamics of ZZS. The generator can then be decomposed as $\cL = \cL_{D}  + \cL_{B}$ where $\cL_{D} = \sum_{i=1}^d \cL_{D,i}$ and $ \cL_{B} = \sum_{i=1}^d \cL_{B,i}$, 
\begin{align*}
    \cL_{D,i} f(x,v) &= v_i \frac{\partial}{\partial x_i} f(x,v)\1_{A_i^C}(x_i,v_i) + c_i(f(T_i(x,v))-f(x,v)) \1_{A_i}(x_i,v_i), \\
    \cL_{B,i} f(x,v) &=  \lambda_i(x,v)[f(x,R_iv)-f(x,v)]\1_{A_i^C}(x_i,v_i),
\end{align*}
and $T_i(x,v)$ corresponds to unfreezing the $i$-th component (we refer to \cite{sticky_pdmp} for a detailed description). An iteration of the scheme \textbf{DBD} in this case proceeds by a first half step of \textbf{D}, which is identical to the continuous sticky ZZS but with $\lambda_i$ temporarily set to $0$. Hence frozen components are unfrozen with rate $c_i$ and then start moving again, or unfrozen components move with their corresponding velocity $v_i$ and become frozen for a random time with rate $c_i$ if they hit $x_i=0$. Then a full step of the usual bounce kernel \textbf{B} is done for the components which are not frozen, while for the frozen components, that is $(x_i,v_i) \in A_i$, the generator $\cL_{B,i}$ does nothing and so the velocity cannot be flipped. So unfreezing is not possible in this step. The iteration ends with another half step of \textbf{D} in a similar fashion to the previous one.

These ideas are more general than the two specific examples we considered and do not introduce further difficulties for our schemes. Finally, notice that a Metropolis correction can be added following Section \ref{sec:metropolis_adjusted_algorithms}, and subsampling is possible following Section~\ref{subsec:subsamp}.

\section{Convergence of the splitting scheme}\label{sec:CVsplitting}

In this section we prove that under suitable conditions the splitting scheme \textbf{DJD} described in Section \ref{sec:intro} is indeed a second order approximation of the original PDMP \eqref{eq:genPDMP}. 

Note that in this section we have a PDMP defined on some arbitrary space $E$ therefore it is not clear what it means to have a derivative, indeed we will typically be interested in the setting $E=\R^d\times \mathcal{V}$ for some set $\mathcal{V}$ which may be a discrete set. Instead of working with a full derivative we will define the directional derivative, $D_\vf$, in the direction $\Phi$ as
\begin{equation*}
D_\vf g(z) = \lim_{t\to 0} \frac{d}{dt} g(\varphi_t(z))
\end{equation*}
for any $g\in C(E)$ for which $t\mapsto g(\varphi_t(z))$ is continuously differentiable in $t$ for every $z$. Note if $E$ is a subset of $\R^d$ for some $d$ and $g$ is continuously differentiable then
$D_\vf g(z) = \vf(z)^T\nabla g(z).$
We extend this definition to multi-dimensional valued functions $G:E\to \R^m$ by defining $D_\vf G(z) = (D_\vf G^i(z))_{i=1}^m$. We define the space $\C_\vf^{k,m}$ to be the set of all functions $g:E\to \R$ which are $k$ times continuously differentiable in the direction $\vf$ with all derivatives $D_\vf^\ell g(z)$ up to order $k$ bounded by a polynomial of order $m$. We endow this space with the norm
\begin{equation*}
    \lVert g\rVert_{\C_\vf^{k,m}}:=\sup_{z\in E}\frac{\lvert g(z)\rvert +\sum_{\ell=1}^k\lvert D_\vf^\ell g(z)\rvert}{1+\lvert z\rvert^{m}}. 
\end{equation*}

Let us make the following assumptions.
\begin{assumption}\label{ass:determisticflowmap}
Let $\Phi$ be a globally Lipschitz vector field defined on $E$ and assume that $D_\Phi \Phi$ is well-defined.
\end{assumption}

\begin{assumption}\label{ass:switching_rate}
The switching rate $\lambda:E\to [0,\infty)$ is twice continuously differentiable in the direction $\Phi$ and $\lambda,D_\vf\lambda, D_\vf^2\lambda$ grow at most polynomially. We denote by $m_\lambda$ a constant such that $\lVert \lambda \rVert_{\C_\vf^{2,m_\lambda}}<\infty$.
\end{assumption}

\begin{assumption}\label{ass:kernelexp}
Let $Q$ be a probability kernel defined on $E$. We shall consider the operator $Q:\C_b(E)\to \C_b(E)$ defined by
\begin{equation}
    Q g(z) = \int g(\tilde{z}) Q(z,d\tilde{z}), \quad \text{ for any } g \in \C_b(E).
\end{equation}
Moreover we assume that $Q$ has moments of all orders and $Q g$ has at most polynomial growth of order $m$ whenever $g$ has at most polynomial growth of order $m$. 
For any $m \in \mathbb{N}$, and $g\in \C_\vf^{1,m}$ we assume the following distribution is well-defined:
\begin{equation}
    (D_\vf Q)g(z) = D_{\Phi}(Qg)(z).
\end{equation}
As an abuse of notation we shall write $D_\vf Q$ also as a kernel. We assume for any $m\in \mathbb{N}$, and $g\in \C_\vf^{1,m}$
\begin{equation}\label{eq:kernelexpfirst}
     \left\lvert Qg(\varphi_s(z))-Qg(z) \right \rvert \leq C s(1+\lvert z\rvert^m) \lVert g\rVert_{\C_\vf^{1,m}},
\end{equation}
and also that there exists a constant $C$ such that for any $g\in \C_\vf^{2,m} $
\begin{equation}\label{eq:kernelexpsecond}
    \left\lvert Qg(\varphi_s(z))-Qg(z) -s D_\Phi Qg(z) \right\rvert
     \leq C s^2 (1+\lvert z\rvert^m)\lVert g\rVert_{\C_\vf^{2,m}}. 
\end{equation}
\end{assumption}

\begin{assumption}\label{ass:derivative_estimate}
The closure $(\cL, D(\cL))$ of the operator $(\cL, \C_c^1(E))$ in $L_\mu^2$ generates a $C_0$-semigroup $\cP_t$. If $g\in \C_\vf^{2,0}$ then we assume that $\cP_tg$ is also twice continuously differentiable in the direction $\Phi$ and $\cL\cP_t g$ is continuously differentiable in the direction $\vf$. Moreover we assume $D_\vf^2\cP_t g$ and $D_\vf\cL\cP_t g$ are both polynomially bounded for finite $t$ and for some $C>0, R\in \R, m_{\cP}\in \mathbb{N}$
\begin{equation*}
   \lvert D_\vf\cP_t g(z)\rvert+ \lvert D_\vf^2\cP_t g(z)\rvert +\lvert D_\vf\cL\cP_t g(z)\rvert \leq C (1+\lvert z\rvert^{m_{\cP}}) e^{R t} \lVert g \rVert_{\C_\vf^{2,0}}.
\end{equation*}

\end{assumption}

\begin{assumption}\label{ass:momentcond}
    Let $\overline{Z}_{t_k}$ denote the approximation obtained by the splitting scheme \textbf{DJD}. Assume that for each $k$, $\overline{Z}_{t_k}$ has moments of all orders and moreover for every $M\in \N$ there exists some $\overline{G}_M$ such that
    \begin{equation*}
        \sup_{m\leq M}\mathbb{E}_z[\lvert \overline{Z}_{t_k}\rvert^m] \leq \overline{G}_M(z).
    \end{equation*}
\end{assumption}

\begin{theorem}\label{thm:weakerror}
Let $Z_t$ be a PDMP corresponding to the generator \eqref{eq:genPDMP}. Assume that Assumption \ref{ass:determisticflowmap} to Assumption \ref{ass:momentcond} hold. 
Then there exist constants $C,R$ such that for any $g\in \C_\vf^{2,0}\cap D(\cL)$ we have for some $M\in\N$
\begin{equation*}
    \sup_{k\leq n}\lvert \mathbb{E}[g(Z_{t_k})]-\mathbb{E}[g(\overline{Z}_{t_k})]\rvert \leq C e^{R t_n} \overline{G}_M(z)\delta^3 n \lVert g\rVert_{\C_\vf^{2,0}}.
\end{equation*}
\end{theorem}

\begin{proof}
    The proof follows a similar approach to \citet[Theorem 4.24]{bertazzi2021approximations} and can be found in Appendix \ref{sec:proof_weakerror}.
\end{proof}

The Theorem shows that splitting schemes of the form \textbf{DJD} give second order approximations of PDMPs, under the assumptions we stated. Indeed, the term $\delta^3 n$ equals $\delta^2 t_n$, where $t_n$ is the time horizon of the continuous time process, and thus we find a $\delta^2$ dependence. Compared to \citet{bertazzi2021approximations} our estimate contains a term that is exponentially increasing in the time horizon. This term could be handled by assuming e.g. geometric convergence of the derivatives of the semigroup for the continuous time PDMP. To the best of our knowledge, aside from the results in \citet{bertazzi2021approximations} there are no known results establishing such estimates for PDMPs. The technical nature of our proof also makes the application of this idea challenging, but we see no reason why this approach should not give a uniform in time estimate of the weak error similarly to \citet{bertazzi2021approximations}. Finally, let us comment on the choice of the class of test functions for which our result holds, that is $\C_\vf^{2,0}\cap D(\cL)$. This choice is explained by the fact that it is necessary to consider functions that are twice differentiable in the direction of the deterministic motion to obtain bounds on the error over one time step. For this reason, it cannot be expected to have a result e.g. in Wasserstein distance, which is not suited for (continuous time) PDMPs as we discussed previously.

\begin{example}[ZZS continued]\label{ex:ZZS_weakerror} Recall the Zig-Zag sampler from Example \ref{ex:ZZ} let us verify the Assumption \ref{ass:determisticflowmap} to \ref{ass:momentcond} in this case. In order to have a smooth switching rate we replace $\lambda_i(x,v)$ by
\begin{equation*}
    \lambda_i(x,v) = \log\left(1+\exp(v_i\partial_i\pot(x))\right).
\end{equation*}
    This is shown to be a valid switching rate in \citet{AndrieuLivingstone}. We will assume that $\pot\in \C^2$ with bounded second and third derivatives. Let us now consider each assumption in turn.
    
    \textit{Assumption \ref{ass:determisticflowmap}:} In this case $\Phi(x,v)=(v,0)^T$ which is smooth and globally Lipschitz.

    \textit{Assumption \ref{ass:switching_rate}:} Since $\lambda_i$ is the composition of smooth maps and $\pot$ we have that $\lambda_i$ has the same smoothness in $x$ as $\pot$ and hence $x\mapsto \lambda_i(x,v)$ is $\C^2$. As $s\mapsto \log(1+e^s)$ grows at most linearly, has first and second derivatives bounded by $1$ we have that $\lambda_i, \nabla_x\lambda_i$ and $\nabla_x^2\lambda_i$ are all polynomially bounded.

\textit{Assumption \ref{ass:kernelexp}:} The proof of this can be found in Appendix \ref{subsec:proof_ex_ZZS}.

\textit{Assumption \ref{ass:derivative_estimate}:} By \citet{AndrieuLivingstone} we have that $\cP_t$ is a strongly continuous semigroup on $L_\mu^2$ with generator $(\cL,D(\cL))$ given as the closure of $(\cL, C_c^1(E))$. Moreover we have that the assumptions of \citet[Theorem 17]{pdmp_inv_meas} are satisfied and hence $\cP_t g(x,v)$ is differentiable in $x$. Following the proof of \citet[Theorem 17]{pdmp_inv_meas} one also has
\begin{equation*}
    \lvert \nabla_x \cP_tg \rvert \leq C(1+\lvert x\rvert^m)e^{Rt} \lVert g\rVert_{\C_\vf^{1,0}}.
\end{equation*}
    Note here since $D_\vf g(x,v) = v^T\nabla_x g$ we have that $\C_\vf^{k,0}$ coincides with the space of continuous functions which are $k$-times continuously differentiable in the variable $x$. 
    By the same arguments one can also obtain 
    \begin{equation*}
    \lvert \nabla_x^2 \cP_tg \rvert \leq C(1+\lvert x\rvert^m)e^{Rt} \lVert g\rVert_{\C_{\vf}^{2,0}}.
\end{equation*}

\textit{Assumption \ref{ass:momentcond}:} This will be established in  Theorem~\ref{thm:ZZS_ergodic_details} (see Lemma \ref{lem:drift_DBD}), in which we show that the chain satisfies a geometric drift condition for a function that bounds any power.
\end{example}

\section{Ergodicity of splitting schemes of BPS and ZZS}\label{sec:ergodicity}
We shall now focus on results on ergodicity of splitting schemes of BPS and ZZS. In particular we show existence of an  invariant distribution, characterise the set of all invariant distributions, and establish convergence of the law of the process to such distributions with geometric rate. Importantly, we make sure that the geometric convergence has the expected dependence on the step size and that the estimates are stable as $\delta$ decreases to $0$. The statements can be found in Section \ref{sec:mainresults_ergodicity}, respectively in Theorems \ref{thm:ergodicBPS} for BPS and \ref{thm:ergodicity_zzs} for ZZS. We obtain our theorems applying the classical theory of \cite{meyn_tweedie_1993}, which is based on minorisation and drift conditions. In Section \ref{sec:strategy_ergodicity} we explain the strategy that we follow to obtain such results.


\subsection{Main results}\label{sec:mainresults_ergodicity}
Let us now state more precisely the result on  ergodicity we shall obtain for splitting schemes of BPS and ZZS.  For a given probability distribution $\mu$, we define its $V$-norm as $\lVert \mu\rVert_V := \sup_{\lvert g\rvert \leq V} \lvert \mu(g)\rvert$. We shall show that the chains admit a unique invariant distribution $\mu_\delta$ and that there exist constants $C'',\tilde \kappa, \delta_0>0$ and a function $V$ such that for $n\geq 1$ and any probability distributions $\mu,\mu'$ it holds
\begin{equation}\label{eq:ergodic}
    \|\mu P_\delta^n - \mu' P_\delta ^n \|_V  \leqslant  C'' \tilde \kappa^{n\delta} \|\mu  - \mu' \|_V, \quad \text{ for all } \delta \in [0,\delta_0].
\end{equation}
Note that $C'',\tilde \kappa, \delta_0,V$ are all independent of $\delta$. Clearly, taking $\mu'=\mu_\delta$ we obtain geometric convergence to the invariant distribution of the splitting scheme. 

For splitting schemes of the BPS, we work under the following condition.
\begin{assumption}\label{assu:BPSscheme}
The dimension is $d \geq 2$, the velocity equilibrium $\nu$ is the uniform measure on $\mathbb S^{d-1}$.
There exists $C>0$ such that  
\[\frac1C |x|^2 - C \leqslant  \psi(x) \leqslant C|x|^2 + C\,,\qquad  \frac1C |x| - C \leqslant |\na \psi(x)| \leqslant C|x| + C\]
for  all $x\in\R^d$. Moreover, $\|\na^2 \psi\|_\infty < \infty$ and, without loss of generality, $\inf \psi = 1$.
\end{assumption}

Notice that, when $d=1$, the BPS and the ZZS coincide, in which case we refer to Theorem~\ref{thm:ergodicity_zzs} below.
Our result of ergodicity for  splitting schemes of the BPS is the following.
\begin{theorem}\label{thm:ergodicBPS}
Consider any scheme of the BPS  based on the decomposition \textbf{D},\textbf{R},\textbf{B}. Under Assumption~\ref{assu:BPSscheme}, the chains have a unique stationary distribution and there exist $\delta_0,C''>0,\tilde \kappa\in(0,1)$  and $V:\R^d\times\mathbb S^{d-1} \rightarrow [1,+\infty)$ satisfying 
\[ \text{for all } x\in\R^d,v\in \mathbb S^{d-1},\qquad e^{|x|/a}/a \leqslant V(x,v) \leqslant a e^{a|x|}\]
for which \eqref{eq:ergodic} holds.
\end{theorem}
\begin{proof}
The proof can be found in Appendix \ref{sec:proofs_ergodicity_bps}.
\end{proof}

More care is required for the \textbf{DBD} scheme of the ZZS since this Markov chain has periodicity and is not irreducible, which is reminiscent of the discrete-space Zig-Zag chain studied in  \citet{monmarche_kin_walks}. Let us illustrate this behaviour by considering the one dimensional setting. Let $(x,v)$ be the initial condition of the process. Since $v$ has magnitude $1$, the position component $x$ can only vary by multiples of the step size $\delta$. Thus for a fixed initial condition $(x,v)$ the process remains on a grid  $(x+\delta\mathbb{Z})\times \{-1,1\}$. Moreover, after a single step of the scheme there are two possible outcomes: either the velocity does not change, in which case $x$ moves to $x+\delta v$, or the velocity is flipped and the position remains the same. This means that the change in the position (by amounts of $\delta$) plus half the difference in the velocity always changes by $\pm 1$ each step and hence is equal to the number of steps in the scheme up to multiples of two, i.e.
\begin{equation*}
    \frac{\overline{X}_{n\delta}-x}{\delta} + \frac{1}{2}(\overline{V}_{n\delta}-v) \in n+2\mathbb{Z}.
\end{equation*}
As a consequence, with probability $1$ the chain alternates between two disjoint sets, depending on whether $n$ is even or odd. Therefore, the chain is periodic and not ergodic. To overcome this issue, we consider the chain with one step transition kernel given by $P_\delta^2 = P_\delta P_\delta$, i.e. we restrict to the case of an even number of steps. 
For a given initial condition $(x,v)\in \R^d\times \{-1,1\}^d$, the Markov chain $P_\delta^2$ lives on the following grid
\begin{equation}\label{eq:D_minorisation}
    D(x,v):=\{(y,w) \in C \times\{\pm 1\}^d: (y_i,w_i)\in D_1(x_i,v_i) \, \text{ for all } i=1,\dots,d\},
\end{equation}
where $D_1(x_i,v_i):=D_+(x_i,v_i)\cup D_-(x_i,v_i),$ with
\begin{equation}\notag
    \begin{aligned}
        D_+(x_i,v_i) & := \{(y_i,w_i): w_i=v_i, \, y_i=x_i+m\delta,\, m\in  2\mathbb{Z}\}, \\
        D_-(x_i,v_i) & := \{(y_i,w_i): w_i=-v_i,\, y_i=x_i+m\delta,\, m\in  2\mathbb{Z}+1\}.
    \end{aligned}
\end{equation}
Hence the Markov chain $P_\delta^2$ is aperiodic, though it is not irreducible on $\R^d$ and therefore has (infinitely) many invariant measures which depend on the initial condition $(x,v)$. The ergodicity of $P_\delta^2$ can nevertheless be characterised as, for a given initial condition $(x,v)$, it is irreducible on $D(x,v)$. Notice that  the chain at odd steps lives on the disjoint set $\{(y,w):y=x+m\delta, \, m\in \mathbb{Z}, \, w\in\{\pm 1\}^d \} \setminus D(x,v)$.

In this case we show in Theorem \ref{thm:ergodicity_zzs} that the Markov chain with transition kernel $P_\delta^2$ is irreducible on $D(x,v)$, has a unique invariant measure, $\pi_\delta^{x,v}$, and is geometrically ergodic. 
Now we can characterise all the invariant measures of the Markov chain with transition kernel $P_\delta^2$ defined on $\R^d\times \{-1,1\}^d$ as the closed convex hull of the set $\{\pi_\delta^{x,v}:x\in \R^d, v\in \{-1,1\}^d\}$. Now consider the Markov chain with transition kernel $P_\delta^2$ on $\R^d\times\{-1,1\}^d$. For any initial distribution $\mu$ we have convergence of $\mu P_\delta^{2n}$ to some measure $\pi^\mu_\delta$ as $n$ tends to $\infty$ and $\pi^\mu_\delta$ is given by
\begin{equation}\label{eq:mi_mu}
    \pi^\mu_\delta(\varphi)  = (\mu\pi_\delta^{x,v})(\varphi):= \int_{\R^d\times\{-1,1\}^d}\int_{\R^d\times\{-1,1\}^d} \varphi(y,w) \pi_\delta^{x,v}(dy,dw) \mu(dx,dv).
\end{equation}

We use the next assumption to verify that Theorem \ref{thm:Harris} applies for initial conditions drawn from probability distributions with support on $D(x,v)$.
\begin{assumption}\label{ass:ergodicity_zzs}
Consider switching rates $\lambda_i(x,v) = (v_i\partial_i \pot (x))_+ + \gamma_i(x)$ for $i=1,\dots,d$. $\pot \in \mathcal{C}^2(\mathbb{R}^d)$ and the following conditions hold:
\begin{enumerate}[label=(\alph*)]
    \item There exists a scalar $\tilde y\geq 0$ such that for all $y > \tilde y$ 
    $$\underline{\lambda}(y):=\min_{i=1,\dots,d} \quad \min_{(x,v): \, x_iv_i\in [\tilde y,y], \,\lvert x_j \rvert \in [\tilde y,y] \text{ for all }j\neq i } \,\, \lambda_i(x,v)>0. $$
    
    \item For $\lvert x\rvert\geq R$ for some $R>0$
	\begin{equation}\label{eq:boundongamma}
	   \sup_{t\in(0,1)}\sup_{\substack{y_1,y_2\in B(x,t\sqrt{d}),\\v,w\in \{-1,1\}^d}} e^{ \left(t^2(\lvert(v+w)^T\nabla^2 \pot(y_1))_i\rvert+2 t\lvert(w\nabla^2 \pot(y_2))_i\rvert\right)}\gamma_i(x+vt)e^{ t v_i\partial_i\pot(x)}\leq \gamma_0<1.
	\end{equation}
	
	\item  Denote as $B(x,\delta\sqrt{d})$ the ball with centre at $x$ and radius $\delta\sqrt{d}$. Then 
	\begin{equation*}
		\lim_{\lVert x\rVert \to \infty} \,\, \sup_{y_1,y_2\in B(x,\delta\sqrt{d})}\frac{\max\{1,\lVert \nabla^2 \pot(y_1) \rVert\}}{\lvert \partial_i \pot(y_2) \rvert} = 0 \qquad \text{for all } 0\leq \delta \leq \delta_0,\, i=1,\dots,d,
	\end{equation*}
 where $\delta_0=2(1+\gamma_0)^{-1}$, for $\gamma_0$ as in  part (b).
\end{enumerate}
\end{assumption}
Part (a) in Assumption \ref{ass:ergodicity_zzs} is inspired by  \citet[Assumption 3]{BierkensRoberts} and is used to show that a minorisation condition holds. This condition is either a consequence of properties of the target, or else can be enforced by taking a non-negative excess switching rate, in which case $\gamma_i(x)$ can be chosen to be a continuous function $\gamma_i:\mathbb{R}^d\to  (0,\infty)$. In principle one could prove a minorisation condition using the techniques of \citet{Bierkensergodicity}, but this is beyond the scope of this paper. Part (b) is a condition on the decay of the refreshment rate, while Part (c) is similar to  Growth Condition 3 in \citet{Bierkensergodicity} and is satisfied for instance if $\pot$ is strongly convex with globally Lipschitz gradient. These two conditions are used to show that a drift condition holds. 

\begin{theorem}\label{thm:ergodicity_zzs}
Consider the splitting scheme \textbf{DBD} for ZZS. Suppose Assumption~\ref{ass:ergodicity_zzs} holds. Then there exist $C'',\delta_0>0,\tilde\kappa\in(0,1)$ and $V:\R^d\times\{-1,1\}^d\rightarrow [1,\infty)$ satisfying 
\[\text{for all } (x,v)\in \R^d\times\{-1,1\}^d\,,\qquad\prod_{i=1}^d \left(1+2\lvert \partial_i\pot(x)\rvert \right)^{-\frac12} \leq \frac{V(x,v)}{\exp(\beta\pot(x))} \leq \prod_{i=1}^d \left(1+2\lvert \partial_i\pot(x)\rvert \right)^\frac12\]
for all $\beta\in(0,1/2)$ such that, for all $\delta\in(0,\delta_0]$,  the following holds:
\begin{enumerate}
    \item  Fix $(x,v)\in\R^d\times\{-1,1\}^d$ and consider $P^2_\delta = P_\delta P_\delta$, transition kernel on  $D(x,v)$. Then $P^2_\delta$ admits a unique invariant distribution and the inequality \eqref{eq:ergodic} holds with $P_\delta$ replaced by $P_\delta^2$ with these $C'',\tilde\kappa,V$ for any $\mu,\mu'$ having support on $D(x,v)$.
    \item For any probability measure $\mu$ on $\R^d\times\{-1,1\}^d$ with $\mu(V) <\infty$, we have that $\mu P_\delta^{2n}$ converges as $n\rightarrow\infty$ to the measure $\pi^\mu_\delta$ given by \eqref{eq:mi_mu} where $\pi^{x,v}_\delta$ is the unique invariant measure of $P_\delta^2$ on $D(x,v)$ and we have
\begin{equation}\label{eq:convergence_random_IC_DBD}
    \|\mu P_\delta^{2n} -\pi^\mu_\delta\|_V \leq C''  \tilde{\kappa}^{n\delta} \,\int\|\delta_{(x,v)}-\pi_\delta^{x,v}\|_V\mu(dx,dv)\,.
\end{equation}
\end{enumerate}

\end{theorem}
\begin{proof}
The proof can be found in Appendix \ref{sec:ergodicity_DBD_zzs}.
\end{proof}

Under similar assumptions we establish geometric ergodicity of schemes \textbf{DRBRD}, \RDBDR of ZZS, where the switching rates in the \textbf{B} part are $\lambda_i(x,v)=(v_i\partial_i\pot(x))_+$, i.e. the canonical rates, while refreshments in the \textbf{R} part are independent draws from $\Unif(\{\pm 1\}^d)$ with rate $\gamma(x):\mathbb{R}^d\to[0,\infty).$ The rigorous statement of this result, Theorem \ref{thm:ergo_others_zzs}, and its proof can be found in Appendix \ref{sec:proof_ergo_others_zzs}.

\subsection{Proof strategy}\label{sec:strategy_ergodicity}
Let us start by stating the following classical result, due to Meyn and Tweedie (see \citet{meyn_tweedie_1993} for the original result, while here the specific statement is based on \citet[Theorem 1.2]{HairerMattingly2008}, see also  \citet[Theorem S.7]{BPS_Durmus} for the explicit constants).
\begin{theorem}\label{thm:Harris}
Consider a Markov chain with transition kernel $P$ on a set $E$. Suppose that there exist constants $\rho\in[0,1)$, $C,\alpha>0$, a function $V:E\rightarrow [1,+\infty)$ and a probability measure $\nu$ on $E$ such that the two following conditions are verified:
\begin{enumerate}
    \item Drift condition: for all $x\in E$,
    \begin{equation}\label{eq:drift_condition}
        P V(x) \leq \rho V(x) + C.
    \end{equation}
    \item Local Dobelin condition: for all $x\in E$ with $V(x) \leqslant 4C/(1-\rho)$,
    \[\delta_x P  \geqslant \alpha \nu \,.\]
\end{enumerate}
Then, for all probability measures $\mu,\mu'$ on $E$ and all $n\in\N$,
\begin{equation}\label{eq:statement_Harris}
    \|\mu P^n - \mu' P^n\|_V \leqslant \frac{C}{\alpha}\kappa^n\|\mu -\mu'\|_V 
\end{equation}
where $\kappa = \max(1-\alpha/2,(3+\rho)/4)$. Moreover $P$ admits a unique stationary distribution $\mu_*$ satisfying $\mu_*(V) < \infty.$
\end{theorem}

\begin{remark}\label{rmk:ergo_nominorisation}
Under the Drift condition \eqref{eq:drift_condition} alone, following the proof of  \citet[Theorem 1.2]{HairerMattingly2008} in the case $\alpha=0$ we get that for all probability measures $\mu,\mu'$ on $E$,
\[\|\mu P - \mu' P\|_V \leqslant (\rho +2C) \|\mu-\mu'\|_V\,.
 \]
We use this inequality in the discussion that follows.
\end{remark}

We prove geometric ergodicity of our splitting schemes by showing that the assumptions of Theorem~\ref{thm:Harris} are satisfied. 
More precisely, both in the case of BPS and of ZZS, we obtain a local Doeblin (or minorisation) condition with constant $\alpha$ after $n_*=\lceil t_*/\delta \rceil$ steps, where $t_*>0$  plays the role of physical time and $n_*$ is the number of steps needed to travel for an equivalent time. Here $t_*,\alpha$ are independent of $\delta$. On the other hand, we show that the drift condition holds for one step of the kernel with constants $\rho = 1-b\delta$ and $C = D\delta$, where $b,D$ and the Lyapunov function $V$ are independent of $\delta$. This implies that for any $s>0$ and any $\delta \in (0,\delta_0]$ 
\begin{eqnarray*}
(P_\delta)^{\lceil s/\delta\rceil} V & \leqslant & \po 1 - b\delta\pf^{\lceil s/\delta\rceil} V + D \delta \sum_{k=0}^{\lceil s/\delta\rceil-1} \po 1 - b\delta\pf^k  \leqslant  e^{-b s} V + \frac{D}b\,.
\end{eqnarray*}
Applying Theorem~\ref{thm:Harris}, we get for $P_\delta^{n_*}$ a long-time convergence estimate which is uniform over $\delta\in(0,\delta_0]$, that is for all $\delta\in(0,\delta_0]$ and $n\geq 1$ we find 
\begin{equation}\notag
    \|\mu (P_\delta^{n_*})^n - \mu' (P_\delta^{n_*})^n\|_V \leqslant \frac{C'}{\alpha}  \kappa^{n} \,\|\mu -\mu'\|_V,
\end{equation}
where $C'=D/b$ and $\kappa=\max(1-\alpha/2,(3+e^{-bt_*})/4)$. Observe that the rhs does not depend on $\delta.$
Using the observation in Remark \ref{rmk:ergo_nominorisation}, we can get convergence in $V$-norm for $P^n_\delta$. Indeed for $n=m n_*+r$ with $r<n_*$ we have
\begin{eqnarray}
    \|\mu P_\delta^n - \mu' P_\delta ^n \|_V 
    & \leqslant & \frac{C'}{\alpha} \kappa^m \|\mu P_\delta^r - \mu' P_\delta^r\|_V \nonumber\\
    & \leqslant & \frac{C'}{\alpha }   \po 1+ 2 C'\pf  \kappa^m \|\mu  - \mu' \|_V \nonumber\\
    & \leqslant & C'' \tilde \kappa^{n\delta} \|\mu  - \mu' \|_V, \notag 
\end{eqnarray}
where $\tilde{\kappa}=\kappa^{1/(t_*+\delta_0)}\in(0,1)$ and $C''=C'   \po 1+ 2 C'\pf /(\alpha  \kappa)$ are independent from $\delta$. 
Here we used that with computations identical to above we get the drift condition $P_\delta^r V\leq (1-b\delta)V+D(1-(1-b\delta)^r)/b \leq V+C' $, which is enough for the current purpose. As a conclusion, the estimates given in Theorems~\ref{thm:ergodicBPS} and \ref{thm:ergodicity_zzs} (or in Appendices \ref{sec:proofs_ergodicity_bps} and \ref{sec:proofs_ergo_zzs} for more details)   give the expected dependency in $\delta$ for the convergence rate of the process toward equilibrium.

\section{Expansion of the invariant measure of splitting schemes for BPS and ZZS}\label{sec:expansion_mu}

In this section we investigate the bias in the invariant measure of different splittings of BPS and ZZS and draw conclusions on which schemes are to be preferred. This analysis sheds light on the dependence of the bias on the structure of the splitting scheme as well as on the refreshment rate, which is the only parameter of our algorithms other than the step size. Our theoretical and numerical investigations highlight how poor choices of the refreshment rate can have a negative effect on the bias of some schemes. In practice it is hard to define good values of $\lambda_r$ a priori, as this is case-dependent and theory only exists for factorised targets \citep{bierkens2018highdimensional}. For these reasons, we shall focus on the robustness of the various schemes to this parameter.  
In this section we assume the following condition on $\mu_\delta$, which is motivated by Theorems \ref{thm:ergodicBPS} and \ref{thm:ergodicity_zzs}, giving existence and uniqueness of a stationary distribution $\mu_\delta$, and Theorem \ref{thm:weakerror}, leading to the conjecture that $\mu_\delta$ is a second order approximation of $\mu$. 
\begin{assumption}\label{ass:exp_inv_meas}
    The processes corresponding to our splitting schemes have an invariant distribution with density
\begin{align}\label{eq:ansatz}
    \mu_\delta(x,v) = \mu(x,v) (1-\delta^2 f_{2}(x,v)+\mathcal{O}(\delta^4)),
\end{align}
where $\mu(x,v) = \nu(v)  \pi(x)$, $\pi$ is the target and $\nu$ is a distribution
on the velocity vector.
\end{assumption}
\begin{remark}
{In Section \ref{sec:ergodicity} we discuss cases where a splitting scheme may admit more than one invariant measure. In such cases it is not immediately clear what the expansion \eqref{eq:ansatz} means. 
In order to make \eqref{eq:ansatz} consistent as $\delta \to 0$, in those cases we consider $\mu_\delta$ as the limit of the law of the splitting scheme as the number of steps tends to infinity when the process is started according to $\mu$.}
\end{remark}
\begin{remark}
In order to establish Assumption \ref{ass:exp_inv_meas} rigorously one would typically rely on estimates on the semigroup, for instance establishing decay of its derivatives. For diffusion processes analogous results have been established using a variety of techniques \citep{CrisanOttobre,Kopec16}, but adapting these techniques to PDMPs is challenging. One of the main reasons for this is that PDE based approaches often make use of smoothing and, unlike diffusion processes, PDMPs do not have smoothing properties (for example, the proof in \citet{Kopec16} relies on use of the Bismut-Elworthy formula, which is not true for PDMP). One result in this direction is Theorem~5.13 in \citet{bertazzi2021approximations}, which establishes convergence of the first order derivative for the semigroup of a ZZS. Notice that in order to establish \eqref{eq:ansatz} we would also require higher order derivatives.
\end{remark}

Finding a function $f_2$ as in \eqref{eq:ansatz} allows us to compare the asymptotic bias of different splitting schemes of BPS and ZZS and is thus crucial to determine which splitting scheme to use for MCMC computation. 
Due to the complexity of the equations, we obtain an explicit expression for $f_2$ only in specific cases such as the one-dimensional setting. The extension to the higher dimensional setting is left for future work.


In Section \ref{sec:result_bias} we give the general, analytic expression for $f_2$ for one-dimensional splitting schemes in the case in which the velocity takes values $\pm 1$. Observe that in this context the BPS with velocity on the unit sphere and the ZZS coincide, hence this result applies to either process. For the explicit statement we refer the reader to Proposition \ref{prop:f_2_1D}, which is obtained following the approach of \cite{Leimkuhler_rational}.
In Section \ref{sec:onedimtargets_expinvmeas} we give a heuristic argument suggesting that the scheme \textbf{RDBDR}, that is Algorithm \ref{alg:Metropolis_RDBDR_BPS} in the case of BPS, performs well compared to the alternatives and is robust to poor choices of the refreshment rate. We support this argument considering a one-dimensional standard normal target distribution and plotting the TV distance to the truth as given by Proposition \ref{prop:f_2_1D}, as well as giving the results from a numerical simulation. We remark that in the case of ZZS the refreshment rate can be set to $0$ and thus we prefer the splitting \textbf{DBD}.
In Section \ref{sec:inv_meas_RDBDR} we follow a different approach and fully characterise the invariant measure of splitting scheme \textbf{RDBDR} with $\lambda_r\geq 0$ in the one-dimensional case. The result, which can be found in Proposition \ref{prop:mu_delta1D}, is obtained verifying that \textbf{RDBDR} is skew-reversible with respect to a particular perturbation of the true target.


\subsection{Main result}\label{sec:result_bias}

In order to obtain $f_2$ we follow the approach of \cite{Leimkuhler_rational}, which we now briefly illustrate before stating our main result.
First, using the Baker-Campbell-Hausdorff (BCH) formula (see e.g. \citet{bonfiglioli_bch}) we can find $\cL_2$ such that
\begin{equation}\notag
    \mathbb{E}_{x,v}[f(\overline{X}_\delta,\overline{V}_\delta)] = f(x,v)+\delta \mathcal{L}f(x,v) + \delta^3\mathcal{L}_2f(x,v)+\mathcal{O}(\delta^4).
\end{equation}
Here $\cL$ is the infinitesimal generator of the continuous time process. 
Integrating both sides with respect to $\mu_\delta$ and using that $\mu_\delta$ is an invariant measure for the splitting scheme we obtain
\begin{align*}
    \int f(x,v) \mu_\delta(x,v)dxdv &= \int f(x,v) \mu_\delta(x,v)dxdv+\delta \int \mathcal{L}f(x,v) \mu_\delta(x,v)dxdv  \\
    & \quad + \delta^3\int \mathcal{L}_2f(x,v)\mu_\delta(x,v)dxdv+\mathcal{O}(\delta^4).
\end{align*}
Substituting $\mu_\delta$ with the expansion \eqref{eq:ansatz} we have
\begin{align*}
      0&= \delta \int \mathcal{L}f(x,v)\mu(x,v) dxdv-\delta^3 \int \mathcal{L}f(x,v)\mu(x,v)  f_{2}(x,v)dxdv \\
      & \quad + \delta^3\int \mathcal{L}_2f(x,v)\mu(x,v) dxdv+\mathcal{O}(\delta^4)
\end{align*}
and since $\mu$ is an invariant measure for BPS we have $\int \cL f d\mu=0$ which gives
\begin{equation}\label{eq:main_expansion_invmeas}
    \mathcal{L}^* (\mu f_2) = \mathcal{L}_2^* \mu.
\end{equation}
Here $\cL^*$ and $\cL_2^*$ are the adjoints of $\cL$ and $\cL_2$ in $L^2$ with respect to Lebesgue measure\footnote{Note that one could alternatively work in the weighted space $L^2(\mu)$. This proved useful in the case of Langevin diffusions, as shown in \citet{LeimkuhlerMatthewsStoltz}.}. 
A compatibility condition is required to ensure that there is a unique solution to \eqref{eq:main_expansion_invmeas}. Because both $\mu_\delta$ and $\mu$ are probability densities, integrating \eqref{eq:ansatz} gives the requirement 
\begin{equation}\label{eq:compatibility_condition}
    \int f_2(x,v) \mu(x,v)  dx dv = 0.
\end{equation}

Solving the problem defined by \eqref{eq:main_expansion_invmeas}-\eqref{eq:compatibility_condition} is much more complicated for PDMPs than for the Langevin case considered in \citet{Leimkuhler_rational}. This is due to the difficult expressions for the adjoint $\cL^*$, which in general is an integro-differential operator. Nonetheless, we are able to obtain $f_2$ restricting to the one dimensional case with velocity in $\{\pm 1\}$. We are ready to state the main result of this section.

\begin{proposition}\label{prop:f_2_1D}
Consider the one-dimensional setting with state space $\mathbb{R}\times \{ \pm 1\}$ and target distribution $\mu(x,v) = \pi(x)\nu(v)$ with $\pi \propto \exp(-\pot)$ and $\nu = \Unif(\{\pm 1\})$. Let $\lambda_r \geq 0$ be the refreshment rate. 
Then the function $f_2$ that solves \eqref{eq:main_expansion_invmeas}-\eqref{eq:compatibility_condition} is
\begin{align*}
    f_2(x,+1) &= f_2^+(0) + \int_0^x \left( \left(\frac{\lambda_r}{2} +(-\partial \pot(y))_+ \right) g(y)-\frac{\cL^*_2 \mu(y,+1)}{\mu(y,+1)} \right)dy, \\
    f_2(x,-1) & = f_2(x,+1) +g(x),
\end{align*}
where 
\begin{align*}
     g(x) &= \exp \left( \pot(x)\right)\int_{-\infty}^x
     \left( \frac{\cL^*_2 \mu(y,+1)}{\mu(y,+1)} + \frac{\cL^*_2 \mu(y,-1)}{\mu(y,-1)} \right) \exp(-\pot(y))dy, \\
     f_2^+(0) & = -\int_{-\infty}^\infty \Big(  \frac{g(x)}{2} +  \int_0^x \left( \left( \frac{\lambda_r}{2} + (-\partial \pot(y))_+ \right)  g(y)-\frac{\cL^*_2 \mu(y,+1)}{\mu(y,+1)}\right)dy \Big)\pi(x)dx.
\end{align*}
\end{proposition}
\begin{proof}
The proof can be found in Appendix \ref{sec:compute_f2_1d}.
\end{proof}

Therefore, $f_2$ can be obtained for any splitting scheme by plugging the corresponding $\cL_2^*$ and the particular choice of target $\pi$ in Proposition \ref{prop:f_2_1D}. 
For example, in the case of \RDBDR an application of Proposition \ref{prop:f_2_1D} and substitution of the appropriate term $\cL_2^*$ (see Proposition \ref{prop:expansion_adjoint_bps} in Appendix \ref{app:computing_cl2*}) gives
\begin{equation}\label{eq:f_2_RDBDR}
        f_2(x,+1) = f_2(x,-1) = \frac{1}{24}\left( \int_{-\infty}^\infty \pot''(y)\pi(y)dy -  \pot''(x) \right).
    \end{equation}
We observe that in the one dimensional case the second order term of the bias of scheme \RDBDR is always independent of the refreshment rate and of $v$. 

\subsection{{Comparison of splitting schemes}}\label{sec:onedimtargets_expinvmeas}
{Proposition \ref{prop:f_2_1D} gives the analytic expression for $f_2(x,v)$, but it remains difficult to conclude which scheme has the smallest bias for a given target due to the complexity of the operators $\cL_2^*$ (see Proposition \ref{prop:expansion_adjoint_bps} in Appendix \ref{app:computing_cl2*}). Here we give a heuristic argument suggesting that scheme \textbf{RDBDR} is in general a reasonable choice, and support this in the standard Gaussian case from a theoretical and numerical point of view (see \Cref{fig:bias_inv_measure_gauss}). }

{First recall that, neglecting refreshment events, ZZS and BPS can be interpreted as limits of lifted Metropolis-Hastings algorithms based on the deterministic proposal mechanism $(x,v)\mapsto(x+v\delta,v)$, where in fact the jump mechanism arises as limits of the acceptance-rejection steps \citep{BierkensRoberts}. Hence, the jump part \textbf{B} of ZZS and BPS corrects the error incurred by the drift part \textbf{D} and it is then sensible to expect a smaller bias for schemes which divide the drift term in two half steps, interrupted by a correction step \textbf{B}, as opposed to taking a full step of \textbf{D} all at once. Moreover, intuitively it seems best to introduce the refreshment part \textbf{R} without disrupting the interaction between parts \textbf{D} and \textbf{B}. This motivates incorporating refreshment half-steps at the beginning and end of each iteration. For these reasons we expect scheme \textbf{RDBDR} (or alternatively \textbf{DBD} in the case of ZZS) to have the smallest bias. }

\begin{figure}[t]
\begin{subfigure}[t]{0.48\textwidth}
    \includegraphics[width=\textwidth]{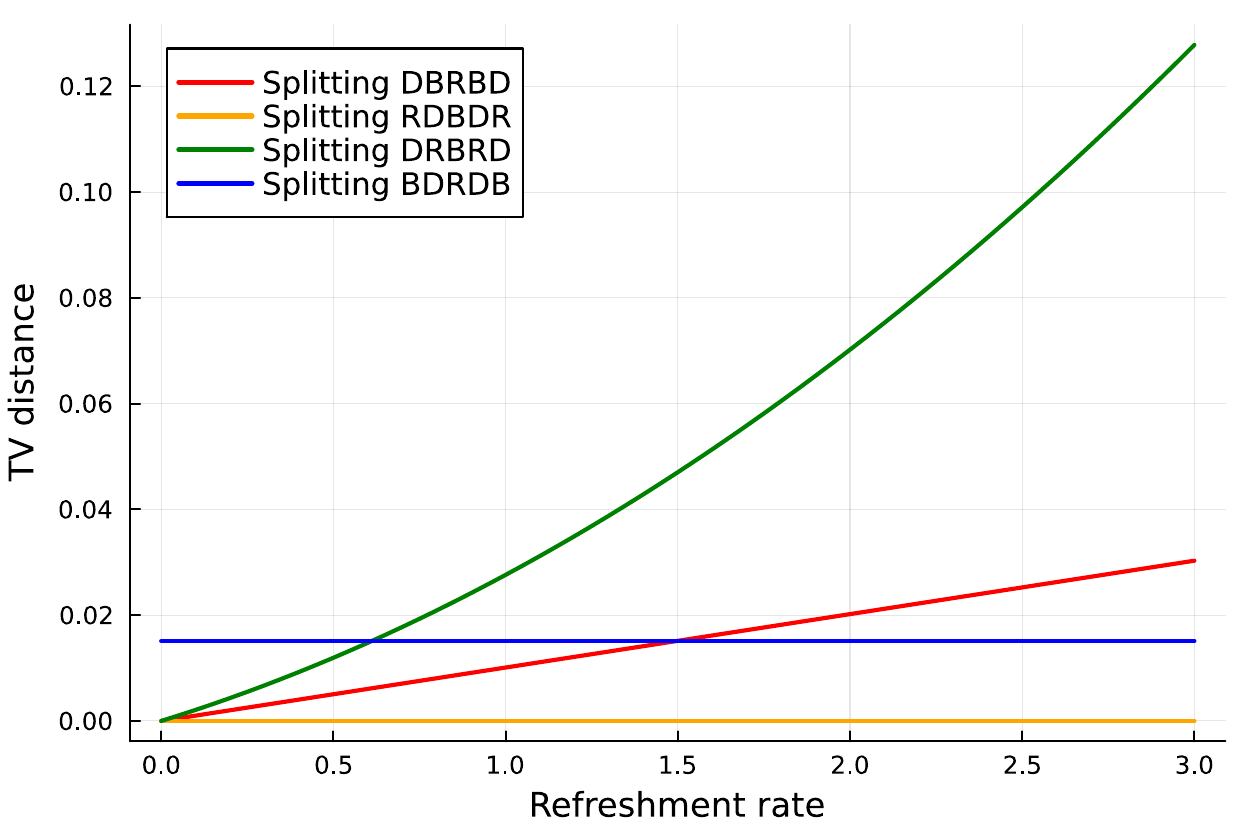}
    \label{fig:TVdist_with_refreshment}
\end{subfigure}
\hfill
\begin{subfigure}[t]{0.48\textwidth}
    \includegraphics[width=\textwidth]{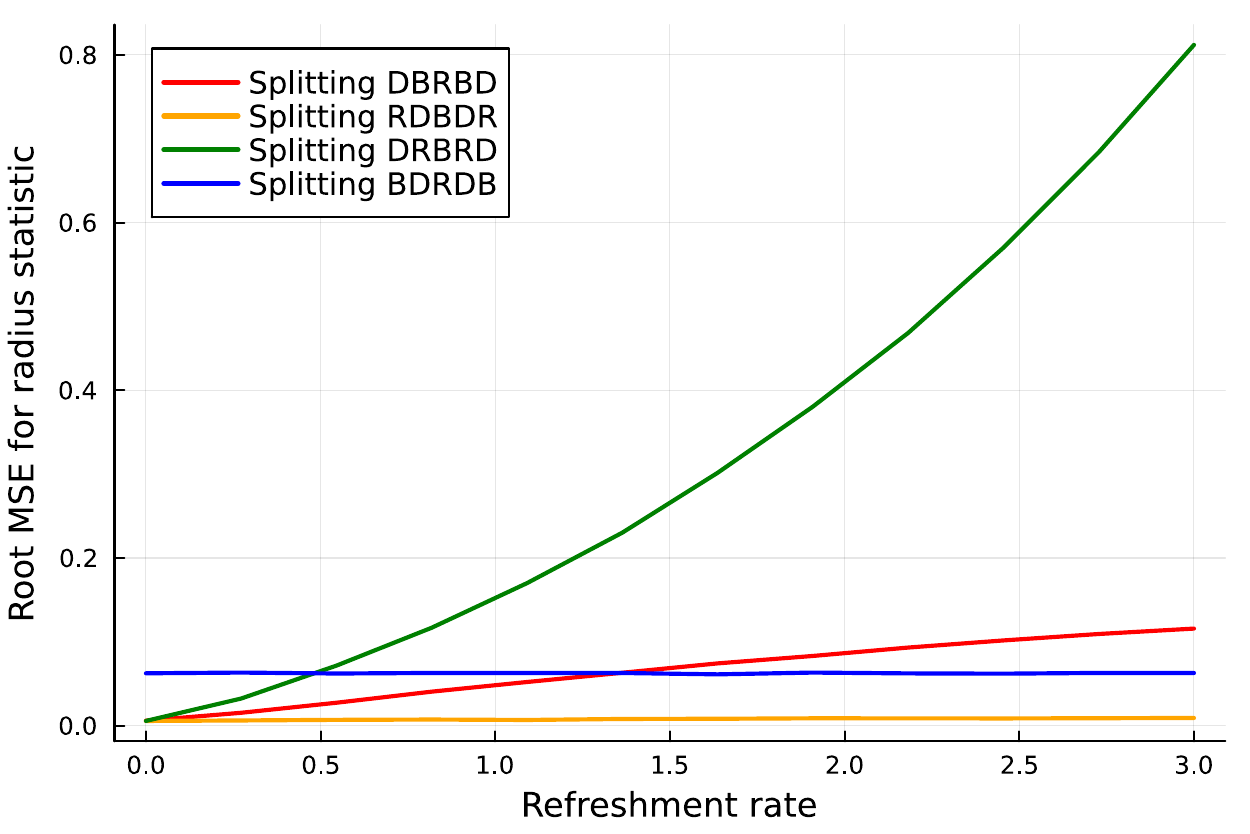}
    \label{fig:bps_empiricalerrror_1d}
\end{subfigure}
\caption{\emph{Left}: total variation distance to the target up to second order obtained with Proposition \ref{prop:f_2_1D}, and \emph{right}: square root MSE in the estimation of the radius statistic, i.e. $x^2$, for a one-dimensional standard normal target distribution. In both plots $\delta = 0.5$, while for the right plot the number of iterations is $N=2 \cdot 10^5$, the position is initialised at the target distribution and the velocity from the uniform distribution on the unit sphere, and the experiment is repeated $300$ times.}
\label{fig:bias_inv_measure_gauss}
\end{figure}

{In \Cref{fig:bias_inv_measure_gauss} we consider a standard normal target and confirm our intuitive arguments comparing scheme \textbf{RDBDR} with schemes \textbf{DRBRD} and \textbf{DBRBD}, which include the refreshment part in between the \textbf{DBD} part, and also with scheme \textbf{BDRDB}, which violates our intuitive principles but is in fact equivalent to \textbf{RDBDR} up to a shift of a half-step.  All these schemes have the same cost of one gradient computation per iteration, since in \textbf{BDRDB} it is sufficient to keep track of the gradient at the previous iteration.
In \Cref{fig:bias_inv_measure_gauss} we consider both the error incurred in the estimation of a test statistic and the theoretical TV distance between the limiting distribution of the schemes and the standard normal distribution. The TV distance is derived from the analytic expression of $f_2$ as follows. By marginalising $\mu_\delta$ with respect to $\nu = \Unif(\{\pm 1\})$ we obtain $\pi_\delta(x) = \pi(x) (1-(\nicefrac{\delta^2}{2})(f_2(x,+1)+f_2(x,-1))) + \mathcal{O}(\delta^4).$
Hence we can express the TV distance between $\pi$ and $\pi_\delta$  as
\begin{equation}\label{eq:tvdist_expinvmeas}
    \lVert \pi- \pi_\delta \rVert_{TV}=  \frac{\delta^2}{2} \sup_A \left\lvert \int_A(f_2(x,+1)+f_2(x,-1))\pi(x)dx \right\rvert + \mathcal{O}(\delta^4).
\end{equation}
The $\delta^2$ contribution of the TV distance can be computed by plugging in the expressions for $f_2$ found in Proposition \ref{prop:f_2_1D}, where we substitute for each splitting scheme the corresponding operator $\cL_2^\star$ obtained in Proposition \ref{prop:expansion_adjoint_bps} in Appendix \ref{app:computing_cl2*}.} 

The test statistic we use for this example is the radius statistic, i.e.  $T(x)=\lVert x\rVert$ the Euclidean norm of the position component. For each time algorithm we estimate the radius statistic using the ergodic average of that chain, i.e.
\begin{equation*}
    \hat{T}(\bar{X}) = \frac{1}{N} \sum_{n=1}^N T(\bar{X}_{t_n})
\end{equation*}
this is compared to the true radius $\pi(T)=1$. By repeating this over $M=300$ independent realisations $\bar{X}^1,\ldots,\bar{X}^M$ of the algorithm we obtain an estimate for the root mean square error (MSE) as
\begin{equation*}
    \sqrt{\frac{1}{M}\sum_{m=1}^M|\hat{T}(\bar{X}^m)-\pi(T)|^2}.
\end{equation*}

{\Cref{fig:bias_inv_measure_gauss} suggests that indeed \textbf{RDBDR} compares well to the other schemes. In particular, its bias is independent of the refreshment rate, contrarily to schemes \textbf{DRBRD} and \textbf{DBRBD}. Finally, scheme \textbf{BDRDB} has a positive bias which is independent of the choice of $\lambda_r$. 
We notice that the theoretical and numerical behaviour show similar dependence on the structure of the scheme and on the refreshment rate, even though the plots in \Cref{fig:bias_inv_measure_gauss} show a different metric for each case.
We refer the reader to \Cref{sec:proofs_propositions_invmeas} for further comparisons, which corroborate our heuristic arguments further.}

\subsection{Characterisation of the invariant measure of \textbf{RDBDR} in one dimension}\label{sec:inv_meas_RDBDR}
In fact, in 1D, for the scheme \textbf{RDBDR}, we can get an explicit expression for the invariant measure. Note that the result is independent of $\lambda_r$ and holds also for the case $\lambda_r=0,$ in which case the scheme coincides with \textbf{DBD}.
\begin{proposition}\label{prop:mu_delta1D}
    Consider the scheme \textbf{RDBDR} for BPS or ZZS in one dimension, where the velocity is refreshed from $\nu = \Unif(\{\pm 1\})$. For some $x\in\R$ and $\delta>0$, define the probability distribution $\mu_\delta$ with support on the grid $\{y\in\R:y=x+n\delta,\, n\in\mathbb{Z}\}\times \{\pm 1\}$ given by
    $$\mu_\delta(y,v)\propto e^{-\pot_\delta(y)},$$
    where $\pot_\delta(x) =\psi(x)$ and for $y=x+nv\delta$, $n\in\N$
    $$\pot_\delta(y) = \psi(x) + \delta  v \sum_{\ell=1}^{n} \pot'(x+(\ell-1/2)v\delta).$$
    Then the distribution $\mu_\delta$ is stationary for the chain which is initialised at $x$ and with step size $\delta$.
    Moreover, under the conditions of Theorem \ref{thm:ergodicity_zzs} we obtain that $\mu_\delta$ is ergodic, in the sense that for all bounded functions
    \[\lim_{N\to\infty} \frac{1}{N}\sum_{n=1}^{N} f(\Xbar_{t_n},\Vbar_{t_n}) = \mu_\delta(f) \quad \mathbb{P}_{x,v}-\text{a.s.}\]
    
\end{proposition}
\begin{proof}
The proof can be found in Appendix \ref{subsec:proof_mu_delta1D}.
\end{proof}


Once again it is clear that the scheme is unbiased in the Gaussian case $\pot(x)=x^2/(2\sigma^2)$ (in the sense that $\psi_\delta(y) = \psi(y)$ for all $y =x+n\delta v$, i.e. the BPS is ergodic with respect to the restriction of the true Gaussian target to the grid, and moreover the target measure is invariant for the scheme). More generally, for $y=x+vn\delta$ with $n\in\mathbb{N}$ we get
\begin{eqnarray*}
    \psi(y) 
    &= &  \psi(x) + \sum_{\ell=1}^n \int_{-\delta/2}^{\delta/2} \psi'(x+v(\ell-1/2) \delta+u) \dd u  \\
    & = & 
    \psi_\delta(y) +  \frac12 \sum_{\ell=1}^n \int_{-\delta/2}^{\delta/2} u^2 \psi^{(3)}(x+v(\ell-1/2) \delta) \dd u + O(n \delta^5) \\
    & = & \psi_\delta(y) +  \frac{\delta^2}{24} \int_x^y \psi^{(3)}(u)\dd u  + O( \delta^4|x-y|)\,.
\end{eqnarray*}
Setting $x=0$ this gives $\pot_\delta=\pot+\delta^2 f_2 + O(\delta^4)$ with $f_2(y) = (\nicefrac{1}{24}) (\pot''(y)- \pot''(0))$, which is the same as what follows from Proposition \ref{prop:f_2_1D} (see Equation \eqref{eq:f_2_RDBDR}) since the term $f_2^+(0)$ in \eqref{eq:f_2_RDBDR} was introduced to make $\exp(-\pot)(1+\delta^2 f_2)$ a probability distribution and would appear also in the present context. Hence Propositions \ref{prop:f_2_1D} and \ref{prop:mu_delta1D} agree.

\section{{Scaling of the rejection probability of adjusted algorithms}}\label{sec:scaling_rej_prob}
In this section we study the average rejection probability of Algorithms \ref{alg:Metropolis_DBD_ZZS} and \ref{alg:Metropolis_RDBDR_BPS}, focusing in particular on its dependence on the step size. The main results of the section are Proposition \ref{prop:meanacceptrate_zzs} and Proposition \ref{prop:meanacceptrate_bps}, dealing with ZZS and BPS respectively. In both cases we consider the average rejection probability, where the initial state is fixed and the randomness is in the proposal for the next state, and prove it has dependence on the step size of order three. {This behaviour is akin to each step of the leapfrog integrator used in HMC}, as shown by \citet{stoltz_trstanova}.

We begin considering the average rejection rate of the Metropolis adjusted ZZS (Algorithm \ref{alg:Metropolis_DBD_ZZS}).
\begin{proposition}\label{prop:meanacceptrate_zzs}
    Suppose $\pot$ is smooth and its derivatives are growing at most polynomially and let $\delta\in[0,\delta_0]$ for $\delta_0>0$. The average rejection probability of Algorithm \ref{alg:Metropolis_DBD_ZZS} satisfies
    $$\E[1-\alpha((x,v),(\tilde X,\tilde V))] = \delta^3 G(x,v) + \tilde G(x,v,\delta),$$
    where $G=G_1+G_2$,
    \begin{align*}
        &G_1(x,v)  =  \frac{1}{2} \max\left(0,\sum_{i=1}^d \lambda_i(x,v) v_i \sum_{k\neq i}  v_k \partial_{ik}\pot(x) \right) \\
        & G_2(x,v) = \frac{1}{24} \max\left(0,-\sum_{\alpha:\lvert \alpha\rvert =3} D^\alpha\pot(x) v^\alpha\right),
    \end{align*}
    while $\tilde G(x,v,\delta)\sim\mathcal{O}(\delta^4)$ for any $x,v$ and grows at most polynomially in $x$ for any $\delta\in[0,\delta_0]$.
\end{proposition}
\begin{proof}
    The proof can be found in Appendix \ref{sec:proofs_rejectionprob}.
\end{proof}

In the next example we apply Proposition \ref{prop:meanacceptrate_zzs} to a factorised target, for which we find that in stationarity the expected value of $G$ is proportional to $\sqrt{d}$ for large $d.$
\begin{example}[Factorised target distribution]
    Consider a potential of the form $\pot(x)=\sum_{i=1}^d \pot_i(x_i).$  Proposition \ref{prop:meanacceptrate_zzs} gives for $(x,v)\sim\mu$ 
    $$G(x,v) \stackrel{d}{=} \frac{1}{24} \max\left(0,\sum_{i=1}^d v_i \pot_i^{(3)}(x_i) \right). $$
    Clearly, this equals zero if $\pot$ is an independent Gaussian distribution, in which case by the proof of the Proposition we also find $\tilde G=0$ and \textbf{DBD} has the correct stationary distribution. Alternatively, consider the assumption $\lvert\pot_i^{(3)}\rvert\leq M$. In this case $\max(0,\sum_{i=1}^d v_i \pot_i^{(3)}(x_i) )\leq M\max(0,\sum_{i=1}^d v_i)$ and we can write $\sum_{i=1}^d v_i =  2\textnormal{Bin}(d,1/2) - d $
    where the equality is in distribution and $\textnormal{Bin}$ denotes the binomial distribution. In the large $d$ regime $\sum_{i=1}^d v_i$ is approximately distributed as a centred Gaussian distribution with variance $d$, hence
    \[\mathbb{E}_\mu\left[G(x,v)\right] \approx \frac{M}{24} \delta^3 \sqrt{d}, \]
    which means that the leading order of the acceptance rate is stable if $\delta$ scales as $d^{-\frac 1 6}$.
\end{example}

Now we focus on the Metropolis adjusted BPS, stating a similar result to Proposition~\ref{prop:meanacceptrate_zzs}. 
\begin{proposition}\label{prop:meanacceptrate_bps}
    Suppose $\pot$ is smooth and its derivatives are growing at most polynomially, and moreover that for the points for which $\lVert\nabla\pot(x)\rVert=0$ it holds that $\det(\nabla^2\pot(x))\neq 0.$ Let $\delta\in[0,\delta_0]$ for $\delta_0>0$. Consider an initial condition $x$ which is a draw from a distribution that is absolutely continuous with respect to $\pi$. Then the average rejection probability of Algorithm \ref{alg:Metropolis_RDBDR_BPS} satisfies
    $$\E[1-\alpha((x,v),(\tilde X,\tilde V))] = \delta^3 G(x,v) + \tilde G(x,v,\delta),$$
    where $G=G_1+G_2$,
    \begin{align*}
        & G_1(x,v) =  \frac{1}{8}\lambda(x,v)\max\left( 0,\langle v,\nabla^2 \pot(x) v\rangle - \langle R(x) v,\nabla^2 \pot(x) R(x)v\rangle\right) \\
        & G_2(x,v) = \frac{1}{24} \max\left( 0,-\sum_{\alpha:\lvert \alpha\rvert =3} D^\alpha\pot(x) v^\alpha\right),
    \end{align*}
    while $\tilde G(x,v,\delta)\sim o(\delta^3)$ for any $x,v$ and grows at most polynomially in $x$ for any $\delta\in[0,\delta_0]$.
\end{proposition}
\begin{proof}
    The proof can be found in Appendix \ref{sec:proofs_rejectionprob}.
\end{proof}

\section{Numerical experiments}\label{sec:numerical}
In this section we discuss some numerical simulations for the proposed samplers. The codes for all these experiments can be found at \url{https://github.com/andreabertazzi/splittingschemes_PDMP}.

\subsection{Image deconvolution using a total variation prior}\label{sec:imaging_experiments}

In this section we test the unadjusted ZZS (Algorithm \ref{alg:splitting_DBD_ZZS}) on an imaging inverse problem, which we solve with a Bayesian approach. In the following we shall refer to an image either as a $N\times N$ matrix or as a vector of length $d=N^2$, which is obtained by placing each column of the matrix below the previous one. In both cases each entry corresponds to a pixel.
Now denote as $x\in \R^d$ the image we are interested in estimating and $y\in \R^d$ the observation. The observation is related to $x$ via the statistical model
\begin{equation*}
    y=Ax+\xi,
\end{equation*}
where $A$ is a $d\times d$-dimensional matrix which may be degenerate and ill-conditioned and $\xi$ a $d$-dimensional Gaussian random variable with mean zero and variance $\sigma^2 I_d$. 
The forward problem we consider is given by a blurring operator, i.e. $A$ acts by a discrete convolution with a kernel $h$. In our examples $h$ will be a uniform blur operator with blur length $9$.
The \emph{likelihood} of $y$ given $x$ is
\begin{equation*}
    \begin{aligned}
        &p(y|x) \propto e^{-f_y(x)},\quad \textnormal{for } f_y(x) = \frac{1}{2\sigma^2} \lVert Ax-y\rVert^2.
    \end{aligned}
\end{equation*}
As \emph{prior distribution} on $x$ we choose the total variation prior: $ p(x) \propto e^{-g(x)},$ where $g(x) =\theta \lVert x\rVert_{TV}$ for $\theta>0$ and $\lVert x\rVert_{TV}$ is the total variation of the image $x$ (see \cite{rudin1992nonlinear}) and is given by
\begin{equation*}
    \lVert x\rVert_{TV} := \sum_{i,j=1}^{N-1} (\lvert x_{i+1,j}-x_{i,j}\rvert +\lvert x_{i,j+1}-x_{i,j}\rvert ).
\end{equation*}
{The total variation prior is a log concave non-differential improper prior often used as a benchmark for Bayesian Imaging problems \citep{PMZ20,oliveira2009adaptive, P16}.}
The total variation prior corresponds to the $\ell^1$-norm of the discrete gradient of the image and promotes piecewise constant reconstructions. Note that this prior is not smooth and hence we cannot directly apply the gradient based algorithms such as the unadjusted Langevin algorithm (ULA) or ZZS. Therefore we approximate $g$ with a Moreau-Yosida envelope
\begin{equation*}
    g^\lambda(x)  = \min_{z\in \R^d} \left\{g(z)+\frac{1}{2\lambda}\lVert x-z\rVert^2\right\}.
\end{equation*}
By \citet[Proposition 12.19]{rockafellar2009variational} we have that $g^\lambda$ is Lipschitz differentiable with Lipschitz constant $\lambda^{-1}$ and 
\begin{align*}
    &\nabla g^{\lambda} (x) = \frac{1}{\lambda}(x-\prox_g^\lambda (x)),\quad \textnormal{where }\,\, \prox_g^\lambda (x )=  \arg\min_{z\in \R^d} \left\{g(z)+\frac{1}{2\lambda}\lVert x-z\rVert^2\right\}.
\end{align*}
Using Bayes theorem, we have the \emph{posterior distribution}
\begin{equation*}
    p(x| y) \propto e^{-f_y(x)-\theta g^{\lambda}(x)}.
\end{equation*}
We select the optimal $\theta$ by using the SAPG algorithm \citep{vidal2020maximum,de2020maximum} and we choose $\lambda$ based on the guidelines given in \citet{DurmusMoulinesPereyra}, which set {$\lambda=5/L_f$} where $L_f$ is the Lipschitz constant of $f_y$. {Despite this being a log concave probability distribution}, sampling from this model using MCMC schemes is difficult because $x$ is usually very high dimensional and the problem is ill-conditioned. {In order to be able to characterise the nature of the conditioning we introduce a strongly convex term to the posterior, and consider as target distribution}
\begin{eqnarray}\label{eq:posterior_tv}
    {\pi(x) \propto e^{-f_y(x)-\theta g^{\lambda}(x) -\frac{1}{2}m\lVert x\rVert^2}}.
\end{eqnarray}
In this case the unadjusted Langevin algorithm can be very expensive to run since the step size is limited by $2/L$, where $L=L_f+\lambda^{-1}{+m}$ is the Lipschitz constant of $\nabla \log \pi$. The values of the parameters in the considered example are summarised in Table~\ref{tab:imaging_params}. 

We are interested in drawing samples from the distribution \eqref{eq:posterior_tv}, and in particular we compare the unadjusted ZZS (Algorithm \ref{alg:splitting_DBD_ZZS}, abbreviated as UZZS in the plots), ULA, the continuous ZZS, as well as {the discretization of the underdamped Langevin Algorithm considered by \citet{KZSS}, which is a strongly second order accurate integrator abbreviated as UBU in the plots.} 

To implement the continuous ZZS we can compute the Lipschitz constant of the gradient of the negative log-posterior, $L$, and thus we can implement the exact ZZS using the Poisson thinning technique based on the simple bound
\begin{equation}\label{eq:bounds_zzs_tv}
    \lambda_i(x+vt,v) \leq t L \sqrt{d}  + \lambda_i(x,v).
\end{equation}

In order to compare the computational cost of the continuous ZZS to the unadjusted ZZS, {ULA and UBU} we count each proposal for an event time obtained by Poisson thinning as a gradient evaluation and thus as an iteration. Indeed, an update of the computational bounds requires the evaluation of $\lambda_i(x,v)$ {for some $i$ and thus the proximal operator has to be calculated. This is the most relevant cost in the algorithm so we view it as equivalent cost to one step of ULA, UZZS and UBU}. 
To estimate the posterior mean for the continuous ZZS we compute the time average $T^{-1}\int_0^T X_t dt$. For each of these algorithms ULA, UZZS and continuous ZZS we have used $10^6$ gradient evaluations to obtain the estimates, {this includes a burn in phase which uses $10\%$ of the iterations}.

In order to have the best convergence speed for ULA we use a step size of {$L^{-1}$}, for UZZS there is not a stability barrier so we may take much larger step size and for these experiments we use ${2L^{-1/2}}$ as the step size.  {For UBU the performance based on theoretical guarantees is not competitive since this algorithm scales poorly with the conditioning number \citep{cheng2018underdamped, durmus2016sampling}, which is very large in these examples. However we have implemented UBU with a friction parameter $\gamma=2$, which is outside the range covered by the theoretical guarantees, and with step size $L^{-\frac{1}{2}}$. Note that although the convergence rate for the discretisation is not known for this regime, the continuous SDE converges with a rate $O(\sqrt{m})$ for $m$-strongly log-concave probability measures \citep{Cau23}. Therefore using a step-size of $O(L^{-\frac{1}{2}})$ has the potential of achieving a convergence rate $O(\kappa^{-\frac{1}{2}})$.}  For comparison of the bias we have used a long run ($10^7$ iterations) of the MYMALA algorithm \citep{DurmusMoulinesPereyra} which produces unbiased samples of $\pi$ including the Moreau-Yosida approximation of $g$, and we use these samples as our reference for the true distribution. Note that MYMALA mixes substantially slower than ULA, UZZS, and UBU hence is not competitive with these algorithms in terms of convergence speed.

\begin{table}[t]
    \centering
    \begin{tabular}{c|c|c|c|c|c|c|}
        $\sigma$ & L & d & $\lambda$ & $\theta$ & $m$\\
        \hline
        $0.0024$ & $2.12\times 10^{5}$ & $65536$ & $2.83\times 10^{-5}$ & $10.74$ & 1\\
    \end{tabular}
    \caption{{Summary of the problem parameters for Section~\ref{sec:imaging_experiments}.}}
    \label{tab:imaging_params}
\end{table}
\begin{table}[ht]
    \centering
    \begin{tabular}{|c|p{3.5cm}p{3.5cm}p{3.5cm}|}
    \hline
        Algorithm & {IMMSE} of posterior mean vs $x$ & {IMMSE} of posterior mean vs MALA & {IMMSE} of posterior standard deviation vs MALA\\
         \hline 
         ULA & $1.9 \times 10^{-3}$ & $6.4\times 10^{-5}$ & $3.1\times 10^{-5}$\\
         UZZS & $1.9 \times 10^{-3}$ & $\mathbf{5.9\times 10^{-5}}$ & $2.9\times 10^{-5}$\\
         UBU & $\mathbf{1.9 \times 10^{-3}}$ & $5.9\times 10^{-5}$ & $\mathbf{2.9\times 10^{-5}}$\\
         cts ZZS & $5.2 \times 10^{-3}$ & $2.7\times 10^{-3}$ & $2.7\times 10^{-3}$\\
         \hline
    \end{tabular}
    \caption{A summary of the error calculated by the {IMMSE} for two different statistics: the posterior mean and the pixelwise posterior standard deviation. We compare the posterior mean against the ground truth $x$ and against the posterior mean obtained using a long run of MYMALA. For the pixelwise posterior standard deviation we compare against a long run of MYMALA.}
    \label{tab:imaging_results}
\end{table}

{Figure \ref{fig:cman_tv} shows} the ground truth image $x$, the observed data $y$, the estimated posterior mean using the different samplers and the true posterior mean obtained using a long run of MYMALA. These figures show that the posterior mean obtained via {UZZS, UBU or ULA appear visually} the same as the true posterior mean whereas the continuous {ZZS sampler has} clearly not converged after $10^6$ gradient evaluations. {We have plotted the standard deviation for each pixel which provides an indication
of the level of confidence in each pixel value, as measured by the model} 
and the results are shown in Figure~\ref{fig:stdev}. Using the long run of MYMALA for comparison we see that ULA and UZZS provide a good estimate of the standard deviation of the model, {however, looking closely at the boundary of the image one can see that ULA, UBU and UZZS are overestimating the standard deviation at these pixels}. {As a measure of distance between two images we use the image mean square error (IMMSE) defined as}
\begin{equation*}
{    \mathrm{IMMSE}(x,y)  = \frac{1}{N^2}\sum_{i,j=1}^N |x_{i,j}-y_{i,j}|^2}
\end{equation*}
The {IMMSE} between the pixelwise standard deviation is reported in Table \ref{tab:imaging_results} along with the {IMMSE} of the posterior mean, these show UZZS consistently has smaller error than ULA and both have significantly smaller {IMMSE} than continuous ZZS.

\begin{figure}
\centering
\begin{subfigure}[t]{0.24\linewidth}
    \centering
    \includegraphics[width=\linewidth]{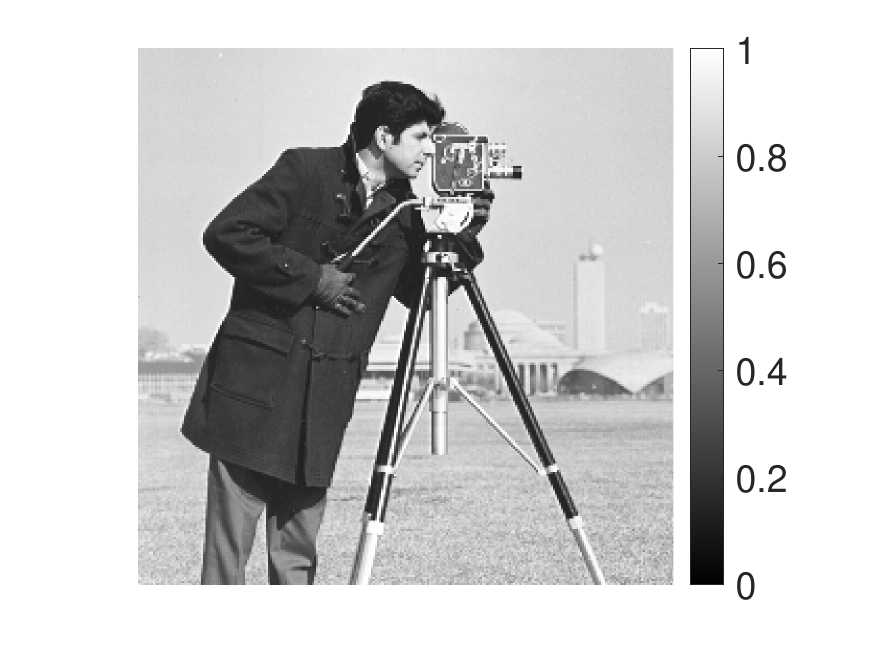}
    \caption{Ground truth, $x$}
\end{subfigure}
\begin{subfigure}[t]{0.24\linewidth}
    \centering
    \includegraphics[width=\linewidth]{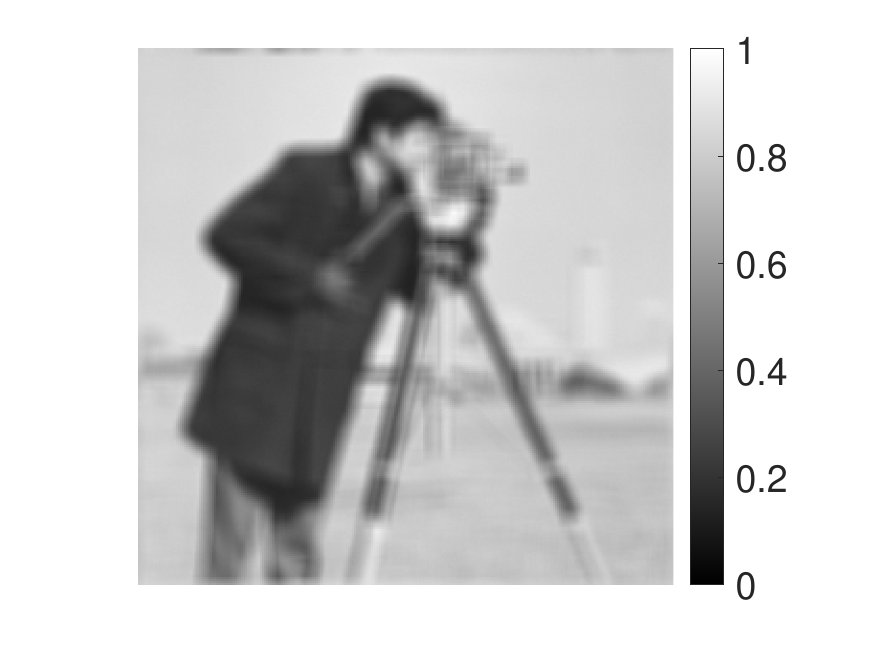}
    \caption{Data, $y$}
\end{subfigure}
\begin{subfigure}[t]{0.24\linewidth}
    \centering
    \includegraphics[width=\linewidth]{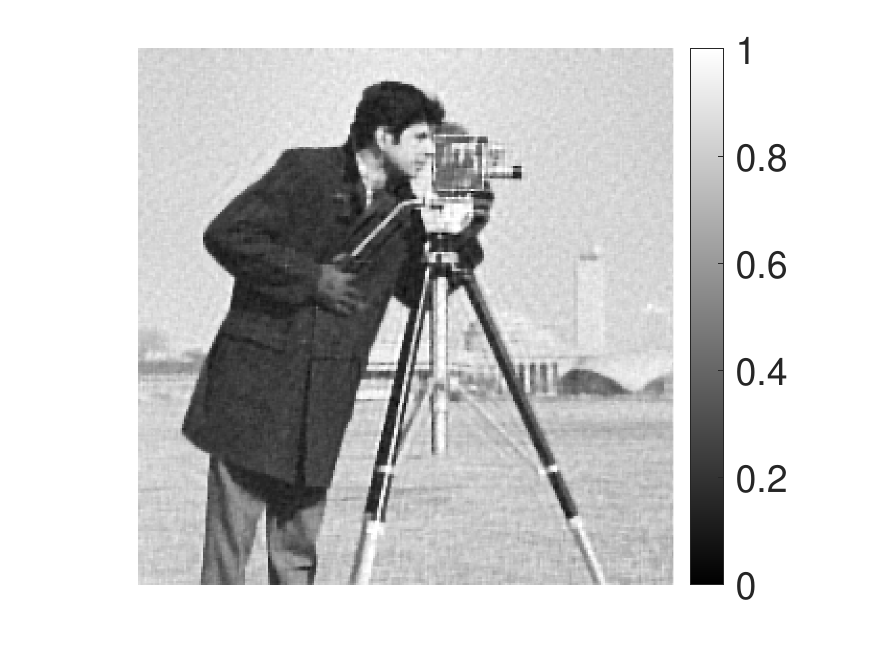}
    \caption{MYMALA}
\end{subfigure}\\
\begin{subfigure}[b]{0.24\linewidth}
    \centering
    \includegraphics[width=\linewidth]{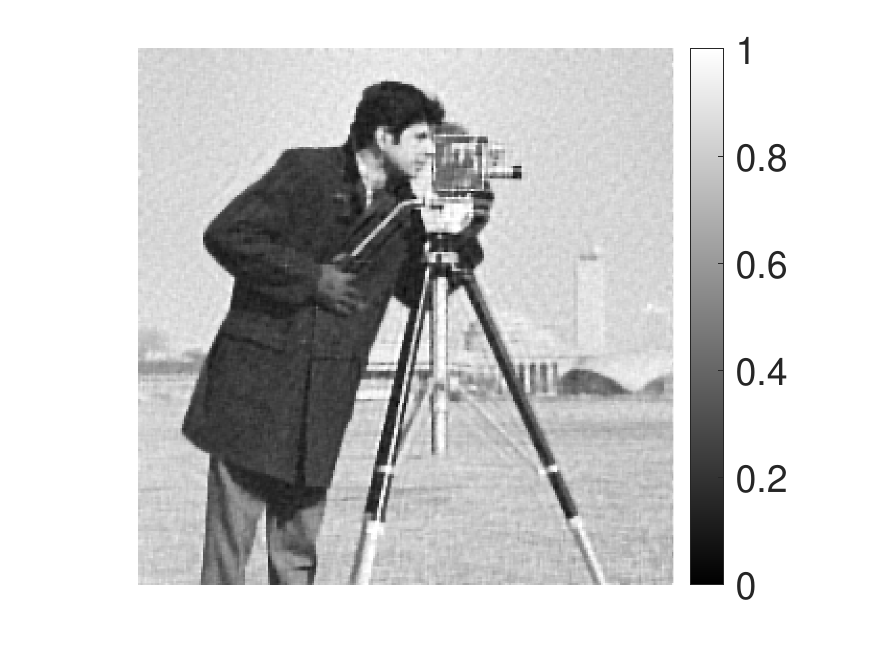}
    \caption{ULA}
\end{subfigure}
\begin{subfigure}[b]{0.24\linewidth}
    \centering
    \includegraphics[width=\linewidth]{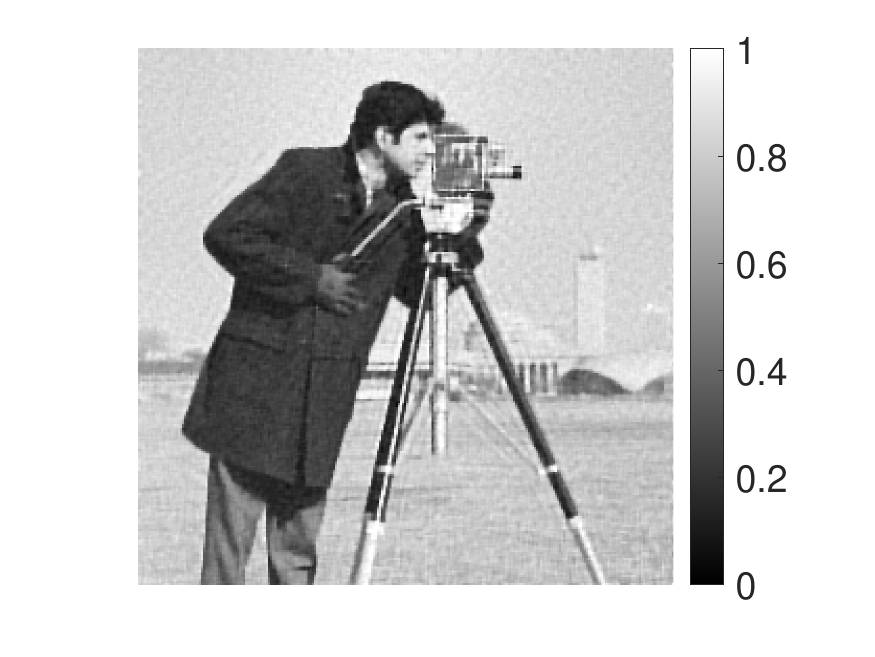}
    \caption{UBU}
\end{subfigure}
\begin{subfigure}[b]{0.24\linewidth}
    \centering
    \includegraphics[width=\linewidth]{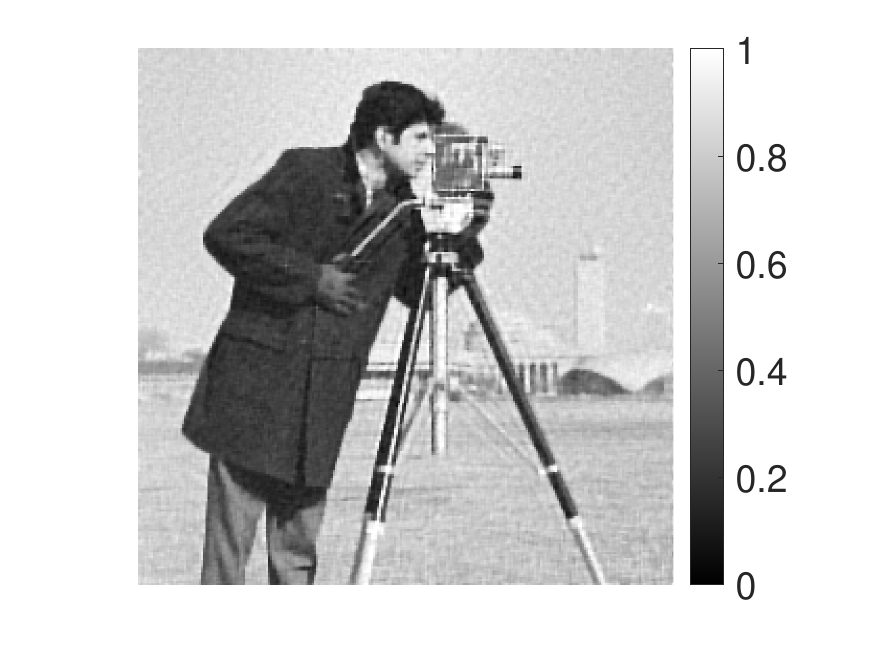}
    \caption{UZZS}
\end{subfigure}
\begin{subfigure}[b]{0.24\linewidth}
    \centering
    \includegraphics[width=\linewidth]{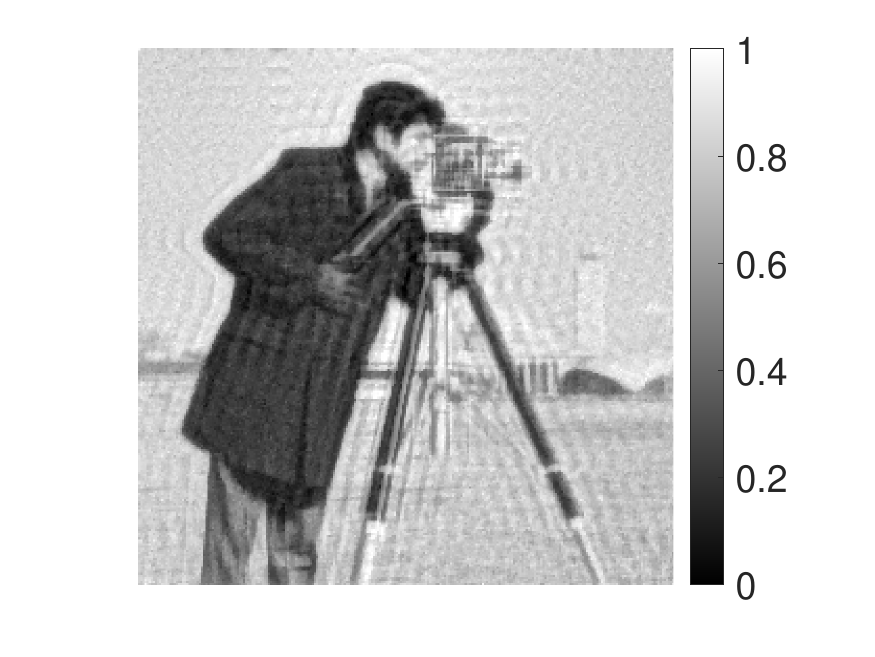}
    \caption{ZZS continuous}
\end{subfigure}
\caption{Results for the reconstruction of cameraman image using a TV prior. The first row shows the original image $x$, the data $y$ and the posterior mean using a long run of MYMALA. Each panel (d)-(i) shows the posterior mean obtained after $10^6$ iterations with the labelled algorithm.}
\label{fig:cman_tv}
\end{figure}
\begin{figure}[t]
    \centering
    \begin{subfigure}[t]{0.24\linewidth}
        \includegraphics[width=\linewidth]{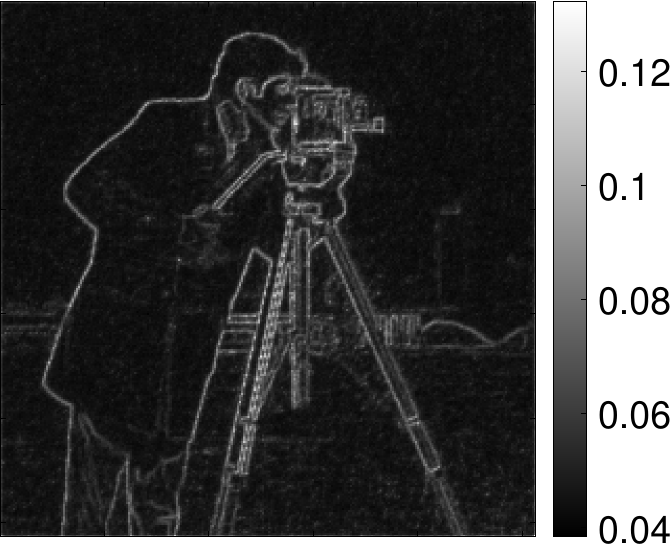}
        \caption{ULA}
    \end{subfigure}
    \begin{subfigure}[t]{0.24\linewidth}
        \includegraphics[width=\linewidth]{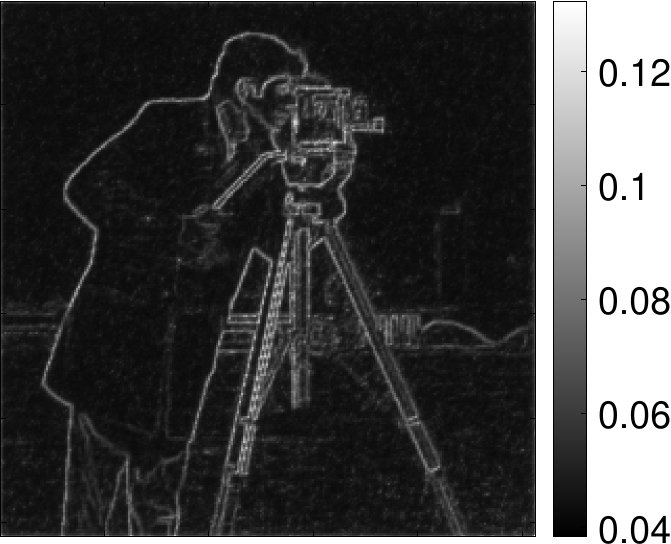}
        \caption{UBU}
    \end{subfigure}
    \begin{subfigure}[t]{0.24\linewidth}
        \includegraphics[width=\linewidth]{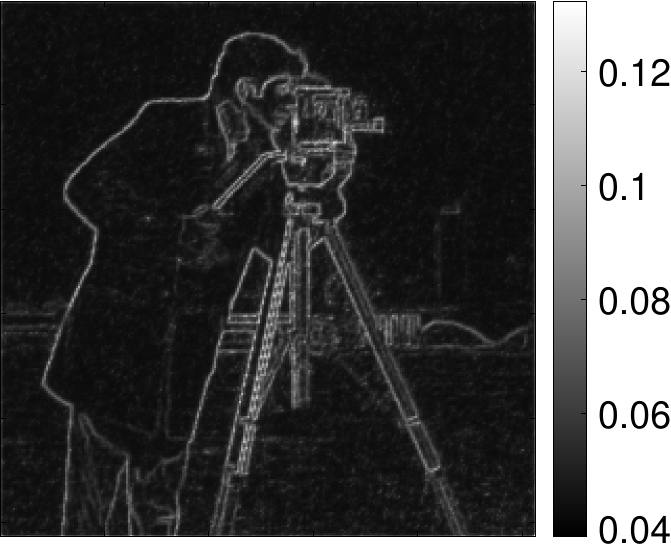}
        \caption{UZZS}
    \end{subfigure}
    \begin{subfigure}[t]{0.24\linewidth}
        \includegraphics[width=\linewidth]{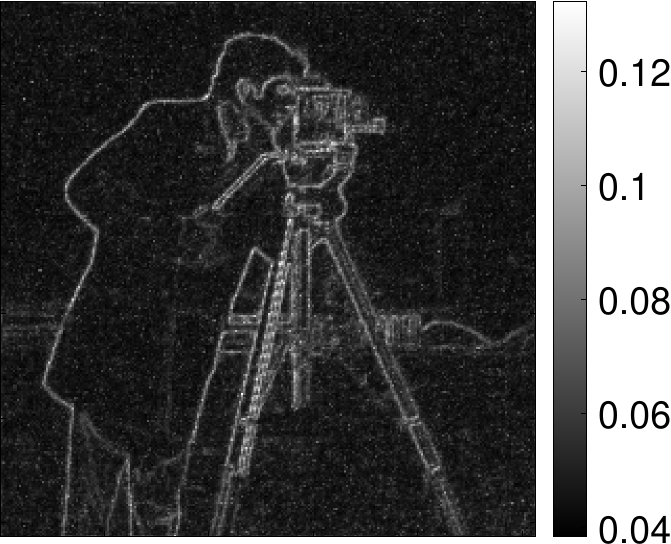}
        \caption{MYMALA}
    \end{subfigure}
\caption{Results for the pixelwise standard deviation for the cameraman image using a TV prior. 
}
\label{fig:stdev}
\end{figure}
\begin{figure}[t]

\begin{subfigure}{0.49\textwidth}
    \centering
    \includegraphics[width=0.85\textwidth]{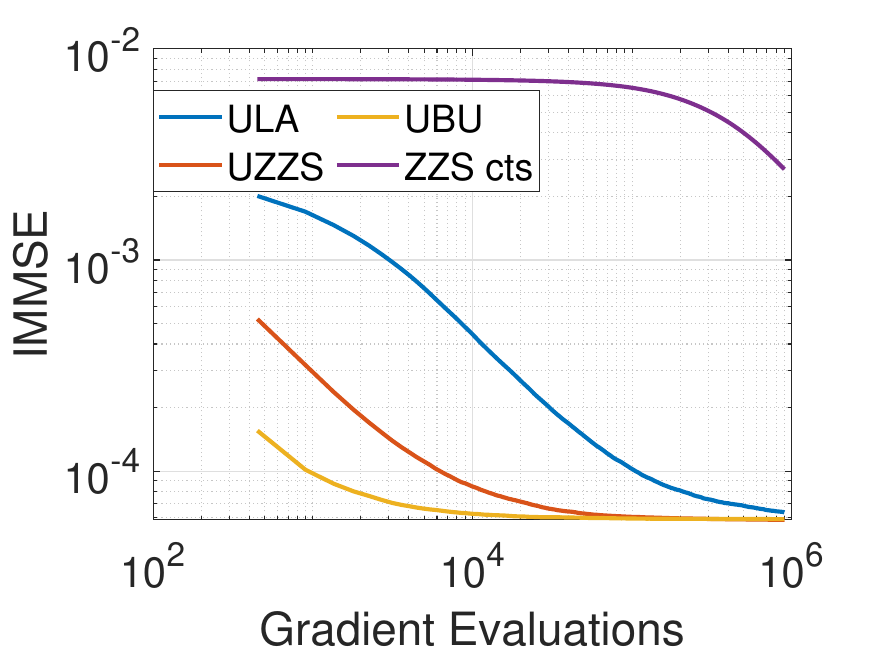}
\end{subfigure}
\begin{subfigure}{0.49\textwidth}
    \centering
    \includegraphics[width=0.85\textwidth]{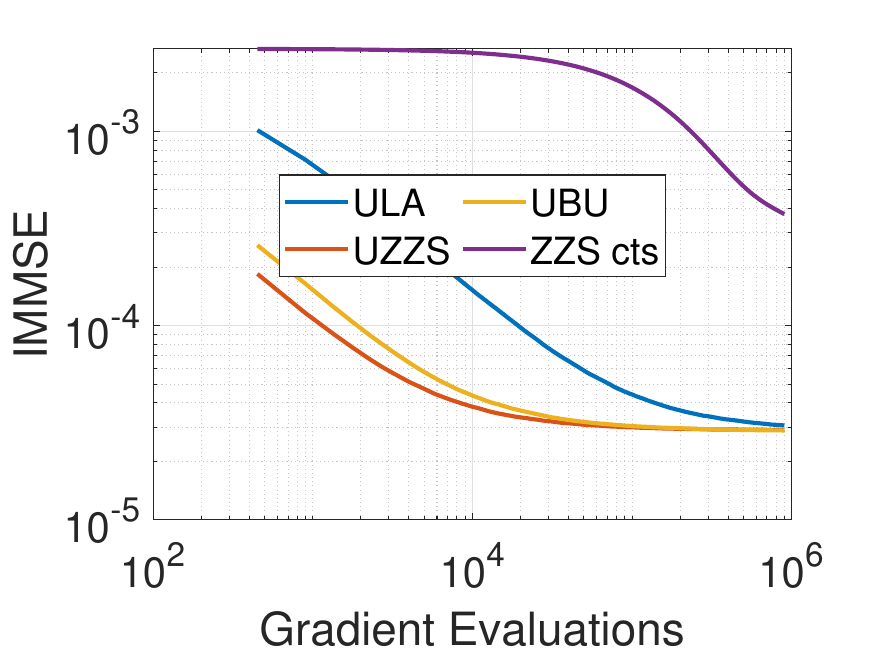}
\end{subfigure}
\caption{{IMMSE} between posterior mean (\emph{left}) and pixelwise posterior {standard deviation} (\emph{right}) estimated by the algorithms {ULA, UBU, UZZS, and the continuous ZZS} using a long run of MYMALA as a reference for the truth.}
\label{fig:mse_mala}
\end{figure}
The convergence is further investigated in Figure \ref{fig:mse_mala}. Figure \ref{fig:mse_mala} shows the {IMMSE} using MYMALA as the truth for the mean and pixelwise variance by each algorithm (ULA, {UBU}, UZZS and continuous ZZS). These show that UZZS converges the fastest of the {four} algorithms in terms of the {standard deviation while UBU performs the best for the first moment}.  
This is likely due to the fact that the step size of ULA must be very small or otherwise the process explodes to infinity, while for UZZS larger step sizes can be selected and in these experiments the step size is $500$ times larger for UZZS. 
This constitutes a major difference because every iteration is computationally intensive since the dimension is very high (see Table \ref{tab:imaging_params}) and notably each iteration involves solving an optimisation problem, which in our simulations is solved by the Chambolle-Pock algorithm \citep{chambolle2011first}.

Finally, let us compare the unadjusted ZZS with the continuous time ZZS. It is clear from our experiments that ZZS performs poorly compared to its discretisation. The reason is twofold. First, the major drawback of Poisson thinning using the bounds \eqref{eq:bounds_zzs_tv} is that a considerable proportion of the proposed event times are rejected (in our examples the rejection rate is around $70-80\%$). Moreover, the rates $\lambda_i$ are very large in the current framework and the process can have even $10^9$ switches per continuous time unit. {In comparison, the ULA requires $2\times 10^5$ iterations per unit time, while UZZS and UBU require $460$ iterations per unit time for this example.} This means that {significantly more} gradient computations are required to travel a reasonable distance { for the continuous time ZZS} and thus the process itself is expensive to run. The combination of these two phenomena implies an important loss of efficiency, which explains the results of our simulations.

\subsection{{Chain of interacting particles}}\label{subsec:example_particles}

In this section we consider a problem which will serve as an illustration of a typical context where ZZS is favoured with respect to other samplers. This toy model presents features that are similar to the molecular system considered in \citet{weisman}, where splitting schemes involving velocity bounces have proven efficient. We consider a chain of $N$ particles in 1D, labelled $1$ to $N$. The particles interact through two potentials: a chain interaction, where the particle $i$ interacts with the particles $i-1$ and $i+1$; and a mean-field interaction, where each particle interacts with all the others. For $x\in\R^N$, the potential is thus of the form
\[\psi(x) = \sum_{i=1}^{N-1} V(x_i-x_{i+1}) + \frac{1}{2N} \sum_{i,j=1}^N W(x_i-x_j)\,, \]
where $V$ is the chain potential and $W$ is the mean-field potential.  In the following we take
\[V(s) = s^4, \qquad W(s) = -\sqrt{1+s^2}\,,\]
for $s\in\R$, i.e. the chain interaction is an anharmonic quartic potential which constrains two consecutive particles in the chain to stay close, while the mean-field interaction induces a repulsion from the rest of the system. Although this specific $\psi$ is an academic example meant for illustration purpose, its general form is classical in statistical physics. 

This target distribution presents two significant challenges: the term representing the mean-field interactions requires $O(N^2)$ computations, which becomes impractical for large $N$; the quartic potential causes numerical instabilities for standard diffusion-based samplers, which then require truncations of their drifts to avoid diverging at infinity. Our samplers are able to tackle both challenges without any modifications.

We shall consider as test statistic the empirical variance of the system of particles around their barycentre, that is
\begin{equation}\label{eq:test_particles}
    v(x)= \frac1{2N^2} \sum_{i,j=1}^N \po x_i-x_j\pf^2\,.
\end{equation}
Notice that $\psi$ is invariant by translation of the whole system, so that $e^{-\psi}$ is not integrable on $\R^N$. However, we are interested in a translation-invariant function, so we consider $e^{-\psi}$ as a probability density on the subspace $\{x\in \R^N,\bar x =0\}$, which amounts to looking at the system of particles from its centre of mass. In practice, we run particles in $\R^N$  without constraining their barycentre to zero, which does not change the output as long as we estimate the expectations of translation-invariant functions. 

The forces $\nabla \psi$ can be decomposed in two parts, one of which (the chain interaction) is unbounded and not globally Lipschitz but is relatively cheap to compute (with a complexity $O(N)$), while the second part (the mean-field interaction) is bounded but numerically expensive (with a complexity $O(N^2)$). If this decomposition is not taken into account and we simply run a classical MCMC sampler based on the computation of $\na \psi$, then the step size is required to be very small due to the non-Lipschitz part of the forces, leading to a poor numerical efficiency since  the costly mean-field force has to be computed at  each step. Alternatively, in order to simulate a continuous-time PDMP (e.g. the BPS) via thinning for this model, the non-Lipschitz part of the potential would induce high-order polynomial bounds on the jump rate, leading to many jump proposals and a poor numerical efficiency. This issue would be even more critical with 3D particles and singular potentials such as the Lennard-Jones one considered in~\citet{weisman}.

Now, as was already discussed in Section~\ref{subsec:subsamp} for subsampling, PDMPs and their splitting schemes can be used with a splitting of the forces\footnote{The word \emph{splitting} is here used for two different notions which should not be confused; one the one hand, the Strang splitting~\eqref{eq:Strang} used to define a discrete-time scheme, and on the other hand the decomposition of the forces $\na \psi$  in several parts, each part being treated with a jump mechanism. It is possible to do a splitting scheme of a PDMP without splitting the forces, as it is possible to sample by thinning a continuous-time PDMP with splitted forces (e.g. the ZZS where $\na\psi$ is splitted on the canonical basis). By contrast, for multi-time-step methods for HMC, splitting forces is directly related to splitting the operator $-\nabla\psi\cdot  \na_v$.}. Let us illustrate the idea for ZZS. We consider a ZZS where the switching rate of the $i$-th velocity has the form
\[\lambda_i(x,v) = \po v_i (V'(x_i - x_{i+1})-V'(x_{i-1}-x_i))\pf_+  + \frac{1}N  \sum_{j\neq i} \po v_i W'(x_i-x_j)\pf_+, \]
where for the particles $1$ and $N$ we set $x_0=x_1$ and $x_{N+1}=x_N$ to cancel out the corresponding terms. The corresponding continuous-time ZZS has the correct invariant measure (once centred) as shown by \citet{ZZ}. We consider our \textbf{DBD} splitting to approximate this ZZS.
In particular, we shall take advantage of the fact that $|W'(s)| \leqslant 1$ for all $s\in\R$ to use the Poisson thinning technique \citep{lewis_shedler_thinning} to deal with the mean-field interactions.
We can sample the jump times of the $i$-th velocity as follows. First, we sample two jump times with rates respectively $\po v_i (V'(x_i - x_{i+1})-V'(x_{i-1}-x_i))\pf_+$ and $1$. The latter rate is an upper bound for the mean field force. If both times are larger than the step size $\delta$, then the velocity is not flipped. Else, if the time corresponding to the first rate is smaller than $\delta$ and than that corresponding to the second, then we flip the $i$-th velocity. Alternatively, if the second time is smaller than $\delta$ and than the first, we draw $J\sim \Unif(\{1,\dots,N\})$ and we flip the sign of the $i$-th velocity with probability $\po v_i W'(x_i-x_J)\pf_+$. Since in this case the rates are not canonical due to the splitting of forces, more than one jump per component is possible and thus this procedure is repeated until the end of the time step has been reached.
This results in $O(1)$ computations per particle on average, hence $O(N)$ for the whole system.
For BPS, we consider the splitting scheme \textbf{RDBDR}, where part \textbf{B} uses similar ideas as above. The main difference lies in the fact that jumps due to the mean-field interaction are obtained using the reflection operator that is defined by the same force that caused the event time. Pseudo-codes for the jump part of ZZS and BPS can be found in Appendix~\ref{sec:pseudocodes_particles}.

We shall compare our methods to a suitable Hamiltonian Monte Carlo (HMC) algorithm. First of all, we consider the variant of HMC that has refreshments from the Laplace distribution \citep{Nishimura05122024}, since this is more suitable to deal with the quartic interaction \citep{Livingstone_Faulkner}. Then, we introduce a splitting of the forces and construct a multi-time-step method in accordance to the ideas introduced e.g. in \citet{tuckerman1992reversible,Morrone_etal,Neal,isokin,shahbaba2014split,lagardere2019pushing}. The general idea is to split the generator as $\mathcal L = \mathcal L_1 + \mathcal L_2$,  where $\mathcal L_1$ represents the Hamiltonian dynamics with only the ``cheap" part of the forces $\na \psi$, denoted as $\na \psi_1$, while $\cL_2 = \nabla \psi_2 \cdot \na_v$ where $\na\psi_2$ represents the most expensive part of the forces. Given a step-size $\delta$ and some $K \in \mathbb N$, one starts with a Strang splitting approximation
\[e^{\delta \mathcal L} \simeq e^{\frac{\delta}2 \mathcal L_2}\po e^{\frac{\delta}{K} \mathcal L_1}\pf^K e^{\frac{\delta}2 \mathcal L_2}\,.\]
Then, we use a splitting scheme to approximate $e^{\frac{\delta_1}{K} \mathcal L_1}$, and we take a stochastic gradient version of the expensive force in $\cL_2$, that is we draw $J\sim\Unif(\{1,\dots,N\})$ and we consider only the force between each particle and the $J$-th particle.
This approximation is then repeated $M$ times for some $M\in\mathbb N$. The reason we do this stochastic gradient approximation is that otherwise HMC is prohibitively slow due to the $\mathcal O(N^2)$ complexity of the mean-field interaction. The pseudo-code for this HMC algorithm is given in Appendix~\ref{sec:pseudocodes_particles}.

\begin{figure}[t]
\centering
\begin{subfigure}{0.99\textwidth}
  \centering
    \includegraphics[width=0.49\textwidth]{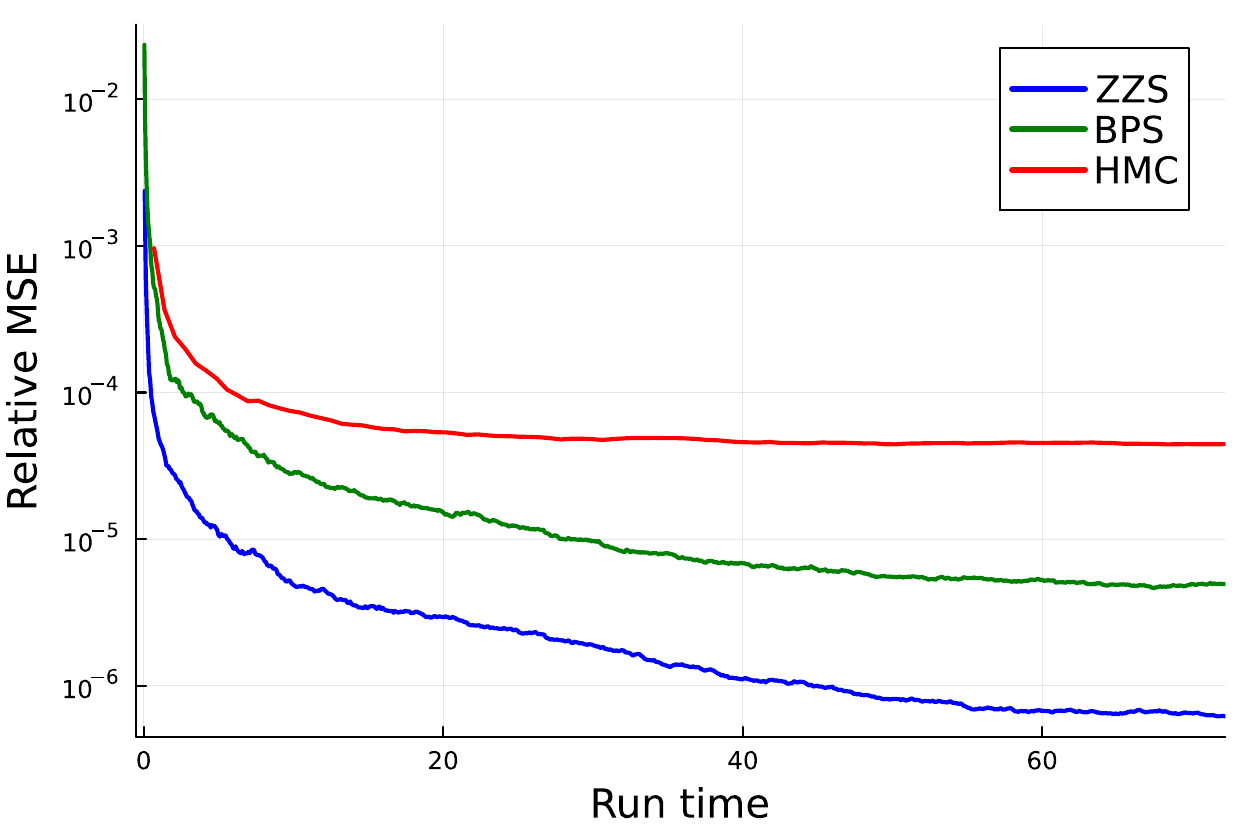}
    \hfill
    \includegraphics[width=0.49\textwidth]{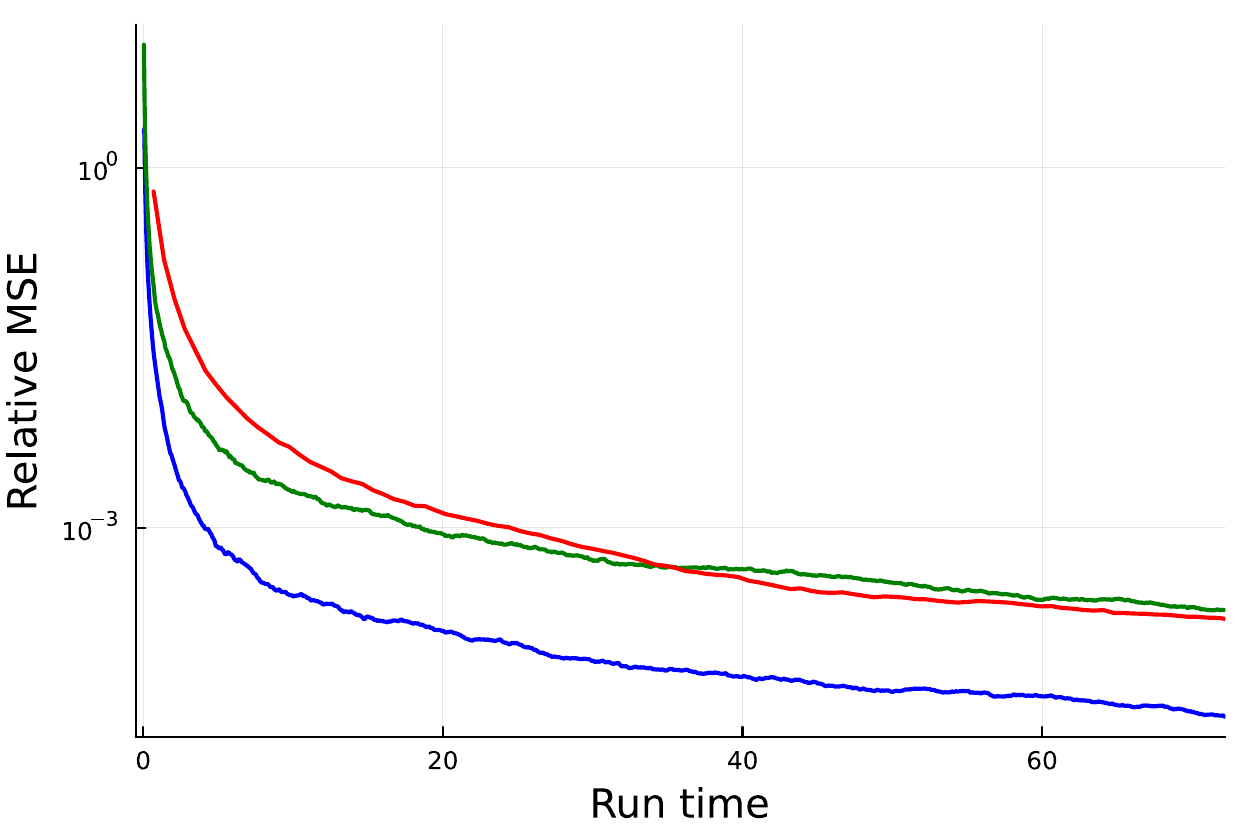}
    \caption{Results for $N=25$. The ZZS has step-size $1e-2$, BPS has step-size $5e-3$ and refreshment rate $\lambda_r=1$, while HMC has step size $3.3e-3$, $M=5$, $K=3.$ The plots show the average over $100$ simulations. }
\end{subfigure} 

\vspace{5pt}

\begin{subfigure}{0.99\textwidth}
  \centering
    \includegraphics[width=0.49\textwidth]{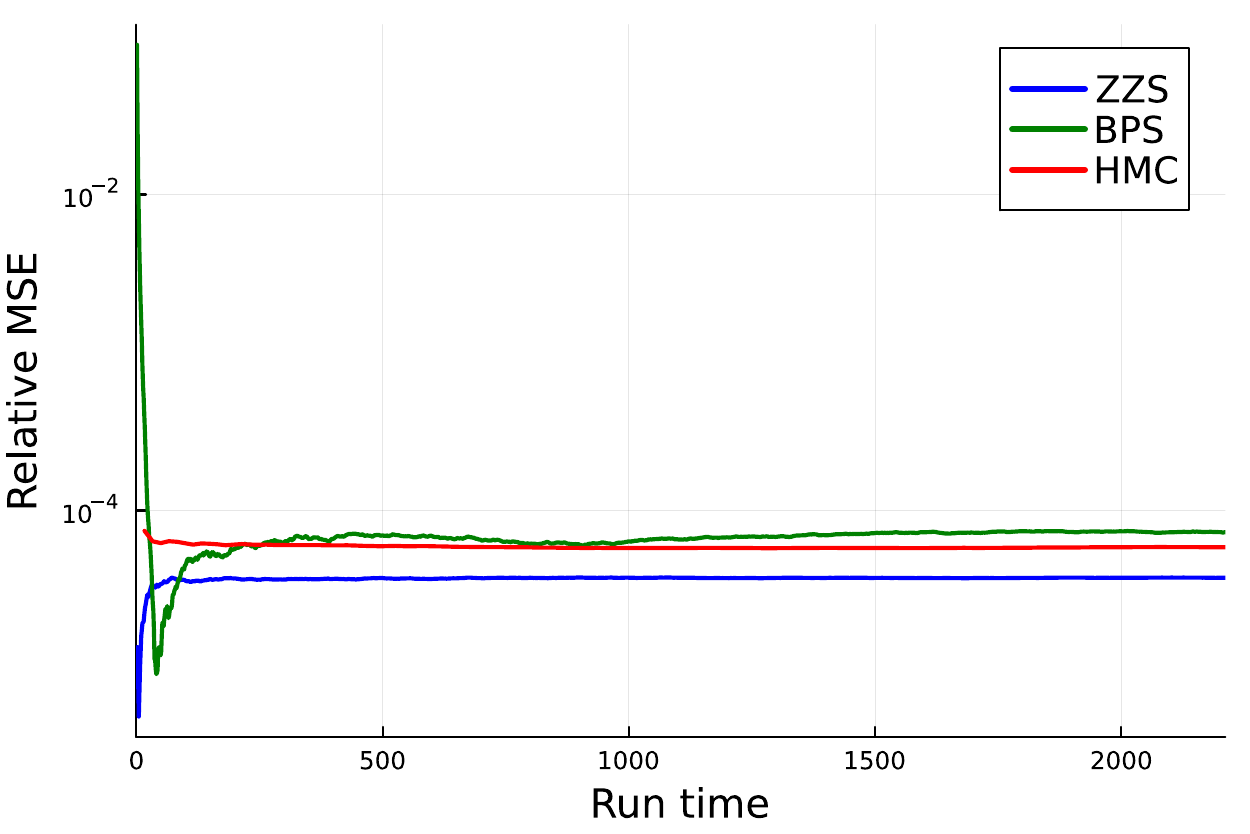}
    \hfill
    \includegraphics[width=0.49\textwidth]{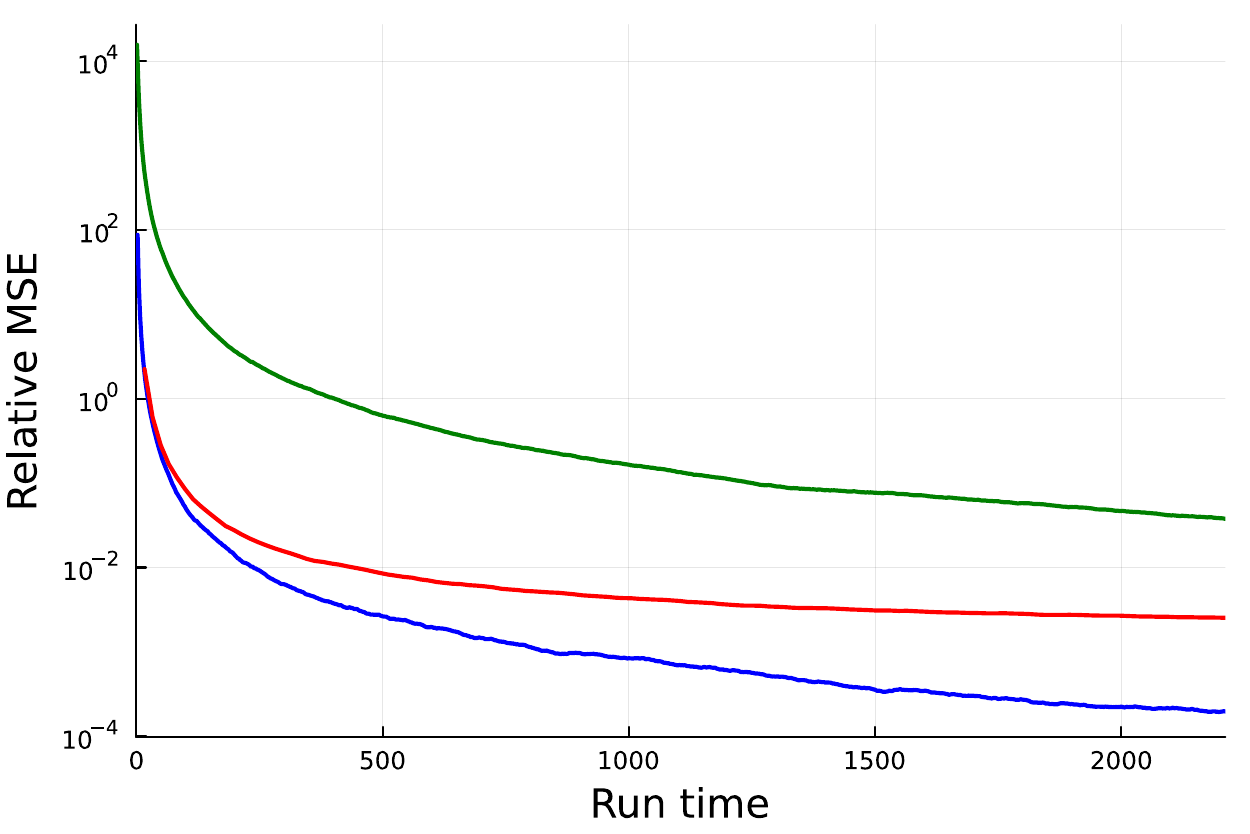}
    \caption{Results for $N=100$. The ZZS has step-size $1e-2$, BPS has step-size $5e-3$ and refreshment rate $\lambda_r=1$, while HMC has step size $3e-2$, $M=5$, $K=3.$ The plots show the average over $5$ simulations.}
\end{subfigure} 

\caption{Numerical simulations in the setting of Section \ref{subsec:example_particles}. In all plots the $y$-axis shows the relative mean square error (MSE), that is the MSE normalised by the square of the true value, which in this case corresponds to a long run of the Metropolis-adjusted BPS. The \emph{left} and \emph{right} plots show the relative MSE for the estimation respectively of the mean and variance of the test function \eqref{eq:test_particles}. The initial position of each particle is obtain drawing from the standard Gaussian distribution.}
\label{fig:var_and_varvar_particles_N=25}
\end{figure}

We observe that the ZZS and the HMC algorithm described above are obtained with similar ideas, but applied in a different order. Indeed, the ZZS was derived first splitting the forces at the level of the continuous-time process and then applying our splitting scheme, whereas for the HMC algorithm the forces were split at the level of the discretisation (cf. previous footnote). Hence, in the case of ZZS we do not introduce an additional numerical error since we are approximating a continuous-time PDMP that has the right invariant distribution. Moreover, ZZS naturally adapts the (random) step-size for the expensive part of the forces by leveraging the Poisson thinning technique, avoiding tuning an additional parameter that arises from multi-time-step procedures. For more comparisons between random jumps and deterministic multi-time-step splitting methods in some practical cases, we refer to \cite{weisman} and to \cite{gouraudJumps} which was released after (and motivated by) the present work.

The results of our simulations are shown in Figure~\ref{fig:var_and_varvar_particles_N=25}. The parameters of all samplers are obtained by performing a grid search over the parameter space.
It is clear from Figure~\ref{fig:var_and_varvar_particles_N=25} that ZZS gives cheaper and more accurate estimates of the mean and variance of the empirical variance in both cases considered. In particular, we see that ZZS clearly outperforms HMC and BPS. This holds even though the calibration of ZZS is considerably simpler since it involves only the tuning of the step size, while for HMC and BPS one has to tune additionally respectively $K$ and $M$, and $\lambda_r$. We remark that, in our simulations, BPS is robust to ``reasonable" choices of $\lambda_r$.
BPS is competitive for $N=25$ particles, but for higher $N$ it requires a step size that is too small to match the accuracy of ZZS, making its convergence slow. This is possibly due to the stochasticity in the reflection operator of BPS corresponding to the mean field interactions. More refined schemes for BPS might improve its performance, but these are outside of the scope of this paper, where we are interested in testing our general-purpose splitting schemes.


\acks{\sloppy 
The authors thank the three anonymous referees for their valuable comments, which helped substantially improve this article.
AB acknowledges funding from the Dutch Research Council (NWO) as part of the research programme `Zigzagging through computational barriers' with project number 016.Vidi.189.043. AB also acknowledges funding by the European Union (ERC-2022-SyG, 101071601). Views and opinions expressed are however those of the authors only and do not necessarily reflect those of the European Union or the European Research Council Executive Agency. Neither the European Union nor the granting authority can be held responsible for them. PD acknowledges funding from the Engineering and Physical Sciences Research Council (EPSRC) grant EP/V006177/1. PM acknowledges funding from the  French ANR grant 
SWIDIMS (ANR-20-CE40-0022) and from the European Research Council (ERC) under the European Union’s Horizon 2020 research and innovation program (grant agreement No 810367),
project EMC2.}


\newpage

\appendix


\section{Pseudo-codes for the algorithms with subsampling}\label{sec:pseudocodes_subsampling}
Here we give pseudo-codes for the ZZS algorithms with subsampling described in Section~\ref{subsec:subsamp}. Algorithms~\ref{alg:zzs_PT} and \ref{alg:zzs_woPT} are suited for scenarios where the target distribution is of the form $\pi(x) \propto\exp(- \frac{1}{M} \sum_{j=1}^M \pot_j(x) )$. Algorithm~\ref{alg:zzs_PT} in addition requires availability of a function $\beta_i(x,v_i)$ such that 
\begin{equation*}
    (v_i \partial_i \pot_j(x))_+ \leq \beta_i(x,v_i)
\end{equation*}
for all $j\in\{1,\dots,M\}.$ These can be extended to the case of BPS by using a stochastic version of the reflection operator, which is of the form 
\begin{equation*}
    R_j(x)v=v-2\frac{\langle v,\nabla_x\pot_j(x)\rangle}{\lvert\nabla_x\pot_j(x)\rvert^2} \nabla_x\pot_j(x)
\end{equation*}
when the jump is caused by the $j$-th term of $\pot$, that is $\pot_j$.

\begin{algorithm}
\DontPrintSemicolon
\caption{Part \textbf{B} for the ZZS with subsampling and Poisson thinning}
\label{alg:zzs_PT}
\KwIn{Initial condition $(x,v)\in\mathbb{R}^d\times\{\pm 1\}^d$, step size $\delta$.}
\KwOut{Updated velocity vector $v$.}
\For{$i \gets 1$ \KwTo $d$}{
    $t \gets 0$\;
    Compute the Poisson thinning bound $\beta_i(x,v_i)$\;
    $\tau \sim \Exp(\beta_i(x,v_i))$\;
    \While{$\tau \leq \delta - t$}{
            $t \gets t + \tau$\;
            $J \sim \Unif(\{1,\dots,M\})$\;
            $U\sim\Unif[0,1]$\;
            \If{$U \leq  \frac{(v_i\partial_i \pot_J(x))_+}{\beta_i(x,v_i)}$}{
                $v_i \gets -v_i$\;
            }
            $\tau \sim \Exp(\beta_i(x,v_i))$\;
            
    }
}
\Return $v$\;
\end{algorithm}

\begin{algorithm}
\DontPrintSemicolon
\caption{Part \textbf{B} for the ZZS with subsampling without Poisson thinning}
\label{alg:zzs_woPT}
\KwIn{Initial condition $(x,v)\in\mathbb{R}^d\times\{\pm 1\}^d$, step size $\delta$.}
\KwOut{Updated velocity vector $v$.}
\For{$i \gets 1$ \KwTo $N$}{
    $t \gets 0$\;
    $J \sim \Unif(\{1,\dots,M\})$\;
    $\tau \sim \Exp((v_i\partial_i \pot_J(x))_+)$\;
    \While{$\tau \leq \delta - t$}{
            $t \gets t + \tau$\;
            $v_i \gets -v_i$\;
            $J \sim \Unif(\{1,\dots,M\})$\;
            $\tau \sim \Exp((v_i\partial_i \pot_J(x))_+)$\;
    }
}
\Return $v$\;
\end{algorithm}

\section{Proofs of Section \ref{sec:CVsplitting}}
\subsection{Proof of Theorem \ref{thm:weakerror}}\label{sec:proof_weakerror}

Fix $g\in \C_\vf^{2,0}\cap D(\cL)$. By a telescoping sum we have
\begin{equation*}
     \mathbb{E}_z[g(\Zbar_{t_n})] -\mathbb{E}_z[g(Z_{t_n})]= \sum_{k=0}^{n-1}(\mathbb{E}_z[\cP_{t_n-t_{k+1}}g(\Zbar_{t_{k+1}})] -\mathbb{E}_z[\cP_{t_n-t_k}g(\Zbar_{t_k})]).  
\end{equation*}
For each $k\in \{0,\ldots, n-1\}$, set $f_k(y,s) = \cP_{t_n-t_{k}-s}g(y)$ then we have
\begin{equation*}
     \mathbb{E}_z[g(\Zbar_{t_n})] -\mathbb{E}_z[g(Z_{t_n})] = \sum_{k=0}^{n-1} \mathbb{E}_z[f_k(\Zbar_{t_{k+1}},\delta)-f_k(\Zbar_{t_k},0)].  
\end{equation*}
By conditioning on $\Zbar_{t_k}$ it is sufficient to prove that
\begin{equation}\label{eq:boundkthterm}
    \lvert \mathbb{E}_z[f_k(\Zbar_{\delta},\delta)]-f_k(z,0)\rvert \leq  R (1+\lvert z\rvert^M) \lVert g\rVert_{\C_\vf^{2,0}}\delta^3.
\end{equation}
Indeed if we have that \eqref{eq:boundkthterm} holds then by Assumption \ref{ass:momentcond} we have
\begin{align*}
    \lvert \mathbb{E}_z[g(Z_{t_n})] -\mathbb{E}_z[g(\Zbar_{t_n})]\rvert &\leq C\delta^3\sum_{k=0}^{n-1}e^{R(t_n-t_k)}\mathbb{E}_z[ \lf(\Zbar_{t_k})] \\
    &\leq C \lVert g\rVert_{\C_\vf^{2,0}} e^{Rt_n} \delta^3 n \overline{G}_M(z),
\end{align*}
which gives the desired result. It remains to show that \eqref{eq:boundkthterm} holds. 

As done in \citet{bertazzi2021approximations} we rewrite the lhs as
\begin{align}\label{eq:lhsexpansion}
    \mathbb{E}_z[f_k(\Zbar_{\delta},\delta)]-f_k(z,0) & =\mathbb{E}_z[f_k(\Zbar_{\delta},\delta)] - f_k(\varphi_{\delta}(z),\delta)+f_k(\varphi_{\delta}(z),\delta)-f_k(z,0).
\end{align}
In particular, with identical steps to \citet{bertazzi2021approximations} we can rewrite the last two terms on the left hand side of \eqref{eq:lhsexpansion} using the fundamental theorem of calculus and that $\partial_s f_k(z,s) = -\cL f_k(z,s)$: 
\begin{align*}
    f_k(\varphi_{\delta}(z),\delta)-f_k(z,0) = - \int_0^{\delta} \lambda(\varphi_r(z))[Q(f_k(\cdot,r))(\varphi_{r}(z))-f_k(\varphi_r(z),r)]dr.
\end{align*}
Then we compute the expectation in the right hand side of \eqref{eq:lhsexpansion}, collecting a term for the case of no jumps, a single jump and the case of multiple jumps
\begin{align*}
    &  \mathbb{E}_z[f_k(\Zbar_{\delta},\delta)]-f_k(z,\delta) =  \\
    & = \int \int_0^{\delta} Q(\varphi_{\delta/2}(z),d\tilde{z}) \Big( f_k(\varphi_{\delta/2}(\tilde{z}),\delta) - f_k(\varphi_\delta(z),\delta) \Big) \lambda(\varphi_{\delta/2}(z))e^{-s\lambda(\varphi_{\delta/2}(z))} e^{-(\delta-s)\lambda(\tilde{z})} ds \label{eq:we_term1}\tag{$\dagger$} \\
    & \quad + \sum_{\ell=2}^\infty \mathbb{E}_z[(f_k(\Zbar_{\delta},\delta)-f_k(\varphi_\delta(z),\delta)) \mathbbm{1}_{\{\ell\text{ events}\}}] \label{eq:we_term2}\tag{$\ddag$} \\
    & \quad - \int_0^{\delta} \lambda(\varphi_r(z))[Q(f_k(\cdot,r))(\varphi_{r}(z))-f_k(\varphi_r(z),r)]dr \label{eq:we_term3}\tag{$\ddag\dagger$}.
\end{align*}
Observe that the sum in the second term \eqref{eq:we_term2} can be truncated from $\ell=3$ onward as we only wish to get an order $\delta^3$ local error. Indeed, we have $\lvert f_k\rvert \leq \lVert g\rVert_\infty$ and hence
\begin{align*}
    & \left\lvert\sum_{\ell=3}^\infty \mathbb{E}_z[(f_k(\Zbar_{\delta},\delta)-f_k(\varphi_\delta(z),\delta)) \mathbbm{1}_{\{\ell\text{ events}\}}]\right\rvert  \leq 2\lVert g \rVert_\infty  \mathbb{P}_z(\ell\geq 3 \text{ events})\\ 
    & \leq 2\lVert g\rVert_\infty \int \int_0^\delta \int_0^{\delta-s_1} \int_0^{\delta-s_1-s_2} \lambda(\varphi_{\delta/2}(z)) e^{-s_1\lambda(\varphi_{\delta/2}(z)) } \lambda(z_1) e^{-s_2\lambda(z_1) } \lambda(z_2) e^{-s_3\lambda(z_2) } \\
    &  \quad Q(\varphi_{\delta/2}(z),dz_1)Q(z_1,dz_2) Q(z_2,dz_3) \,ds_1 ds_2 ds_3 \\
    & \leq 2\lVert g\rVert_\infty \int (1-e^{-\delta \lambda(\varphi_{\delta/2}(z))})(1-e^{-\delta \lambda(z_1)})(1-e^{-\delta \lambda(z_2)}) Q(\varphi_{\delta/2}(z),dz_1)Q(z_1,dz_2) \\
    & \leq 2 \delta^3 \lVert g\rVert_\infty \int \lambda(\varphi_{\delta/2}(z)) \lambda(z_1) \lambda(z_2) Q(\varphi_{\delta/2}(z),dz_1)Q(z_1,dz_2) \\
    & \leq 2 \delta^3 \lVert g\rVert_\infty \int \lambda(\varphi_{\delta/2}(z)) \lambda(z_1) Q(\varphi_{\delta/2}(z),dz_1)Q\lambda(z_1)\\
    & \leq 2 \delta^3 \lVert g\rVert_\infty \, \lambda(\varphi_{\delta/2}(z)) Q(\lambda(\cdot) Q\lambda(\cdot))(\varphi_{\delta/2}(z)) 
\end{align*}
where we used that $1-\exp(-z)\leq z$ for $z\geq 0$. By Assumption \ref{ass:kernelexp} we have that $\lambda Q(\lambda Q\lambda)$ is polynomially bounded and therefore we can bound \eqref{eq:we_term2} by
\begin{align*}
    \eqref{eq:we_term2} &= \int Q(\varphi_{\delta/2}(z),dz_1) Q(z_1,dz_2) \Big( f_k(\varphi_{\delta/2}(z_2),\delta)-f_k(\varphi_\delta(z),\delta) \Big) \cdot \\
    & \quad \cdot \int_0^\delta \int_0^{\delta-s_1} \lambda(\varphi_{\delta/2}(z))e^{-s_1\lambda(\varphi_{\delta/2}(z))}  \lambda(z_1)e^{-s_2\lambda(z_1)}  e^{-(\delta-s_1-s_2 )\lambda(z_2)} ds_2 \,ds_1 \\
    & \quad + \mathcal{O}(\lVert g\rVert_\infty (1+\lvert z\rvert^{3m_\lambda})\delta^3) .
\end{align*}
Here and throughout we use the notation $F(z,\delta,g)=\mathcal{O}(\lVert g\rVert_{\C_\vf^2} (1+\lvert z\rvert^{m})\delta^n)$ to mean that
\begin{equation*}
    \limsup_{\delta\to 0} \sup_{z\in E}\sup_{g}\frac{\lvert F(z,\delta,g)\rvert }{\lVert g\rVert_{\C_\vf^2} \delta^n(1+\lvert z\rvert^m)} \leq C.
\end{equation*}

We Taylor expand several terms in order to verify that the local error is of order $\delta^3$. We use repeatedly the following expansions:
\begin{align*}
      &\lambda(\varphi_s(z)) = \lambda(z) + s D_\vf \lambda(z)+ 
        s^2R(z,\tilde{s};\lambda) ,\\
      &f_k(\varphi_s(z),\delta) = f_k(z,\delta) + sD_\vf f_k(z,\delta) + s^2R(z,\tilde{s};f_k) ,\\
     &R(z,\tilde{s};g)= D_\Phi^2g(\varphi_{\tilde{s}}(z))/2,\\
    & f_k(z,s) = f_k(z,0) -s \cL f_k(z,0) + s^2 \cL^2 f_k(z,\tilde{s})/2,
\end{align*}
for some $\tilde{s}\in [0,s]$ (note that $\tilde{s}$ may vary with each term so when we use these expansions we include an index to distinguish different incidents of $\tilde{s}$). 
Note that by Assumption \ref{ass:derivative_estimate} we have $\lVert f_k\rVert_{\C_\vf^2} \leq Ce^{R(t_n-t_k)}(1+\lvert z\rvert^{m_\cP}) \lVert g\rVert_{\C_\vf^{2,0}}$ which gives us a bound on the remainder terms.
Applying the expansions above to \eqref{eq:we_term1} we obtain
\begin{align*}
    & \eqref{eq:we_term1}  = \int  Q(\varphi_{\delta/2}(z),d\tilde{z}) \Big( f_k(\tilde{z},0) -\delta \cL f_k(\tilde{z},0) + \frac{\delta^2}{2} \cL^2 f_k(\tilde{z},\tilde{s}_2) + \frac{\delta}{2} D_\vf f_k(\tilde{z},\delta)  \\
    & + \frac{\delta^2}{8}R(\tilde{z},\tilde{s}_2;f_k)- f_k(z,0) +\delta \cL f_k(z,0) - \frac{\delta^2}{2} \cL^2 f_k(z,\tilde{s}_3) - \delta D_\vf f_k(z,\delta) - \frac{\delta^2}{2} R(z,\tilde{s}_3;f_k)\Big)\\
    &\int_0^{\delta} \Big(\lambda(z) + \frac{\delta}{2}D_\vf\lambda(z) + \frac{1}{8}\delta^2R(z,\tilde{s}_4;\lambda) \Big)\\
    & \Big( 1 -s\lambda(\varphi_{\delta/2}(z))-(\delta-s)\lambda(\tilde{z}) + \frac{1}{2}( s\lambda(\varphi_{\delta/2}(z))+(\delta-s)\lambda(\tilde{z}))^2 e^{-\eta} \Big) ds \\
    = & \int  Q(\varphi_{\delta/2}(z),d\tilde{z})\Big( f_k(\tilde{z},0) - f_k(z,0)  \Big) \int_0^{\delta} \Big(\lambda(z) + \frac{\delta}{2}D_\vf \lambda(z)  \Big) \Big( 1 -s\lambda(z)-(\delta-s)\lambda(\tilde{z})  \Big) ds\\
    & +\delta\int  Q(\varphi_{\delta/2}(z),d\tilde{z})\Big( - \cL f_k(\tilde{z},0) + \frac{1}{2} D_\vf f_k(\tilde{z},0) + \cL f_k(z,0)  - D_\vf f_k(z,0) \Big)\\
    & \int_0^{\delta} \Big(\lambda(z) + \frac{\delta}{2}D_\vf \lambda(z)  \Big) \Big( 1 -s\lambda(z)-(\delta-s)\lambda(\tilde{z})  \Big) ds+ e^{R(t_n-t_k)}\mathcal{O}(\lVert g\rVert_{\C_\vf^{2,0}}(1+\lvert z\rvert^M)\delta^3) 
\end{align*}
where we used $$\eta \in [0,s\lambda(\varphi_{\delta/2}(z))+(\delta-s)\lambda(\tilde{z})] $$ in the first equality and further Taylor expansions to obtain the second equality. Here $M=3m_{\lambda}+m_{\cP}$.
Now using Assumption \ref{ass:kernelexp} we can expand the $Q$ term
\begin{align*}
    \eqref{eq:we_term1} 
    = & \int  \Big(Q(z,d\tilde{z}) +\frac{\delta}{2}D_\vf Q(z,d\tilde{z})  \Big)\Big( f_k(\tilde{z},0) - f_k(z,0)  \Big) \\
    & \quad \int_0^{\delta} \Big(\lambda(z) + \frac{\delta}{2}D_\vf \lambda(z)  \Big) \Big( 1 -s\lambda(z)-(\delta-s)\lambda(\tilde{z})  \Big) ds\\
    & +\delta\int  Q(z,d\tilde{z})\Big( - \cL f_k(\tilde{z},0) + \frac{1}{2} D_\vf f_k(\tilde{z},0) + \cL f_k(z,0)  - D_\vf f_k(z,0) \Big)\\
    & \int_0^{\delta} \Big(\lambda(z) + \frac{\delta}{2}D_\vf \lambda(z)  \Big) \Big( 1 -s\lambda(z)-(\delta-s)\lambda(\tilde{z})  \Big) ds+ e^{R(t_n-t_k)}\mathcal{O}(\lVert g\rVert_{\C_\vf^{2,0}}(1+\lvert z\rvert^M)\delta^3)  
\end{align*}


Term \eqref{eq:we_term2} can be expanded as
\begin{align*}
    \eqref{eq:we_term2}  =& \int Q(\varphi_{\delta/2}(z),d\tilde{z})Q(z_1,dz_2)  ( f_k(z_2,0) -f_k(z,0) )  \!\!\int_0^\delta \int_0^{\delta-s_1}\!\!\!\!\! \big(\lambda(z) + \delta D_\vf(\lambda) (\varphi_{\tilde{s}_4}(z))  \big) \lambda(z_1) \\
    & \Big(1+(-s_1\lambda(\varphi_{\delta/2}(z)) -s_2\lambda(z_1) -(\delta-s_1-s_2 )\lambda(z_2)) e^{-\xi} \Big) ds_2 \,ds_1 \\
    &+  e^{R(t_n-t_k)}\mathcal{O}(\lVert g\rVert_{\C_\vf^{1,0}} (1+\lvert z\rvert^{m_{\cP}+3m_\lambda}) \delta^3) \\
    = & \frac{\delta^2}{2} \int Q(\varphi_{\delta/2}(z),d\tilde{z})Q(z_1,dz_2) \Big( f_k(z_2,0) -f_k(z,0) \Big)   \lambda(z) \lambda(z_1) \\
    &+ e^{R(t_n-t_k)}\mathcal{O}(\lVert g\rVert_{\C_\vf^{1,0}} (1+\lvert z\rvert^{m_{\cP}+3m_\lambda}) \delta^3)\\
    = & \frac{\delta^2}{2} \int Q(z,dz_1)Q(z_1,dz_2) \Big( f_k(z_2,0) -f_k(z,0) \Big)   \lambda(z) \lambda(z_1) \\
    &+ e^{R(t_n-t_k)}\mathcal{O}(\lVert g\rVert_{\C_\vf^{1,0}} (1+\lvert z\rvert^{m_{\cP}+3m_\lambda}) \delta^3),
\end{align*}
where $$\xi \in [0,s_1\lambda(\varphi_{\delta/2}(z)) +s_2\lambda(z_1) +(\delta-s_1-s_2 )\lambda(z_2)].$$

By Assumption \ref{ass:kernelexp} we can expand the term \eqref{eq:we_term3} as follows:
\begin{align*}
    \eqref{eq:we_term3} = & -\int_0^\delta \lambda (\varphi_r(z)) \left(Qf_k(\cdot,r)(z) + r \int D_\vf Q(z,d\tilde{z})f_k(\tilde{z},r) -f_k(\varphi_r(z),r) \right) dr \\
    & \quad + e^{R(t_n-t_k)}\mathcal{O}(\delta^3\lVert g\rVert_{C_\vf^{2,0}}(1+\lvert z\rvert^{m_\lambda})).
\end{align*}
By Taylor's theorem 
\begin{align*}
    \eqref{eq:we_term3} = & -\int_0^\delta \lambda (\varphi_r(z)) \left(Qf_k(\cdot,0)(z)-r\int Q(z,d\tilde{z}) \cL f_k(\tilde{z},0) + r \int D_\vf Q(z,d\tilde{z})(f_k(\tilde{z},0)-r\cL f_k(\tilde{z},\tilde{r}))\right).\\
    &\left. -\left(f_k(z,0)+rD_{\vf}f_k(z,0)-r\cL f_k(z,0)\right) \right) dr \\
    &+ e^{R(t_n-t_k)}\mathcal{O}(\delta^3\lVert g\rVert_{C_\vf^{2,0}}(1+\lvert z\rvert^{m_\lambda+m_{\cP}}).
\end{align*}
Note that $\int D_\vf Q(z,d\tilde{z})\cL f_k(\tilde{z},\tilde{r}) = Q(D_\vf \cL f_k(\cdot,\tilde{r}))(z) = e^{R(t_n-t_k)}\mathcal{O}((1+\lvert z\rvert^{m_{\cP}})\lVert g\rVert_{\C_\vf^{2,0}})$. Using this and Taylor expanding $\lambda(\varphi_r(z))$ we have
\begin{align*}
    \eqref{eq:we_term3} = & -\int_0^\delta \lambda (z) \Big(Qf_k(\cdot,0)(z)-r\int Q(z,d\tilde{z}) \cL f_k(\tilde{z},0) + r \int D_\vf Q(z,d\tilde{z})f_k(\tilde{z},0)\\
    & -\left(f_k(z,0)+rD_{\vf}f_k(z,0)-r\cL f_k(z,0)\right) \Big) dr -\int_0^\delta rD_{\vf}\lambda (z) \left(Qf_k(\cdot,0)(z)-f_k(z,0) \right) dr \\
    &+ e^{R(t_n-t_k)}\mathcal{O}(\delta^3\lVert g\rVert_{C_\vf^{2,0}}(1+\lvert z\rvert^{m_\lambda+m_{\cP}}).
\end{align*}
Evaluating the integral over $r$
\begin{align*}
    \eqref{eq:we_term3} = & - \lambda (z) \left(Qf_k(\cdot,0)(z)\delta-\frac{1}{2}\delta^2\int Q(z,d\tilde{z}) \cL f_k(\tilde{z},0) + \frac{1}{2}\delta^2 \int D_\vf Q(z,d\tilde{z})f_k(\tilde{z},0)\right.\\
    &\left. -\left(f_k(z,0)\delta+\frac{1}{2}\delta^2 D_{\vf}f_k(z,0)-\frac{1}{2}\delta^2\cL f_k(z,0)\right) \right) - \frac{1}{2}\delta^2D_{\vf}\lambda (z) \left(Qf_k(\cdot,0)(z)-f_k(z,0) \right)\\
    &+ e^{R(t_n-t_k)}\mathcal{O}(\delta^3\lVert g\rVert_{C_\vf^{2,0}}(1+\lvert z\rvert^{m_\lambda+m_{\cP}}).
\end{align*}

\textbf{First order terms.} Terms of order $\delta$ appear only in \eqref{eq:we_term1} and \eqref{eq:we_term3} and clearly they cancel out.

\textbf{Second order terms.} In \eqref{eq:we_term1} we can further expand terms of the form $D_\vf (f)(\Tilde{z},\delta)$ and rearrange as
\begin{align*}
   \text{Order } \delta^2 \text{ of } \eqref{eq:we_term1}  =& 
    \delta^2 \Bigg[ \int \lambda(z) \Big( -Q(z,d\tilde{z})  \cL f_k(\tilde{z},0)  + Q(z,d\tilde{z}) \frac{1}{2} D_\vf f_k(\tilde{z},0) \\
    &  +  \frac{1}{2}D_\vf Q(z,d\tilde{z}) (f_k(\tilde{z},0)-f_k(z,0)) \Big)\\
   &   + \lambda(z) ( \cL f_k(z,0) -   D_\vf f_k(z,0)\rangle ) \\
   &  + \int  Q(z,d\tilde{z}) (f_k(\tilde{z},0)- f_k(z,0)) 
   \Big(-\frac{1}{2}\lambda(z)(\lambda(\tilde{z})+ \lambda(z)) + \frac{1}{2}D_\vf\lambda(z)  \Big) \Bigg] \\ 
   = & \delta^2 \Bigg[ \int \lambda(z) \Big( -Q(z,d\tilde{z})  \cL f_k(\tilde{z},0)  + Q(z,d\tilde{z}) \frac{1}{2} D_\vf f_k(\tilde{z},0)\rangle \\
    &  + \underbrace{\frac{1}{2}D_\vf Q(z,d\tilde{z}) (f_k(\tilde{z},0)-f_k(z,0))}_{\textnormal{Term } A} \Big)\\
   &   +\underbrace{\lambda(z) \Big( \cL f_k(z,0) - D_\vf f_k(z,0)) -\frac{1}{2} \lambda(z) \int Q(z,d\tilde{z}) (f_k(\tilde{z},0)-f_k(z,0))  \Big) }_{\textnormal{Term } B}\\
   & + \underbrace{\frac{1}{2}D_\vf \lambda(z) \int  Q(z,d\tilde{z}) (f_k(\tilde{z},0)-f_k(z,0))}_{\textnormal{Term } C}  \\
   &  -\frac{1}{2} \int  Q(z,d\tilde{z}) (f_k(\tilde{z},0)-\underbrace{f_k(z,0)}_{\textnormal{Term } D})  
   \lambda(z)\lambda(\tilde{z}) \Bigg].
\end{align*}
For term \eqref{eq:we_term2} we have
\begin{align*}
    \text{Order } \delta^2 \text{ of } \eqref{eq:we_term2} & =\frac{1}{2}\delta^2 \int Q(z,dz_1) Q(z_1,dz_2) ( f_k(z_2,0) -\underbrace{f_k(z,0)}_{\textnormal{Term } D} ) \lambda(z)  \lambda(z_1).
\end{align*}
Similarly for \eqref{eq:we_term3} we have 
\begin{align*}
    \text{Order } \delta^2 \text{ of } \eqref{eq:we_term3}  = &  - \frac{\delta^2}{2} \Bigg( \underbrace{\int D_\vf (Q)(z)\lambda(z) (f_k(\tilde{z},0)-f_k(z,0)) }_{\textnormal{Term } A} \\
    & + \underbrace{\int Q(z,d\tilde{z}) D_\vf (\lambda) (z) (f_k(\tilde{z},0)-f_k(z,0)) }_{\textnormal{Term } C}\\
    & + \int Q(z,d\tilde{z}) \lambda(z) (-\cL f_k(\tilde{z},0) +\underbrace{ \cL f_k(z,0)-D\vf (f_k)(z,0)}_{\textnormal{Term } B} \Bigg).        
\end{align*}
After cancellations we obtain
\begin{align*}
    & \text{Order } \delta^2 \text{ of } \eqref{eq:we_term1}+\eqref{eq:we_term2}+\eqref{eq:we_term3} = \delta^2 \lambda(z) \Bigg( \int \Big( -Q(z,d\tilde{z})  \cL f_k(\tilde{z},0)  + Q(z,d\tilde{z}) \frac{1}{2} D_\vf (f_k)(\Tilde{z},0)\Big)  \\
    & -\frac{1}{2} \int Q(z,d\tilde{z}) f_k(\tilde{z},0) \lambda(\tilde{z}) + \frac{1}{2}\int Q(z,d\tilde{z}) Q(\tilde{z},dz_2)  f_k(z_2,0) \lambda(\tilde{z}) +\frac{1}{2} \int Q(z,d\tilde{z})  \cL f_k(\tilde{z},0) \Bigg)\\
    & =   \frac{1}{2}\delta^2 \lambda(z) \Bigg( \int Q(z,d\tilde{z})\Big( -  \cL f_k(\tilde{z},0)  + D_\vf (f_k)(\Tilde{z},0)  + \lambda(\tilde{z})\int  Q(\tilde{z},dz_2)  [f_k(z_2,0)  -  f_k(\tilde{z},0)]\Big)  \Bigg)\\
    &=0,
\end{align*}
where the last equality follows by the definition of the generator of the PDMP. Therefore we have shown that second order terms cancel out.


\subsection{Proofs for Example \ref{ex:ZZS_weakerror}}\label{subsec:proof_ex_ZZS}
Let us verify that Assumption \ref{ass:kernelexp} holds. Note that
\begin{equation*}
    (D_{\vf} Q)(g)(x,v) = v^T \sum_{i=1}^d\nabla_x\left(\frac{\lambda_i(x,v)}{\lambda(x,v)} g(x,R_iv)\right).  
\end{equation*}
Therefore by Taylor's theorem we have for some $\eta$ between $x$ and $x+sv$
\begin{equation*}
    Q g(x+sv,v) - Qg (x,v) -(D_{\vf} Q)(g)(x,v) = \frac{1}{2}s^2\sum_{i=1}^dv^T\nabla_x^2\left(\frac{\lambda_i(x,v)}{\lambda(x,v)} g(x,R_iv)\right)\bigg\rvert_{x=\eta} v.   
\end{equation*}
It remains to show that $\frac{\lambda_i}{\lambda},\nabla_x\left(\frac{\lambda_i(x,v)}{\lambda(x,v)} \right),\nabla_x^2\left(\frac{\lambda_i(x,v)}{\lambda(x,v)}\right) $ are bounded. It is clear that $0<\lambda_i/\lambda\leq 1$. Let us consider the first derivative. Set $\Xi(s)= \log(1+e^s)$ so that $\lambda_i(x,v)=\Xi(v_i\partial_i\pot(x))$ and note that
\begin{align*}
    0\leq \frac{\Xi'(s)}{\Xi(s)} \leq 1, \quad 0\leq \frac{\Xi''(s)}{\Xi(s)} \leq 1.
\end{align*}
Now
\begin{equation*}
    \left\lvert\nabla_x\left(\frac{\lambda_i(x,v)}{\lambda(x,v)} \right)\right\rvert = \left\lvert\frac{\nabla_x\lambda_i(x,v)}{\lambda(x,v)} -\frac{\lambda_i(x,v)\nabla_x\lambda(x,v)}{\lambda(x,v)^2}\right\rvert \leq \left\lvert\frac{\nabla_x\lambda_i(x,v)}{\lambda(x,v)}\right\rvert+\sum_{j=1}^d\left\lvert\frac{\nabla_x\lambda_j(x,v)}{\lambda(x,v)}\right\rvert.
\end{equation*}
So it remains to show that $\nabla_x\lambda_i/\lambda$ is bounded. Using the bounds on $\Xi$ we have
\begin{equation*}
    \left\lvert\frac{\nabla_x\lambda_i(x,v)}{\lambda(x,v)}\right\rvert\leq \left\lvert\frac{\nabla_x\lambda_i(x,v)}{\lambda_i(x,v)}\right\rvert =  \left\lvert\frac{\Xi'(v_i\partial_i\pot(x))}{\Xi(v_i\partial_i\pot(x))} v_i\nabla_x\partial_i\psi(x)\right\rvert \leq \left\lvert\nabla_x\partial_i\psi(x)\right\rvert.
\end{equation*}
This is bounded by our assumptions on $\pot$. Let us now consider $\nabla_x^2\left(\frac{\lambda_i(x,v)}{\lambda(x,v)}\right)$:
\begin{align*}
    \left\lvert\nabla_x^2\left(\frac{\lambda_i(x,v)}{\lambda(x,v)} \right)\right\rvert &= \Bigg\lvert\frac{\nabla_x^2\lambda_i(x,v)}{\lambda(x,v)} -\frac{2\nabla_x\lambda_i(x,v)\nabla_x\lambda(x,v)^T}{\lambda(x,v)^2} \\
    & \quad +\frac{2\lambda_i(x,v)\nabla_x\lambda(x,v)\nabla_x\lambda(x,v)^T}{\lambda(x,v)^3}-\frac{\lambda_i(x,v)\nabla_x^2\lambda(x,v)}{\lambda(x,v)^2}\Bigg\rvert.
\end{align*}
Using the bound on $\nabla_x\lambda_i/\lambda$ we can bound all the terms aside from the one involving the second derivative of $\lambda$ so it suffices to consider 
\begin{align*}
    \left\lvert\frac{\nabla_x^2\lambda_i(x,v)}{\lambda(x,v)}\right\rvert &= \left\lvert\frac{\Xi''(v_i\partial_i\pot(x))}{\sum_j\Xi(v_j\partial_j\pot(x))} \nabla_x\partial_i\psi(x)\nabla_x\partial_i\psi(x)^T + \frac{\Xi'(v_i\partial_i\pot(x))}{\sum_j\Xi(v_j\partial_j\pot(x))} v_i\nabla_x^2\partial_i\psi(x)\right\rvert\\&\leq \left\lvert \nabla_x\partial_i\psi(x)\nabla_x\partial_i\psi(x)^T\right\rvert + \left\lvert \nabla_x^2\partial_i\psi(x)\right\rvert 
\end{align*}
This is bounded since $\pot$ has bound second and third order derivatives.







\section{Proofs of ergodicity for splitting schemes of the BPS}\label{sec:proofs_ergodicity_bps}

To fix ideas, we consider a splitting scheme for a BPS with unitary velocity with generator $\cL$ decomposed as $\cL=\cL_1+\cL_2+\cL_3$ with
\begin{align*}
\cL_1 f(x,v) &= v^T \na_x f(x,v),\\
\cL_2 f(x,v) &= \po v^T \na \psi(x)\pf_+ \po f(x,R(x)v) - f(x,v)\pf,\\
\cL_3 f(x,v) &= \lambda_r \int_{\mathbb S^{d-1}} \po f(x,w) - f(x,v) \pf \dd w\,.
\end{align*} 
Write $P_{t}^j = e^{t \cL_j}$ the associated semi-groups for $j\in\cco 1,3\ccf$. 
We shall show the following statement, which implies Theorem \ref{thm:ergodicBPS}.
\begin{theorem}\label{thm:BPS_ergodic_details}
Consider any scheme of the BPS  based on the decomposition \textbf{D},\textbf{R},\textbf{B}. Under Assumption~\ref{assu:BPSscheme}, the following hold:
\begin{enumerate}
    \item There exist $a,b,C,\delta_0>0$  and a function $V$ (with $a,b,C,\delta_0,V$ depending only on $\psi$ and $\lambda_r$, but not on $\delta$) such that, for all $x,v$,
\[ e^{|x|/a}/a \leqslant V(x,v) \leqslant a e^{a|x|}\]
and for all $\delta \in (0,\delta_0]$ and all $x,v$,
\begin{equation}
P_\delta V(x,v) \leqslant \po 1 - b \delta\pf V(x,v) + C \delta \,.\label{eq:BPSLyapunov}    
\end{equation}
\item For all $R>0$, there exist $c,\delta_0>0$ and a probability measure $\nu$ on $E$ such that for all $x,v$ with $|x|\leqslant R$ and all $\delta\in(0,\delta_0]$, setting $n_* = \lceil 4R/\delta \rceil$,
\begin{equation}
    \label{eq:BPSDoeblin}
\delta_{x,v} P_\delta^{n_*} \geqslant c \nu\,.
\end{equation}

\end{enumerate}
\end{theorem}

\subsection{Lyapunov function}

In this section, we prove \eqref{eq:BPSLyapunov}. Under Assumption~\ref{assu:BPSscheme}, let $W(x) = \sqrt{\psi(x)}$ so that $\na W$ is bounded and, as $\lvert x\rvert \rightarrow \infty$, $\liminf |\na W(x)|>0$  and $\na^2 W(x) \rightarrow 0$.

\subsubsection*{Parameters}

Let $c_W>0$ be such that $|\na W(x)| \geqslant c_W$ for $|x|$ large enough. Let
\[q = \int_{\mathbb S^{d-1}}\1_{w_1  \leqslant - 1/2}\dd w >0\,. \]
Let $\varphi\in\mathcal C^2(\R)$ be a non-decreasing function such that $\varphi'(0) >0$ and
\[\varphi(\theta) \left\{\begin{array}{ll}
=1 & \text{for }\theta \leqslant -\frac12 c_V\\
\geqslant 2-\frac{q}{2} & \text{for }\theta \geqslant -\frac14 c_V\\
=2 & \text{for }\theta \geqslant 1.
\end{array}\right.\]
Let
\[\kappa = \frac{4\lambda_r  }{c_W}\,,\qquad  M =  \frac{\kappa }{\varphi\po  \frac{q\lambda_r}{4\kappa} \pf -\varphi\po - \frac{q\lambda_r}{4\kappa} \pf} \]
(where we used that $\varphi'(0)>0$ so that the denominator is positive). Finally, let
\[\varepsilon = \frac{\lambda_r q}{16 \|\varphi'\|_\infty}\]
and let $R>0$ be such that for all $x\in\R^d$ with $|x|\geqslant R$,
\[|\na W(x)| \geqslant c_W\,,\qquad  2W(x) \geqslant M, \qquad \text{and} \qquad \|\na^2 W(x)\|\leqslant \varepsilon\,.\]

\subsubsection*{Preliminary computation}

In all this section we denote by the same letter $C$ various constants (independent from $t$) that change from line to line and we assume that $t\in(0,\delta_0]$ with $\delta_0 =  1/(\|\na W\|_\infty \kappa )$.
 We consider a Lyapunov function
\[V(x,v) = e^{\kappa W(x)} \varphi(\theta(x,v))\,,\qquad \theta(x,v) = v^T \na W(x)\,.\]
This choice is inspired by the Lyapunov function of the continuous time BPS that was studied in \citet{BPS_Durmus}.
Notice that $|\theta|$ is bounded by $\|\na W\|_\infty$. More generally, for a non-negative $\mathcal C^2$ function $g$, we bound
\begin{eqnarray*}
P_{t}^1 \po e^{\kappa W} g\circ\theta\pf(x,v) &= & e^{\kappa W(x+tv)} g (v^T \na W(x+tv)),
\end{eqnarray*} 
using that $e^{tr} \leqslant 1+tr+ r^2t^2 e^{r\delta_0}/2$ for $t\leqslant \delta_0$ as
\begin{eqnarray*}
e^{\kappa W(x+tv)} & \leqslant & e^{\kappa \po W(x) + t v^T \na W(x) + \frac12 t^2 \|\na^2 W\|_\infty\pf }  \\
& \leqslant & e^{\kappa W(x)} \po 1+ t \kappa v^T \na W(x) + Ct^2\pf
\end{eqnarray*}
 and, similarly,
\begin{eqnarray*}
g (v\cdot \na W(x+tv)) &\leqslant& g(\theta(x,v)) + t v^T \na^2 W(x) v  g'(\theta(x,v)) + Ct^2\\
& \leqslant & g(\theta(x,v)) + \po \varepsilon \|g'\|_\infty + C\1_{|x|\leqslant R}\pf t + Ct^2 
\end{eqnarray*}
(here and below the successive constants $C$ implicitly involve the suprema of $g$, $|g'|$ and $|g''|$ over $[-\|\na W\|_\infty,\|\na W\|_\infty]$). Combining these two bounds and using that $|t\kappa\theta|\leqslant 1$, we get 
\begin{equation}\label{eq:Pt1}
    P_t^1 \po e^{\kappa W}g\circ \theta\pf  \leqslant e^{\kappa W} \po (1+t \kappa \theta ) g\circ \theta + \varepsilon \|g'\|_\infty  t + C t^2\pf + C t \,.
\end{equation}

Second, for any function $g$,
\begin{align*}
& P_t^2 \po e^{\kappa W}g\circ \theta\pf (x,v)\\
&= e^{\kappa W(x)} \po e^{-t(v^T\na \psi(x))_+} g\po v^T \na W(x)\pf \!+\! \po 1-  e^{-t(v^T\na \psi(x))_+}\pf g\po (R(x)v)^T\na W(x)\pf \pf \\
& =  e^{\kappa W(x)}  \po g(\theta(x,v)) + \po 1-  e^{-2t W(x) \theta_+(x,v)}\pf\po  g\po -\theta(x,v)\pf -g\po \theta(x,v)\pf \pf   \pf\,.
\end{align*}
In the second equality we used $(v^T \nabla \pot(x))_+ = 2W(x) (v^T\nabla W(x))_+.$ 
In particular, if $g$ is non-decreasing, 
\begin{align}
&P_t^2 \po e^{\kappa W}g\circ \theta\pf (x,v) \notag \\
&\leqslant   e^{\kappa W(x)}  \po g(\theta(x,v)) + \po 1-  e^{-t M \theta_+(x,v)}\pf\po  g\po -\theta(x,v)\pf -g\po \theta(x,v)\pf \pf  + t C \1_{|x|\leqslant R} \pf \nonumber\\
& \leqslant   e^{\kappa W(x)}  \po g(\theta(x,v))  + t M \theta_+(x,v)\po  g\po -\theta(x,v)\pf -g\po \theta(x,v)\pf \pf    +  Ct^2 \pf  + t C \1_{|x|\leqslant R}\,. \label{eq:Pt2}
\end{align}

Third, for any function $g$, similarly,
\begin{align}
&P_t^3 \po e^{\kappa W}g\circ \theta\pf (x,v) \notag \\
& = e^{\kappa W(x)} \po g\po \theta (x,v)\pf + \po 1 - e^{-\lambda_r t}\pf \po \int_{\mathbb S^{d-1}} g(w^T \na W(x)) \dd w - g\po \theta (x,v)\pf\pf  \pf\nonumber\\
&\leqslant e^{\kappa W(x)} \po g\po \theta (x,v)\pf  + \lambda_r t\po \int_{\mathbb S^{d-1}} g(w\cdot \na W(x)) \dd w - g\po \theta (x,v)\pf\pf  + Ct^2 \pf\label{eq:Pt3}
\end{align}
In particular, for $g=\varphi$, using the rotation-invariance of the uniform measure on the sphere and the fact that $\varphi$ is increasing and with values in $[1,2]$, we bound
\[\int_{\mathbb S^{d-1}} \varphi(w\cdot \na W(x)) \dd w = \int_{\mathbb S^{d-1}} \varphi(w_1|\na W(x)|) \dd w  \leqslant  \1_{|x|\leqslant R} + A\]
with
\[A= \sup_{|x|\geqslant R} \int_{\mathbb S^{d-1}} \varphi(w_1|\na W(x)|) \dd w \,.\]
As a consequence, using that $e^{\kappa W(x,v)}\1_{|x|\leqslant R}$ is bounded, we obtain
\begin{equation}
P_t^3 \po e^{\kappa W}\varphi \circ \theta\pf   
\leqslant     e^{\kappa W} \po \varphi \circ \theta  + \lambda_r t\po A - \varphi \circ \theta\pf  + Ct^2 \pf + Ct\label{eq:Pt3_bis}
\end{equation}

\subsubsection*{Combining the computations}

Now, to fix ideas, let us start with the \textbf{DRBRD} case. In the following, for convenience, we write $f(\theta)=f\circ \theta$ for functions $f$. From our previous computations, using first \eqref{eq:Pt1} and then \eqref{eq:Pt3} and \eqref{eq:Pt3_bis} (keeping track only of the terms of order $0$ or $1$ in $t$, the higher order ones being bounded and gathered in the term $t^2 C$),
\begin{eqnarray*}
\lefteqn{P_{t}^1 P_{t}^3 P_{2t}^2 P_{t}^3 P_{t}^1 V}\\
&\leqslant & P_{t}^1 P_{t}^3 P_{2t}^2 P_{t}^3 \po e^{\kappa W} \po (1+t\kappa \theta) \varphi(\theta) + t\varepsilon \|\varphi'\|_\infty +t^2 C\pf  \pf + tC \\
&\leqslant & P_{t}^1 P_{t}^3 P_{2t}^2 \po e^{\kappa W} \po (1+t\kappa \theta) \varphi(\theta) + \lambda_r t(A-\varphi(\theta)) +  t\varepsilon \|\varphi'\|_\infty +t^2 C\pf  \pf  +  tC \\
& := & P_{t}^1 P_{t}^3 P_{2t}^2 \po e^{\kappa W} \Psi(\theta)  \pf  +  tC \,,
\end{eqnarray*}
where we defined
\[\Psi(z)=(1+t\kappa z) \varphi(z) + \lambda_r t(A-\varphi(z)) +  t\varepsilon \|\varphi'\|_\infty +t^2 C \,.\]
Assume that $t\leqslant  1/(2\lambda_r+2\|\na W\|_\infty \kappa)$. Under this condition $\Psi$ is non-decreasing on $[-\|\na W\|_\infty,\|\na W\|_\infty]$ (which, we recall, contains the image of $\theta$), since $\varphi$ is non-decreasing  and $(1+t(\kappa z-\lambda_r))>0$ if $|z|\leqslant \|\na W\|_\infty$. Thus, applying \eqref{eq:Pt2} but with $t$ replaced by $2t$) and then \eqref{eq:Pt3}, \eqref{eq:Pt3_bis} and \eqref{eq:Pt1} again,
\begin{align*}
\lefteqn{P_{t}^1 P_{t}^3 P_{2t}^2 P_{t}^3 P_{t}^1 V}\\
&\leqslant   P_{t}^1 P_{t}^3 \po e^{\kappa W} \po  \Psi(\theta) + 2tM \theta_+ \po \Psi(-\theta)-\Psi(\theta)\pf + t^2 C \pf \pf  +  tC \\
&\leqslant  P_{t}^1 P_{t}^3 \po e^{\kappa W} \po  \Psi(\theta) + 2tM \theta_+  \po \varphi(-\theta)-\varphi(\theta)\pf   + t^2 C\pf \pf  +  tC \\
& \leqslant P_{t}^1  \po e^{\kappa W}\!\! \po (1\!+ \!t\kappa\theta)\varphi(\theta)   \!+ \!2 \lambda_r t(A-\varphi(\theta)) \!+ \!2tM \theta_+ \po \varphi(-\theta)\!-\!\varphi(\theta)\pf +  t\varepsilon \|\varphi'\|_\infty +t^2 C\pf  \pf  \!+\!  tC \\
& \leqslant   e^{\kappa W} \!\! \po (1\!+\!  2t\kappa\theta)\varphi(\theta) \!+\! 2\lambda_r t(A-\varphi(\theta))  \!+\! 2tM  \theta_+ \po \varphi(-\theta)-\varphi(\theta)\pf +  2 t\varepsilon \|\varphi'\|_\infty +t^2 C_0\pf     +  tC \,,
\end{align*}
for some $C_0,C>0$ (in the last inequality, we use that $t$ is small enough so that the function which multiplies $e^{\kappa W}$ is non-negative in order to apply \eqref{eq:Pt1}). 
The choice of $\varepsilon$ ensures that
\[4\varepsilon \|\varphi'\|_\infty \leqslant \lambda_r q/4 =: \eta. \]
So, if we have that, for all $r\in\R$,
\begin{equation}\label{eq:eta}
\kappa r + \lambda ( A - \varphi(r)) + M r_+ \po \varphi(-r)-\varphi(r)\pf \leqslant - \eta\,, 
\end{equation}
then, using that $\varphi(\theta)\in[1,2]$, 
we will conclude that, for $t\leqslant \min( \eta/C_0,1/(\lambda_r+\|W\|_\infty\kappa))/2$, 
\[P_{t}^1 P_{t}^3 P_{2t}^2 P_{t}^3 P_{t}^1 V \leqslant e^{\kappa W} \po\varphi(\theta )- 2\eta t + \frac12 \eta t  + \frac12 \eta t  \pf + t C \leqslant \po 1 - \frac{t \eta}{2}\pf V + t C \,,\]
which will conclude the proof of the first part of Theorem~\ref{thm:ergodicBPS} for the scheme \textbf{DRBRD}.

Let us check what happens for different splitting orders. The computations are similar, in particular an important point is to check that the bound for $P_t^1$ (resp. $P_t^2$) is always used for a function $g$ which is non-negative (resp. non-decreasing in $\theta$), which is ensured each time by the assumption that $t$ is small enough, as in the previous  \textbf{DRBRD} case.  For instance, for \textbf{BDRDB}, we get, for $t$ small enough, 
\begin{align*}
& \lefteqn{P_{t}^2 P_{t}^1 P_{2t}^3 P_{t}^1 P_{t}^2 V}\\
&\leqslant   P_{t}^2 P_{t}^1 P_{2t}^3 P_{t}^1 \po e^{\kappa W} \po \varphi(\theta) +t M \theta_+ (\varphi(-\theta)-\varphi(\theta)) + t^2 C\pf   \pf  +  tC \\
&\leqslant   P_{t}^2 P_{t}^1 P_{2t}^3  \po e^{\kappa W} \po (1+t\kappa \theta)\varphi(\theta) + t \varepsilon \|\varphi'\|_\infty  +t M \theta_+ (\varphi(-\theta)-\varphi(\theta)) + t^2 C\pf   \pf  +  tC \\
&\leqslant   P_{t}^2 P_{t}^1  \po e^{\kappa W} \po (1+t\kappa \theta)\varphi(\theta) + 2\lambda_r t(A-\varphi(\theta)) +  t \varepsilon \|\varphi'\|_\infty  +t M \theta_+ (\varphi(-\theta)-\varphi(\theta)) + t^2 C\pf   \pf  +  tC \\
&\leqslant   P_{t}^2    \po e^{\kappa W} \po (1+2t\kappa \theta)\varphi(\theta) + 2\lambda_r t(A-\varphi(\theta)) +  2t \varepsilon \|\varphi'\|_\infty  +t M \theta_+ (\varphi(-\theta)-\varphi(\theta)) + t^2 C\pf   \pf  +  tC \\
&\leqslant   e^{\kappa W} \po (1+2t\kappa \theta)\varphi(\theta) + 2\lambda_r t(A-\varphi(\theta)) +  2t \varepsilon \|\varphi'\|_\infty  +2 t M \theta_+ (\varphi(-\theta)-\varphi(\theta)) + t^2 C\pf      +  tC \,,
\end{align*}
which is exactly the same bound as in the \textbf{DRBRD} case, as expected since we only keep track of the first order terms in $t$, which coincide for all splitting orders. The other cases are similar.

\subsection*{Conclusion}
Let us check that \eqref{eq:eta} holds with our choices of parameters. Using that $\varphi(r) \leqslant  2$ for all $r\in\R$  and $\varphi(r)=1$ for all $r \leqslant -c_V/2$ and that $|\na V(x)|\geqslant c_V$ if $|x|\geqslant R$,
\[A \leqslant q \times 1  + (1-q) \times 2 = 2 - q \,.\]
Then, for $r \geqslant -c_V/4$, since $\varphi(r) \geqslant 2-q/2$,
\[A-\varphi(r) \leqslant -q/2\,.\]
The choice of $\kappa$ ensures that, for all $r\leqslant - c_V/4$,
\[\kappa r + \lambda_r ( A - \varphi(r)) \leqslant \kappa r + \lambda_r ( 1 -q) \leqslant -\lambda_r q  \,. \] 
For $r\in [c_V/4, q \lambda_r /(4\kappa)]$,
\[ \kappa r + \lambda_r ( A - \varphi(r)) \leqslant \kappa r -  \lambda_r q/2 \leqslant -\lambda_r q/4\,.\]
Finally, for $r\geqslant q\lambda_r /(4\kappa)$, the choice of $M$ ensures that
\begin{align*}
&\kappa r + \lambda_r ( A - \varphi(r)) + M r \po \varphi(-r)-\varphi(r)\pf  \\
&\quad\leqslant  \po \kappa - M  \po \varphi\po  \frac{q\lambda_r}{4\kappa} \pf -\varphi\po -  \frac{q\lambda_r}{4\kappa} \pf\pf \pf r  - \lambda_r q/2\\
&\quad \leqslant - \lambda_r q/2\,,
\end{align*}
which concludes the proof of the first part of Theorem~\ref{thm:ergodicBPS}.

\subsection{Doeblin condition}

In this section we establish \eqref{eq:BPSDoeblin}.
The proof is an adaptation from \citet[Lemma 5.2]{Monmarche2016} (see also \citet[Lemma 11]{BPS_Durmus}), with the additional difficulty that time is discrete.  The proof essentially relies on a controllability argument: the point is to construct trajectories starting at any point of a given ball and reaching any point of another suitable ball. Since we only need to lower bound the transition probability, we can focus for simplicity on particular trajectories. Specifically, in the following, we will be interested in trajectories which, during the time $4R$ (i.e. in the first $n_*$ iterations),  have exactly three refreshments, occurring at times close to $0$, $2R$ and $4R$, and no bounces (cf. the events $\mathcal A_1$ and $\mathcal A_2$ below). Then, for these trajectories, the density of the final state will essentially be given by the densities of the velocities $V_1,V_2,V_3$ sampled at the refreshment events (the densities of $V_1$ and $V_2$ providing the density of the final position, while the density of $V_3$ gives the density of the final velocity).  

We formalise this argument.  Fix $R>0$ and consider an initial condition $(x,v)\in\R^d\times\mathbb S^{d-1}$ with $|x|\leqslant R$. We want to prove that the law of the process at time $n_* = \lceil 4R/\delta\rceil $ is bounded below by a positive constant (independent from $\delta$ small enough) times the Lebesgue measure on some domain, uniformly in $x$ with $|x|\leqslant R$. Since the process moves at unitary speed, for all $n\in\N$, $|\Xbar_n|\leqslant R + n\delta$, and thus
\[\mathbb P\po \text{no bounce in the $n$ first steps}\pf \geqslant e^{-h(n\delta)}\]
with $h(t) = \sup\{\|\na \psi(x)|, |x| \leqslant R +t\}$. In the absence of bounces, depending on the  splitting order, the chain behaves either as the chain given by the splitting \textbf{DRD} or \textbf{RDR}. To fix ideas, we only consider the \textbf{DRD} case in the following (corresponding either to \textbf{DRBRD}, \textbf{DBRBD} or \textbf{BDRDB}), the proof being similar in the other case. 

For $k\geqslant 1$, denote by $T_k$ the number of transitions of the chain between the $(k-1)^{th}$ refreshment and the $k^{th}$ one, so that $(T_k)_{k\in\N}$ is an i.i.d. sequence of geometric random variables with parameter $1-e^{-\delta \lambda_r}$, independent from the random variables used to define the bounces. Similarly, let $(V_k)_{k\in\N}$ be the i.i.d. sequence of random variables uniformly distributed on the sphere used to define the velocities at the refreshment times. For a small parameter $\eta>0$ to be fixed later on, consider the events
\begin{align*}
\mathcal A_1 &= \left\{T_1+T_2+T_3\leqslant n_*,\,   | \delta T_j-2R| \leqslant \eta R \text{ for }j=2,3 \right\}\\
\mathcal A_2 &= \{\delta (T_4-1) > 2\eta R, \text{ no bounce in the $n_*$ first steps}\}\,.    
\end{align*}
In particular, under $\A_1\cap \A_2$, exactly three refreshments occur during the $n_*$ first transitions. Then, for all Borel set $\mathcal B$ of $\mathbb R^d\times\mathbb S^{d-1}$,
\begin{eqnarray*}
\mathbb P\po \Zbar_{n_*} \in \mathcal B\pf & \geqslant & \mathbb P\po \Zbar_{n_*} \in \mathcal B,\mathcal A_1, \mathcal A_2\pf\\
& = & \mathbb P\po (\tilde X,V_3)\in\mathcal B,\mathcal A_1,\mathcal A_2\pf\\
& = & \mathbb P(\mathcal A_2) \mathbb P\po (\tilde X,V_3)\in\mathcal B,\mathcal A_1\pf 
\end{eqnarray*}
with, in the \textbf{DRD} case for instance,
\[\tilde X = x+\delta\po \po T_1 - \frac12\pf  v+T_2 V_1 + T_3 V_2 + \po n_*-T_1-T_2-T_3+\frac12\pf V_3\pf \,. \]
(Indeed, the \textbf{DRD}  scheme    starts and ends with a half step of transport). Notice that $\mathbb P(\A_2)$ is lower bounded uniformly in $\delta\in(0,1]$ as
\[\mathbb P(\A_2) \geqslant e^{-h(4R+1)} e^{- \lambda_r(2\eta R+1)}\,.\]
On the other hand, writing \[\mathcal I=\{(t_1,t_2,t_3)\in\N^3,\ t_1+t_2+t_3\leqslant n_*, |\delta t_j -2R|\leqslant \eta R\text{ for } j=2,3 \}\,,\]
we get
\begin{align*}
& \mathbb P\po (\tilde X,V_3)\in\mathcal B,\A_1\pf\\
= &\sum_{\mathcal I}  (1-e^{-\delta\lambda_r})^3 e^{-\delta \lambda_r(t_1+t_2+t_3-3)}\\
 & \int_{(\mathbb S^{d-1})^3} \1_{\mathcal B}\po  x+\delta\po \po t_1 - \frac12\pf v+t_2v_1+t_3v_2+\po n_*-t_1-t_2-t_3 + \frac12\pf v_3\pf ,v_3\pf \dd v_1\dd v_2 \dd v_3\\
\geqslant &
\sum_{\mathcal I}  (1-e^{-\delta\lambda_r})^3 e^{-\delta \lambda_r(t_1+t_2+t_3-3)}
\inf_{|x'|\leqslant R(1+2\eta)}
\int_{(\mathbb S^{d-1})^3} \1_{\mathcal B}\po x'+ \delta(t_2v_1+t_3v_2),v_3\pf \dd v_1\dd v_2 \dd v_3  \,.
\end{align*}
By the rotation invariance of the uniform measure on the sphere, for fixed $t,s>0$, $sV_1+t V_2$ has the same distribution as $V_1 |s w+tV_2|$ where $w$ is a fixed unitary vector. Then, $|sw+t V_2|^2 = s^2 + t^2 + 2st V_2^T w$ and $V_2^T w$ has on $[-1,1]$ the probability density proportional to $u\mapsto (1-u^2)^{d/2-1}$  (this is here we use that $d\geq 2$), which is lower bounded by a positive constant on  $[-1+\varepsilon,1-\varepsilon]$ for all $\varepsilon>0$. Assuming that $\eta \leqslant 1/4$ and considering for $y\in\R^d$  the ring $\mathcal R(y) = \{x\in\R^d,\  4\eta R \leqslant |x-y| \leqslant 4R(1-2\eta)\}$, we get that $sV_2+t V_3$ has a density which is lower bounded on $\mathcal R(0)$ by a  constant $\alpha>0$ which is independent from $t,s\in[R(2-\eta),R(2+\eta)]$.

As a consequence, for $(t_1,t_2,t_3)\in\mathcal I$ and $x'\in\R^d$ with $|x'|\leqslant R(1+2\eta)$, the law of $(x'+\delta(t_2V_1+t_3V_2),V_3)$ has a density lower bounded by $\alpha $ on $\mathcal R(x') \times \mathbb S^{d-1}$. Assuming that $\eta\leqslant 1/16$, let $\mathcal R' = \{y\in\R^d,\ R(1+6\eta)\leqslant |y|\leqslant R(3-10\eta)\}$ (which has a non-zero Lebesgue measure). The triangle inequality implies that $\mathcal R' \subset \mathcal R(x')$ if $|x'|\leqslant R(1+2\eta)$. As a consequence,
\[\mathbb P\po (\tilde X,V_3)\in\mathcal B,\A_1\pf \geqslant \mathbb P\po \mathcal A_1  \pf  \alpha \int_{\mathcal R'\times\mathbb S^{d-1}}\1_{\mathcal B}(y,v) \dd y \dd v\,,\]
which concludes the proof of \eqref{eq:BPSDoeblin} since $\mathbb P(\mathcal A_1)$ converges to a positive constant as $\delta\rightarrow 0$.


\section{Ergodicity for splitting schemes of ZZS}\label{sec:proofs_ergo_zzs}

\subsection{Splitting DBD}\label{sec:ergodicity_DBD_zzs}
In this Section we focus on splitting scheme DBD for ZZS. In order to prove Theorem \ref{thm:ergodicity_zzs} we prove a drift condition for the one step transition kernel $P_\delta$, implying a similar conclusion for $P^2_\delta$, as well as a minorisation condition after an even number of steps $n^*$. These two statements, together with the aperiodicity of $P^2$, enable us to conclude on the geometric ergodicity of the kernel $P^2$.

\begin{theorem}\label{thm:ZZS_ergodic_details}
Consider the splitting scheme \textbf{DBD} for ZZS. Suppose Assumption~\ref{ass:ergodicity_zzs} holds. Then the following hold:
\begin{enumerate}
\item  Let $\beta \in (0,1/2) $, $\phi(s) = \tfrac{1}{2} \textnormal{sign}(s) \ln{(1+2|s|)}$ and
\begin{equation}\label{eq:lyap_zzs}
	V(x,v) = \exp{\left( \beta \pot(x) + \sum_{i=1}^d \phi(v_i \partial_i \pot(x)) \right) }.
\end{equation}
There exist constants $b\in(0,1), C<\infty$ such that for all $(x,v)$ and all $\delta\in(0,\delta_0)$ with $\delta_0=2(1+\gamma_0)^{-1}$, where $\gamma_0$ is as in Assumption \ref{ass:ergodicity_zzs}(b), it holds that
	\begin{equation}\label{eq:driftcondZZS}
	P_\delta V(x,v) \leq 
	(1-b \delta) V(x,v) + C\delta. 
	\end{equation}
	
    \item For any $R>0$ consider a set $C=[-R,R]^d$. 
    For some $L>0$ and for $\tilde y$ as in Assumption \ref{ass:ergodicity_zzs}(a) let 
\begin{equation}\label{eq:N_smallset}
    n_* := 2 +  \frac{4\tilde y+2R}{\delta} + 2\left\lceil \frac{L}{\delta} \right\rceil.
\end{equation}
which satisfies $n_*\in 2\N$. For $(x,v)\in C\times\{\pm 1\}^d$  define the set $D(x,v)$ given by \eqref{eq:D_minorisation}.
Then for any $(y,w)\in D(x,v)\cap (C\times\{\pm 1\}^d)$ and $\delta \in (0,\delta_0]$ for $\delta_0>0$ it holds that
\[\delta_{y,w} P_\delta^{n_*} \geqslant c \nu\,.\]
where $c$ is independent of $\delta$ and $\nu$ is uniform over $D(x,v)\cap (C\times\{\pm 1\}^d)$.
\end{enumerate}
\end{theorem}

We observe that the Lyapunov function defined in \eqref{eq:lyap_zzs} is the same of the continuous time ZZS of \citet{Bierkensergodicity}.
In Section \ref{sec:minorisation_DBD} we prove the minorisation condition, in Section \ref{sec:drift_DBD} we prove the drift condition, while in Section \ref{sec:proof_erg_randomIC_zzs} we prove Equation \eqref{eq:convergence_random_IC_DBD}.

\subsubsection{Minorisation condition}\label{sec:minorisation_DBD}
We now prove a minorisation condition for splitting scheme DBD of ZZS. 
In the following Lemma we consider the  one-dimensional setting, for which the reasoning is similar to that of the proof of a minorisation condition for the continuous ZZS done in \citet[Lemma B.2]{bertazzi2020adaptive}. 

\begin{lemma}\label{lem:smallset_ZZS_1D}
Consider the splitting scheme DBD of ZZS with step size $\delta \leq \delta_0$. Suppose Assumption \ref{ass:ergodicity_zzs}(a) holds for some $\tilde y\geq 0$ and consider a set $C=[-R,R]$ for  $R>0$. 
For $L>0$ and for $\tilde y$ as in Assumption \ref{ass:ergodicity_zzs}(a) let $n_*$ as in \eqref{eq:N_smallset}.
For $(x,v)\in C\times \{\pm 1\}$ define the set $D(x,v):= D_+(x,v) \cup D_-(x,v)$, where 
\begin{equation}\label{eq:D_+andD_-}
    \begin{aligned}
        D_+(x,v) & := \{(y,w): w=v, \, y=x+m\delta,\, m\in  2\mathbb{Z}\}, \\
        D_-(x,v) & := \{(y,w): w=-v,\, y=x+m\delta,\, m\in  2\mathbb{Z}+1\}.
    \end{aligned}
\end{equation}
Then for any $(y,w)\in D(x,v)\cap (C\times \{\pm 1\})$ it holds that
\begin{equation*}
    \mathbb{P}_{(y,w)}((X_{n_*},V_{n_*})\in \cdot) \geq b \nu(\cdot )
\end{equation*}
where $b$ is independent of $\delta$ and $\nu$ is uniform over $D(x,v) \cap (C\times \{\pm 1\})$.  
\end{lemma}


\begin{proof}
Let $C=[-R,R]$ for a fixed $R>0$ and let $x\in C$. We shall consider only the case of $v=+1$, as the same arguments extend to the symmetric case $v=-1$. In particular observe that if the process is started in set $D(x,+1)$ (respectively $D(x,-1)$), then after an even number of iterations it will again be in $D(x,+1)$ ($D(x,-1)$). This means that the process lives on $D(x,+1)$ (respectively on $D(x,-1)$). To shorten the notation we denote by $D_+,D_-$ the sets $D_+(x,+1),D_-(x,+1)$ as defined in \eqref{eq:D_+andD_-}. Below we focus on the case of an initial condition in $D_+$, while the case of $D_-$ follows with an identical reasoning and obvious changes.

Let $n_*$ as in \eqref{eq:N_smallset} and define
$$\overline{\lambda}:= \max_{x\in C} \max_{y: \lvert y-x\rvert\leq n_*\delta ,v=\pm 1} \lambda(y,v)$$ 
which is the largest switching rate that can be reached within $n_*$ iterations starting in $C$. Note that our definition of $n_*$, that is \eqref{eq:N_smallset}, implies that $n_*\delta$ is upper bounded by a constant as $\delta\leq \delta_0$ and thus $\overline{\lambda}$ can be chosen independently of the step size $\delta$.
Recall $\underline{\lambda}>0$.

From here on we shall denote the initial condition as $(y,w)\in D_+ $, and  without loss of generality we shall assume $\tilde y = x + \ell \delta$ for some $\ell \in \N$, where $\tilde y$ is as in Assumption \ref{ass:ergodicity_zzs}(a). We want to lower bound the probability that after $n_*$ iterations the process is in measurable sets $\Bbar \subset D$.
We consider two cases: in the first one the final state of the process is of the form $(X_{n_*},V_{n_*}) = (z,-1) \in D_- \cap (C\times\{\pm 1\})$, while in the second case $(X_{n_*},V_{n_*}) = (z,+1) \in D_+\cap (C\times\{\pm 1\})$.

\subsubsection*{First case} Consider the case in which the final state has negative velocity, i.e. $V_{n_*}=-1$.
To lower bound the probability of reaching this state, we consider the case in which only one switching event takes place. Let $z=y+m\delta$ with $m\in\N$ odd. Then in order to have $(X_{n_*},V_{n_*})=(z,-1)$ with exactly one event taking place at time $N_1$ it must be that 
$$y + (N_1-1)\delta - (n_*-N_1)\delta =z.$$
Thus we find that the event should take place at
$$N_1 = \frac{z-y}{2\delta} + \frac{n_*+1}{2}. $$
In order to guarantee the switching rate is strictly positive it must also be that $X_{N_1}\geq \tilde y$, i.e. $y+(N_1-1)\delta\geq \tilde y$ and thus $N_1 \geq 1+ (\tilde y-y)/\delta$. Note $N_1 < n_*$ as required.
Denote the position at the time of the switching event by $\tilde{x}=y+\delta(N_1-1/2)$. Then the probability of exactly one event taking place at iteration $N_1$ is given by
\begin{align*}
    \int_0^\delta \lambda(\tilde{x},1) \exp(-s\lambda(\tilde{x},1)) \exp(-(\delta-s)\lambda(\tilde{x},-1))  ds 
    & \geq \delta \underline{\lambda} \exp(-\delta \overline{\lambda}).
\end{align*}
The probability of this path is simple to lower bound, since upper bounding the switching rates gives a smaller probability:
\begin{align*}
     \mathbb{P}_{(y,+1)}((X_{n_*},V_{n_*}) = (z,-1))  &\geq \underbrace{\prod_{n=0}^{N_1-1} \exp(-\delta \lambda(y+(n+1/2)\delta))}_{\textnormal{no jump before }N_1} \times \underbrace{\delta \underline{\lambda} \exp(-\delta \overline{\lambda})}_{\textnormal{a jump at }N_1} \times \\
    & \qquad \times  \underbrace{\prod_{n=0}^{n_*-N_1} \exp(-\delta \lambda(y+(N_1-1 - n)\delta))}_{\textnormal{no jump after }N_1} \\
    & \geq \exp(-(N_1-1)\delta \overline{\lambda}) \,\, \delta \underline{\lambda} \exp(-\delta \overline{\lambda}) \,\,\exp(-(n_*-N_1)\delta \overline{\lambda}) \\
    & \geq 2\exp(-(n_*-1)\delta \overline{\lambda}) \,\, \underline{\lambda} \exp(-\delta_0 \overline{\lambda}) \times \left(\frac{1}{2} \frac{\delta M}{M} \right)\\
    & \geq 2 \exp(-(n_*-1)\delta \overline{\lambda}) \,\, \underline{\lambda} \exp(-\delta_0 \overline{\lambda}) (2R-\delta) \left( \nu(-1) \times \frac{1}{M}\right),
\end{align*}
where $M \in \N$ is the number of points in $D_+\cap(C\times \{\pm 1\})$.
In the last line we used that $\delta M \geq 2R-\delta$. Recall that $\delta \leq \delta_0$. This concludes as $(n_*-1)\delta \leq 4\tilde y + R + 2L + 3\delta_0 $ and $2R-\delta\geq 2R-\delta_0$. 

\subsubsection*{Second case}
We now focus on the case in which $V_{n_*}=+1$. We shall find an appropriate lower bound by restricting to the case in which exactly two switching events take place. Denoting the times of the two events as $N_1,N_2$, if the final position is $z$ it must be
\begin{equation}\notag
    y + (N_1-1)\delta - (N_2-1) \delta + (n_*-N_1-N_2) \delta = z
\end{equation}
which implies 
\begin{equation}\label{eq:minor_constr_N2}
    N_2 = \frac{y-z}{2\delta} + \frac{n_*}{2}.
\end{equation}
Moreover, at event times the process should be in regions with strictly positive switching rate:
\begin{align*}
    & y+(N_1-1/2)\delta \geq \tilde y,\\
    & y+(N_1-1)\delta-(N_2-1/2)\delta \leq -\tilde y.
\end{align*}
These imply respectively 
\begin{align*}
    & N_1\geq \frac{\tilde y-y}{\delta} + 1 =: \underline{N}_1,\\
    & N_2 \geq \frac{y+\tilde y}{\delta} + N_1.
\end{align*}
Since $N_2$ is determined by \eqref{eq:minor_constr_N2}, we enforce that the second inequality holds:
\begin{align*}
    \frac{y-z}{2\delta} + \frac{n_*}{2} \geq \frac{y+\tilde y}{\delta} + N_1
\end{align*}
which implies 
\begin{align*}
    N_1 \leq \frac{n_*}{2} - \frac{y+2\tilde y+z}{2\delta} =: \overline{N}_1.
\end{align*}
Now to obtain the right dependence on $\delta$, we shall take $n_*$ such that $\overline{N}_1-\underline{N}_1$ is increasing as $1/\delta$. It holds 
\begin{align*}
    \overline{N}_1-\underline{N}_1 = \frac{n_*-2}{2} -\frac{4\tilde y-y+z}{2\delta}
\end{align*}
and thus it is sufficient to take
$$n_* = 2 +  \frac{4\tilde y-y+z}{\delta} + 2\left\lceil \frac{L}{\delta} \right\rceil$$
for some constant $L>0$, as with this choice $\overline{N}_1-\underline{N}_1 =\left\lceil  L/\delta \right\rceil $.

Using the results above we find
\begin{align*}
    \mathbb{P}_{(y,+1)}((X_{n_*},V_{n_*}) = (z,+1)) & \geq \sum_{N_1=\underline{N}_1}^{\overline{N}_1} \Bigg[ \underbrace{\prod_{n=0}^{N_1-1} \exp(-\delta \lambda(y+(n+1/2)\delta))}_{\textnormal{no jumps before }N_1} \times \underbrace{\delta \underline{\lambda} \exp(-\delta \overline{\lambda})}_{\textnormal{a jump at }N_1} \times \\
    & \qquad \times  \underbrace{\prod_{m=0}^{N_2-1} \exp(-\delta \lambda(y+(N_1-1 - (m+1/2))\delta))}_{\textnormal{no jumps until }N_2} \times \underbrace{\delta \underline{\lambda} \exp(-\delta \overline{\lambda})}_{\textnormal{a jump at }N_2} \\
    & \qquad \times \underbrace{\prod_{\ell=0}^{n_*-N_1-N_2}  \exp(-\delta \lambda(y+(N_1-1 - (N_2-1) + (\ell+1/2))\delta)) }_{\textnormal{no jumps after }N_2} \Bigg] \\
    & \geq \sum_{N_1=\underline{N}_1}^{\overline{N}_1} \exp(-\delta \overline{\lambda} n_* \delta)) \delta^2 \underline{\lambda}^2 \exp(-2\delta \overline{\lambda}) \\
    & = \left\lceil \frac{L}{\delta} \right\rceil \exp(-  \overline{\lambda} n_* \delta)) \delta^2 \underline{\lambda}^2 \exp(-2\delta \overline{\lambda}) \\
    & \geq L  \exp(-\delta \overline{\lambda} n_*  )) \underline{\lambda}^2 \exp(-2\delta_0 \overline{\lambda}) \delta\\
    & \geq 2L  \exp(-\delta \overline{\lambda} n_*  )) \underline{\lambda}^2 \exp(-2\delta_0 \overline{\lambda}) (2R-\delta_0) \left( \nu(+1) \times \frac{1}{M}\right).
\end{align*}
Similarly to above it is now sufficient to note that $n_*\delta \leq 4\delta_0+4\tilde y+2R+2L.$

\subsubsection*{Conclusion}
To conclude it is sufficient to observe that the conditions above hold for any choice of $x,y,z \in C$ since $n_*$ is as in \eqref{eq:N_smallset}.
\end{proof}

\subsubsection*{Multidimensional case}
To extend to the higher dimensional setting, first observe that it is possible to apply the same ideas in the proof of Lemma \ref{lem:smallset_ZZS_1D} to each component, in particular requiring that the events happen when all components of the process are outside of the rectangle $[-\tilde y,+\tilde y]^d$. This implies that Assumption \ref{ass:ergodicity_zzs}(a) can be used to lower bound the probability of flipping each component of the velocity vector. Hence each coordinate can be controlled independently of the others. It is clear that the following minorisation condition is implied: let $C=[-R,R]^d$ for $R>0$, $(x,v)\in C\times \{\pm 1\}^d$, and let $D(x,v)$ as in \eqref{eq:D_minorisation}; then for all $(y,w) \in (x,v)$ it holds that 
\begin{equation*}
    \mathbb{P}_{(y,w)}((X_{n_*},V_{n_*})\in \cdot) \geq b^d \,\,\nu_d(\cdot ),
\end{equation*}
where $n_*,b$ are as in Lemma \ref{lem:smallset_ZZS_1D} and $\nu_d$ is the uniform distribution over states in the grid $D(x,v)\cap (C\times \{\pm 1\}^d)$.

\subsubsection{Drift condition}\label{sec:drift_DBD}
Let us first characterise in the following Lemma the law of the jump part of the process. This result is then used to prove the wanted drift condition in Lemma~\ref{lem:drift_DBD} below.
\begin{lemma}\label{lem:probofajump}
	Let $\tilde{V}_t^x$ denote the PDMP corresponding to the generator $\cL_2$ (for this process $x$ acts as a parameter). Suppose that $\lambda_i(x,v)$ is independent of $v_j$ for $j\neq i$. Then for any $w\in \{\pm 1\}^d$ we have
	\begin{equation*}
	\mathbb{P}_v(\tilde{V}_t^x =w) = \prod_{i=1}^d\frac{\lambda_i(x,R_i w)+\frac{w_i}{v_i}\lambda_i(x,v)e^{-(\lambda_i(x,v)+\lambda_i(x,R_iv))t}}{\lambda_i(x,v)+\lambda_i(x,R_i v)}.
	\end{equation*} 
\end{lemma}

\begin{proof}
	To simplify notation we will suppress the dependence on $x$ and set $\Lambda_i(v)= \lambda_i(v)+\lambda_i(-v) = \lambda_i(x,v)+\lambda_i(x,R_i v)$. Since $\tilde{V}_t^x$ jumps according to $\lambda_i$ which does not depend on $v_j$ we have that the coordinates of $\tilde{V}_t^x$ are all independent. Hence it is sufficient to show
	\begin{equation*}
		\mathbb{P}_{v_i}((\tilde{V}_t^x)^i =w_i) = \frac{\lambda_i(-w_i)+\frac{w_i}{v_i}\lambda_i(v_i)e^{-\Lambda_i(v_i)t}}{\Lambda_i(v_i)}.
	\end{equation*} 
	Therefore it is sufficient to consider the setting $d=1$. Define for any $t\geq0, v,w\in \{1,-1\}$
	\begin{equation*}
\varphi_t(v;w) := \frac{\lambda(-w)+\frac{w}{v}\lambda(v)e^{-\Lambda(v)t}}{\Lambda(v)}.
	\end{equation*}
	If we show that for all $t\geq0, v,w\in \{1,-1\}$
	\begin{equation}\label{eq:BKEforphi}
	\partial_t \varphi_t(v;w) = \cL_2\varphi_t(v;w), \quad  \varphi_0(v;w) = \1_w(v),
	\end{equation}
	then $\varphi_t$ coincides with the semigroup applied to $\1_w$ and we have the desired result
	\begin{equation*}
	\varphi_t(v;w) = \mathbb{E}_v[\varphi_0(\tilde{V}_t;w)]= \mathbb{P}_v[\tilde{V}_t=w].
	\end{equation*}
It is straightforward to confirm the initial condition $\varphi_0(v;w) = \1_w(v)$ holds. So it remains to show that $\varphi_t$ satisfies the PDE~\eqref{eq:BKEforphi}. Note that
\begin{align*}
\partial_t\varphi_t(v;w) &= -\frac{w}{v}\lambda(v)e^{-\Lambda(v)t}\\
\cL_2\varphi_t(v;w) &= \lambda(v)\left(\varphi_t(-v;w)-\varphi_t(v;w) \right)\\
& = \lambda(v)\left(\frac{\lambda(-w)-\frac{w}{v}\lambda(-v)e^{-\Lambda(v)t}}{\Lambda(v)}-\frac{\lambda(-w)+\frac{w}{v}\lambda(v)e^{-\Lambda(v)t}}{\Lambda(v)} \right)\\
&=-\lambda(v)\left(\frac{\frac{w}{v}\lambda(-v)e^{-\Lambda(v)t}}{\Lambda(v)}+\frac{\frac{w}{v}\lambda(v)e^{-\Lambda(v)t}}{\Lambda(v)} \right)=-\lambda(v)\frac{w}{v}e^{-\Lambda(v)t}.
\end{align*}
Therefore we have that \eqref{eq:BKEforphi} holds.
\end{proof}

\begin{lemma}\label{lem:drift_DBD}
	Consider the splitting scheme DBD of ZZS. Let $\lambda_i(x,v)=(v_i\partial_i \pot(x))_++\gamma_i(x)$ and let Assumption~\ref{ass:ergodicity_zzs} be verified. Let $V$ be the function defined in \eqref{eq:lyap_zzs}. Then there exist constants $\rho\in(0,1), C<\infty$ such that for all $(x,v)$ and all $t\in(0,t_0)$ with $t_0<(1+\gamma_0)^{-1}$
	\begin{equation}\label{eq:drift DBD}
        \overline{P}_t V(x,v) =    P_t^D P_{2t}^B P_t^D V(x,v) \leq 
	(1-\rho t) V(x,v) + Ct. 
    \end{equation}
\end{lemma}

\begin{proof}
    We start the proof by deriving a first bound for $\overline{P}_t V$, where $V$ is the function defined in \eqref{eq:lyap_zzs} and which it is our goal to prove is a Lyapunov function for the chain. This will result in the inequality \eqref{eq:driftDBD_prod}. Afterwards, we divide our proof in two cases, based on whether we are considering an initial condition $(x,v)$ which is inside or outside of a large enough compact set. In both cases we shall prove that \eqref{eq:drift DBD} holds with the right time dependence. The most challenging case is when $(x,v)$ is outside of a large set, where involved computations are required.
    
	For a function $g(x,v)$ conditioning on the event $v=w$ and using Lemma~\ref{lem:probofajump} we have
	\begin{align*}
	    \overline{P}_t g(x,v) &= \sum_{w\in \{\pm1\}^d} g(x+vt+wt,w)\\
        & \quad \prod_{i=1}^d\left[\frac{\lambda_i(x+vt,R_iw)+\frac{w_i}{v_i}\lambda_i(x+vt,v)e^{-(\lambda_i(x+vt,v)+\lambda_i(x+vt,R_iv))t}}{\lambda_i(x+vt,v)+\lambda_i(x+vt,R_iv)}\right].
	\end{align*}
	Plugging in our Lyapunov function $V$ we have
 \begin{equation}
     \begin{aligned}\label{eq:lypfuncexp}
	    \frac{\overline{P}_t V(x,v)}{V(x,v)} &=  \sum_{w\in \{\pm1\}^d}\!\!\!\! \frac{V(x+vt+wt,w)}{V(x,v)}\\
     & \quad \prod_{i=1}^d\left[\frac{\lambda_i(x+vt,R_iw)+\frac{w_i}{v_i}\lambda_i(x+vt,v)e^{-(\lambda_i(x+vt,v)+\lambda_i(x+vt,R_iv))t}}{\lambda_i(x+vt,v)+\lambda_i(x+vt,R_iv)}\right].
	\end{aligned}
 \end{equation}
	By Taylor's theorem there  exists $\overline{x}_1=\overline{x}_1(x,v,w,t)\in B(x,t\sqrt{d})$ such that
	\begin{equation*}
	\pot(x+vt+wt) = \pot(x)+t\langle v+w,\nabla \pot(x)\rangle +\frac{t^2}{2}(v+w)^T\nabla^2 \pot(\overline{x}_1)(v+w).
	\end{equation*}
	Therefore we can rewrite \eqref{eq:lypfuncexp} as
	\begin{equation}\label{eq:taylor_exp_driftDBD}
	\begin{aligned}
	\frac{\overline{P}_t V(x,v)}{V(x,v)} = \sum_{w\in \{\pm1\}^d} e^{\frac{t^2}{2}(v+w)^T\nabla^2 \pot(\overline{x}_1)(v+w)}\prod_{i=1}^de^{t(v_i+w_i)\beta\partial_i\pot(x)+\phi(w_i\partial_i \pot(x+vt+wt))-\phi(v_i\partial_i\pot(x))}\\
	\times \left[\frac{\lambda_i(x+vt,R_iw)+\frac{w_i}{v_i}\lambda_i(x+vt,v)e^{-(\lambda_i(x+vt,v)+\lambda_i(x+vt,R_iv))t}}{\lambda_i(x+vt,v)+\lambda_i(x+vt,R_iv)}\right].
	\end{aligned}
	\end{equation}
	
	Since $\lvert \phi'(s)\rvert \leq 1$ for all $s$, by Taylor's Theorem we have
	\begin{equation*}
	\phi(w_i\partial_i \pot(x+vt+wt))-\phi(v_i\partial_i\pot(x))\leq \lvert w_i\partial_i \pot(x+vt+wt)-w_i\partial_i\pot(x)\rvert+\phi(w_i\partial_i \pot(x))-\phi(v_i\partial_i\pot(x)).
	\end{equation*}
	Then we can write
	\begin{equation}\label{eq:driftDBD_prod}
	    \frac{\overline{P}_t V(x,v)}{V(x,v)} \leq \sum_{w\in \{\pm1\}^d} K_1\prod_{i=1}^d I(i)
	\end{equation}
	with
	\begin{align*}
	K_1 &=  e^{\frac{t^2}{2}(v+w)^T\nabla^2 \pot(\overline{x}_1)(v+w)}e^{\sum_{i=1}^d\lvert w_i\partial_i \pot(x+vt+wt)-w_i\partial_i\pot(x)\rvert}\\
	I(i)&=e^{t(v_i+w_i)\beta\partial_i\pot(x)+\phi(w_i\partial_i \pot(x))-\phi(v_i\partial_i\pot(x))} \\
    &\quad \times \frac{\lambda_i(x+vt,R_iw)+\frac{w_i}{v_i}\lambda_i(x+vt,v)e^{-(\lambda_i(x+vt,v)+\lambda_i(x+vt,R_iv))t}}{\lambda_i(x+vt,v)+\lambda_i(x+vt,R_iv)} .
	\end{align*}
	We shall now use the bound \eqref{eq:driftDBD_prod} as a starting point, focusing first on the setting in which $(x,v)$ is outside of a large enough compact set.

	\subsection*{Bound outside of a compact set}
	We split the product in \eqref{eq:driftDBD_prod} into four cases: (i) $w_i=v_i$ and $v_i\partial_i\pot(x+vt)>0$; (ii) $w_i=v_i$ and $v_i\partial_i\pot(x+vt)<0$; (iii) $w_i=-v_i$ and $v_i\partial_i\pot(x+vt)>0$; (iv) $w_i=-v_i$ and $v_i\partial_i\pot(x+vt)<0$.
	
	Consider first case (i). Let $i$ be such that $w_i=v_i$ and $v_i\partial_i \pot(x+vt)>0$. Then
	\begin{align*}
	    I(i)=e^{2tv_i\beta\partial_i \pot(x)}\frac{\lambda_i(x+vt,R_i v)+\lambda_i(x+vt,v)e^{-(\lambda_i(x+vt,v)+\lambda_i(x+vt,R_i v))t}}{\lambda_i(x+vt,v)+\lambda_i(x+vt,R_i v)}.
	\end{align*}
	Using the form of $\lambda_i$ we can write this as
	\begin{align*}
	    I(i)=\frac{\gamma_i(x+vt)e^{2\beta tv_i\partial_i\pot(x)}(1-e^{-(\lvert \partial_i\pot(x+vt)\rvert+2\gamma_i(x+vt))t})}{\lvert \partial_i\pot(x+vt)\rvert+2\gamma_i(x+vt)}+e^{-(\lvert \partial_i\pot(x+vt)\rvert+2\gamma_i(x+vt))t+2tv_i\beta\partial_i\pot(x)}.
	\end{align*}
	Using that $1-e^{-z}\leq z$ for all $z>0$ (note we make use of this inequality several times in the following computations) we find
	\begin{align*}
	    I(i)\leq \gamma_i(x+vt)e^{2\beta tv_i\partial_i\pot(x)}t+e^{-(\lvert \partial_i\pot(x+vt)\rvert+2\gamma_i(x+vt))t+2tv_i\beta\partial_i\pot(x)}.
	\end{align*}
	By Assumption \ref{ass:ergodicity_zzs}(b) we can bound $\gamma_i$ for $\lvert x\rvert \geq R$ with $R$ sufficiently large and we have $v_i\partial_i\pot(x+vt)\geq 1+\gamma_i(x+vt)$ which gives 
	\begin{align*}
    	I(i)\leq\frac{1+( v_i\partial_i\pot(x+vt)+\gamma_i(x+vt))}{\lvert \partial_i\pot(x+vt)\rvert+2\gamma_i(x+vt)}e^{-\lvert \partial_i\pot(x+vt)\rvert t+2tv_i\beta\partial_i\pot(x)} \leq 2 e^{-\lvert \partial_i\pot(x+vt)\rvert t+2tv_i\beta\partial_i\pot(x)} .
	\end{align*}
	
	For case (ii), let $i$ be such that $w_i=v_i$ and $v_i\partial_i \pot(x+vt)<0$. Then
	\begin{align*}
	I(i)&=e^{2\beta tv_i\partial_i\pot(x)}\frac{\lvert \partial_i\pot(x+vt)\rvert+\gamma_i(x+vt)+\gamma_i(x+vt)e^{-(\lvert \partial_i\pot(x+vt)\rvert+2\gamma_i(x+vt))t}}{\lvert \partial_i\pot(x+vt)\rvert+2\gamma_i(x+vt)}\leq e^{2\beta tv_i\partial_i\pot(x)}.
	\end{align*}
	
	For case (iii), let $i$ be such that $w_i=-v_i$ and $v_i\partial_i \pot(x+vt)>0$. Then
	\begin{equation*}
	    I(i)= e^{\phi(-v_i\partial_i \pot(x))-\phi(v_i\partial_i\pot(x))}\frac{\lambda_i(x+vt,v)-\lambda_i(x+vt,v)e^{-(\lambda_i(x+vt,v)+\lambda_i(x+vt,R_i v))t}}{\lambda_i(x+vt,v)+\lambda_i(x+vt,R_iv)}.
	\end{equation*}
	For $s>0$ it holds that $\phi(-s)-\phi(s) = -\ln (1+2 s)$
	and hence
	\begin{equation*}
	I(i)=\frac{\lambda_i(x+vt,v)}{1+2 v_i\partial_i\pot(x)}\frac{1-e^{-(\lambda_i(x+vt,v)+\lambda_i(x+vt,-v))t}}{\lambda_i(x+vt,v)+\lambda_i(x+vt,R_i v)}\leq \frac{\lambda_i(x+vt,v)}{1+2 v_i\partial_i\pot(x)}t.
	\end{equation*}
	
	For case (iv), let $i$ be such that $w_i=-v_i$ and $v_i\partial_i \pot(x+vt)<0$. Then
	\begin{equation*}
	I(i)=e^{\phi(-v_i\partial_i \pot(x))-\phi(v_i\partial_i\pot(x))}\frac{\gamma_i(x+vt)-\gamma_i(x+vt)e^{-(\lambda_i(x+vt,v)+\lambda_i(x+vt,R_iv))t}}{\lambda_i(x+vt,v)+\lambda_i(x+vt,R_i v)}.
	\end{equation*}
	For $s<0$ we have $
	\phi(-s)-\phi(s) = \ln (1+2 |s|)$, and thus we obtain
	\begin{equation*}
	I(i)\leq\gamma_i(x+vt)(1+2\lvert\partial_i\pot(x)\rvert)t.
	\end{equation*}
	
	Combining these estimates we have for $\lvert x\rvert \geq R$ with R sufficiently large
	\begin{align*}
	    & \frac{\overline{P}_t V(x,v)}{V(x,v)} 
	    \leq \sum_{w\in \{\pm1\}^d}\! K_1\!\prod_{i:w_i=v_i,\, v_i\partial_i\pot>0} (\gamma_i(x+vt)e^{2\beta tv_i\partial_i\pot(x)}t+e^{-(\lvert \partial_i\pot(x+vt)\rvert+2\gamma_i(x+vt))t+2tv_i\beta\partial_i\pot(x)})\\
	    &\times\prod_{i:w_i=v_i,\, v_i\partial_i\pot <0}e^{2\beta tv_i\partial_i\pot(x)}  \prod_{i:w_i=-v_i,\, v_i\partial_i\pot>0}\frac{\lambda_i(x+vt,v)}{1+2 v_i\partial_i\pot(x)}t\prod_{i:w_i=-v_i,\, v_i\partial_i\pot<0}\gamma_i(x+vt)(1+2\lvert\partial_i\pot(x)\rvert)t.
	\end{align*}
	
	Now consider $K_1$. By Taylor's theorem there exists $\overline{x}_2\in B(x,2\sqrt{d}t)$ such that
	\begin{equation*}
	    	K_1 \leq  \exp\left(\sum_{i=1}^d \left(\frac{t^2}{2}\lvert((v+w)^T\nabla^2 \pot(\overline{x}_1))_i\rvert+t\lvert(w\nabla^2 \pot(\overline{x}_2))_i\rvert\right)\rvert v_i+w_i\rvert\right).
	\end{equation*}
	Using this bound and the four cases above we now obtain
	\begin{align*}
	    &\frac{\overline{P}_t V(x,v)}{V(x,v)}\leq\!\!\!\! \!\prod_{i:v_i\partial_i\pot>0} \!\!\!\!\! e^{ 2t^2\lvert (v^T\nabla^2 \pot(\overline{x}_1))_i\rvert+\frac{2 t}{2}\lvert(v\nabla^2 \pot(\overline{x}_2))_i\rvert} \,\,(\gamma_i(x+vt)e^{2\beta tv_i\partial_i\pot(x)}t+e^{-(1-2\beta)\lvert \partial_i\pot(x+vt)\rvert t-2\gamma_i(x+vt)t})\\
	    &\times \prod_{i: v_i\partial_i\pot<0} e^{ \left(\frac{t^2}{2}(2\lvert v^T\nabla^2 \pot(\overline{x}_1))_i\rvert+2 t\lvert(v\nabla^2 \pot(\overline{x}_2))_i\rvert\right)+2\beta tv_i\partial_i\pot(x)} + \sum_{w\in \{\pm1\}^d\setminus\{v\}}\! t^{|\{i:w_i\neq v_i\}|} \\
	    & \times \!\!\! \prod_{i:w_i=v_i, v_i\partial_i\pot>0} \!\!\!\!\!\!\! e^{ \left(t^2(\lvert(v+w)^T\nabla^2 \pot(\overline{x}_1))_i\rvert+2 t\lvert(w\nabla^2 \pot(\overline{x}_2))_i\rvert\right)}(\gamma_i(x+vt)e^{2\beta tv_i\partial_i\pot(x)}t+e^{-(1-2\beta)\lvert \partial_i\pot(x+vt)\rvert t-2\gamma_i(x+vt)t})\\
	    &\times\prod_{i:w_i=v_i, v_i\partial_i\pot<0} e^{ \left(t^2(\lvert(v+w)^T\nabla^2 \pot(\overline{x}_1))_i\rvert+2 t\lvert(w\nabla^2 \pot(\overline{x}_2))_i\rvert\right)}e^{2\beta tv_i\partial_i\pot(x)} \prod_{i:w_i=-v_i, v_i\partial_i\pot>0} \frac{\lambda_i(x+vt,v)}{1+2 v_i\partial_i\pot(x)}\\
	    &\times \prod_{i:w_i=-v_i, v_i\partial_i\pot<0} e^{-t_0\lvert \partial_i\pot(x+vt)\rvert}(1+2\lvert\partial_i\pot(x)\rvert).
	\end{align*}
	By \eqref{eq:boundongamma} we have
	\begin{align*}
	    &\frac{\overline{P}_t V(x,v)}{V(x,v)}\leq\! \!\prod_{i:v_i\partial_i\pot>0} (\gamma_0 t+e^{ 2t^2\lvert (v^T\nabla^2 \pot(\overline{x}_1))_i\rvert+\frac{2 t}{2}\lvert(v\nabla^2 \pot(\overline{x}_2))_i\rvert}e^{-(1-2\beta)\lvert \partial_i\pot(x+vt)\rvert t-2\gamma_i(x+vt)t})\\
	    &\times \prod_{i: v_i\partial_i\pot<0} e^{ \left(\frac{t^2}{2}(2\lvert v^T\nabla^2 \pot(\overline{x}_1))_i\rvert+2 t\lvert(v\nabla^2 \pot(\overline{x}_2))_i\rvert\right)+2\beta tv_i\partial_i\pot(x)}\\
	    &+\!\!\! \sum_{w\in \{\pm1\}^d\setminus\{v\}}\!\!\! t^{|\{i:w_i\neq v_i\}|} \!\!\!\!\! \!\prod_{i:w_i=v_i, v_i\partial_i\pot>0}\!\!\!\! (\gamma_0t+e^{ \left(t^2(\lvert(v+w)^T\nabla^2 \pot(\overline{x}_1))_i\rvert+2 t\lvert(w\nabla^2 \pot(\overline{x}_2))_i\rvert\right)}e^{-(1-2\beta)\lvert \partial_i\pot(x+vt)\rvert t-2\gamma_i(x+vt)t})\\
	    &\times\prod_{i:w_i=v_i, v_i\partial_i\pot<0} e^{ \left(t^2(\lvert(v+w)^T\nabla^2 \pot(\overline{x}_1))_i\rvert+2 t\lvert(w\nabla^2 \pot(\overline{x}_2))_i\rvert\right)}e^{2\beta tv_i\partial_i\pot(x)} \prod_{i:w_i=-v_i, v_i\partial_i\pot>0} \frac{\lambda_i(x+vt,v)}{1+2 v_i\partial_i\pot(x)}\\
	    &\times \prod_{i:w_i=-v_i, v_i\partial_i\pot<0} e^{-t_0\lvert \partial_i\pot(x+vt)\rvert}(1+2\lvert\partial_i\pot(x)\rvert).
	\end{align*}
	Since $\beta <1/2$, by Assumption~\ref{ass:ergodicity_zzs}(c) there exists $\beta_1$ such that for $\lvert x\rvert\geq R$ with $R$ sufficiently large
	\begin{align*}
	    &\frac{\overline{P}_t V(x,v)}{V(x,v)}\leq\! \!\prod_{i:v_i\partial_i\pot>0} (\gamma_0 t+e^{-\beta_1\lvert \partial_i\pot(x+vt)\rvert t}) \prod_{i: v_i\partial_i\pot<0} e^{ -\beta_1 \lvert\partial_i\pot(x)\rvert  t}\\
	    &+ \sum_{w\in \{\pm1\}^d\setminus\{v\}}\! t^{|\{i:w_i\neq v_i\}|} \!\prod_{i:w_i=v_i, v_i\partial_i\pot>0} (\gamma_0t+e^{ -\beta_1\lvert \partial_i\pot(x+vt)\rvert t})\prod_{i:w_i=v_i, v_i\partial_i\pot<0} e^{-\beta_1 t \lvert \partial_i\pot(x)\rvert}\\
	    &\times  \prod_{i:w_i=-v_i, v_i\partial_i\pot>0} \frac{\lambda_i(x+vt,v)}{1+2 v_i\partial_i\pot(x)}\prod_{i:w_i=-v_i, v_i\partial_i\pot<0} e^{-t_0\lvert \partial_i\pot(x+vt)\rvert}(1+2\lvert\partial_i\pot(x)\rvert).
	\end{align*}
	For $\lvert x\rvert \geq R$ with $R$ sufficiently large $\lambda_i(x+vt,v)/(1+2 v_i\partial_i\psi(x)) \leq 1$ and by \eqref{eq:boundongamma} we have $\gamma_i(x)(1+2\lvert\partial_i\pot(x)\rvert) \leq 1$. We also have that $\lvert \nabla \pot(x+vt)\rvert \geq M$ for any $M>0$ for $\lvert x\rvert \geq R$ with $R$ sufficiently large. Then we have
	\begin{align*}
	    &\frac{\overline{P}_t V(x,v)}{V(x,v)}\leq (\gamma_0 t+e^{-\beta_1M t})^{\lvert \{i: v_i\partial_i\psi(x+vt)>0\}\rvert} e^{ -\beta_1Mt \lvert \{i: v_i\partial_i\psi(x+vt)<0\}\rvert}\\
	    &+ \sum_{w\in \{\pm1\}^d\setminus\{v\}}\! t^{|\{i:w_i\neq v_i\}|} \!\prod_{i:w_i=v_i, v_i\partial_i\pot>0} (\gamma_0t+e^{ -\beta_1M t})\prod_{i:w_i=v_i, v_i\partial_i\pot<0} e^{-\beta_1 t M}.
	\end{align*}	
	Since $e^{ -\beta_1M t} \leq \gamma_0t+e^{ -\beta_1M t}$ we obtain
	\begin{align*}
	    &\frac{\overline{P}_t V(x,v)}{V(x,v)} \leq (\gamma_0 t+e^{-\beta_1M t})^{\lvert \{i: v_i\partial_i\psi(x+vt)>0\}\rvert} (\gamma_0 t+e^{-\beta_1M t})^{ \lvert \{i: v_i\partial_i\psi(x+vt)<0\}\rvert}\\
	    &  + \sum_{w \in \{\pm1\}^d\setminus\{v\}}\! t^{|\{i:w_i\neq v_i\}|} (\gamma_0 t+e^{-\beta_1M t})^{\lvert \{i:w_i=v_i, v_i\partial_i\psi(x+vt)>0\}\rvert} (\gamma_0 t+e^{-\beta_1M t})^{ \lvert \{i:w_i=v_i, v_i\partial_i\psi(x+vt)<0\}\rvert}. 
	\end{align*}
	Hence 
	\begin{align*}
	    \frac{\overline{P}_t V(x,v)}{V(x,v)} &\leq (\gamma_0 t+e^{-\beta_1M t})^d +  \sum_{w\in \{\pm1\}^d: \,w\neq v} t^{|\{i:w_i\neq v_i\}|} (\gamma_0 t+e^{-\beta_1M t})^{d-|\{i:w_i\neq v_i\}|}  \\
	    & =  \sum_{w\in \{\pm1\}^d} t^{|\{i:w_i\neq v_i\}|} (\gamma_0 t+e^{-\beta_1M t})^{d-|\{i:w_i\neq v_i\}|}  \\
	    & =  \,\, \sum_{k=0}^d \binom{d}{k} t^k\, (\gamma_0 t+e^{-\beta_1M t})^{d-k} \\
	    & = \left((1+\gamma_0) t+e^{-\beta_1M t}\right)^d .
	\end{align*}
    To show that \eqref{eq:drift DBD} holds for $\lvert x\rvert \geq R$ it is sufficient to show that $(1+\gamma_0) t+e^{-\beta_1M t} < 1-\rho t$ for some $\rho>0$. Indeed in that case $1-\rho t <1$ and thus $$((1+\gamma_0) t+e^{-\beta_1M t})^d <(1-\rho t)^d < 1-\rho t.$$ Note that for $t \leq t_0$, with $t_0\in[0,1]$, it holds that $e^{-\beta_1M t} \leq 1 - c t$ for $c = \frac{1-e^{-\beta_1M t_0 }}{t_0}$. Then for $t \leq t_0$ we have 
	$$(1+\gamma_0) t+e^{-\beta_1M t} \leq  1-t(c-1-\gamma_0).$$
	Then it is needed that $ c> 1+\gamma_0$, that is $t_0$ should be such that
	\begin{equation}\label{eq:conditiononM}
	    \frac{1-e^{-\beta_1 M t_0}}{t_0} > 1+ \gamma_0.
	\end{equation}
	Note we can always increase $M$ by taking $R$ larger. Choose $M$ such that $e^{-\beta_1 M t_0}<1-t_0(1+\gamma_0)$, which is possible since $t_0<(1+\gamma_0)^{-1}$, then \eqref{eq:conditiononM} holds. Hence \eqref{eq:drift DBD} holds for $\lvert x\rvert \geq R$ with $\rho=(1-e^{-\beta_1 M t_0})t_0^{-1}- 1- \gamma_0 $.
	\subsection*{Bound inside of a compact set}
	It remains to show that \eqref{eq:drift DBD} holds for $\lvert x\rvert \leq R$. Let $C=\{x:\lvert x\rvert \leq R\}\times\{\pm1\}^d$. Recall $t < 1$ and $\pot\in\mathcal{C}^2 $. We shall use the inequality $e^{tr} \leqslant 1+tr+t^2 r^2 e^{r}/2 \leq 1+t(r + e^{3  r}/2) $, which holds for for $t\leqslant 1, r>0$.
	First of all we consider the term in the sum corresponding to the case $w=v$. Bounding the probability of this event by $1$ we find
	\begin{align*}
	    K_1\prod_{i=1}^d I(i) & \leq e^{t^2 2v^T\nabla^2 \pot(\overline{x}_1)v + \frac{2}{2}\sum_{i=1}^d\lvert \partial_i \pot(x+2vt)-\partial_i\pot(x)\rvert + 2t\beta \langle v,\nabla\pot(x)\rangle} \\
	    & \leq 1+t (A(x,v,t) + e^{3A(x,v,t)}/2),
	\end{align*}
	where $A(x,v,t)= 2v^T\nabla^2 \pot(\overline{x}_1)v + (2/2) \sum_{i=1}^d \lvert \langle v,\nabla \partial_i \pot(\overline{x}_2)\rangle \rvert +2\beta \langle v,\nabla\pot(x)\rangle$. Taking the maximum of $A$ over $(x,v,t) \in C \times \{ \pm 1\}^d \times (0,1)$ we find
	\begin{align*}
	    K_1\prod_{i=1}^d I(i)  \leq 1+t \overline{A}.
	\end{align*}
	
	Let us now consider the remaining elements in the sum. Here we take advantage that a velocity flip is an order $t$ event. Consider for the moment only the $i$-th component of the velocity vector. The probability that this is flipped (i.e. $w_i=-v_i$) satisfies 
	\begin{align*}
	    &\frac{\lambda_i(x+vt,R_i w)-\lambda_i(x+vt,v)e^{-(\lambda_i(x+vt,v)+\lambda_i(x+vt,R_i v))t}}{\lambda_i(x+vt,v)+\lambda_i(x+vt,R_iv)} \\
        &\leq \frac{\lambda_i(x+vt,v) (1-e^{-t(\lvert \partial_i\pot(x+tv)\rvert + 2\gamma_i(x+tv))})}{\lvert \partial_i\pot(x+tv)\rvert + 2\gamma_i(x+tv)} \\
	    & \leq t\lambda_i(x+vt,v) \\
	    & \leq t \max_{i=1,\dots,d,\, (x,v)\in C\times\{\pm 1 \}^d,\, t\in(0,1)} \lambda_i(x+vt,v).
	\end{align*}
	Here we used that $1-\exp(-z) \leq z$ for $z\geq 0$. All other probabilities can be bounded by $1$ and hence 
	\begin{align*}
	    \sum_{w \neq v} K_1\prod_{i=1}^d I(i) \leq t \max_{i=1,\dots,d,\, (x,v)\in C\times\{\pm 1 \}^d,\, t\in(0,1)} \lambda_i(x+vt,v) \sum_{w \neq v}  \frac{V(x+vt+wt,w)}{V(x,v)}.
	\end{align*}
	
	Since $V$ is continuous under our assumptions we proved that for every compact set $C \times \{\pm 1\}^d$ there exists a constant $B>0$ such that for all $(x,v) \in C \times \{\pm 1\}^d$ it holds
	\begin{equation}\label{eq:drift_depn_t}
	        \overline{P}_t V(x,v) \leq  (1+t B) V(x,v).
	\end{equation}
Therefore we have \eqref{eq:drift DBD} holds for all $x\in \R^d, v\in \{\pm1\}^d$.
\end{proof}

\subsubsection{Proof of Equation \eqref{eq:convergence_random_IC_DBD}}\label{sec:proof_erg_randomIC_zzs}
Let us prove \eqref{eq:convergence_random_IC_DBD}.
Fix a probability measure $\mu$ on $\R^d\times \{\pm 1\}^d$, and let $(x,v)$ be a point in the support of $\mu$. Then we can construct the set $D(x,v)$ corresponding to $(x,v)$ and given by \eqref{eq:D_minorisation}. By Theorem \ref{thm:ergodicity_zzs} there is a unique invariant measure $\pi_\delta^{x,v}$ for the Markov process with kernel $P_\delta^2$, and by \eqref{eq:ergodic} for any probability measures $\nu,\nu'$ on $D(x,v)$
\begin{equation*}
\|\nu P_\delta^{2n} -\nu'P_\delta^{2n}\|_V \leq \frac{C }{\alpha}  \tilde{\kappa}^{n\delta} \,\|\nu-\nu'\|_V.
\end{equation*}
Setting $\nu=\delta_{x,v}$, $\nu'=\pi_\delta^{x,v}$ and using that $\pi_\delta^{x,v}$ is an invariant measure for the kernel $P_\delta^2$ we have
\begin{equation*}
\|\delta_{x,v} P_\delta^{2n} -\pi_\delta^{x,v}\|_V \leq \frac{C }{\alpha}  \tilde{\kappa}^{n\delta} \,\|\nu-\nu'\|_V.
\end{equation*}
Then integrating with respect to the probability measure $\mu$ we obtain
\begin{align*}
    \|\mu P_\delta^{2n} -\mu\pi_\delta^{x,v}\|_V  &= \sup_{\lvert g\rvert \leq V} \left\lvert \int [P_\delta^{2n}g(x,v) -\pi_\delta^{x,v}(g)] \mu(dx,dv)\right\rvert\\
    &\leq \int \sup_{\lvert g\rvert \leq V} \left\lvert  P_\delta^{2n}g(x,v) -\pi_\delta^{x,v}(g) \right\rvert \mu(dx,dv)\\
    &\leq \int  \left\lVert  \delta_{x,v}P_\delta^{2n} -\pi_\delta^{x,v} \right\rVert_V \mu(dx,dv)\\
    &\leq \frac{C }{\alpha}  \tilde{\kappa}^{n\delta} \,\int\|\delta_{(x,v)}-\pi_\delta^{x,v}\|_V\mu(dx,dv).
\end{align*}

\subsection{Other splitting schemes}\label{sec:proof_ergo_others_zzs}
In this Section we consider splitting schemes \DRBRD and \RDBDR of ZZS and prove geometric ergodicity in Theorem \ref{thm:ergo_others_zzs}. The minorisation and drift conditions are proved in Sections \ref{sec:minorisation_other_zzs} and \ref{sec:drift_other_zzs} respectively.
We shall work under the following assumption. 
\begin{assumption}\label{ass:refresh_zzs}
There exists $\gamma_0\in(0,\infty)$ such that the following conditions for the refreshment rate hold: 
\begin{enumerate}[label=(\alph*)]
    \item there exists $R>0$ for which for any $\lvert x\rvert \geq R$ it holds that 
    $$ \gamma(x) \prod_{j=1}^d (1+\lvert \partial_j \pot(x)\rvert) \leq  \gamma_0.$$
    \item For $\lvert x\rvert > R$ for some $R>0$ it holds that
    $$ \sup \gamma(x+vt)
    e^{t\lvert \nabla \pot(x)\rvert +\frac{t^2}{2}(v+w)^T\nabla^2 \pot(y_1)(v+w) + \lvert \nabla \pot(y_2)\rvert } \prod_{i=1}^d (1+2\lvert \partial_i \pot(x)\rvert) \leq \gamma_0,
    $$
    where the supremum is over $t\in(0,1),\,y_1,y_2\in B(x,t\sqrt{d}),\,v,w\in \{-1,1\}^d.$
\end{enumerate}
\end{assumption}

\begin{theorem}\label{thm:ergo_others_zzs}
Consider splitting schemes \DRBRD and \RDBDR of ZZS. Suppose Assumption \ref{ass:ergodicity_zzs}(a), (c) holds for switching rates $\lambda_i(x,v) = (v_i\partial_i\pot(x))_+$. Suppose moreover that the refreshment rate $\gamma$ satisfies Assumption \ref{ass:refresh_zzs}(a) for scheme \RDBDR and Assumption \ref{ass:refresh_zzs}(b) for scheme \textbf{DRBRD}. Then statements (1) and (2), as well as Equation \eqref{eq:convergence_random_IC_DBD}, hold. In particular (2) holds with $\delta_0 < 2(1+2\gamma_0+\gamma_0^2)^{-1}$ for \RDBDR and with $\delta_0<2(1+2\gamma_0)^{-1}$ for \textbf{DRBRD}.
\end{theorem}

\subsubsection{Minorisation condition}\label{sec:minorisation_other_zzs}
\subsubsection*{Splitting \DRBRD} 
The chain obtained by \DRBRD has the same periodic behaviour of DBD. Hence this case can be treated in the same way and a minorisation condition follows by the same reasoning used in Section~\ref{sec:minorisation_DBD} for splitting DBD.


\subsubsection*{Splitting \RDBDR} 
In this case we give a sketch of the proof.
The chain obtained by \RDBDR breaks the grid behaviour exhibited by DBD because of the two refreshment steps at the beginning and end of each step. Indeed, consider the one-dimensional case and recall the definition of the grid $D(x,v)$ as in Lemma \ref{lem:smallset_ZZS_1D}. Since $v=\pm 1$, there are two disjoint grids: $D(x,+1)$ and $D(x,-1)$, with the idea being that after even steps of DBD the process lives on the same grid it started from. However, the process \RDBDR can swap between one grid and the other by having a velocity refreshment. Indeed, starting the process at $(x,+1)$ and having a velocity flip due to a refreshment at the end of the first step and having no other jumps, we find the state of the process is $(X_{2},V_{2}) = (x,-1)$. Therefore after even steps this process lives on the grid $D(x)=\{y:y=x+m\delta, \, m\in \mathbb{Z} \}$. If  the initial and final condition are on the same grid $D(x,v)$, then no refreshment is required and one can simply use the proof of the scheme DBD. On the other hand, if the two states are on different grids, i.e. one is on $D(x,+1)$ and the other on $D(x,-1)$, then a refreshment is required to choose the right grid. 

In order to maintain the right dependence on the step size $\delta$ it is required to give the process additional $\lceil \frac{M}{\delta} \rceil$ iterations, for a constant $M>0$. Indeed with this modification the probability of having a refreshment in the first $\lceil \frac{M}{\delta} \rceil$ is constant has a lower bound which is independent of $\delta$, assuming $\delta\leq \delta_0$ for some $\delta_0>0$ (see for instance the second case in the proof of Lemma \ref{lem:smallset_ZZS_1D}). After the first $\lceil \frac{M}{\delta} \rceil$ iterations the process is on the right grid and Lemma \ref{lem:smallset_ZZS_1D} can be applied with the further constraint that no (more) refreshments take place. Note that this event again has a lower bounded probability independent of $\delta$. Since in the first $\lceil \frac{M}{\delta} \rceil$ iterations the process can go out of the initial compact set $C=[-R,R]$, it follows that the Lemma should be applied with set $\tilde{C}=[-R-M,R+M]$.

The extension to the multidimensional case follows by applying this same intuition to every component.

\subsubsection{Drift condition}\label{sec:drift_other_zzs}

Let us start with an auxiliary result.
\begin{lemma}\label{lem:drift_refreshpart}
Suppose the refreshment rate $\gamma$ satisfies Assumption \ref{ass:refresh_zzs}(a). Then $P_t^R V\leq (1+\gamma_0 t)V + Mt$, where $\gamma_0$ is as in Assumption \ref{ass:refresh_zzs} and $M$ independent of $t$.
\end{lemma}
\begin{proof}
Let $V$ be as in Lemma \ref{lem:drift_DBD}. Applying the transition kernel $P^R_t$ to $V$ we find
\begin{align*}
    P^R_t V(x,v) & =V(x,v)  e^{-t\gamma(x)}  + \frac{1}{2^d} (1-e^{-t\gamma(x)}) \sum_{w\neq v} V(x,w)
    \\
    & = V(x,v) e^{-t\gamma(x)} +  (1-e^{-t\gamma(x)})  V(x,v) \frac{1}{2^d}\sum_{w\neq v} \frac{V(x,w)}{V(x,v)} \\
    & = V(x,v)\Big( e^{-t\gamma(x)} +  (1-e^{-t\gamma(x)})  \frac{1}{2^d}\sum_{w\neq v} \prod_{j:v_j\neq w_j} (1+\lvert \partial_j \pot(x)\rvert ) \Big)\\
    & \leq V(x,v) \left( e^{-t\gamma(x)} + t  \gamma(x)  \prod_{j=1}^d (1+\lvert \partial_j \pot(x)\rvert )\right).
\end{align*}
Clearly for $x$ inside of a compact set this implies $P^R_t V(x,v) \leq (1+Bt) V(x,v)$ by taking maximum over $x$. Outside of a compact set we use Assumption \ref{ass:refresh_zzs} to obtain 
\begin{align*}
    P^R_t V(x,v) & \leq V(x,v) (1+t\gamma_0).
\end{align*}
\end{proof}

\begin{lemma}\label{lem:drift_RDBDR}
Consider the splitting scheme \RDBDR of ZZS. Suppose Assumptions \ref{ass:ergodicity_zzs}(c) and \ref{ass:refresh_zzs}(a) hold. Then there exist a function $V$ and constants $\rho\in(0,1)$, $C>0$ such that for any $t\in (0,t_0)$ with $t_0 < (1+2\gamma_0+\gamma_0^2)^{-1}$ it holds that 
$$P^R_t P_t^D P_{2t}^B P_t^DP^R_t \,V(x,v) \leq (1-\rho t)V(x,v) + Ct.$$
\end{lemma}
\begin{proof}
Let $V$ be as in Lemma \ref{lem:drift_DBD}. In the current context the result of the Lemma is that for all $t \in(0,t_0)$ with $t_0<1$ it holds that $P_t^D P_{2t}^B P_t^D V(x,v) \leq(1-\rho t) V(x,v) + Bt$ where $\rho=(1-e^{-R t_0})t_0^{-1}- 1$ for $R$ sufficiently large such that $\rho>0$.
Applying Lemmas \ref{lem:drift_DBD} and \ref{lem:drift_refreshpart} we  obtain
\begin{align*}
    P^R_t P_t^D P_{2t}^B P_t^DP^R_t V(x,v) &\leq (1+t\gamma_0) P^R_t P_t^D P_{2t}^B P_t^D V(x,v) + Mt \\
    & \leq (1+t\gamma_0) (1-\rho t) P^R_t  V(x,v) + t(M+(1+\gamma_0) B)  \\
    & \leq (1+t\gamma_0)^2 (1-\rho t) V(x,v) + t(M(2+\gamma_0)+(1+\gamma_0) B).
\end{align*}
It is left to ensure that $(1+t\gamma_0)^2 (1-\rho t) \leq (1-\tilde{\rho}t)$ for $\tilde{\rho}>0$. We have
\begin{align*}
    (1+t\gamma_0)^2 (1-\rho t) \leq (1-t(\rho-2\gamma_0 - \gamma_0^2)).
\end{align*}
Hence it is needed that
$$ \frac{(1-e^{-R t_0})}{t_0}- 1 > 2\gamma_0 + \gamma_0^2$$
and thus that $e^{-R t_0} < 1-t_0(1+2\gamma_0+\gamma_0^2),$ which is valid as $R$ can be taken as large as needed and  $t_0 <(1+2\gamma_0+\gamma_0^2)^{-1}.$ 
\end{proof}

\begin{lemma}
Consider the splitting scheme \DRBRD of ZZS.
Suppose Assumptions \ref{ass:ergodicity_zzs}(c) and \ref{ass:refresh_zzs} hold. Then there exist a function $V$ and constants $\rho\in(0,1)$, $C>0$ such that for any $t\in (0,t_0)$ with $t_0 < (1+2\gamma_0)^{-1}$ it holds that 
$$P_t^D P_t^R P_{2t}^B P_t^R P_t^D \,V(x,v) \leq (1-\rho t)V(x,v) + Ct.$$
\end{lemma}
\begin{proof}
Let $V$ be as in Lemma \ref{lem:drift_DBD}. Observe that by Lemma \ref{lem:drift_refreshpart} we have that 
\begin{align*}
    P_t^R P_t^D V(x,v) & = P_t^R V(x+vt,v) \leq (1+\gamma_0t ) V(x+vt,v) + Mt
\end{align*}
and thus $P_t^R P_t^D V(x,v) \leq (1+\gamma_0t) P_t^D V(x,v) + Mt.$
Then 
\begin{align*}
    P_t^D P_t^R P_{2t}^B P_t^R P_t^D V(x,v) & \leq (1+\gamma_0t) P_t^D P_t^R P_{2t}^B P_t^D V(x,v) + Mt
\end{align*}
and 
\begin{align*}
    P_t^D P_t^R P_{2t}^B P_t^D V(x,v) & = e^{-t\gamma(x+vt)} P_t^D P_{2t}^B P_t^D V(x,v) \\
    & \quad + V(x,v) (1-e^{-t\gamma(x+vt)})\sum_{w\in \{\pm 1\}^d} \frac{1}{2^d} \frac{V(x+vt+wt,w)}{V(x,v)}.
\end{align*}
The first term corresponds to the case of no refreshments, while in the second term a refreshment takes place. For the first term we can directly apply Lemma \ref{lem:drift_DBD}, which in the current context shows that for $t<t_0<1$ it holds $P_t^D P_{2t}^B P_t^D V(x,v) \leq (1-\rho t) V(x,v) + Mt$ for $\rho=(1-e^{-\beta_1 M t_0})t_0^{-1}- 1.$ The second term can be rewritten as in \eqref{eq:taylor_exp_driftDBD}, that is for $\overline{x}_1=\overline{x}_1(x,v,w,t)\in B(x,t\sqrt{d})$
\begin{align*}
    &\frac{V(x+vt+wt,w)}{V(x,v)} \\
    & \quad = \exp\left(\beta(\pot(x+vt+wt)-\pot(x)) + \sum_{i=1}^d (\phi(w_i\partial_i \pot(x+vt+wt))-\phi(v_i\partial_i\pot(x)) )\right) \\
    & \quad = e^{t\lvert \nabla \pot(x)\rvert +\frac{t^2}{2}(v+w)^T\nabla^2 \pot(\overline{x}_1)(v+w) + \lvert \nabla \pot(\overline{x}_2)\rvert } \prod_{i=1}^d (1+2\lvert \partial_i \pot(x)\rvert).
\end{align*}
Using Assumption~\ref{ass:refresh_zzs}(b) we find 
\begin{align*}
    (1-e^{-t\gamma(x+vt)})\sum_{w\in \{\pm 1\}^d} \frac{1}{2^d} \frac{V(x+vt+wt,w)}{V(x,v)} &\leq t\gamma(x+vt) \sum_{w\in \{\pm 1\}^d} \frac{1}{2^d} \frac{V(x+vt+wt,w)}{V(x,v)} \\
    & \leq  t \gamma_0.
\end{align*}
Therefore we have shown 
\begin{align*}
    P_t^D P_t^R P_{2t}^B P_t^R P_t^D V & \leq (1+\gamma_0t) (1-\rho t + t\gamma_0) V  
    + \tilde{M} t \leq (1-t(\rho - 2\gamma_0)) V + \tilde{M} t.
\end{align*}
Hence it is sufficient to ensure that $\rho>2\gamma_0$, which can be done similarly to the proof of Lemma \ref{lem:drift_RDBDR}.
\end{proof}

\section{Proofs of Section \ref{sec:expansion_mu}}\label{app:exp_inv_meas}
In this section we collect statements and proofs that are not included in Section \ref{sec:expansion_mu}. 

\subsection{The generator of BPS and its adjoint}
The generator of BPS and its adjoint play an important role in the proofs below and therefore we recall them in this section.
Recall the generator of BPS:
\begin{align*}
	\cL_{BPS} f(x,v) &= \langle v, \nabla_x f(x,v)\rangle + \lambda_1(x,v)[f(x,R(x)v)-f(x,v)]\\
    & \quad +\lambda_r \int \big(f(x,w) - f(x,v)\big) \nu(w)dw.
\end{align*}
The adjoint of $\cL_{BPS}$ in $L^2$, that is the operator $\cL^*_{BPS}$ such that $\int g\cL_{BPS} f=\int f \cL^*_{BPS} g$, is given by
\begin{align*}
    \cL^*_{BPS} g(x,v) 
    & = -\langle v,\nabla_x g(x,v)\rangle  + ((g\lambda_1)(x,R(x)v)-(g\lambda_1)(x,v)) \\& \quad + \lambda_r \left(\nu(v)\int g(x,y)dy  - g(x,v)\right),
\end{align*}
which can be written as $\cL^*_{BPS} g(x,v) =( \cL^*_{D} +\cL^*_{B}+\cL^*_{R}) g(x,v)$ for
\begin{align*}
    & \cL^*_{D}g(x,v) = -\langle v,\nabla_x g(x,v)\rangle, \\
    & \cL^*_{B} g(x,v)= g(x,R(x)v)\lambda_1(x, R(x)v) - g(x,v) \lambda_1(x, v), \\ 
    & \cL^*_{R} g(x,v)= \lambda_r \left(\nu(v)\int g(x,y)dy  - g(x,v)\right).
\end{align*}

\subsection{Proof of Proposition \ref{prop:f_2_1D}}\label{sec:compute_f2_1d}
We start by focusing on the left hand side of \eqref{eq:main_expansion_invmeas}, i.e. $\cL^*_{BPS}(\mu f_2)$. We find since $\mu$ is rotationally invariant in $v$
\begin{align*}
    \cL^*_{BPS} (\mu f_2) (x,v) &= \mu(x,v) \Big\{ \langle v,\nabla \pot(x)\rangle f_2(x,v) - \langle v,\nabla_x f_2(x,v)\rangle + (-\langle v,\nabla \pot(x)\rangle)_+ f_2(x,R(x)v)\\
    & \quad -\langle v,\nabla \pot(x)\rangle_+ f_2(x,v) + \lambda_r \int f_2(x,y)\nu(dy) - \lambda_r f_2(x,v) \Big\}.
\end{align*}
We shall consider the case of $v=\pm 1$, hence $\nu = (1/2) \delta_{+1}+(1/2)\delta_{-1}$. In particular this choice satisfies Assumption \ref{ass:velocity_distribution} below. Introduce the notation $f^+_2(x)=f_2(x,1)$, $f^-_2(x)=f_2(x,-1)$. We have in the 1-dimensional setting
\begin{align*}
    \cL^*_{BPS} (\mu f_2) (x,+1) &= -\mu(x,+1) \Big\{  (f_2^+)' (x) +((-\pot' (x))_++\lambda_r/2 )f_2^+(x) - (\lambda_r/2 +(-\pot' (x))_+)f^-_2(x)\Big\}, \\
    \cL^*_{BPS} (\mu f_2) (x,-1) &= +\mu(x,-1) \Big\{  (f_2^-)' (x) + ((+\pot' (x))_+ + \lambda_r/2 )f^+_2(x)-(\lambda_r/2 +(+\pot' (x))_+ )f_2^-(x) \Big\}.
\end{align*}
Define function $h$ such that $h=\cL^*_2 \mu$, and also $h^+(x)=h(x,+1)$ and $h^-(x)=h(x,-1)$.
Therefore we wish to solve the following system of ODEs
\begin{equation}\label{eq:system_1d}
    \begin{cases}
    (f_2^+)' (x) = -(\lambda_r/2+(-\pot' (x))_+)f_2^+(x) + (\lambda_r/2 +(-\pot' (x))_+)f^-_2(x)-h^+(x),\\
    (f_2^-)' (x)  =- (\lambda_r/2 +(+\pot' (x))_+)f^+_2(x)+(\lambda_r/2+(+\pot' (x))_+ )f_2^-(x)+h^-(x),
    \end{cases}
\end{equation}
with compatibility condition \eqref{eq:compatibility_condition}, which in this case can be written as
\begin{equation}\label{eq:compatibility_1d}
    \int_{-\infty}^\infty (f_2^+(x)+f_2^-(x)) \pi(x)  dx  = 0
\end{equation}
with $\pi(x)=  \mu(x,1)+\mu(x,-1)$.
Let us find a solution to \eqref{eq:system_1d} for a generic (continuous and locally lipschitz) function $h$. Start by subtracting the first line to the second line in \eqref{eq:system_1d}:
\begin{equation}\label{eq:system1d_step1}
    (f_2^-)' (x)-(f_2^+)' (x) = ((\pot' (x))_+-(-\pot' (x))_+)(f_2^-(x)-f_2^+(x)) + h_s(x),
\end{equation}
where $h_s(x)=h^+(x)+h^-(x)$. Define $g=f^-_2-f_2^+$ and notice that $(\pot' (x))_+-(-\pot' (x))_+=\pot' (x)$. Then we can rewrite \eqref{eq:system1d_step1} as
\begin{equation}\notag
    g' (x) = \pot' (x) g(x) + h_s(x).
\end{equation}
Solving this ODE using an integrating factor we find
\begin{equation}\notag
    g(x) = \exp\left(\pot(x)\right)\lim_{y\to-\infty} \left[ \exp\left(-\pot(y)\right) g(y) \right] +\exp\left(\pot(x)\right)\int_{-\infty}^x h_s(y) \exp(-\pot(y))dy .
\end{equation}
Recall that $g=f^--f^+$ and $f^+,f^-$ satisfy \eqref{eq:compatibility_1d}. In order for $f_2$ to define a proper density we require
\begin{equation*}
    \int_{-\infty}^\infty g(x) \pi(x)dx < \infty.
\end{equation*}
For this to hold it must be that $\lim_{y\to-\infty} \exp\left(-\pot(y)\right) g(y)=0$ and thus
\begin{equation}\label{eq:1d_g}
    g(x) = \exp\left(\pot(x)\right)\int_{-\infty}^x h_s(y) \exp(-\pot(y))dy .
\end{equation}
Since $f_2^-(x) = f_2^+(x) + g(x)$ and plugging this in the first equation of \eqref{eq:system_1d} we obtain the ODE
\begin{equation}\notag
    (f_2^+)' (x) = (\lambda_r/2 +(-\pot' (x))_+)g(x)-h^+(x)
\end{equation}
which can be integrated as
\begin{equation}\label{eq:1d_f2+}
    f_2^+(x) = f_2^+(0) + \int_0^x \left( (\lambda_r/2 +(-\pot' (y))_+)g(y)-h^+(y)\right)dy.
\end{equation}
It follows that
\begin{equation}\label{eq:1d_f2-}
\begin{aligned}
    f_2^-(x) & = f_2^+(0) + \int_0^x \left( (\lambda_r/2 +(-\pot' (y))_+)g(y)-h^+(y)\right)dy \\ 
    & \quad +\exp\left(\pot(x)\right)\int_{-\infty}^x h_s(y) \exp(-\pot(y))dy.
    \end{aligned}
\end{equation}
Finally we compute $f^+_2(0)$ enforcing the compatibility condition \eqref{eq:compatibility_1d}. Plugging \eqref{eq:1d_f2+} and \eqref{eq:1d_f2-} in \eqref{eq:compatibility_1d} we find the condition
\begin{align}\label{eq:cond_f2+}
    f_2^+(0) & = -\int_{-\infty}^\infty \Big(  g(x)/2 +  \int_0^x \left( (\lambda_r/2 +(-\pot' (y))_+)g(y)-h^+(y)\right)dy \Big)\pi(x)dx.
\end{align}

\subsection{Computing $\cL^*_2$}\label{app:computing_cl2*}
In this section we obtain the terms $\cL^*_2$ for the splitting schemes \textbf{DBRBD}, \textbf{RDBDR}, \textbf{DRBRD}, \textbf{BDRDB}. This is necessary in order to find the analytic expression of $f_2$ given by Proposition \ref{prop:f_2_1D}. We shall focus on the case of BPS with a general dimensionality of the process. Recall that in the one-dimensional case this coincides with the ZZS.

Let us consider the following assumption on the invariant distribution of the velocity vector.
\begin{assumption}\label{ass:velocity_distribution}
The invariant measure for the velocity component $\nu$ satisfies the following conditions:
\begin{enumerate}[]
    \item \emph{Invariance under rotations}: $\nu(w)=\nu(v)$ for any $v,w$ such that $\lvert v\rvert = \lvert w \rvert$;
    \item \emph{Mean zero}: $\mathbb{E}_\nu[V]=0$;
    \item \emph{Isotropic}: 
    for some $b>0$ it holds that $\textnormal{Cov}_\nu(V) = b I.$
\end{enumerate}
\end{assumption}
These properties hold for instance if $\nu$ is the standard Gaussian distribution, as well as if $\nu$ is the uniform on the unit sphere (in that case $b=1/d$). We find the following expressions for $\cL_2^*.$
\begin{proposition}\label{prop:expansion_adjoint_bps}
Let Assumption \ref{ass:velocity_distribution} hold and define 
\begin{align*}
    &A(x,v)=\frac{3}{2}\lambda_r \Big(b(\lvert \nabla\pot(x)\rvert^2 -\Delta \pot(x))  + 2\langle v,\nabla \pot(x) \rangle \lambda_1(x,R(x)v) + \langle v,\nabla^2 \pot(x)v\rangle\Big),\\
    & B(x,v) = \frac{3}{2} \lambda_1(x,R(x)v) \Big( \langle v,\nabla^2 \pot(x) v\rangle - \langle R(x)v,\nabla^2 \pot(x) R(x)v\rangle\Big)+ \frac{1}{2} \langle v,\nabla_x (\langle v,\nabla^2 \pot(x)v\rangle)\rangle ,\\
    & C(x,v)=3\lambda_1(x,R(x)v)\big( -2\langle v,\nabla^2 \pot(x) v\rangle + \langle v,\nabla \pot(x)\rangle^2\big) -  \langle v,\nabla(\langle v,\nabla^2 \pot(x) v\rangle)\rangle,\\
    & D(x,v) = \frac{3}{2}\lambda_r \Big(b(\lvert \nabla\pot(x)\rvert^2 -\Delta \pot(x)) +\! \langle v,\nabla^2 \pot(x)v\rangle  +\langle v,\nabla \pot(x) \rangle \big(3\lambda_1(x,R(x)v)+\lambda_1(x,v)\big)\Big).
\end{align*}
The splitting scheme \DBRBD satisfies
\begin{align*}
    \cL^*_2 \mu(x,v) &= \frac{\mu(x,v)}{12} \Big(A(x,v)+B(x,v) \Big).
\end{align*}
The splitting scheme \textbf{RDBDR} satisfies
\begin{align*}
    \cL^*_2 \mu(x,v) & = \frac{\mu(x,v)}{12} B(x,v).
\end{align*}
The splitting scheme \textbf{DRBRD} satisfies
\begin{align*}
    \cL^*_2 \mu(x,v) & = \frac{\mu(x,v)}{12} \Bigg( D(x,v)  + B(x,v) + \frac{3}{2}\lambda_r^2 \langle v,\nabla \pot(x)\rangle \Bigg).
\end{align*}
The splitting scheme \textbf{BDRDB} satisfies 
\begin{align*}
    \cL^*_2 \mu(x,v) &= \frac{\mu(x,v)}{12} \Big( -A(x,v)+ C(x,v) \Big).
\end{align*}
\end{proposition}

\begin{remark}
    Clearly, if $\cL^*_2 \mu = 0$ then $f_2$ must be a constant that satisfies \eqref{eq:compatibility_condition} and hence it must be that $f_2=0$, i.e. the second order term in $\mu_\delta$ is zero. This is the case for instance for scheme \textbf{RDBDR} when the target is a multidimensional standard Gaussian. Indeed in Section  \ref{sec:metropolis_adjusted_algorithms} we proved that \textbf{RDBDR} is unbiased for standard Gaussian targets, thus this is a consistent result.
    In the same setting, we observe that $f_2=0$ for schemes \textbf{DBRBD} and \textbf{DRBRD}  when the refreshment rate is $\lambda_r=0$. This is an expected result, as when $\lambda_r=0$ these schemes coincide with \textbf{RDBDR}. In Figure~\ref{fig:bps_err_rad_1d_zerorefresh} we confirm that, for a one-dimensional standard Gaussian and when $\lambda_r=0$, the scheme \textbf{DBD} is unbiased, while the scheme \textbf{BDB} is of second order.

\end{remark}

\begin{figure}[t]
    \centering
    \includegraphics[width=0.45\textwidth]{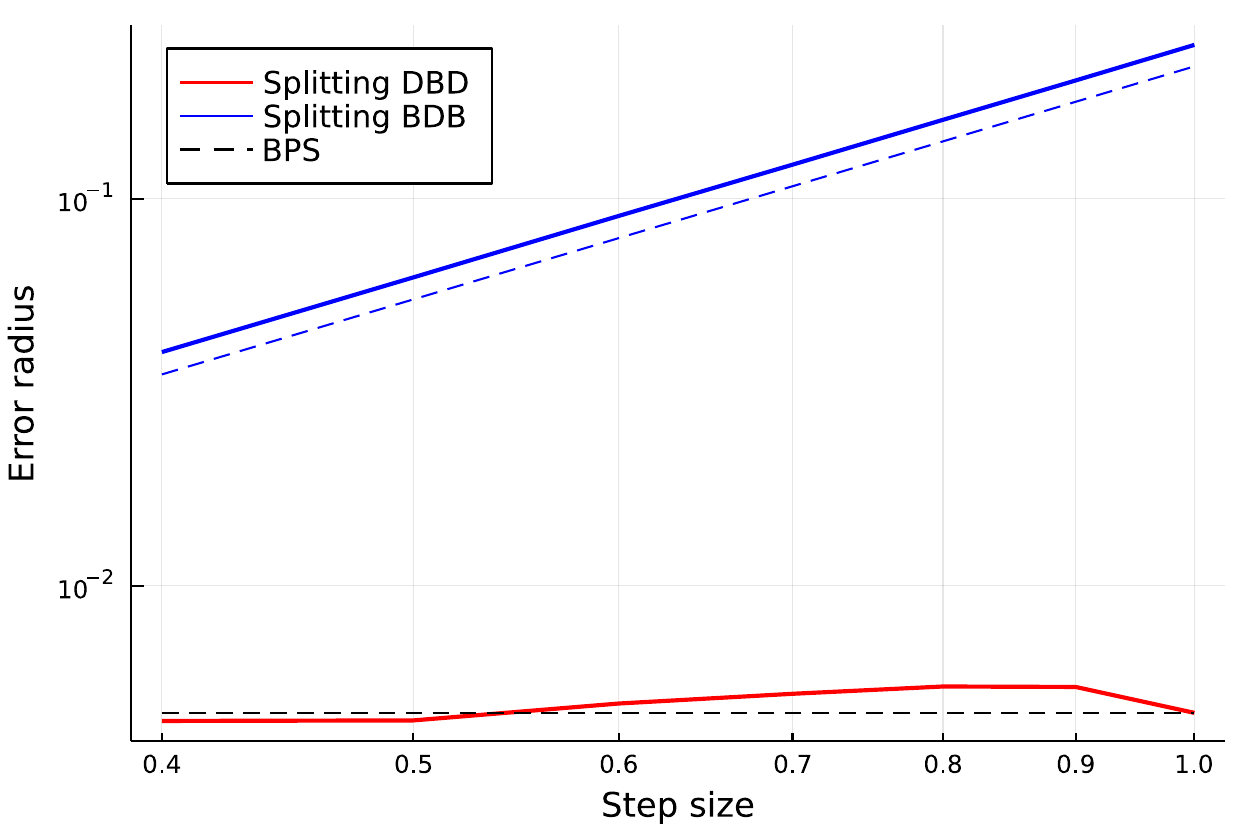}
    \caption{Error for the radius statistic for a one-dimensional standard Gaussian target. Here  $\lambda_r = 0$ for both schemes \textbf{DBD} and \textbf{BDB}. The \emph{dashed, blue line} corresponds to second order convergence. The time horizon is fixed to $T=10^5$ and the number of iterations is $N=T/\delta$.}
    \label{fig:bps_err_rad_1d_zerorefresh}
\end{figure}

In Proposition \ref{prop:expansion_adjoint_bps} we consider symmetric splitting schemes of the form
$$e^{\delta \cL_S} = e^{\frac{\delta}{2}\cL_A} e^{\frac{\delta}{2}\cL_B}e^{\delta\cL_C}e^{\frac{\delta}{2}\cL_B}e^{\frac{\delta}{2}\cL_A}.$$
For such schemes, the Baker-Campbell-Haussdorff formula gives that the adjoint of the generator of the chain satisfies the following decomposition:
\begin{align*}
    \cL^*_S & = \cL^* + \frac{\delta^2}{12} \Big([\cL^*_C,[\cL^*_C,\cL^*_A+\cL^*_B]] + [\cL^*_B,[\cL^*_B,\cL^*_A]]+[\cL^*_C,[\cL^*_B,\cL^*_A]] + [\cL^*_B,[\cL^*_C,\cL^*_A]]   \\
    & \quad -\frac{1}{2} [\cL^*_B,[\cL^*_B,\cL^*_C]] -\frac{1}{2} [\cL^*_A,[\cL^*_A,\cL^*_C]] -\frac{1}{2} [\cL^*_A,[\cL^*_A,\cL^*_B]]\Big) + \mathcal{O}(\delta^4)\\
    & = \cL^* + \delta^2 \cL^*_2 + \mathcal{O}(\delta^4).
\end{align*}
where $\cL^*= \cL^*_A+\cL^*_B+\cL^*_C$ is the adjoint of the generator of the process with generator $\cL= \cL_A+\cL_B+\cL_C$. In order to prove Proposition \ref{prop:expansion_adjoint_bps} it is then sufficient to compute the commutators above for the four schemes we are interested in.
In Section \ref{sec:compute_commutators} we obtain the first and second order commutators of BPS, while in Section \ref{sec:skeleton_proof_expadjoint} we use the BCH formula and the obtained results to prove Proposition \ref{prop:expansion_adjoint_bps}.

\subsubsection{Computing the commutators of BPS}\label{sec:compute_commutators}
In this section we compute the first and second order commutators for the various components of the adjoint of the BPS. 


In the following we shall rely on the following equalities
\begin{align*}
    & \cL^*_D \mu(x,v) = \langle v,\nabla \pot(x) \rangle \mu(x,v), \\
    & \cL^*_B \mu(x,v) = -  \langle v,\nabla \pot(x)\rangle \mu(x,v), \\
    & \cL^*_R \mu(x,v) =  0.
\end{align*}
To obtain $\cL^*_D \mu(x,v)$ we used the trivial, but useful, identity
\begin{align*}
    &\nabla_x \mu(x,v) = -\nabla \pot(x) \mu(x,v).
\end{align*}
The following lemma groups two other useful identities.
\begin{lemma}\label{lem:identities_switchingrates}
For $\lambda(x,v)=\langle v,\nabla \pot(x)\rangle_+$ it holds that
\begin{align}
    & \label{eq:lambdaR-lambda} \lambda_1(x,R(x)v)-\lambda_1(x,v) = -\langle v,\nabla \pot(x)\rangle,\\ 
    & \lambda_1(x,R(x)v)+\lambda_1(x,v) = +\lvert \langle v,\nabla \pot(x)\rangle \rvert.\label{eq:lambdaR+lambda}
\end{align}
\end{lemma}

\textbf{First order commutators.}
Let us start  computing the three first order commutators $[\cL^*_B,\cL^*_D]$, $[\cL^*_R,\cL^*_D]$, and $[\cL^*_R,\cL^*_B]$, which are essential to compute higher order commutators. This is done below respectively in Lemmas \ref{lem:BD_bps}, \ref{lem:RD_bps}, \ref{lem:RB_bps}.

\begin{lemma}\label{lem:BD_bps}
Let $g$ be a suitable function. It holds that
\begin{align*}
    [\cL^*_B,\cL^*_D] g(x,v) & = - \langle R(x)v,(\nabla_x g)(x,R(x)v)\rangle \lambda_1(x,R(x)v) + \langle v,\nabla_x g(x,v)\rangle \lambda_1(x,v)\\
    & \quad + \langle v, \nabla_x\Big(g(x,R(x)v)\lambda_1(x, R(x)v) - g(x,v) \lambda_1(x, v) \Big) \rangle.
\end{align*}
In particular if $g=\mu$
\begin{align*}
    [\cL^*_B,\cL^*_D] \mu(x,v) & = \mu(x,v) \Big( \langle v, \nabla \pot (x)\rangle \big(\langle v,\nabla \pot(x)\rangle - \lvert \langle v,\nabla \pot(x)\rangle \rvert \big) -\langle v,\nabla^2 \pot(x))v\rangle \Big).
\end{align*}
\end{lemma}
\begin{remark}
Alternative ways to write $ [\cL^*_B,\cL^*_D] \mu(x,v)$ can be found using the identities in Lemma \ref{lem:identities_switchingrates}. We find
\begin{align*}
    [\cL^*_B,\cL^*_D] \mu(x,v)  & = \mu(x,v) \Big(\lambda_1^2(x,R(x)v)-\lambda_1^2(x,v)+\langle v,\nabla \pot(x)\rangle^2 - \langle v,\nabla^2 \pot(x)v\rangle\Big) \\
    & = \mu(x,v) \Big(\lambda_1^2(x,R(x)v)-\lambda_1^2(x,v)+\langle v,(\nabla \pot(x)\nabla \pot(x)^T -\nabla^2 \pot(x))v\rangle\Big).
\end{align*}
\end{remark}
\begin{proof}
We have
\begin{align*}
    [\cL^*_B,\cL^*_D] g(x,v) & =  \cL^*_B(-\langle v,\nabla_x g(x,v)\rangle) - \cL^*_D ( g(x,R(x)v)\lambda_1(x, R(x)v) - g(x,v) \lambda_1(x, v) ) \\
    & = - \langle R(x)v,(\nabla_x g)(x,R(x)v)\rangle \lambda_1(x,R(x)v) + \langle v,\nabla_x g(x,v)\rangle \lambda_1(x,v)\\
    & \quad + \langle v, \nabla_x\Big(g(x,R(x)v)\lambda_1(x, R(x)v) - g(x,v) \lambda_1(x, v) \Big) \rangle
\end{align*}
and hence
\begin{align*}
    [\cL^*_B,\cL^*_D] \mu(x,v) & =  
    - \mu(x,v) \langle v,\nabla \pot(x)\rangle (\lambda_1(x,R(x)v)+\lambda_1(x,v))\\
    & \quad + \langle v, \nabla_x\Big( \mu(x,v)(\lambda_1(x,R(x)v)-\lambda_1(x,v)) \Big)\rangle\\
    & = \mu(x,v) (\lambda_1^2(x,R(x)v)-\lambda_1^2(x,v))  \\ 
    & \quad - \langle v , \nabla_x \Big( \mu(x,v) \langle v,\nabla \pot(x)\rangle \Big)\rangle.
\end{align*}
Then note that
\begin{align*}
    \langle v , \nabla_x \Big( \mu(x,v) \langle v,\nabla \pot(x)\rangle \Big)\rangle & = \langle v,\nabla^2 \pot(x) v - \nabla \pot(x) \langle v,\nabla \pot(x)\rangle\, \rangle \, \mu(x,v).
\end{align*}
and hence
\begin{align*}
    [\cL^*_B,\cL^*_D] \mu(x,v) & = \mu(x,v) \Big( \langle v, \nabla \pot (x) \big(\langle v,\nabla \pot(x)\rangle - \lvert \langle v,\nabla \pot(x)\rangle \rvert \big)\rangle -\langle v,\nabla^2 \pot(x))v\rangle \Big).
\end{align*}
\end{proof}


\begin{lemma}\label{lem:RD_bps}
Let $g$ be a suitable function. It holds that
\begin{align*}
    [\cL^*_R,\cL^*_D] g(x,v)  & = \lambda_r \nu(v) \Big( \langle v,\int \nabla_x g(x,y)dy\rangle - \int \langle y,\nabla_x g(x,y)\rangle dy \Big).
\end{align*}
In particular if $g=\mu$
\begin{align*}
    [\cL^*_R,\cL^*_D] \mu(x,v) & = -\lambda_r \langle v,\nabla \pot(x) \rangle \mu(x,v).
\end{align*}
\end{lemma}

\begin{proof}
We find
\begin{align*}
    [\cL^*_R,\cL^*_D] g(x,v) & =- \cL^*_R(\langle v, \nabla_x g(x,v)\rangle) - \cL^*_D\left(\lambda_r \left(\nu(v)\int g(x,y)dy  - g(x,v)\right)\right) \\
    & = -\lambda_r \left(\nu(v)\int \langle y, \nabla_x g(x,y)\rangle dy  - \langle v, \nabla_x g(x,v)\rangle\right) \\
    & \quad + \lambda_r \langle v, \nabla_x (\nu(v)\int g(x,y)dy  - g(x,v))\rangle \\
    & = \lambda_r \nu(v) \Big( \langle v,\int \nabla_x g(x,y)dy\rangle - \int \langle y,\nabla_x g(x,y)\rangle dy \Big)
\end{align*}
and thus
\begin{align*}
    [\cL^*_R,\cL^*_D] \mu(x,v) & = \cL^*_R (\langle v,\nabla \pot(x) \rangle \mu(x,v) ) = -\lambda_r \langle v,\nabla \pot(x) \rangle \mu(x,v)
\end{align*}
\end{proof}


\begin{lemma}\label{lem:RB_bps}
Let $g$ be a suitable function. It holds that
\begin{align*}
    [\cL^*_R,\cL^*_B] g(x,v)  & = \lambda_r \Bigg(\nu(v)\int (g(x,R(x)y)\lambda_1(x,R(x)y)- g(x,y)\lambda_1(x,y) )dy  \\
    & \quad + \left(\nu(v)\int g(x,y)dy \right) \langle v,\nabla \pot(x)\rangle \Bigg).
\end{align*}
In particular if $g=\mu$
\begin{align*}
   [\cL^*_R,\cL^*_B] \mu(x,v)=-[\cL^*_R,\cL^*_D] \mu(x,v)  = \lambda_r  \langle v,\nabla \pot(x)\rangle \mu(x,v).
\end{align*}
\end{lemma}

\begin{proof}
Compute
\begin{align*}
    [\cL^*_R,\cL^*_B] g(x,v) & = \cL^*_R( g(x,R(x)v)\lambda_1(x,R(x)v)- g(x,v)\lambda_1(x,v)) \\
    & \quad - \cL^*_B\left(\lambda_r \left(\nu(v)\int g(x,y)dy  - g(x,v)\right)\right)\\
    & = \lambda_r \Big(\nu(v)\int (g(x,R(x)y)\lambda_1(x,R(x)y)- g(x,y)\lambda_1(x,y) )dy  \\
    & \quad - \underbrace{g(x,R(x)v)\lambda_1(x,R(x)v)}_A+ \underbrace{g(x,v)\lambda_1(x,v)}_B\Big)\\
    & \quad - \lambda_r \Bigg(  \left(\nu(R(x)v)\int g(x,y)dy  - \underbrace{g(x,R(x)v)}_A\right)\lambda_1(x,R(x)v) \\
    & \quad -\left(\nu(v)\int g(x,y)dy  - \underbrace{g(x,v)}_B\right)\lambda_1(x,v) \Bigg).
\end{align*}
It is now sufficient to cancel out the terms denoted by A and B that appear twice with opposite signs to obtain the final statement.
To compute $g=\mu$, we can use that $\cL_D + \cL_B$ and $\cL_R$ both preserve the invariant distribution, hence $[\cL^*_R,\cL^*_B] \mu(x,v)=-[\cL^*_R,\cL^*_D] \mu(x,v)$.
\end{proof}

\textbf{Higher order commutators.}
Let us now compute higher order commutators.

\begin{lemma}\label{lem:BRD_bps}
It holds that
\begin{align*}
   [\cL^*_B,[\cL^*_R,\cL^*_D]] \mu(x,v) & = \lambda_r \mu(x,v) \Big(\langle v,\nabla \pot(x) \rangle\Big(   \lambda_1(x,R(x)v) +  \lambda_1(x,v)\Big) \\
    &   \quad +b\,\trace\big( \nabla \pot(x) (\nabla \pot(x))^T-\nabla^2 \pot(x)\big) \Big).
\end{align*}
\end{lemma}

\begin{proof}
Applying Lemma \ref{lem:RD_bps} we obtain
\begin{align*}
    &[\cL^*_B,[\cL^*_R,\cL^*_D]] \mu(x,v)=  -\lambda_r \cL^*_B \Big(  \langle v,\nabla \pot(x) \rangle \mu(x,v) \Big) + [\cL^*_R,\cL^*_D] \Big( \langle v,\nabla \pot(x)\rangle \mu(x,v) \Big) \\
    & = -\lambda_r\Big(  \langle R(x)v,\nabla \pot(x) \rangle \mu(x,v)\lambda_1(x,R(x)v) -  \langle v,\nabla \pot(x) \rangle \mu(x,v) \lambda_1(x,v)\Big) \\
    & \quad + \lambda_r \nu(v) \Big( \langle v,\nabla_x \int  (\langle y,\nabla \pot(x)\rangle \mu(x,y))dy\rangle - \int \langle y,\nabla_x (\langle y,\nabla \pot(x)\rangle \mu(x,y))\rangle dy \Big)\\
    & = \lambda_r \mu(x,v)\langle v,\nabla \pot(x) \rangle\Big(   \lambda_1(x,R(x)v) +  \lambda_1(x,v)\Big) \\
    & \quad - \lambda_r \mu(x,v) \Big(\int (\langle y,\nabla^2 \pot(x) y\rangle  -\langle y,\nabla \pot(x)\rangle^2) \nu(dy) \Big)\\
    & = \lambda_r \mu(x,v) \Big(\langle v,\nabla \pot(x) \rangle\Big(   \lambda_1(x,R(x)v) +  \lambda_1(x,v)\Big) \\
    &   \quad +b\,\trace\big( \nabla \pot(x) (\nabla \pot(x))^T-\nabla^2 \pot(x)\big) \Big).
\end{align*}
Note that in the last line we used that $\langle a,b\rangle^2 = \langle a, b b^T a\rangle$ and that
\begin{equation}
\begin{aligned}\label{eq:trace_computation}
    \int \langle y,(\nabla \pot(x)\nabla \pot(x)^T -\nabla^2 \pot(x))y\rangle \nu(dy) &= \sum_{j=1}^d \sum_{\ell=1}^d (\nabla \pot(x)\nabla \pot(x)^T -\nabla^2 \pot(x))_{j\ell} \int (y_j y_\ell) \nu(dy) \\
    & =b\,\trace\Big(\nabla \pot(x)\nabla \pot(x)^T -\nabla^2 \pot(x)\Big)
\end{aligned}
\end{equation}
which is a consequence of Assumption \ref{ass:velocity_distribution}.
\end{proof}

\begin{lemma}\label{lem:RRB_bps}
It holds that
\begin{align*}
   [\cL^*_R,[\cL^*_R,\cL^*_B]] \mu(x,v) & = -\lambda_r^2 \mu(x,v)\langle v,\nabla \pot(x)\rangle.
\end{align*}
\end{lemma}
\begin{proof}
Since $\cL^*_R \mu(x,v)=0$ we easily find
\begin{align*}
    [\cL^*_R,[\cL^*_R,\cL^*_B]]\mu(x,v) &= \cL^*_R(\lambda_r\langle v,\nabla \pot(x)\rangle \mu(x,v)) = -\lambda_r^2 \mu(x,v)\langle v,\nabla \pot(x)\rangle.
\end{align*}
\end{proof}

\begin{lemma}\label{lem:RRD_bps}
It holds that
\begin{align*}
   [\cL^*_R,[\cL^*_R,\cL^*_D]] \mu(x,v) & = \lambda_r^2 \mu(x,v)\langle v,\nabla \pot(x)\rangle.
\end{align*}
\end{lemma}
\begin{proof}
The result follows from Lemma \ref{lem:RRB_bps}.
\end{proof}

\begin{lemma}\label{lem:RBD_bps}
It holds that
\begin{align*}
    [\cL^*_R,[\cL^*_B,\cL^*_D]] \mu(x,v) &= \lambda_r \mu(x,v) \Bigg(  b\,\trace \Big(\nabla \pot(x)\nabla \pot(x)^T -\nabla^2 \pot(x)\Big) \\
    & \quad -\Big(\lambda_1^2(x,R(x)v)-\lambda_1^2(x,v)+\langle v,(\nabla \pot(x)\nabla \pot(x)^T -\nabla^2 \pot(x))v\rangle\Big)\Bigg)
\end{align*}
\end{lemma}

\begin{proof}
Taking advantage of Lemma \ref{lem:BD_bps}
\begin{align*}
    & [\cL^*_R,[\cL^*_B,\cL^*_D]] \mu(x,v) = \cL^*_R \Big(\mu(x,v) \Big(\lambda_1^2(x,R(x)v)-\lambda_1^2(x,v)+\langle v,(\nabla \pot(x)\nabla \pot(x)^T -\nabla^2 \pot(x))v\rangle\Big) \Big)\\
    & \quad = \lambda_r \mu(x,v)\Bigg( \int \Big(\lambda_1^2(x,R(x)y)-\lambda_1^2(x,y)+\langle y,(\nabla \pot(x)\nabla \pot(x)^T -\nabla^2 \pot(x))y\rangle\Big) \nu(dy)\\
    & \qquad - \Big(\lambda_1^2(x,R(x)v)-\lambda_1^2(x,v)+\langle v,(\nabla \pot(x)\nabla \pot(x)^T -\nabla^2 \pot(x))v\rangle\Big)\Bigg).
\end{align*}
Observe that for $A=\{y:\langle y,\nabla \pot(x)\rangle \geq 0\}$ we have
\begin{align*}
    \int (\lambda_1^2(x,R(x)y)-\lambda_1^2(x,y))\nu(dy) &= \int_{A^C} \langle y,\nabla \pot(x)\rangle^2 \nu(y)dy-\int_{A} \langle y,\nabla \pot(x)\rangle^2 \nu(y)dy\\
    & = 0.
\end{align*}
This can be seen by the change of variables $y'=R(x)y$ in the first integral. The result then follows by using \eqref{eq:trace_computation}.
\end{proof}

\begin{lemma}\label{lem:BRB_bps}
It holds that
\begin{align*}
    [\cL^*_B,[\cL^*_R,\cL^*_B]] \mu(x,v) &  = -\lambda_r \mu(x,v) \langle v,\nabla \pot(x)\rangle (\lambda_1(x,R(x)v)+\lambda_1(x,v)).
\end{align*}
\end{lemma}
\begin{proof}
Consider now
\begin{align*}
    [\cL^*_B,[\cL^*_R,\cL^*_B]] \mu(x,v) & = \cL^*_B(\lambda_r \langle v,\nabla \pot(x)\rangle \mu(x,v))+[\cL^*_R,\cL^*_B](\langle v,\nabla \pot(x)\rangle \mu(x,v))\\
    & = \lambda_r \mu(x,v) \Bigg( \langle R(x) v,\nabla \pot(x)\rangle \lambda_1(x,R(x)v)-\langle v,\nabla \pot(x) \rangle \lambda_1(x,v) \\
    & \quad + \int\Big( \langle R(x)y,\nabla \pot(x)\rangle \lambda_1(x,R(x)y\rangle -\langle y,\nabla \pot(x)\rangle \lambda_1(x,y)\Big)\nu(dy)\\
    & \quad +\int(\langle y,\nabla \pot(x)\rangle \nu(dy) \langle v,\nabla \pot(x)\rangle \Bigg).
\end{align*}
The last term equals zero as $\nu$ has mean zero. Then observe that by Identity \eqref{eq:lambdaR-lambda}
\begin{equation}
\begin{aligned}\label{eq:helper_1}
    & \int\Big( \langle R(x)y,\nabla \pot(x)\rangle \lambda_1(x,R(x)y) -\langle y,\nabla \pot(x)\rangle \lambda_1(x,y)\Big)\nu(dy) =\\
    & \quad = - \int \langle y,\nabla \pot(x)\rangle (\lambda_1(x,R(x)y)+\lambda_1(x,y))\nu(dy)\\
    & \quad = - \Big(\int \lambda_1(x,R(x)y)^2 \nu(dy) - \int\lambda_1(x,y)^2 \nu(dy) \Big)\\
    &\quad = 0,
\end{aligned}
\end{equation}
where the last equality follows by invariance under rotation of $\nu$ as required in Assumption \ref{ass:velocity_distribution}.
Hence we have obtained the statement.
\end{proof}

\begin{lemma}\label{lem:BBD_bps}
It holds that
\begin{align*}
    [\cL^*_B,[\cL^*_B,\cL^*_D]]\mu(x,v) &=2\mu(x,v)\lambda_1(x,R(x)v)  \Big( \langle v,\nabla \pot(x)\rangle^2 - \langle v,\nabla^2 \pot(x)v\rangle \\
    & \quad - \langle R(x)v,\nabla \pot^2 (x) R(x)v\rangle\Big).
\end{align*}
\end{lemma}

\begin{proof}
By Lemma \ref{lem:BD_bps} we find
\begin{align*}
    [\cL^*_B,[\cL^*_B,\cL^*_D]]\mu(x,v) &=\cL^*_B\Big( \mu(x,v)(\lambda_1^2(x,R(x)v)-\lambda_1^2(x,v) + \langle v,\nabla \pot(x)\rangle^2 -\langle v,\nabla^2 \pot(x) v\rangle)\Big) \tag{*} \label{eq:BBD_term1}\\
    & \quad + [\cL^*_B,\cL^*_D] \Big(\langle v,\nabla \pot(x)\rangle \mu(x,v)\Big).\tag{**} \label{eq:BBD_term2}
\end{align*}
Let us treat the two terms separately, starting with \eqref{eq:BBD_term1}. After applying $\cL^*_B$ and using that $R(x)(R(x)v) = v$ the first term becomes
\begin{align*}
    \eqref{eq:BBD_term1} &= \mu(x,v) \Big[ \Big( (\lambda_1^2(x,v)-\lambda_1^2(x,R(x)v) + \langle R(x)v,\nabla \pot(x)\rangle^2 -\langle R(x)v,\nabla^2 \pot(x) R(x)v\rangle \Big)\lambda_1(x,R(x)v)\\
    & \quad - \Big( \lambda_1^2(x,R(x)v)-\lambda_1^2(x,v) + \langle v,\nabla \pot(x)\rangle^2 -\langle v,\nabla^2 \pot(x) v\rangle) \Big)\lambda_1(x,v) \Big]\\
    & = \mu(x,v) \Big[ (\lambda_1^2(x,v)-\lambda_1^2(x,R(x)v)) (\lambda_1(x,R(x)v)+\lambda_1(x,v)) \\
    & \quad + \langle v,\nabla \pot(x)\rangle^2 (\lambda_1(x,R(x)v)-\lambda_1(x,v))\\
    & \quad -\langle R(x)v,\nabla^2 \pot(x) R(x)v\rangle \lambda_1(x,R(x)v)+\langle v,\nabla^2 \pot(x) v\rangle \lambda_1(x,v)\Big].
\end{align*}
Using Identity \eqref{eq:lambdaR-lambda} we obtain that
\begin{align*}
    & \langle v,\nabla \pot(x)\rangle^2 (\lambda_1(x,R(x)v)-\lambda_1(x,v)) = (\lambda_1^2(x,v)-\lambda_1^2(x,R(x)v)) (\lambda_1(x,R(x)v)+\lambda_1(x,v))
\end{align*}
and thus cancelling out the corresponding terms in \eqref{eq:BBD_term1} it follows that
\begin{align*}
    \eqref{eq:BBD_term1}= \mu(x,v) \Big( \langle v,\nabla^2 \pot(x) v\rangle \lambda_1(x,v)-\langle R(x)v,\nabla^2 \pot(x) R(x)v\rangle \lambda_1(x,R(x)v)\Big).
\end{align*}
Focusing now on \eqref{eq:BBD_term2}, we apply Lemma \ref{lem:BD_bps} to find
\begin{align*}
    \eqref{eq:BBD_term2} & = -\langle R(x)v,\nabla_x \big(\langle v,\nabla \pot(x)\rangle \mu(x,v)\big)(x,R(x)v)\rangle \lambda_1(x,R(x)v)\\
    & \quad +\langle v,\nabla_x\big( \langle v,\nabla \pot(x)\rangle \mu(x,v)\big)\rangle \lambda_1(x,v)\\
    & \quad +\langle v,\nabla_x\Big(\langle R(x)v,\nabla \pot(x)\rangle \mu(x,v)\lambda_1(x,R(x)v)-\langle v,\nabla \pot(x)\rangle \mu(x,v) \lambda_1(x,v)\Big)\rangle.
\end{align*}
Recalling that $$\nabla_x(\langle v,\nabla \pot(x)\rangle \mu(x,v))=\mu(x,v)(\nabla^2 \pot(x) v - \nabla \pot(x)\langle v,\nabla \pot(x)\rangle),$$ we find 
\begin{align*}
    \eqref{eq:BBD_term2} & = \mu(x,v)\Big[\left(-\langle R(x)v, \nabla^2 \pot(x) R(x)v\rangle + \langle v,\nabla \pot(x)\rangle^2 \right) \lambda_1(x,R(x)v)\\
    & \quad +\left( \langle v,\nabla^2 \pot(x) v\rangle - \langle v, \nabla \pot(x)\rangle^2 \right)\lambda_1(x,v)\Big]\\
    & \quad -\langle v,\nabla_x\Big(\langle v,\nabla \pot(x)\rangle \mu(x,v)\lvert \langle v,\nabla \pot(x)\rangle\rvert\Big)\rangle   .
\end{align*}
In particular we used Lemma \ref{lem:identities_switchingrates} to write the last term more compactly. The derivative in the last term can be computed as follows
\begin{align*}
    &-\langle v,\nabla_x\Big(\langle v,\nabla \pot(x)\rangle \mu(x,v)\lvert \langle v,\nabla \pot(x)\rangle\rvert\Big)\rangle   =\\
    & =-\mu(x,v) \langle v, \nabla^2 \pot(x) v \lvert \langle v,\nabla \pot(x)\rangle\rvert  -\nabla \pot(x) \langle v,\nabla \pot(x)\rangle \lvert \langle v,\nabla \pot(x)\rangle\rvert\\
    & \quad + \langle v,\nabla \pot(x) \rangle \text{sign}(\langle v,\nabla \pot(x)\rangle) \nabla^2\pot(x)v\rangle\\
    & = -\mu(x,v)\Big( -\langle v,\nabla \pot(x)\rangle^2 \lvert \langle v,\nabla \pot(x)\rangle\rvert + \langle v,\nabla^2 \pot(x)v\rangle \left(\lvert \langle v,\nabla \pot(x)\rangle\rvert + \langle v,\nabla \pot(x) \rangle \text{sign}(\langle v,\nabla \pot(x)\rangle)\right)\Big)\\
    & = -\mu(x,v)\Big( -\langle v,\nabla \pot(x)\rangle^2 \lvert \langle v,\nabla \pot(x)\rangle\rvert + \langle v,\nabla^2 \pot(x)v\rangle 2\lvert \langle v,\nabla \pot(x)\rangle\rvert \Big)\\
    & = \mu(x,v)\lvert \langle v,\nabla \pot(x)\rangle\rvert \big( \langle v,\nabla \pot(x)\rangle^2  - 2\langle v,\nabla^2 \pot(x)v\rangle  \big).
\end{align*}
Hence re-applying Lemma \ref{lem:identities_switchingrates} we find
\begin{align*}
    \eqref{eq:BBD_term2} & =\mu(x,v)\Big[-\langle R(x)v, \nabla^2 \pot(x) R(x)v\rangle \lambda_1(x,R(x)v)\\
    & \quad +2 \langle v, \nabla \pot(x)\rangle^2 \lambda_1(x,R(x)v) -(2\lambda_1(x,R(x)v)+\lambda_1(x,v))\langle v,\nabla^2 \pot(x)v\rangle \Big].
\end{align*}
The proof is now concluded by summing \eqref{eq:BBD_term1} and \eqref{eq:BBD_term2}.
\end{proof}

\begin{lemma}\label{lem:DRB_bps}
It holds that
\begin{align*}
    [\cL^*_D,[\cL^*_R,\cL^*_B]] \mu(x,v)
    & = \lambda_r \mu(x,v) \left(\langle v, \nabla \pot(x)\rangle^2 - \langle v,\nabla^2 \pot(x)v\rangle\right). 
\end{align*}
\end{lemma}

\begin{proof}
Consider now $[\cL^*_D,[\cL^*_R,\cL^*_B]]$:
\begin{align*}
    [\cL^*_D,[\cL^*_R,\cL^*_B]] \mu(x,v) & = \cL^*_D(\lambda_r\langle v,\nabla \pot(x)\rangle \mu(x,v)) - [\cL^*_R,\cL^*_B](\langle v,\nabla \pot(x)\rangle \mu(x,v)) \\
    & = -\langle v,\nabla_x(\lambda_r \langle v,\nabla \pot(x)\rangle \mu(x,v))\rangle\\
    & \quad -\lambda_r\Big( \mu(x,v)\int (-\langle y,\nabla \pot(x)\rangle)( \lambda_1(x,R(x)y)+ \lambda_1(x,y))\nu(dy)\Big)\\
    & = \lambda_r \mu(x,v) \left(\langle v, \nabla \pot(x)\rangle^2 - \langle v,\nabla^2 \pot(x)v\rangle\right). 
\end{align*}
In particular we used  that
\begin{align*}
    & \int (\langle y,\nabla \pot(x)\rangle)( \lambda_1(x,R(x)y)+ \lambda_1(x,y))\nu(dy) = \int \lambda_1(x,y)^2\nu(dy)-\int \lambda_1(x,R(x)y)^2\nu(dy)=0
\end{align*}
which was shown in \eqref{eq:helper_1}.
\end{proof}

\begin{lemma}\label{lem:DRD_bps}
It holds that
\begin{align*}
    [\cL^*_D,[\cL^*_R,\cL^*_D]] \mu(x,v) &= \lambda_r \mu(x,v) \Big( \langle v,\nabla^2 \pot(x) v\rangle -\langle v,\nabla \pot(x)\rangle^2\\
    & \quad + b\,\trace\left(\nabla^2 \pot(x) -\nabla \pot(x)\nabla \pot(x)^T\right)\Big)
\end{align*}
\end{lemma}
\begin{proof}
By Lemma \ref{lem:RD_bps}
\begin{align*}
    [\cL^*_D,[\cL^*_R,\cL^*_D]] \mu(x,v) & = -\lambda_r\cL_D^*(\langle v,\nabla \pot(x)\rangle \mu(x,v)) - [\cL^*_R,\cL^*_D](\langle v,\nabla \pot(x) \rangle \mu(x,v))\\
    & = \lambda_r \mu(x,v) \Bigg( \langle v,\nabla^2 \pot(x) v\rangle -\langle v,\nabla \pot(x)\rangle^2 \\
    & \quad + \int (\langle y,\nabla^2 \pot(x) y\rangle -\langle y,\nabla \pot(x) \rangle^2 )\nu(dy) \Bigg) .
\end{align*}
The statement follows by Equation \eqref{eq:trace_computation}.
\end{proof}

\begin{lemma}\label{lem:DBD_bps}
It holds that
\begin{align*}
    [\cL^*_D,[\cL^*_B,\cL^*_D]] &= \mu(x,v) \Big( -4 \langle v,\nabla \pot(x)\rangle^2  \lambda_1(x,R(x)v) + 7\langle v,\nabla^2 \pot(x) v\rangle\lambda_1(x,R(x)v) \\
    & \quad + \langle v,\nabla_x(\langle v,\nabla^2 \pot(x))v\rangle)\rangle  + \langle R(x)v, \nabla^2 \pot(x) R(x)v\rangle \lambda_1(x,R(x)v) \Big).
\end{align*}
\end{lemma}
\begin{proof}
By Lemma \ref{lem:BD_bps} together with Lemma \ref{lem:identities_switchingrates}
\begin{align*}
    [\cL^*_D,[\cL^*_B,\cL^*_D]]\mu(x,v) &= \cL^*_D \Big( \mu(x,v) \Big( \langle v, \nabla \pot (x) \rangle  \big(\langle v,\nabla \pot(x)\rangle - \lvert \langle v,\nabla \pot(x)\rangle \rvert  \big) -\langle v,\nabla^2 \pot(x)v\rangle \Big)\Big) \label{eq:DBD_term1}\tag{\dag}\\ 
    & \quad - [\cL^*_B,\cL^*_D](\langle v,\nabla \pot(x)\rangle \mu(x,v)).\tag{\dag\dag} \label{eq:DBD_term2}\\
    & = \eqref{eq:DBD_term1} - \eqref{eq:DBD_term2}.
\end{align*}
Consider the two terms separately, starting from the first one:
\begin{align*}
    \eqref{eq:DBD_term1} &= \mu(x,v) \langle v,\nabla \pot(x)\rangle \Big( \langle v, \nabla \pot (x) \rangle\big(\langle v,\nabla \pot(x)\rangle - \lvert \langle v,\nabla \pot(x)\rangle \rvert \big) -\langle v,\nabla^2 \pot(x))v\rangle \Big)\\
    & \quad - \mu(x,v)\langle v, 2\langle v,\nabla \pot(x)\rangle \nabla^2 \pot(x)v - 2\nabla^2 \pot(x) v \lvert \langle v,\nabla \pot(x)\rangle \rvert -\nabla_x(\langle v,\nabla^2 \pot(x))v\rangle )\rangle\\
    & = \mu(x,v) \Big(-2\langle v,\nabla \pot(x)\rangle^2 \lambda_1(x,R(x)v) +\langle v,\nabla^2 \pot(x)\rangle (-3\langle v,\nabla \pot(x)\rangle  +2\lvert \langle v,\nabla \pot(x)\rangle \rvert)\\
    & \quad + \langle v,\nabla_x(\langle v,\nabla^2 \pot(x))v\rangle)\rangle \Big)
\end{align*}
The second term \eqref{eq:DBD_term2} is the same as term \eqref{eq:BBD_term2} in the proof of Lemma \ref{lem:BBD_bps}.
The statement follows taking the difference of the two terms \eqref{eq:DBD_term1} and \eqref{eq:DBD_term2} and using Lemma \ref{lem:identities_switchingrates}.
\end{proof}

\subsubsection{Proof of Proposition \ref{prop:expansion_adjoint_bps}}\label{sec:skeleton_proof_expadjoint}
Now we apply the expressions for the commutators derived in Section \ref{sec:compute_commutators}, thus obtaining the expressions for $\cL^*_2$ of Proposition \ref{prop:expansion_adjoint_bps}. 

Let us start with splitting \textbf{DBRBD}:
\begin{align*}
    &\cL^*_{2} \mu(x,v)  =  \frac{1}{12} \Big([\cL^*_R,[\cL^*_R,\cL^*_D+\cL^*_B]] + [\cL^*_B,[\cL^*_B,\cL^*_D]]+[\cL^*_R,[\cL^*_B,\cL^*_D]] + [\cL^*_B,[\cL^*_R,\cL^*_D]]   \\
    & \quad -\frac{1}{2} [\cL^*_B,[\cL^*_B,\cL^*_R]] -\frac{1}{2} [\cL^*_D,[\cL^*_D,\cL^*_R]] -\frac{1}{2} [\cL^*_D,[\cL^*_D,\cL^*_B]] \Big) \\
    & = \frac{\mu(x,v)}{12} \Bigg( \frac{3}{2}\lambda_r \Big(b\,\text{tr}\big(\nabla \pot(x) \nabla \pot(x)^T - \nabla^2 \pot(x)\big) + 2\langle v,\nabla \pot(x) \rangle \lambda_1(x,R(x)v) + \langle v,\nabla^2 \pot(x)v\rangle\Big)\\
    & \quad + \frac{3}{2} \lambda_1(x,R(x)v) \Big( \langle v,\nabla^2 \pot(x) v\rangle - \langle R(x)v,\nabla^2 \pot(x) R(x)v\rangle\Big)+ \frac{1}{2} \langle v,\nabla_x (\langle v,\nabla^2 \pot(x)v\rangle)\rangle \Bigg). 
\end{align*}
Then focus on splitting \textbf{BDRDB}:
\begin{align*}
    & \cL^*_{2} \mu(x,v)  =  \frac{1}{12} \Big([\cL^*_R,[\cL^*_R,\cL^*_D+\cL^*_B]] + [\cL^*_D,[\cL^*_D,\cL^*_B]] + [\cL^*_R,[\cL^*_D,\cL^*_B]] + [\cL^*_D,[\cL^*_R,\cL^*_B]]  \\
    & \quad -\frac{1}{2} [\cL^*_D,[\cL^*_D,\cL^*_R]] -\frac{1}{2} [\cL^*_B,[\cL^*_B,\cL^*_R]] -\frac{1}{2} [\cL^*_B,[\cL^*_B,\cL^*_D]] \Big) \\
    & = \frac{\mu(x,v)}{12} \Bigg( -\frac{3}{2}\lambda_r \Big(b\,\text{tr}\big(\nabla \pot(x) \nabla \pot(x)^T - \nabla^2 \pot(x)\big) + 2\langle v,\nabla \pot(x) \rangle \lambda_1(x,R(x)v) + \langle v,\nabla^2 \pot(x)v\rangle\Big)\\
    & \quad + 3\lambda_1(x,R(x)v)\Big( -2\langle v,\nabla^2 \pot(x) v\rangle + \langle v,\nabla \pot(x)\rangle^2\Big) - \langle v,\nabla(\langle v,\nabla^2 \pot(x) v\rangle)\rangle\Bigg).
\end{align*}
Consider now \textbf{RDBDR}:
\begin{align*}
    & \cL^*_{2} \mu(x,v)  =  \frac{1}{12} \Big([\cL^*_B,[\cL^*_B,\cL^*_R+\cL^*_D]] + [\cL^*_D,[\cL^*_D,\cL^*_R]] + [\cL^*_B,[\cL^*_D,\cL^*_R]] + [\cL^*_D,[\cL^*_B,\cL^*_R]]  \\
    & \quad -\frac{1}{2} [\cL^*_D,[\cL^*_D,\cL^*_B]] -\frac{1}{2} [\cL^*_R,[\cL^*_R,\cL^*_B]] -\frac{1}{2} [\cL^*_R,[\cL^*_R,\cL^*_D]] \Big) \\
    & = \frac{\mu(x,v)}{12} \Bigg( \frac{3}{2} \lambda_1(x,R(x)v) \Big( \langle v,\nabla^2 \pot(x) v\rangle - \langle R(x)v,\nabla^2 \pot(x) R(x)v\rangle\Big)+ \frac{1}{2} \langle v,\nabla_x (\langle v,\nabla^2 \pot(x)v\rangle)\rangle\Bigg).
\end{align*}
Finally focus on \textbf{DRBRD}:
\begin{align*}
    & \cL^*_{2} \mu(x,v)  =  \frac{1}{12} \Big([\cL^*_B,[\cL^*_B,\cL^*_D+\cL^*_R]] + [\cL^*_R,[\cL^*_R,\cL^*_D]] + [\cL^*_B,[\cL^*_R,\cL^*_D]] + [\cL^*_R,[\cL^*_B,\cL^*_D]]   \\
    & \quad -\frac{1}{2} [\cL^*_R,[\cL^*_R,\cL^*_B]] -\frac{1}{2} [\cL^*_D,[\cL^*_D,\cL^*_B]] -\frac{1}{2} [\cL^*_D,[\cL^*_D,\cL^*_R]]\Big) \\
    & = \frac{\mu(x,v)}{12}  \Bigg( \frac{3}{2}\lambda_r \Big(b\,\text{tr}\big(\nabla \pot(x) \nabla \pot(x)^T - \nabla^2 \pot(x)\big) + \langle v,\nabla \pot(x) \rangle \big(3\lambda_1(x,R(x)v)+\lambda_1(x,v)\big) \\
    & \quad + \langle v,\nabla^2 \pot(x)v\rangle\Big)+ \frac{3}{2} \lambda_1(x,R(x)v)\Big( \langle v,\nabla^2 \pot(x) v\rangle - \langle R(x) v,\nabla^2 \pot(x) R(x) v\rangle\Big) \\
    & \quad +\frac{1}{2} \langle v,\nabla(\langle v,\nabla^2 \pot(x) v\rangle)\rangle + \frac{3}{2}\lambda_r^2 \langle v,\nabla \pot(x)\rangle \Bigg).
\end{align*}

\subsection{Application of Proposition \ref{prop:f_2_1D} to three one-dimensional targets}\label{sec:proofs_propositions_invmeas}

{In this section we give analytic expressions for $f_2$ for the four splitting schemes considered in \Cref{sec:onedimtargets_expinvmeas} for three different one-dimensional target distributions, together with various numerical simulations.}

{The heuristic argument of \Cref{sec:onedimtargets_expinvmeas} and the numerical simulations of Figure~\ref{fig:allschemes_stdgauss_1d} suggest that schemes having \textbf{DBD} as their limit as the refreshment rate goes to zero have a smaller bias in the $x$ component compared to those that converge to \textbf{BDB}.
Similarly to \Cref{sec:onedimtargets_expinvmeas}, here we focus on schemes \textbf{RDBDR}, \textbf{DBRBD}, \textbf{DRBRD}, as well as \textbf{BDRDB}. For these four schemes we computed $\cL_2^*$ in Proposition \ref{prop:expansion_adjoint_bps} and give the corresponding analytic expressions of $f_2$ for three one-dimensional targets: a standard normal distribution (see Proposition \ref{prop:f2_gauss1d}), the distribution corresponding to the potential $\pot(x)=x^4$ (see Proposition \ref{prop:f2_nonlipschitz}), and the Cauchy distribution (see Proposition \ref{prop:f2_cauchy}). 
The results, both according to the theory and numerical simulations, are shown in Figure~\ref{fig:bias_inv_measure}.}
\begin{figure}
\begin{subfigure}{0.49\textwidth}
  \centering
    \includegraphics[width=\textwidth]{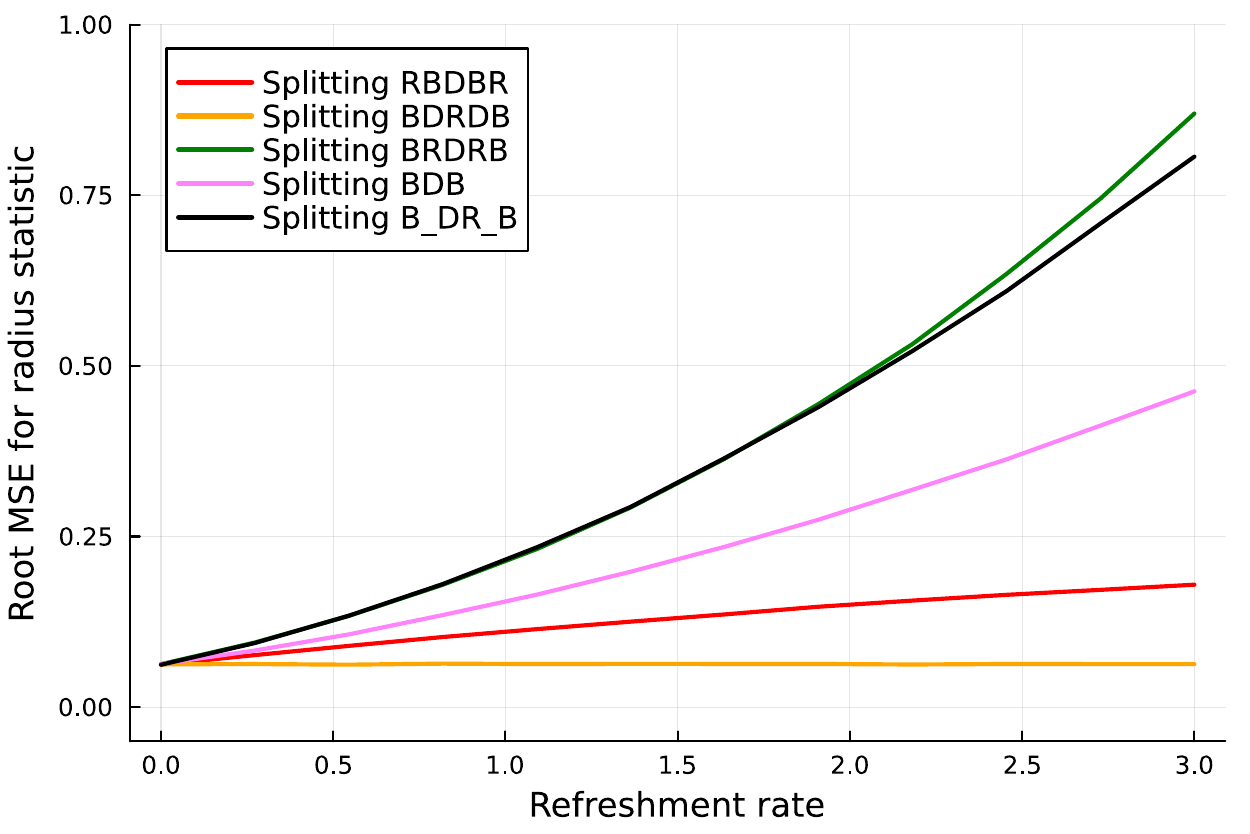}
\end{subfigure}
    \begin{subfigure}{0.49\textwidth}
  \centering
    \includegraphics[width=\textwidth]{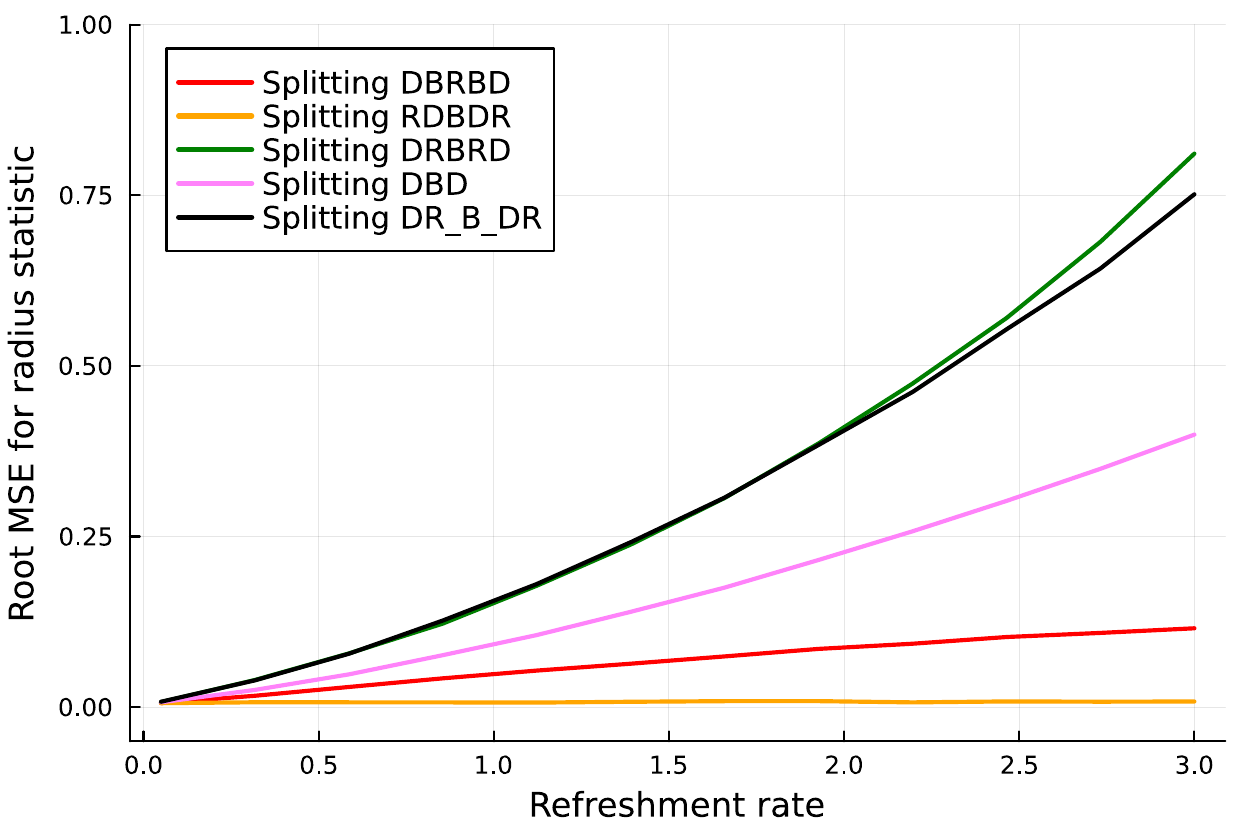}
\end{subfigure}
\caption{Square root of the MSE in the estimation of the radius statistic, $x^2$, with a one-dimensional standard Gaussian target. The step size is set to $\delta=0.5$, the number of iterations is $N=2\times 10^5$, and the experiment is repeated $300$ times. The schemes \textbf{BDB} (\emph{left}) and \textbf{DBD} (\emph{right}) correspond to including the refreshment part in \textbf{B}. In schemes \textbf{B\_DR\_B} (\emph{left}) and \textbf{DR\_B\_DR} (\emph{right}) we denote by \textbf{B} the standard bounce part, by \textbf{DR} the transition kernel which corresponds to having refreshments and deterministic motion together, and we use underscores to divide these two kernels. }
\label{fig:allschemes_stdgauss_1d}
\end{figure}

\begin{figure}[ht]
\begin{subfigure}[t]{0.45\textwidth}
    \includegraphics[width=\textwidth]{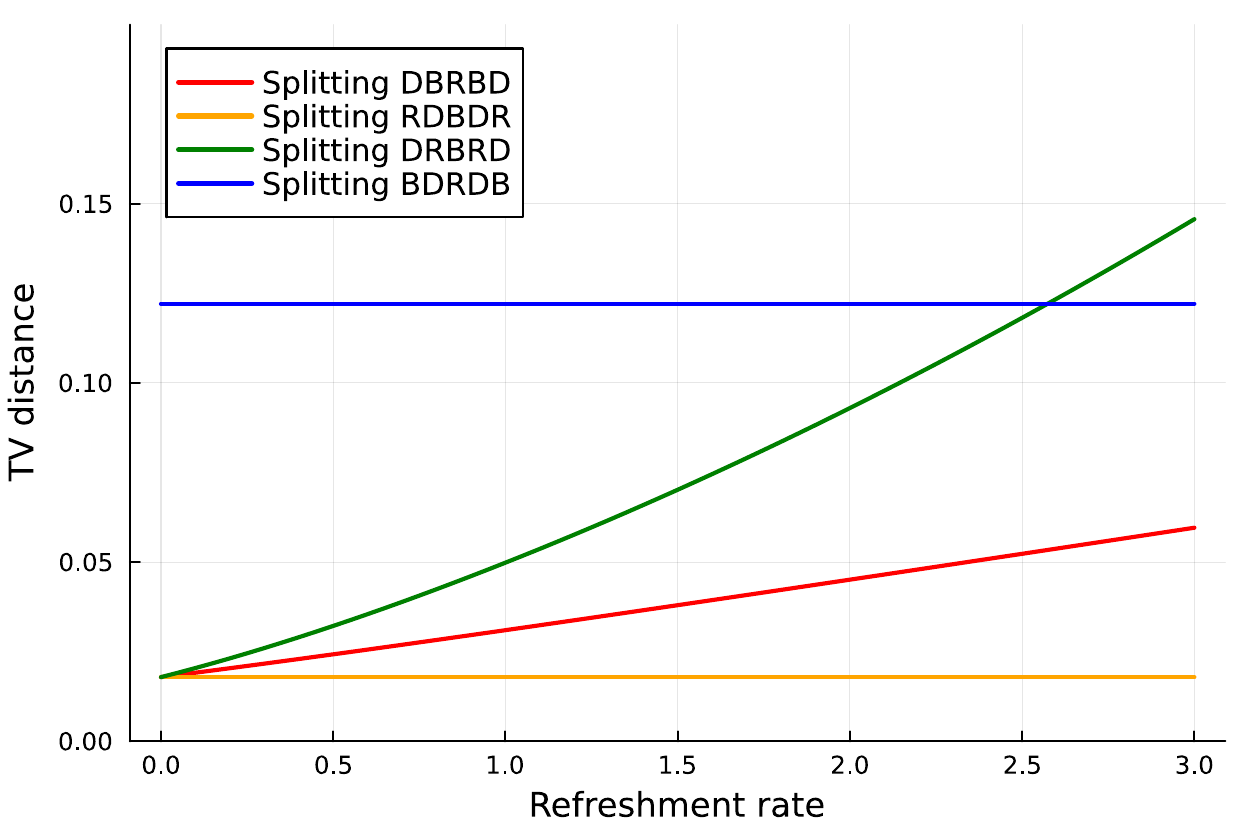}
    \label{fig:TVdist_nonlipschitz}
\end{subfigure}
\hfill
\begin{subfigure}[t]{0.45\textwidth}
    \includegraphics[width=\textwidth]{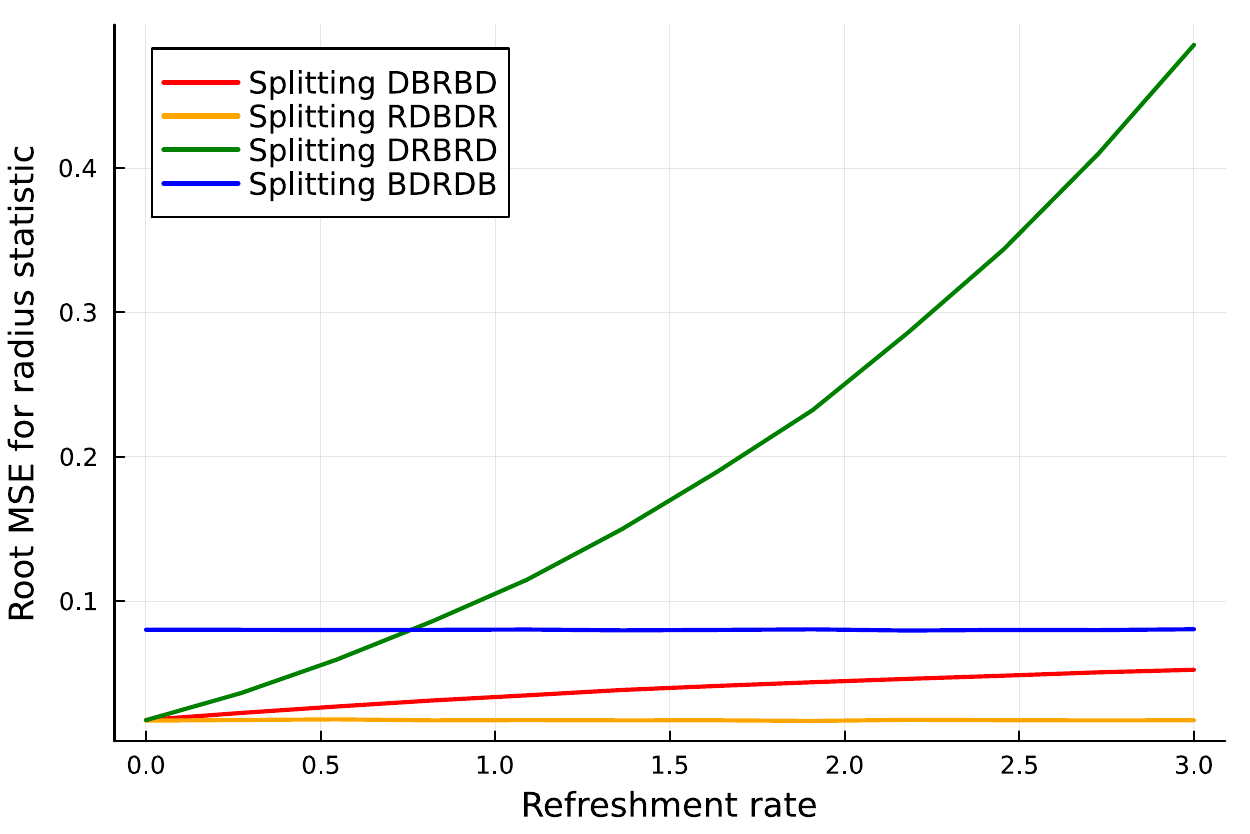}
    \label{fig:bps_empiricalerrror_1d_nonlipschitz}
\end{subfigure}
\caption{Results for the target distribution with potential $\pot(x)=x^4$. The \emph{left} plot shows the TV distance up to the second order term according to Proposition \ref{prop:f2_nonlipschitz}.  Here $\delta = 0.5$, the number of iterations is $N=2 \cdot 10^5$, and the experiment is repeated $50$ times.}
\begin{subfigure}[t]{0.45\textwidth}
    \includegraphics[width=\textwidth]{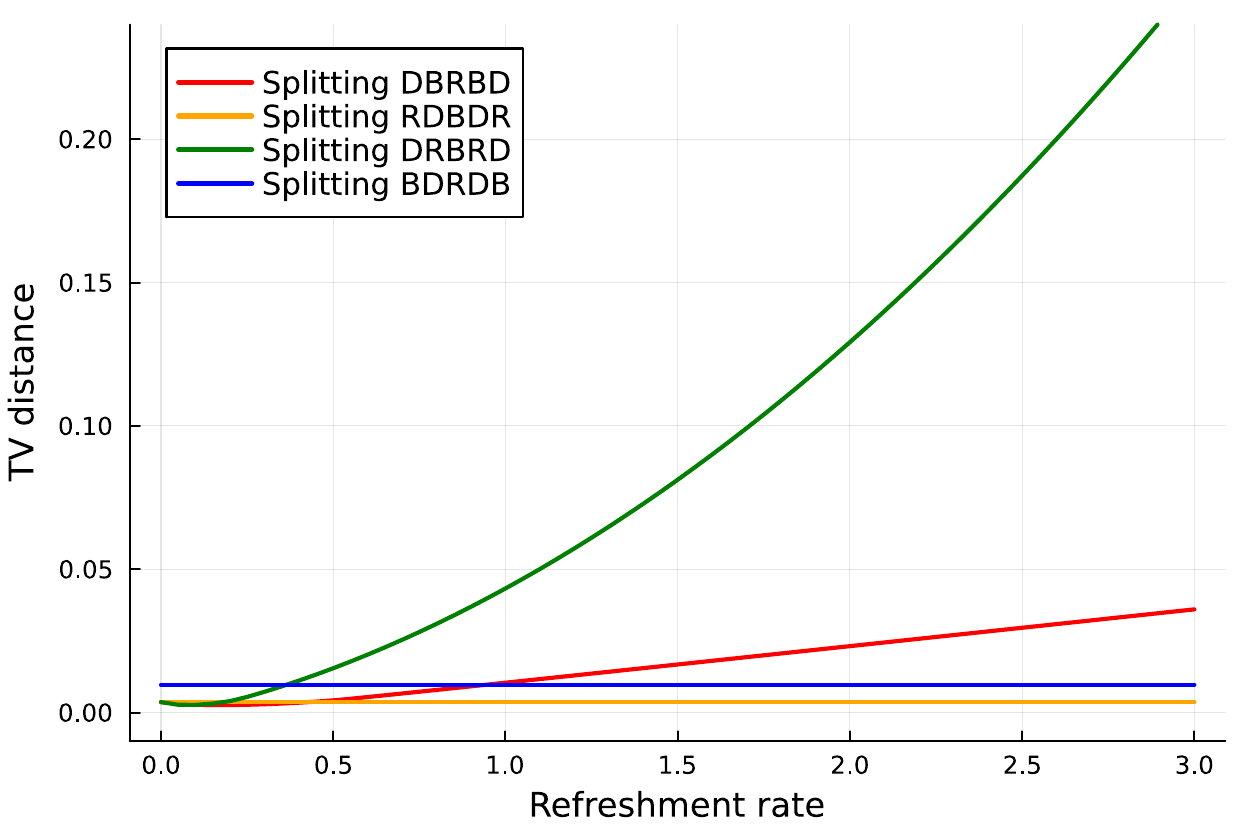}
    \label{fig:biasat0_cauchy}
\end{subfigure}
\hfill
\begin{subfigure}[t]{0.45\textwidth}
    \includegraphics[width=\textwidth]{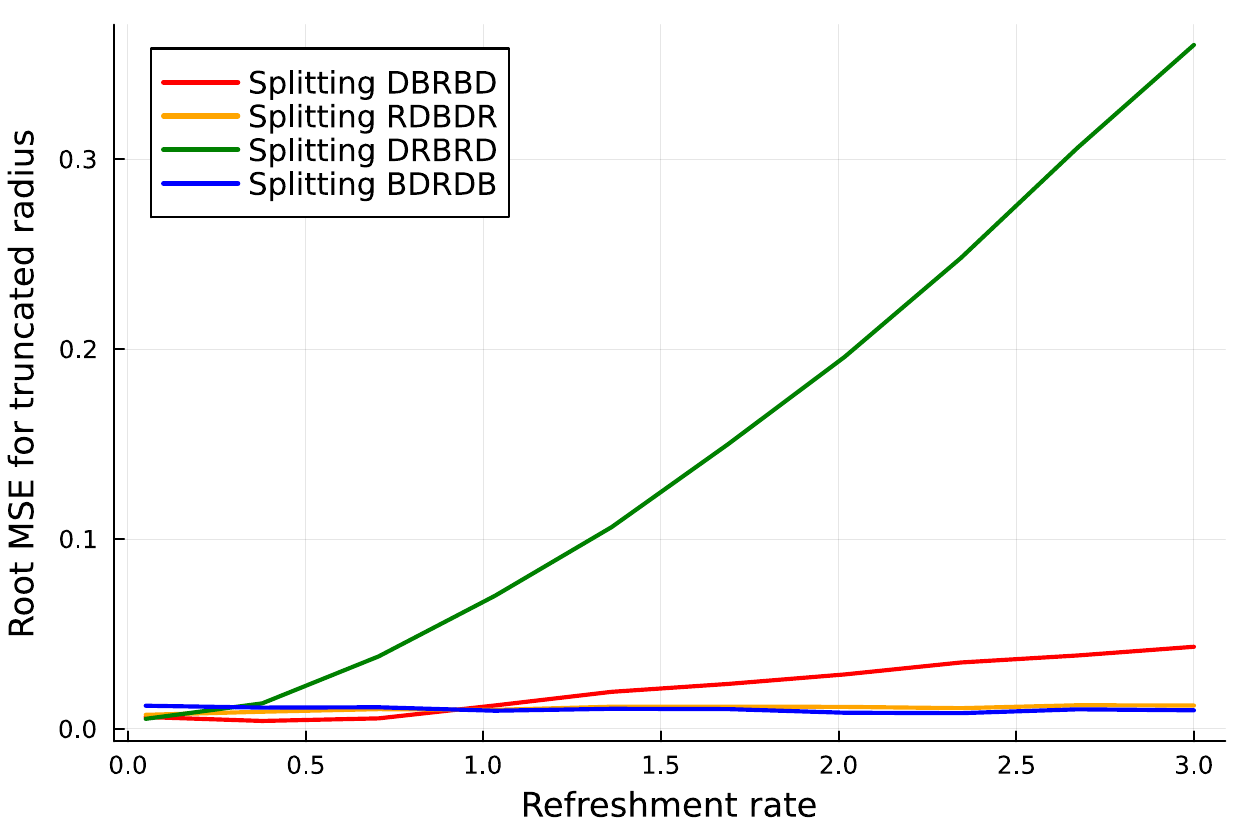}
    \label{fig:bps_empiricalerror_1d_cauchy}
\end{subfigure}
\caption{Results for a Cauchy target distribution with $\gamma=1$. \emph{Left}: TV distance up to the second order term according to Proposition \ref{prop:f2_cauchy}; \emph{right}: square root of the MSE for the estimation of the truncated radius statistic $2 \wedge x^2$. Here $\delta = 0.5$, the number of iterations is $N=4 \cdot 10^6$, and the experiment is repeated $200$ times. }
\label{fig:bias_inv_measure}
\end{figure}

\begin{figure}[ht]
\begin{subfigure}{0.49\textwidth}
  \centering
    \includegraphics[width=\textwidth]{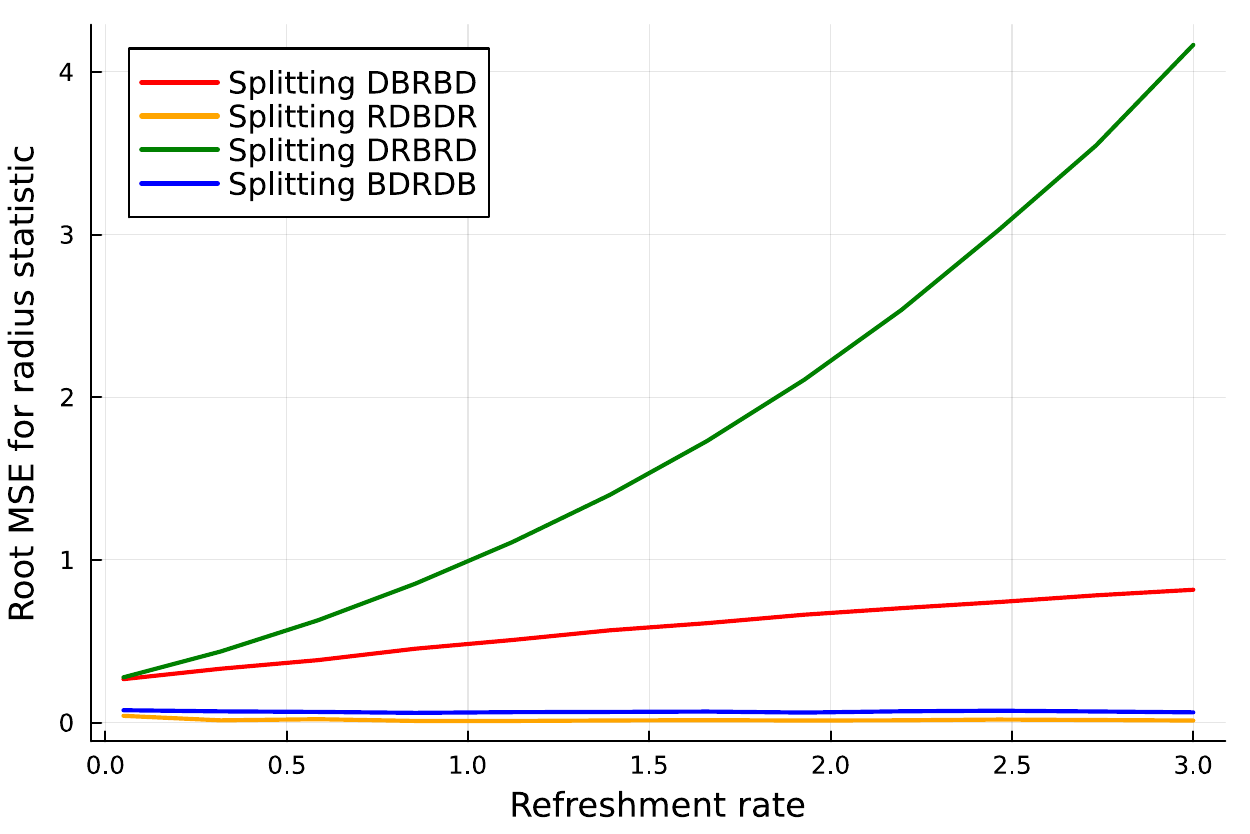}
    \caption{$\rho=0$.}
\end{subfigure}
\hfill
\begin{subfigure}{0.49\textwidth}
  \centering
    \includegraphics[width=\textwidth]{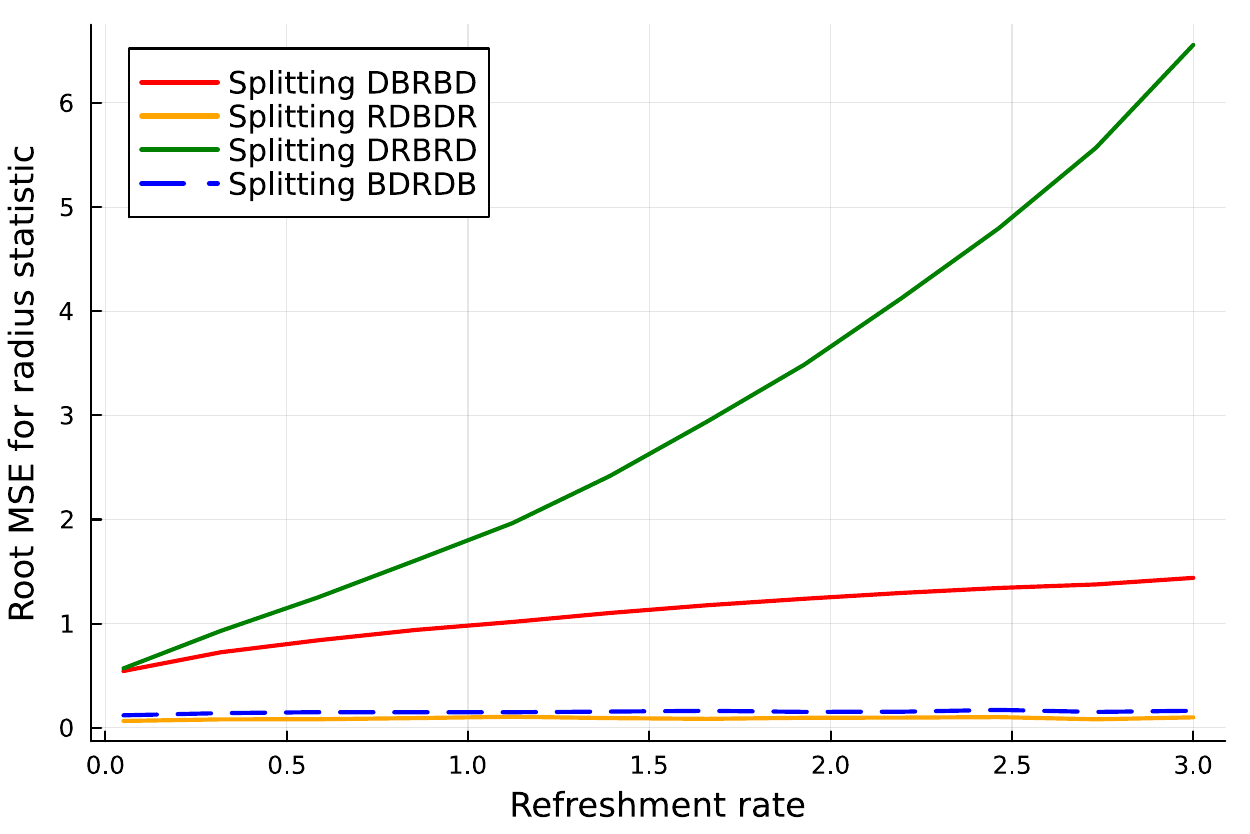}
    \caption{$\rho=0.7$.}
\end{subfigure}
\caption{Square root of the MSE for empirical estimates of the radius statistic for splitting schemes of BPS with a $5$-dimensional Gaussian target with covariance $\Sigma_{ii} =1$, $\Sigma_{ij}=\rho$ for $i\neq j$. The step size is $\delta = 0.5$ and the number of iterations is $N=2 \cdot 10^6$. The position vector is initialised with a draw from the target distribution and the velocity from a draw from the uniform distribution on the unit sphere.}
\label{fig:errinvmeas_multidim}
\end{figure}


{Let us comment on these results. First of all, the theoretical results and the numerical simulations of Figure \ref{fig:bias_inv_measure} are consistent, in the sense that they report the same behaviour although they consider two different metrics. In particular, the bias of schemes \textbf{RDBDR} and \textbf{BDRDB} appears to be independent of the refreshment rate, while \DBRBD and \DRBRD have respectively linear and quadratic dependence. In the one-dimensional case, the plots show that it is best to choose $\lambda_r=0$, which is possible as in this case BPS is irreducible. However, in higher dimensional settings it is necessary to take $\lambda_r >0$ and thus it is essential to use schemes that have good performance for most values of $\lambda_r$. From Figures \ref{fig:bias_inv_measure_gauss} and \ref{fig:bias_inv_measure} it is also clear that \RDBDR typically has the smallest bias out of all the considered splittings.
The experiments in Figure \ref{fig:errinvmeas_multidim} suggest that the findings of the one-dimensional case extend to multi-dimensional targets. In particular, \RDBDR has either a better performance than other splittings or behaves very similarly to \BDRDB both on an independent as well as a correlated Gaussian. }

\subsubsection{Gaussian target}
Let us start with a one-dimensional Gaussian target with mean zero and variance $\sigma^2>0$. 
\begin{proposition}\label{prop:f2_gauss1d}
Let $\pot(x)=x^2/(2\sigma^2)$ for $\sigma^2>0$. Then:
\begin{itemize}
\item For the splitting scheme \textbf{DBRBD} it holds that
\begin{align*}
    f_2(x,+1) =f_2(x,-1) = \frac{\lambda_r}{24} \left(\frac{2\sqrt{2}}{\sigma\sqrt{\pi}} -\frac{x^3}{\sigma^4} \sign(x) \right).
\end{align*}
\item For the splitting scheme \textbf{BDRDB} it holds that
\begin{align*}
    f_2(x,+1) & =\frac{1}{8\sigma^2}- \frac{1}{4\sigma^4} x^2 \mathbbm{1}_{x < 0},\\
    f_2(x,-1) & =\frac{1}{8\sigma^2}- \frac{1}{4\sigma^4} x^2 \mathbbm{1}_{x > 0}.
\end{align*}
\item For the splitting scheme \textbf{RDBDR} it holds that 
\begin{align*}
    f_2(x,+1) =f_2(x,-1) =0.
\end{align*}
\item For the splitting scheme \textbf{DRBRD} it holds that
\begin{equation}\notag
    f_2(x,+1) = f_2(x,-1)  =\frac{\lambda_r}{12} \left( \frac{2\sqrt{2}}{\sigma\sqrt{\pi}} - \frac{\lvert x\rvert^3}{\sigma^4}\right) +\frac{\lambda_r^2}{16}\left(1-\frac{x^2}{\sigma^2} \right) .
\end{equation}
\end{itemize}
\end{proposition}
\begin{figure}[t]
    \centering
    \begin{subfigure}{0.49\textwidth}
    \includegraphics[width=\textwidth]{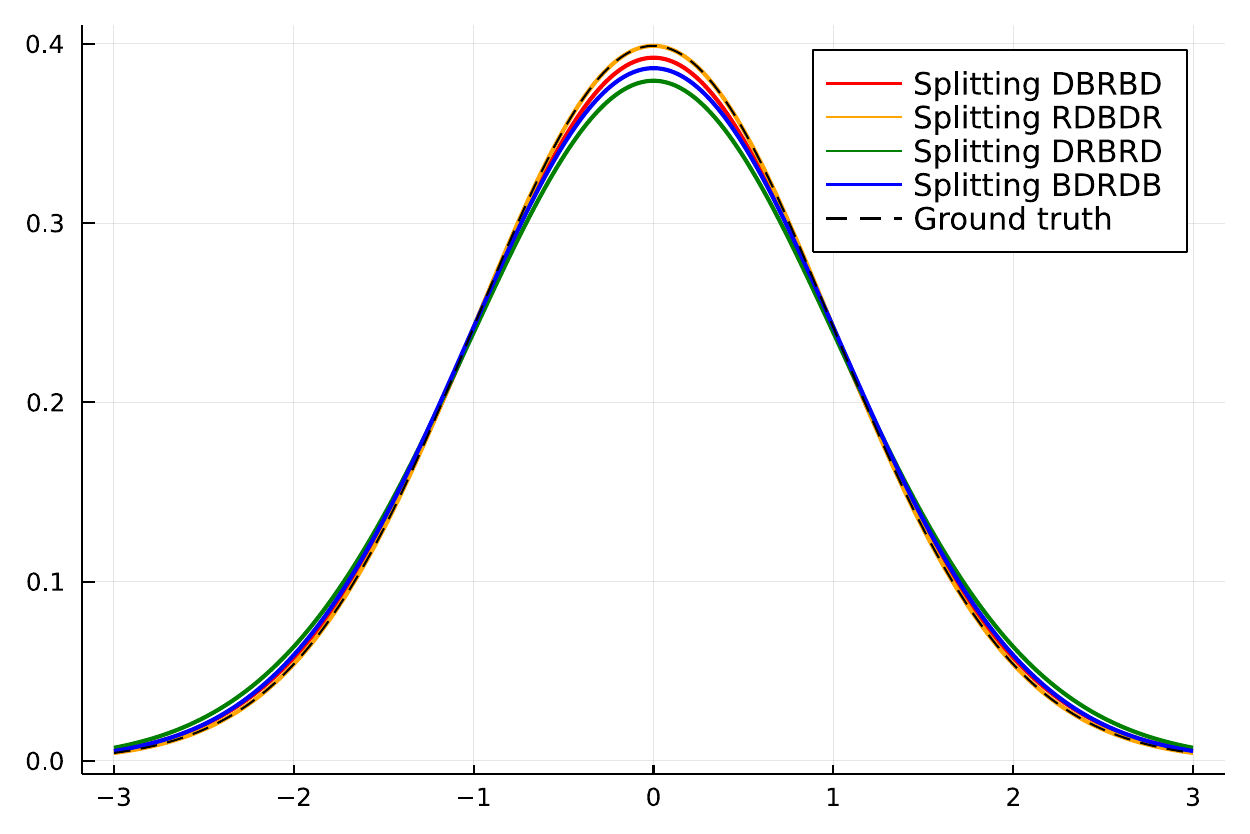}
    \caption{Refreshment rate $\lambda_r=1.0$.}
\end{subfigure}
\hfill
\begin{subfigure}{0.49\textwidth}
    \includegraphics[width=\textwidth]{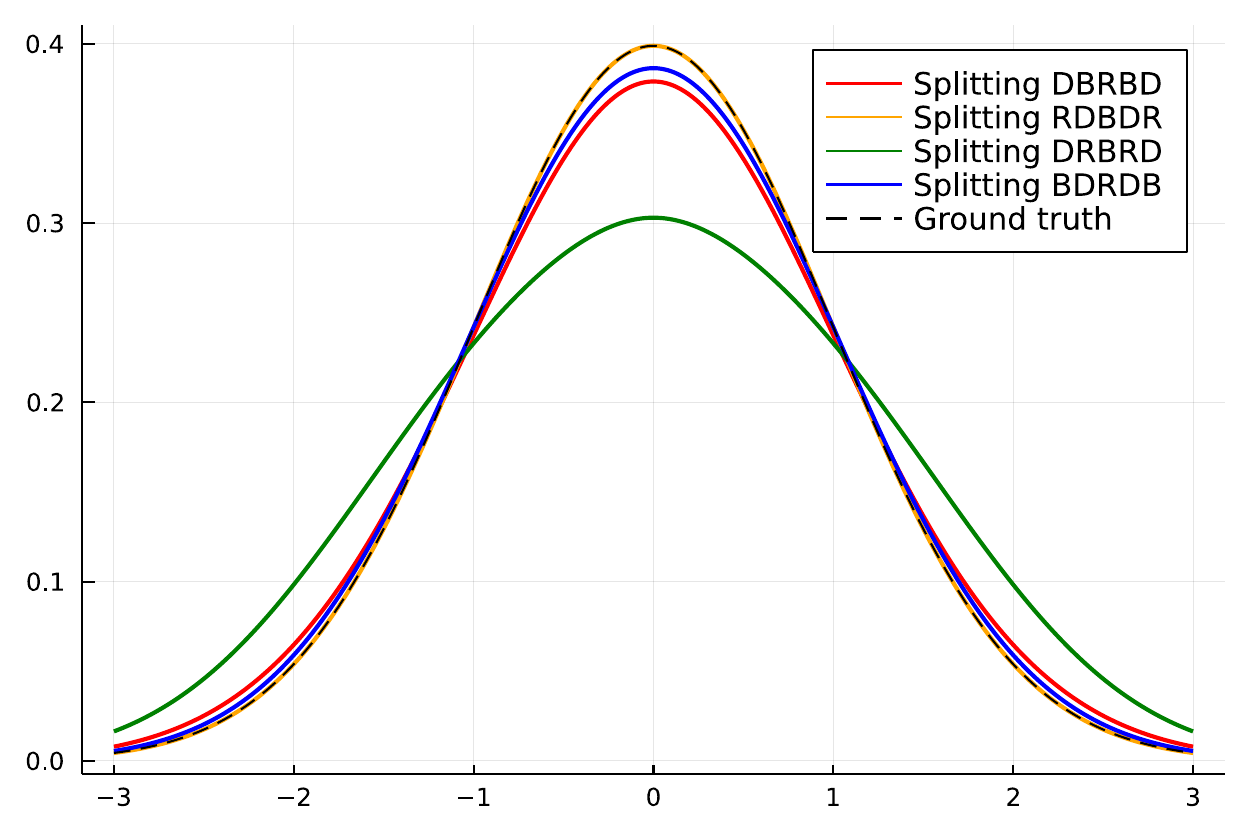}
    \caption{Refreshment rate $\lambda_r=3.0$.}
\end{subfigure}
\caption{Plots of the theoretical invariant measure up to second order for a standard Gaussian target, as given by Proposition \ref{prop:f2_gauss1d}. The step size is $\delta=0.5$.}
\label{fig:inv_meas_bps_1d}
\end{figure}

\subsubsection{Non-Lipschitz potential}
Now we focus on a target distribution with non-Lipschitz potential. 
\begin{proposition}\label{prop:f2_nonlipschitz}
Let $\pot(x)=x^4$. Then:
\begin{itemize}
\item For the splitting scheme \textbf{DBRBD} it holds that
\begin{align*}
    f_2(x,+1) =f_2(x,-1) = \frac{\lambda_r}{7}\left(\frac{1}{2 \Gamma(5/4)} - 2 x^7 \sign(x) \right) + \frac{1}{2}\left(\frac{\Gamma(3/4)}{\Gamma(1/4)} - x^2\right).
\end{align*}
\item For the splitting scheme \textbf{BDRDB} it holds that
\begin{align*}
    &f_2(x,+1) = \frac{5\Gamma(3/4)}{2\Gamma(1/4)} - x^2- 4x^6 \1_{x<0} \,, \\
    &f_2(x,-1) =\frac{5\Gamma(3/4)}{2\Gamma(1/4)} - x^2- 4x^6 \1_{x\geq 0} \,.
\end{align*}

\item For the splitting scheme \textbf{RDBDR} it holds that
\begin{align*}
    f_2(x,+1) = f_2(x,-1) =  \frac{\Gamma(3/4)}{2\Gamma(1/4)} - \frac{1}{2} x^2.
\end{align*}

\item For the splitting scheme \textbf{DRBRD} it holds that
\begin{align*}
    f_2(x,+1) =f_2(x,-1) &=  \frac{\lambda_r}{7}\left(\frac{1}{\Gamma(5/4)} - 4 x^7 \sign(x) \right) + \frac{1}{2}\left(\frac{\Gamma(3/4)}{\Gamma(1/4)} - x^2\right)  \\
    & \quad+ \frac{\lambda_r^2}{8}\left( \frac{1}{4} - x^4 \right).
\end{align*}
\end{itemize}
\end{proposition}

\subsubsection{Heavy tailed target}
Finally we consider a Cauchy distribution $\pi(x)= \gamma /(\pi (\gamma^2+x^2))$ for $\gamma>0$. 
\begin{proposition}\label{prop:f2_cauchy}
Let  $\pot(x)=\ln(\gamma^2+x^2).$  Then:
\begin{itemize}
\item For the splitting scheme \textbf{DBRBD} it holds that
\begin{align*}
    f_2(x,+1) = f_2(x,-1) &= \frac{\lambda_r}{4\gamma}\left( \frac{\pi}{4} - \lvert\arctan(x/\gamma)\rvert +\frac{\gamma \lvert x\rvert }{\gamma^2+x^2}-\frac{1}{\pi}  \right) 
    \\
    &  \quad 
    +\frac{1}{12} \left( \frac{1}{4\gamma^2} +  \frac{x^2-\gamma^2}{(\gamma^2+x^2)^2} \right).
\end{align*}
\item For the splitting scheme \textbf{BDRDB} it holds that
\begin{align*}
    f_2(x,v) = \left(\frac{(x^2-3\gamma^2)^2}{48\gamma^2(x^2+\gamma^2)^2}\right)\1_{xv<0}+\left(\frac{x^4-54x^2\gamma^2+9\gamma^4}{48\gamma^2(x^2+\gamma^2)^2}\right)\1_{xv\geq 0}.
\end{align*}

\item For the splitting scheme \textbf{RDBDR} it holds that
\begin{align*}
    f_2(x,+1) =f_2(x,-1) = \frac{1}{12} \left( \frac{1}{4\gamma^2} +  \frac{x^2-\gamma^2}{(\gamma^2+x^2)^2} \right).
\end{align*}

\item For the splitting scheme \textbf{DRBRD} it holds that
\begin{align*}
    f_2(x,+1) = f_2(x,-1)&=\frac{\lambda_r}{2\gamma}\left( \frac{\pi}{4} - \lvert\arctan(x/\gamma)\rvert +\frac{\gamma \lvert x\rvert }{\gamma^2+x^2}-\frac{1}{\pi}  \right)\\
    & \quad +\frac{1}{12} \left( \frac{1}{4\gamma^2} +  \frac{x^2-\gamma^2}{(\gamma^2+x^2)^2} \right) +\frac{\lambda_r^2}{8}\left(\ln 4-\ln\left(1+\frac{x^2}{\gamma^2}\right)\right).
\end{align*}
\end{itemize}
\end{proposition}

\subsection{Proof of Proposition \ref{prop:mu_delta1D}}\label{subsec:proof_mu_delta1D}
Fix $x\in\R$, $\delta>0$ and let $G(x,\delta):=\{(z,v)\in\R\times\{\pm 1\}:(z-x)/\delta\in\mathbb Z\}$ be the state space of the chain with initial position $x$. For now, let $\mu_\delta$ be any probability measure on $G(x,\delta)$ such that $\mu_\delta(y,w)=\mu_\delta(y,-w)$ for all $(y,w)\in G(x,\delta)$, and let us give a sufficient and necessary condition for it to be invariant by the chain.
 Since such a $\mu_\delta$ is invariant by the refreshment step,  it is invariant for the scheme \textbf{RDBDR} if and only if it is invariant for the scheme \textbf{R'DBD}, where \textbf{R'} is a deterministic flip of the velocity (which, as \textbf{R}, preserves $\mu_\delta$). Besides, from a state $(y,w)\in G(x,\delta)$, one transition of \textbf{R'DBD} can only lead to $(y,w)$ or $(y+\delta w,-w)$, from which it can only stay or come back to the initial $(y,w)$. In other words the pair $\{(y,w),(y+\delta w,-w)\}$ is irreducible for this chain, and thus $\mu_\delta$ is invariant for \textbf{R'DBD} if and only if its restrictions on all these sets for $(y,w)\in G(x,\delta)$ are invariant by this scheme, which by definition  reads
 \begin{align*}
 \forall (y,w)\in G(x,\delta)\,,\qquad
       \mu_\delta(y,w) e^{-\delta\lambda(y+w\delta/2,w)} = \mu_\delta(y+\delta w,-w) e^{-\delta\lambda(y+w\delta/2,-w)}.
   \end{align*}
   It turns out that this is exactly the  skew detailed balance condition \eqref{eq:skewDB} for the scheme \textbf{DBD}.
  Writing that $\mu_\delta(y,w) \propto \exp(-\pot_\delta(y))$ for some $\pot_\delta$ and  recalling that $\lambda(y,w)- \lambda(y,-w)=w \psi'(y)$ for all $y,w$, this is equivalent to
   $$\forall (y,w)\in G(x,\delta),\qquad \pot_\delta(y+\delta w) - \pot_\delta(y)= \delta w \pot'\left(y+\delta w/2\right)\,. $$
Up to an additive constant, the only function $\psi_\delta$ which satisfies this is the one given in the statement of Proposition~\ref{prop:mu_delta1D}. As a conclusion, we have proven that a probability measure on $G(x,\delta)$ which is independent from the velocity is invariant for the scheme \textbf{RDBDR} if and only if it is the one given in the proposition, which concludes the proof of the first statement.


Now we focus on the convergence of empirical means, assuming that the conditions of Theorem~\ref{thm:ergodicity_zzs} are met.  The reference position $x\in\R$ is still fixed.  The long-time convergence established in 
Theorem~\ref{thm:ergodicity_zzs} (for $P_\delta^2$ where $P_\delta$ is one step of the scheme) is well-known to imply an ergodic Law of Large Numbers. In particular,  for all initial conditions in $G(x,\delta)$ and all bounded $f$, distinguishing odd and even indexes, we see that $\frac1N \sum_{k=1}^N f(\overline{Z}_{t_k})$ (where $(\overline{Z}_{t_k})_{k\in\N}$ is a trajectory of the scheme) converges almost surely as $N\rightarrow\infty$ to $\tilde \mu_\delta(f) := (\mu_\delta'(f) + \mu_\delta''(f))/2$, where $\mu_\delta'$ and $\mu_\delta''$ are the unique invariant measures of $P^2_\delta$ on each periodic component of the state space. In particular, $\tilde \mu_\delta$ is an invariant measure for $P_\delta $.  In dimension 1, the scheme \textbf{DBD} is such that for all $y$, for all times, the number of visits of the points $(y,1)$ and $(y,-1)$ differ at most by 1, which implies by ergodicity that $\tilde \mu_\delta(y,w) = \mu_\delta(y,-w) $ for all $(y,w)\in G(x,\delta)$, and we conclude thanks to the first part of the proof.

\section{Proofs of Section \ref{sec:scaling_rej_prob}}\label{sec:proofs_rejectionprob}
\subsection{Proof of Proposition \ref{prop:meanacceptrate_zzs}}
    Recall the expression for the acceptance rate \eqref{eq:MH_prob_zzs} for given initial state $(x,v)$ and proposed state $(\tilde X,R_I v)$, as well as our notation $x_{1/2}=x+v\delta/2$. Let us rewrite the exponent in the acceptance probability using Taylor's theorem. Expanding  $\pot(\tilde{X})$ as well as $\partial_i \pot(x_{1/2})$ around $x$ we find
    \begin{equation}\label{eq:decom_taylor_proof_zzs}
        \pot(x)-\pot(\tilde{X})  + \delta \sum_{i\notin I}v_i\partial_i \pot(x_{1/2}) = \frac{\delta^2}{2}  \sum_{i\notin I} v_i\sum_{j\in I}\partial_{ij}\pot(x)v_j  + \delta^3  H(x,v) + \mathcal{O}(\delta^4),
    \end{equation}
    where 
    \begin{equation*}
        H(x,v) = \frac{1}{8} \sum_{i\notin I} v_i \langle v,\nabla^2 \partial_i \pot(x) v\rangle - \frac{1}{48} \sum_{\alpha:\lvert \alpha\rvert =3}\!\!\! D^\alpha\pot(x) (v+ R_Iv)^\alpha.
    \end{equation*}
    By $\mathcal{O}(\delta^4)$ we denote a remainder term that depends on fourth derivatives of $\pot$, is uniform in $\delta\in[0,\delta_0]$, and under our assumptions increases at most polynomially in $x$.

We shall compute the probability of rejecting the proposed state by partitioning the state on whether there are no events, or one or more components of the velocity are flipped. Consider first the case of no events. In this scenario the second order term in \eqref{eq:decom_taylor_proof_zzs} equals $0$ since $I$ is the empty set. Recalling that $ x_+ - x_+^2/2 \leq 1-1\wedge e^{-x} \leq x_+$ we find
    \begin{align*}
        \E[(1-\alpha((x,v),(\tilde X,\tilde V)))\1_{\text{no events}}] &= \max\left(0,\exp(-\delta\lambda(x_{1/2},v)) \delta^3 H(x,v)\right)  + \mathcal{O}(\delta^4)\\
        & = \delta^3 \max\left(0,-\frac{1}{24} \sum_{\alpha:\lvert \alpha\rvert =3} D^\alpha\pot(x) v^\alpha\right) + \mathcal{O}(\delta^4).
    \end{align*}
    On the complementary event we find
    \begin{align*}
        &\E[(1-\alpha((x,v),(\tilde X,\tilde V)))\1_{\geq \text{ 1 events}}] \\
        & =\frac{\delta^2}{2} \max\left(0,\sum_{I:1\leq\lvert I\rvert\leq d } \sum_{i\notin I} v_i\sum_{j\in I}\partial_{ij}\pot(x)v_j \prod_{i\in I}\left(1-\exp(\-\delta\lambda_i(x_{1/2},v)) \right) \prod_{j\notin I}\exp(\-\delta\lambda_j(x_{1/2},v))\right) \\
        & \quad +\mathcal{O}(\delta^4)\\
        & = \frac{\delta^3}{2} \max\left(0,\sum_{i=1}^d \lambda_i(x_{1/2},v) v_i \sum_{k\neq i}  v_k \partial_{ik}\pot(x) \right) +\mathcal{O}(\delta^4).
    \end{align*}
    In the last line we focused on the terms corresponding to only one component of the velocity being flipped, as these are the only situations with leading order, and we expanded the exponential terms. The term $\lambda_i(x_{1/2},v)$ can be substituted by the term $\lambda_i(x,v)$ because the function $\max(0,a)$ is Lipschitz and $\pot$ is smooth.
    We obtain the statement in the proposition by summing the case of no events to that of one or more events.

\subsection{Proof of Proposition \ref{prop:meanacceptrate_bps}}
    The proof is similar to that of Proposition \ref{prop:meanacceptrate_zzs}, so we only give the main steps. Observe that we only need to focus on the \textbf{DBD} part of the splitting scheme, as refreshments do not affect the rejection probability.
    Similarly to Proposition \ref{prop:meanacceptrate_zzs} we distinguish two cases based on whether a jump happens or not. 
    In the case in which no rejections take place we have that the proposal is $\tilde X=x+v\delta$. We rewrite the exponent in \eqref{eq:MH_bps_2} replacing $\pot(\tilde X)$ and $\nabla \pot(x_{1/2})$ with their Taylor expansions around $x$. This gives
    \begin{align*}
        \pot(x) - \pot(\Tilde X) + \delta \langle v,\nabla \pot(x_{1/2})\rangle = -\frac{1}{24} \delta^3 \sum_{\alpha:\lvert \alpha\rvert=3}D^\alpha\pot(x) v^\alpha + \mathcal{O}(\delta^4).
    \end{align*}
    Following a similar reasoning of Proposition \ref{prop:meanacceptrate_zzs}, we find that in this case
    \begin{align*}
        \E[(1-\alpha((x,v),(\tilde X,\tilde V)))\1_{\text{no reflection}}] &= \frac{\delta^3}{24} \max\left( 0,-\sum_{\alpha:\lvert \alpha\rvert =3} D^\alpha\pot(x) v^\alpha\right) + \mathcal{O}(\delta^4).
    \end{align*}
    On the complementary event we shall use the following Taylor expansion: 
    \begin{align*}
        &\pot(x)-\pot(x +(v+R(x_{1/2})v)\delta/2) = - \frac \delta 2 \langle v+R(x_{1/2})v,\nabla \pot(x_{1/2})\rangle \\
        & \quad + \frac{\delta^2}{8} \left(  \langle v,\nabla^2 \pot(x_{1/2}) v\rangle - \langle R(x_{1/2})v,\nabla^2 \pot(x_{1/2}) R(x_{1/2})v\rangle  \right) + \mathcal{O}(\delta^3).
    \end{align*}
    Notice that the first order term equals zero by definition of the operator $R$.
    Hence on the event that a reflection takes place we find
    \begin{align*}
        &\E[(1-\alpha((x,v),(\tilde X,\tilde V)))\1_{\text{reflection}}] \\
        &\quad =\frac{\delta^3}{8} \max\left( 0, \lambda(x_{1/2},v)(\langle v,\nabla^2 \pot(x_{1/2}) v\rangle - \langle R(x_{1/2})v,\nabla^2 \pot(x_{1/2}) R(x_{1/2})v\rangle)\right) + \mathcal{O}(\delta^4).
    \end{align*}
    It remains to show that we can replace $x_{1/2}$ by $x$ which follows if the $\delta^3$ term is continuous in $x_{1/2}$. Using the definition of $\lambda$ and $R$ we find that 
        \begin{align*}
            &F(y):=\lambda(y,v)\left(\langle v,\nabla^2 \pot(y) v\rangle -\langle R(y)v,\nabla^2 \pot(y) R(y)v\rangle\right)  \\
            & = 2\max\left(0,\left\langle \frac{\nabla\pot(y)}{\lvert \nabla\pot(y)\rvert^{\frac{2}{3}}},v\right\rangle\right)\left\langle v, \frac{\nabla\pot(y)}{\lvert \nabla\pot(y)\rvert^{\frac{2}{3}}}\right\rangle\left\langle v,\nabla^2 \pot(y) \frac{\nabla\pot(y)}{\lvert \nabla\pot(y)\rvert^{\frac{2}{3}}}\right\rangle \\
            &\quad - \max\left(0,\left\langle \frac{\nabla\pot(y)}{\lvert \nabla\pot(y)\rvert^{\frac{4}{5}}},v\right\rangle\right)\left\langle v, \frac{\nabla\pot(y)}{\lvert \nabla\pot(y)\rvert^{\frac{4}{5}}}\right\rangle^2 \left\langle\nabla\pot(y),\nabla^2 \pot(y) \frac{\nabla\pot(y)}{\lvert \nabla\pot(y)\rvert^{\frac{4}{5}}}\right\rangle.
        \end{align*}
        Since $\psi$ is smooth and $z/|z|^{4/5}$ is H\"older continuous with exponent $1/5$ the above function is the composition of smooth and H\"older continuous functions, and hence the difference $F(x_{1/2})-F(x)$ converges to zero as $\delta$ tends to $0$.

\subsection{Bounds on the rejection probability for log-concave targets}
Consider a log-concave target distribution for which the potential is gradient $L$-Lipschitz and the Hessian satisfies $m I_d\preceq \nabla^2\pot(x)\preceq MI_d$, where $I_d$ is the $d$-dimensional identity matrix. We suppose $G_1$ is the leading term in $G$, which is for instance the case when $\partial_{ijk}\pot$ is non-zero only for a small number of indices, and obtain a bound as follows.
\begin{example}[Algorithm \ref{alg:Metropolis_DBD_ZZS}]\label{ex:logconcave_zzs}
    Observe that
    \begin{align*}
        \sum_{i=1}^d \lambda_i(x,v) v_i \sum_{k\neq i}  v_k \partial_{ik}\pot(x) & = \sum_{i=1}^d \lambda_i(x,v) \left(\langle v,\nabla^2\pot(x)v\rangle - \langle R_iv,\nabla^2\pot(x)R_iv\rangle\right).
    \end{align*}
    Using the bounds $\langle v,\nabla^2\pot(x)v\rangle - \langle R_iv,\nabla^2\pot(x)R_iv\rangle \leq d(M-m)$ and $\max(0,a)\leq \lvert a\rvert$ twice we find
    \begin{align*}
        \mathbb{E}_\mu[G_1(x,v)] \leq \frac12 d(M-m) \sum_{i=1}^d \mathbb{E}_\pi\left[\lvert\partial_i\pot(x)\rvert \right]\leq \frac12 d^2 \sqrt{L} (M-m).
    \end{align*}
    In the second inequality we applied Jensen's inequality to \citet[Lemma 2]{Dalalyan2017FurtherAS}, which gives $\mathbb{E}_\pi\left[\lvert\partial_i\pot(x)\rvert \right]\leq \sqrt L.$
    Alternatively, we can assume that for all $v\in\{\pm 1\}^d$ it holds $\lvert v_i \sum_{k\neq i} v_k \partial_{ik}\pot(x) \rvert \leq M$ for all $i=1,\dots,d$, where $M$ is independent of $d$. This holds e.g. for a correlated Gaussian distribution. With this assumption we find with similar computations as above
    $\mathbb{E}_\mu[G_1(x,v)] \leq  \frac12 d \sqrt{L} M.$
\end{example}

\begin{example}[Algorithm \ref{alg:Metropolis_RDBDR_BPS}]\label{ex:logconcave_bps}
    In order to have a fair comparison between BPS and ZZS, we consider BPS with Gaussian velocity. This choice ensures that the Euclidean norm of the velocity vectors of the two samplers are equal (on average). Similarly to Example \ref{ex:logconcave_zzs}, under the assumption that $\nabla\pot$ is $L$-Lipschitz and $m I_d\preceq \nabla^2\pot(x)\preceq MI_d$ it is not hard to obtain the bound 
    $\mathbb{E}_\mu [G_1(x,v)] \lesssim \frac{\sqrt{L}(M-m)}{8}d^2.$

\end{example}

\section{Pseudo-codes for Section \ref{subsec:example_particles}}\label{sec:pseudocodes_particles}
Here we give the pseudo-codes for the jump parts of ZZS and BPS considered in Section~\ref{subsec:example_particles}. These are respectively Algorithm \ref{alg:zzs_particles} and Algorithm \ref{alg:bps_particles}. Both algorithms take as input the gradients $\nabla \pot_1$ and $\nabla\pot_2$, which are defined as follows:
\begin{align*}
    &\nabla\pot_1(x) = \nabla_x \left(\sum_{i=1}^{N-1} V(x_i-x_{i+1}) \right), \\
    &\nabla\pot_2(x,j) = \nabla_x \left( \sum_{i=1}^N W(x_i-x_j)\right).
\end{align*}
The pseudo-code for HMC is given in Algorithm \ref{alg:HMC_particles}.

\begin{algorithm}
\DontPrintSemicolon
\caption{Part \textbf{B} for the ZZS considered in Section \ref{subsec:example_particles}}
\label{alg:zzs_particles}
\KwIn{Initial condition $(x,v)$, step size $\delta$, gradients $\na\pot_1$, $\na\pot_2$.}
\KwOut{Updated velocity vector $v$.}
\For{$i \gets 1$ \KwTo $N$}{
    $t \gets 0$\;
    $\tau_1 \sim \Exp((v_i\partial_i\pot_1(x))_+)$\;
    $\tau_2 \sim \Exp(1)$\;
    \While{$\min(\tau_1, \tau_2) \leq \delta - t$}{
        \If{$\tau_1 < \tau_2$}{
            $t \gets t + \tau_1$\;
            $v_i \gets -v_i$\;
            $\tau_2 \gets \tau_2 - \tau_1$\;
            $\tau_1 \gets \infty$\;
        }
        \Else{
            $t \gets t + \tau_2$\;
            $J \sim \Unif(\{1,\dots,N\})$\;
            $U\sim\Unif[0,1]$\;
            \If{$U \leq  (v_i\partial_i \pot_2(x,J))_+$}{
                $v_i \gets -v_i$\;
                $\tau_1 \sim \Exp((v_i\partial_i\pot_1(x))_+)$\;
            }
            \Else{
                $\tau_1 \gets \tau_1 - \tau_2$\;
            }
            $\tau_2 \sim \Exp(1)$\;
        }
    }
}
\Return $v$\;
\end{algorithm}

\begin{algorithm}
\DontPrintSemicolon
\caption{Part \textbf{B} for the BPS considered in Section \ref{subsec:example_particles}}
\label{alg:bps_particles}
\KwIn{Initial condition $(x,v)$, step size $\delta$, gradients $\na\pot_1$, $\na\pot_2$.}
\KwOut{Updated velocity vector $v$.}
$t \gets 0$\;
$\beta \gets \sqrt{N} \lvert v\rvert $\;
$\tau_1 \sim \Exp(\langle \na\pot_1(x), v \rangle_+)$\;
$\tau_2 \sim \Exp(\beta)$\;
\While{$\min(\tau_1, \tau_2) \leq \delta - t$}{
        \If{$\tau_1 < \tau_2$}{
            $t \gets t + \tau_1$\;
            $v \gets v - 2 \frac{\langle v, \na\pot_1(x)\rangle }{\lvert\na\pot_1(x)\rvert^2} \na\pot_1(x)$\;
            $\tau_2 \gets \tau_2 - \tau_1$\;
            $\tau_1 \gets \infty$\;
        }
        \Else{
            $t \gets t + \tau_2$\;
            $J \sim \Unif(\{1,\dots,N\})$\;
            $U\sim\Unif[0,1]$\;
            \If{$ U \leq \frac{\langle  v,\na \pot_2(x,J)\rangle_+}{\beta}$}{
                $v \gets v - 2 \frac{\langle  v,\na \pot_2(x,J)\rangle}{\lvert\na\pot_2(x,J)\rvert^2} \na\pot_2(x,J)$\;
                $\tau_1 \sim \Exp(\langle \na\pot_1(x), v \rangle_+)$\;
            }
            \Else{
                $\tau_1 \gets \tau_1 - \tau_2$\;
            }
            $\tau_2 \sim \Exp(\beta)$\;
        }
}
\Return $v$\;
\end{algorithm}

\begin{algorithm}
\DontPrintSemicolon
\caption{HMC algorithm considered in Section \ref{subsec:example_particles}}
\label{alg:HMC_particles}
\KwIn{Initial condition $x$, step size $\delta$, number of iterations $n_{\text{iter}}$, parameters $M,K$, gradients $\na\pot_1$, $\na\pot_2$.}
\KwOut{Markov chain $(X_n)_{n=1}^{n_{\text{iter}}}$.}
\For{$j \gets 1$ \KwTo $n_{\text{iter}}$}{
    Refresh $v$ from the standard Laplace distribution\;
    \For{$k \gets 1$ \KwTo $M$}{
        $J \sim \Unif(\{1,\dots,N\})$\;
        $v \gets v - \frac12 \delta \na\pot_2(x,J)$\;
        \For{$l \gets 1$ \KwTo $K$}{
            $v \gets v - \frac{\delta}{2K} \delta \na\pot_1(x)$\;
            $x \gets x + \frac{\delta}{K} \cdot \text{sign}(v)$\;
            $v \gets v - \frac{\delta}{2K} \na\pot_1(x)$\;
        }
        $J \sim \Unif(\{1,\dots,N\})$\;
        $v \gets v - \frac12 \delta \na\pot_2(x,J)$\;
    }
    $X_j\gets x$\;
}
\Return $(X_n)_{n=1}^{n_{\text{iter}}}$\;
\end{algorithm}

\newpage
\bibliography{23-0036}

\end{document}